\documentclass[11pt,reqno]{amsart}
\usepackage[letterpaper,margin=1in,footskip=0.25in]{geometry}
\usepackage{microtype}
\usepackage{amssymb}
\usepackage{mathrsfs}
\usepackage[only,llbracket,rrbracket]{stmaryrd}
\usepackage{mathtools}
\usepackage{enumitem}
\newlength{\enumerateparindent}
\usepackage{multirow,bigdelim}
\usepackage[noadjust]{cite}
\renewcommand{\citepunct}{;\ }

\usepackage{trimclip}
\usepackage{tikz-cd}

\PassOptionsToPackage{pdfusetitle,colorlinks,pagebackref,bookmarksdepth=3,linktoc=all}{hyperref}
\usepackage{bookmark}
\hypersetup{citecolor=[HTML]{2E7E2A},linkcolor=[HTML]{800006},urlcolor=[HTML]{8A0087}}
\usepackage[hyphenbreaks]{breakurl}

\newtheorem{theorem}{Theorem}[section]
\newtheorem{corollary}[theorem]{Corollary}
\newtheorem{lemma}[theorem]{Lemma}
\newtheorem{keypoint}{Key point}
\newtheorem{proposition}[theorem]{Proposition}

\newtheorem{alphthm}{Theorem}

\newtheorem{innercustomthm}{Theorem}
\newenvironment{customthm}[1]
  {\renewcommand\theinnercustomthm{#1}\innercustomthm}
  {\endinnercustomthm}

\theoremstyle{definition}
\newtheorem{definition}[theorem]{Definition}
\newtheorem{convention}[theorem]{Convention}
\newtheorem{example}[theorem]{Example}
\newtheorem{setup}[theorem]{Setup}

\theoremstyle{remark}
\newtheorem{remark}[theorem]{Remark}

\newtheoremstyle{cited}{.5\baselineskip\@plus.2\baselineskip\@minus.2\baselineskip}{.5\baselineskip\@plus.2\baselineskip\@minus.2\baselineskip}{\itshape}{}{\bfseries}{\bfseries .}{5pt plus 1pt minus 1pt}{\thmname{#1}\thmnumber{ #2}\thmnote{ \normalfont#3}}
\theoremstyle{cited}
\newtheorem{citedprop}[theorem]{Proposition}

\newtheoremstyle{citeddef}{.5\baselineskip\@plus.2\baselineskip\@minus.2\baselineskip}{.5\baselineskip\@plus.2\baselineskip\@minus.2\baselineskip}{}{}{\bfseries}{\bfseries .}{5pt plus 1pt minus 1pt}{\thmname{#1}\thmnumber{ #2}\thmnote{ \normalfont#3}}
\theoremstyle{citeddef}
\newtheorem{citeddef}[theorem]{Definition}

\newtheoremstyle{step}{.25\baselineskip\@plus.1\baselineskip\@minus.1\baselineskip}{.25\baselineskip\@plus.1\baselineskip\@minus.1\baselineskip}{\itshape}{}{\bfseries}{\bfseries .}{5pt plus 1pt minus 1pt}{\thmname{#1}\thmnumber{ #2}\thmnote{ \normalfont(#3)}}
\theoremstyle{step}
\newtheorem{claim}{Claim}[theorem]

\DeclareMathOperator{\Coh}{\mathsf{Coh}}
\DeclareMathOperator{\Mod}{\mathsf{Mod}}
\DeclareMathOperator{\Amp}{Amp}
\DeclareMathOperator{\Ass}{Ass}
\DeclareMathOperator{\Bs}{Bs}
\DeclareMathOperator{\Char}{char}
\DeclareMathOperator{\Cl}{Cl}
\DeclareMathOperator{\Cone}{Cone}
\DeclareMathOperator{\Cont}{Cont}
\DeclareMathOperator{\Div}{Div}
\DeclareMathOperator{\Exc}{Exc}
\DeclareMathOperator{\EExt}{\mathscr{E}\kern-2pt\mathit{xt}}
\DeclareMathOperator{\Ext}{Ext}
\DeclareMathOperator{\Fix}{Fix}
\DeclareMathOperator{\FFix}{\mathbf{Fix}}
\DeclareMathOperator{\Frac}{Frac}
\DeclareMathOperator{\Hom}{Hom}
\DeclareMathOperator{\HHom}{\mathscr{H}\kern-3pt\mathit{om}}
\DeclareMathOperator{\RRHHom}{\mathbf{R}\mathscr{H}\kern-3pt\mathit{om}}
\DeclareMathOperator{\RRHom}{\mathbf{R}Hom}
\DeclareMathOperator{\JG}{JG}
\DeclareMathOperator{\Max}{Max}
\DeclareMathOperator{\MaxSpec}{MaxSpec}
\DeclareMathOperator{\Mob}{Mob}
\DeclareMathOperator{\NE}{\mathit{NE}}
\DeclareMathOperator{\NEbar}{\overline{\mathit{NE}}}
\DeclareMathOperator{\Nef}{Nef}
\DeclareMathOperator{\Pic}{Pic}
\DeclareMathOperator{\Proj}{Proj}
\DeclareMathOperator{\PProj}{\underline{Proj}}
\DeclareMathOperator{\Princ}{Princ}
\DeclareMathOperator{\SB}{\mathbf{B}}
\DeclareMathOperator{\Sing}{Sing}
\DeclareMathOperator{\Sp}{\mathfrak{Sp}}
\DeclareMathOperator{\Spa}{Spa}
\DeclareMathOperator{\Spec}{Spec}
\DeclareMathOperator{\Spf}{Spf}
\DeclareMathOperator{\Supp}{Supp}
\DeclareMathOperator{\WDiv}{WDiv}
\DeclareMathOperator{\cent}{center}
\DeclareMathOperator{\coeff}{coeff}
\DeclareMathOperator{\cok}{coker}
\DeclareMathOperator{\prdiv}{div}
\DeclareMathOperator*{\hocolim}{\underrightarrow{\mathrm{hocolim}}}
\DeclareMathOperator{\Int}{int}
\DeclareMathOperator{\im}{im}
\DeclareMathOperator{\Length}{length}
\DeclareMathOperator{\mult}{mult}
\DeclareMathOperator{\nonsnc}{non-snc}
\DeclareMathOperator{\res}{res}
\DeclareMathOperator{\sgn}{sgn}
\DeclareMathOperator{\Span}{span}
\DeclareMathOperator{\supp}{supp}
\DeclareMathOperator{\vol}{vol}

\newcommand{\bA}{\mathbf{A}}
\newcommand{\CC}{\mathbf{C}}
\newcommand{\DD}{\mathbf{D}}
\newcommand{\FF}{\mathbf{F}}
\newcommand{\LL}{\mathbf{L}}
\newcommand{\NN}{\mathbf{N}}
\newcommand{\PP}{\mathbf{P}}
\newcommand{\QQ}{\mathbf{Q}}
\newcommand{\RR}{\mathbf{R}}
\newcommand{\ZZ}{\mathbf{Z}}
\newcommand{\kk}{\mathbf{k}}
\newcommand{\cA}{\mathcal{A}}
\newcommand{\cB}{\mathcal{B}}
\newcommand{\cC}{\mathcal{C}}
\newcommand{\cE}{\mathcal{E}}
\newcommand{\cG}{\mathcal{G}}
\newcommand{\cH}{\mathcal{H}}
\newcommand{\cI}{\mathcal{I}}
\newcommand{\cL}{\mathcal{L}}
\newcommand{\cO}{\mathcal{O}}
\newcommand{\cP}{\mathcal{P}}
\newcommand{\cS}{\mathcal{S}}
\newcommand{\cX}{\mathcal{X}}
\newcommand{\cY}{\mathcal{Y}}
\newcommand{\cZ}{\mathcal{Z}}
\newcommand{\fR}{\mathfrak{R}}
\newcommand{\fb}{\mathfrak{b}}
\newcommand{\fm}{\mathfrak{m}}
\newcommand{\fn}{\mathfrak{n}}
\newcommand{\fp}{\mathfrak{p}}

\newcommand{\sF}{\mathscr{F}}
\newcommand{\sG}{\mathscr{G}}
\newcommand{\sH}{\mathscr{H}}
\newcommand{\sI}{\mathscr{I}}
\newcommand{\sJ}{\mathscr{J}}
\newcommand{\sK}{\mathscr{K}}
\newcommand{\sL}{\mathscr{L}}
\newcommand{\sM}{\mathscr{M}}
\newcommand{\al}{\textup{al}}
\newcommand{\an}{\textup{an}}
\newcommand{\coh}{\textup{c}}
\newcommand{\qc}{\textup{qc}}
\newcommand{\cyc}{\textup{cyc}}
\newcommand{\eff}{\textup{eff}}
\newcommand{\gp}{\textup{gp}}
\newcommand{\id}{\textup{id}}
\newcommand{\op}{\mathrm{op}}
\newcommand{\sm}{\textup{sm}}

\newcommand{\hooklongrightarrow}{\lhook\joinrel\longrightarrow}
\makeatletter
\newcommand{\longtwoheadrightarrow}{\mathrel{\text{\longtwo@rightarrow}}}
\newcommand{\longtwo@rightarrow}{%
  \sbox0{$\m@th\longrightarrow$}%
  \smash{\rlap{\kern0.175\wd0 \clipbox{{.3\width} {-\height} 0pt {-\height}}{$\m@th\longrightarrow$}}}%
  $\m@th\longrightarrow$%
}
\makeatother
\makeatletter
\newcommand*{\da@rightarrow}{\mathchar"0\hexnumber@\symAMSa 4B }
\newcommand*{\da@leftarrow}{\mathchar"0\hexnumber@\symAMSa 4C }
\newcommand*{\xdashrightarrow}[2][]{%
  \mathrel{%
    \mathpalette{\da@xarrow{#1}{#2}{}\da@rightarrow{\,}{}}{}%
  }%
}
\newcommand*{\da@xarrow}[7]{%
  \sbox0{$\ifx#7\scriptstyle\scriptscriptstyle\else\scriptstyle\fi#5#1#6\m@th$}%
  \sbox2{$\ifx#7\scriptstyle\scriptscriptstyle\else\scriptstyle\fi#5#2#6\m@th$}%
  \sbox4{$#7\dabar@\m@th$}%
  \dimen@=\wd0 %
  \ifdim\wd2 >\dimen@
    \dimen@=\wd2 %
  \fi
  \count@=2 %
  \def\da@bars{\dabar@\dabar@}%
  \@whiledim\count@\wd4<\dimen@\do{%
    \advance\count@\@ne
    \expandafter\def\expandafter\da@bars\expandafter{%
      \da@bars
      \dabar@ 
    }%
  }%
  \mathrel{#3}%
  \mathrel{%
    \mathop{\da@bars}\limits
    \ifx\\#1\\%
    \else
      _{\copy0}%
    \fi
    \ifx\\#2\\%
    \else
      ^{\copy2}%
    \fi
  }%
  \mathrel{#4}%
}
\makeatother

\providecommand\given{}
\newcommand\SetSymbol[1][]{\nonscript\:#1\vert\allowbreak\nonscript\:\mathopen{}}
\DeclarePairedDelimiterX\Set[1]\{\}{\renewcommand\given{\SetSymbol[\delimsize]}#1}

\def\hyph{-\penalty0\hskip0pt\relax}
\begin{document}

\title[The relative minimal model program in equal
characteristic zero]{The
relative minimal model program for\\excellent algebraic
spaces and analytic spaces\\in equal characteristic zero}

\author{Shiji Lyu}
\address{Department of Mathematics, Statistics, and Computer Science\\University of Illinois at Chicago\\Chicago, IL
60607-7045\\USA}
\email{\href{mailto:slyu@alumni.princeton.edu}{slyu@alumni.princeton.edu}}
\urladdr{\url{https://homepages.math.uic.edu/~slyu/}}

\thanks{Lyu was supported by an AMS-Simons travel grant}

\author{Takumi Murayama}
\address{Department of Mathematics\\Purdue University\\West Lafayette, IN
47907-2067\\USA}
\email{\href{mailto:murayama@purdue.edu}{murayama@purdue.edu}}
\urladdr{\url{https://www.math.purdue.edu/~murayama/}}

\thanks{Murayama was supported by the National Science
Foundation under Grant Nos.\ DMS-1902616 and DMS-2201251}


\makeatletter
  \hypersetup{
    pdfauthor={Shiji Lyu and Takumi Murayama},
    pdftitle={The relative minimal model program for excellent algebraic spaces and analytic spaces in equal characteristic zero},
    pdfsubject=\@subjclass,pdfkeywords=\@keywords
  }
\makeatother

\begin{abstract}
  We establish the relative minimal model program with scaling for locally projective
  morphisms of quasi-excellent algebraic spaces admitting dualizing complexes,
  quasi-excellent
  formal schemes admitting dualizing complexes, semianalytic germs of complex
  analytic spaces, rigid analytic spaces, Berkovich spaces, and adic spaces
  locally of weakly finite type over a field,
  all in equal
  characteristic zero.
  To do so, we prove finite generation of
  relative adjoint rings associated to projective morphisms of
  such spaces
  using the strategy of Cascini and Lazi\'c and the
  generalization of the Kawamata--Viehweg vanishing theorem to the scheme
  setting recently established by the second author.
  To prove these results uniformly, we prove
  GAGA theorems for Grothendieck duality and dualizing
  complexes to reduce to the algebraic case.
  In addition, we apply our methods to establish the relative minimal model
  program with scaling for spaces of the form above in dimensions $\le 3$ in
  positive and mixed characteristic, and to show that one can run
  the relative minimal model program with scaling for complex analytic spaces without
  shrinking the base at each step.
\end{abstract}

\maketitle

\setcounter{tocdepth}{1}
{\hypersetup{hidelinks}\tableofcontents}

\section{Introduction}
In \cite{BCHM10,HM10}, Birkar, Cascini, Hacon, and M\textsuperscript{c}Kernan
established the relative minimal model program with scaling for
projective morphisms of complex quasi-projective varieties.
Recently,
Villalobos-Paz \cite{VP} established the analogue of this result for
algebraic spaces of finite type over a field of characteristic zero, and
Fujino \cite{Fuj} and Das--Hacon--P\u{a}un \cite{DHP}
established the analogue for complex analytic spaces.\medskip
\par The goal of this paper is to prove the following theorem.
This shows one can give a unified proof of the relative minimal model program
with scaling established in \cite{BCHM10,VP,Fuj,DHP} that
simultaneously applies to other, larger categories of spaces, with appropriate
choices of scaling divisors $A$.
Note that projective morphisms occur naturally in all categories considered:
For example, if $X$ is a space in one of the categories below, then resolutions
of singularities $\tilde{X} \to X$ as produced in
\cite{Hir64,AHV77,Sch99,Tem08,Tem12,Tem18} are projective morphisms.
Together with these results on resolutions of singularities, the vanishing
theorems in \cite{KMM87,Nak87,Mur}, and the weak factorization theorems in
\cite{Wlo03,AKMW02,AT19}, Theorem \ref{thm:introrelativemmp} shows that we now have many of
the key tools of complex birational geometry available in these other categories
of spaces.
\par For the statement below, following Definition \ref{def:qfactorialoverz},
we say that \(X\) is \textsl{\(\mathbf{Q}\)-factorial over
\(Z\)} if for every affinoid subdomain\footnote{In\label{note:affinoidsubdomain} cases
$(\ref{setup:introalgebraicspaces})$ for schemes or in case
$(\ref{setup:introformalqschemes})$, we mean ``affine open.''
In case $(\ref{setup:introalgebraicspaces})$, we mean ``affine \'etale over
\(Z\).''
In case \((\ref{setup:introadicspaces})\), we mean ``affinoid open.''
See \cite[\S B.6]{AT19}, \citeleft\citen{Ber90}\citemid Definition
2.2.1\citepunct \citen{Ber93}\citemid p.\ 21\citeright, and \cite[Definition
7.2.2/2]{BGR84} for the notion of an affinoid subdomain in the
other cases.} \(U \subseteq Z\), the cycle map
\[
  \cyc_\QQ\colon \Div_\QQ\bigl(\pi^{-1}(U)^\al\bigr) \longrightarrow
  \WDiv_\QQ\bigl(\pi^{-1}(U)^\al\bigr)
\]
is surjective.
Here, \(\pi^{-1}(U)^\al\) is the algebraization of the space \(\pi^{-1}(U)\).
\begin{alphthm}[The relative minimal model program with scaling
  in equal characteristic zero]
  \label{thm:introrelativemmp}
Let $\pi\colon X \to Z$ be a locally projective morphism in one of the following
categories, where $X$ and $Z$ are integral and $X$ is normal:
\begin{enumerate}[label=$(\textup{\Roman*})$,ref=\textup{\Roman*}]
  \item[$(0)$]
  \makeatletter
  \protected@edef\@currentlabel{0}
  \phantomsection
  \label{setup:introalgebraicspaces}
  \makeatother
    The category of excellent Noetherian algebraic spaces of equal
    characteristic zero
    over a scheme $S$ admitting dualizing complexes.
  \item\label{setup:introformalqschemes}
    The category of quasi-excellent Noetherian formal schemes of equal
    characteristic zero
    admitting $c$-dualizing complexes.
  \item\label{setup:introcomplexanalyticgerms} The category of semianalytic
    germs of complex analytic spaces.
  \item\label{setup:introberkovichspaces} The category of $k$-analytic spaces
    over a complete non-Archimedean field $k$ of characteristic zero.
  \makeatletter
  \item[{$(\ref*{setup:introberkovichspaces}')$}]
  \protected@edef\@currentlabel{\ref*{setup:introberkovichspaces}'}
  \phantomsection
  \label{setup:introrigidanalyticspaces}
  \makeatother
    The category of rigid
    $k$-analytic spaces over a complete non-trivially valued
    non-Archimedean field $k$ of characteristic zero.
  \item\label{setup:introadicspaces}
    The category of adic spaces locally of weakly finite type over a complete
    non-trivially valued non-Archimedean field $k$ of characteristic zero.
\end{enumerate}
Let $K_X$ be a canonical divisor on $X$ chosen
compatibly with a dualizing complex on $Z$.\footnote{For example, when $Z$ is a
variety over $k$ or in cases $(\ref{setup:introcomplexanalyticgerms})$,
$(\ref{setup:introberkovichspaces})$,
$(\ref{setup:introrigidanalyticspaces})$, and
$(\ref{setup:introadicspaces})$,
we can choose $K_X$ so that
$\cO_{X_\sm}( (K_X)_{\vert X_\sm})
= \det(\Omega_{X_\sm/k})$ where $X_\sm$ is the smooth locus of $X$.}
  \par Suppose $X$ is $\QQ$-factorial over $Z$ (or $\QQ$-factorial in case
  $(\ref{setup:introalgebraicspaces})$)
  and let $\Delta$ be a $\QQ$-divisor such that
  $(X,\Delta)$ is klt.
  Let $A$ be a $\QQ$-invertible sheaf on $X$ such that the following conditions
  hold:
  \begin{enumerate}[label=$(\roman*)$,ref=\roman*]
    \item $A$ is $\pi$-ample.
    \item $K_X+\Delta+A$ is $\pi$-nef.
  \end{enumerate}
  Then, the relative minimal model program with scaling of $A$ over $Z$ exists.
  Moreover, we have the following properties.
  \begin{enumerate}[label=$(\arabic*)$,ref=\arabic*]
    \item\label{thm:mainwhenterm}
      The relative minimal model program with scaling of $A$ over $Z$ terminates after
      a finite sequence of flips and divisorial contractions
      over every affinoid subdomain
      $U \subseteq Z$ for which
      there exists a rational number $c \in (-\infty, 1]$ such that
      $(cK_{X}+\Delta)_{\vert U}$ is $\pi_{\vert \pi^{-1}(U)}$-big.
    \item If there exists an affinoid covering $Z = \bigcup_j U_j$ such that
      each $U_j$ satisfies the condition in $(\ref{thm:mainwhenterm})$, then
      the relative minimal model program with scaling of $A$ over $Z$ yields
      a commutative diagram
      \begin{equation}\label{eq:merodiag}
        \begin{tikzcd}[column sep=0.4em]
          (X,\Delta) \arrow[dashed]{rr}\arrow{dr}[swap]{\pi} & & (X_m,\Delta_m)
          \arrow{dl}{\pi_m}\\
          & Z
        \end{tikzcd}
      \end{equation}
      where $X \dashrightarrow X_m$ is a rational/meromorphic map and for each $j$,
      $\pi_m^{-1}(U_j) \to U_j$ is either the relative analytification of a minimal
      model over $U_j$ (when $(K_X+\Delta)_{\rvert\pi^{-1}(U_j)}$ is
      $\pi$-pseudoeffective) or the relative analytification of a Mori fibration
      over $U_j$ (when $(K_X+\Delta)_{\rvert\pi^{-1}(U_j)}$ is not
      $\pi$-pseudoeffective).
    \item If $(K_X+\Delta)_{\rvert\pi^{-1}(U)}$ is not $\pi$-pseudoeffective for
      every affinoid subdomain $U \subseteq Z$, then 
      the relative minimal model program with scaling of $A$
      yields a commutative diagram \eqref{eq:merodiag}
      where $X \dashrightarrow X_m$ is a meromorphic map and $\pi_m$ is
      the relative analytification of a Mori fibration over every affinoid
      subdomain $U \subseteq Z$.
  \end{enumerate}
\end{alphthm}
\par We note the hypotheses on the scaling divisor $A$ can be weakened in case
$(\ref{setup:introalgebraicspaces})$.
See Theorems \ref{rem:MMP} and \ref{thm:cl1365} and Corollaries \ref{cor:cl1366}
and \ref{cor:cl1367}.
In case $(\ref{setup:introalgebraicspaces})$,
the partially defined map $X \dashrightarrow X_m$ is rational in the sense of
\cite[\href{https://stacks.math.columbia.edu/tag/0EMM}{Tag
0EMM}]{stacks-project}.
In case $(\ref{setup:introcomplexanalyticgerms})$,
the partially defined map $X \dashrightarrow X_m$ is meromorphic in
the sense of Remmert \cite[Def.\ 15]{Rem57}
(see also \cite[Definition 1.7]{Pet94}).
In cases $(\ref{setup:introberkovichspaces})$,
$(\ref{setup:introrigidanalyticspaces})$, and $(\ref{setup:introadicspaces})$,
the partially defined map $X \dashrightarrow X_m$ is meromorphic in the sense of
Morrow and Rosso \cite[Definition 3.2]{MR23}.
\par In addition to the results in \cite{BCHM10,VP,Fuj,DHP} mentioned above,
as far as we are aware, the only known case of
$(\ref{setup:introformalqschemes})$,
$(\ref{setup:introberkovichspaces})$,
$(\ref{setup:introrigidanalyticspaces})$, and $(\ref{setup:introadicspaces})$
is the case when $X$ is a rigid analytic
surface.
In this case, the relative minimal model program is known
\cite{Uen87,Mit11}, and also holds when $\Char(k) > 0$.
For case $(\ref{setup:introalgebraicspaces})$,
the relative minimal model program for schemes
holds without the
assumption on characteristic in dimension $2$ \cite{Sha66,Lic68,Lip69,Tan18} and
for residue characteristics
$\notin \{2,3,5\}$ in dimension $3$
\cite{Kaw94,Kol21qfac,TY,BMPSTWW,Sti}.
The relative minimal model program for morphisms $X \to Z$ where $X$ is either a
three-dimensional algebraic space over an algebraically closed field of
characteristic zero or a three-dimensional complex analytic space that is
Moishezon locally over $Z$ is proved in \cite{Sho96}.
\subsection{Key points in the proof of Theorem
\ref{thm:introrelativemmp}}\label{sect:keypoints}
\par We now discuss three key elements in the proof of Theorem
\ref{thm:introrelativemmp}.
The first key point is the following:
\begin{keypoint}
  All rings appearing in Theorem \ref{thm:introrelativemmp} are excellent.
  Thus, it suffices to prove our results on the
  minimal model program in the algebraic setting for schemes or algebraic spaces
  and then use GAGA-type theorems.
\end{keypoint}
\par Theorem \ref{thm:introrelativemmp} therefore
illustrates the power of working in the
general context of excellent rings and schemes:
All rings appearing in
these different contexts are excellent \cite{Fri67,Mat73,Kie69,Con99,Duc09}, and
hence we can use the GAGA theorems from
\cite{Ser56,EGAIII1,Kop74,Ber93,Hub07,Poi10,AT19} to move 
between the algebraic and analytic settings.
To implement this strategy
in this paper, we prove GAGA-type theorems for dualizing complexes and
Grothendieck duality in \S\ref{sect:dualityandgaga}, which allow us to move from
settings $(\ref{setup:introformalqschemes})$,
$(\ref{setup:introcomplexanalyticgerms})$,
$(\ref{setup:introberkovichspaces})$,
$(\ref{setup:introrigidanalyticspaces})$, and
$(\ref{setup:introadicspaces})$ to the algebraic setting.
This strategy using GAGA was previously used by
Schoutens \cite{Sch99} (in the rigid analytic case)
and Temkin \cite{Tem12,Tem18} for resolutions of singularities,
by Mitsui \cite{Mit11} for the bimeromorphic geometry of rigid analytic
surfaces, and by Abramovich and Temkin \cite{AT19}
for weak factorization of birational maps.
However, as far as we are aware, our GAGA-type theorems for Grothendieck duality
and dualizing complexes are new in cases 
$(\ref{setup:introcomplexanalyticgerms})$,
$(\ref{setup:introberkovichspaces})$,
$(\ref{setup:introrigidanalyticspaces})$, and $(\ref{setup:introadicspaces})$
(the case for formal schemes is proved in
\cite{ATJLL99}).
\par The special case of $(\ref{setup:introalgebraicspaces})$
when $X$ and $Z$ are schemes
answers a question of Koll\'ar \cite[(23)]{Kolhol} and is
of particular interest separate from its role described above.
This is because of the important role (quasi-)excellent schemes play in the
birational geometry of algebraic varieties, for example in proving resolutions
of singularities \cite{Hir64}, the theory of generic limits \cite{dFM09,Kolhol}
and the proof of the ACC conjecture for log canonical thresholds in the smooth
case or the case when the singularities lie in a bounded family
\cite{dFEM10,dFEM11}, and cases of the ACC
conjecture for minimal log discrepancies in dimension three
\cite{Kaw15}.\medskip
\par While the GAGA theorems described above work over every affinoid subdomain
of the base space $Z$, they cannot be applied globally on $Z$.
Thus, we require a new ingredient that will allow us to glue steps of the
relative minimal model program with scaling together that are constructed over
each member of an affinoid covering.
The solution to this gluing problem is the following:
\begin{keypoint}\label{keypoint:scaling}
  Scaling has two roles: Termination and Gluing.
\end{keypoint}
\par One of the key insights in \cite{BCHM10} is that although it is unknown
whether the relative minimal model program always terminates, one can show that
it terminates as long as one assumes the boundary divisor is big, and
one keeps track of an appropriate scaling divisor $A$
and uses it to choose contraction morphisms at each step of the
relative minimal model program.
A more recent insight originating in the uniqueness results due to Koll\'ar
\cite{Kol21qfac} and utilized by Villalobos-Paz
in \cite{VP}
for algebraic spaces
is that
scaling has another role: Scaling enables one to choose steps of the relative
minimal model program \emph{uniquely,} and hence one can glue together steps of the
relative minimal model program constructed locally on affinoid subdomains of the
base.
The new insight in this paper is that we can adapt this idea
to all categories stated in Theorem \ref{thm:introrelativemmp},
even though the transition maps between rings of sections over affinoid
subdomains of the base are not \'etale or even of finite type.
The approach we take in this paper is based on recent work of Enokizono and
Hashizume \cite{EH}, who solved this gluing problem for projective morphisms of
quasi-excellent algebraic stacks with dualizing
complexes.
See Theorem \ref{thm:gluemmp}.
\par These gluing methods apply outside of equal characteristic zero as well.
Using recent progress on the minimal model program for excellent schemes for
surfaces and threefolds in positive and mixed characteristic
\cite{Kaw94,Tan18,Kol21qfac,TY,BMPSTWW}, we can show that Theorem
\ref{thm:introrelativemmp} extends to positive and mixed characteristics, as
long as we assume that $\dim(X) = 2$ or that $\dim(X) = 3$, $\dim(\pi(X)) > 0$,
and the residue characteristics of local rings of $Z$ are not in $\{2,3,5\}$.
Note that the special case $(\ref{setup:introalgebraicspaces})$ when $X$
and $Z$ are schemes is already interesting since \cite[Theorem G]{BMPSTWW} assumes that
$Z$ is quasi-projective over an excellent domain admitting a dualizing complex.
\begin{customthm}{\ref*{thm:introrelativemmp}\textsuperscript{\textit{p}}}[The
  relative minimal model program with scaling in dimensions $\le 3$ in
  positive and mixed characteristic]\label{thm:introrelativemmpcharp}
  Fix notation as in the first paragraph of Theorem \ref{thm:introrelativemmp}
  with the words ``of (equal) characteristic zero'' and with case
  $(\ref{setup:introcomplexanalyticgerms})$ omitted.
  \par Let $\Delta$ be a $\QQ$-Weil divisor on $X$ such that one of the following
  conditions holds:
  \begin{itemize}
    \item $\dim(X) \le 2$ and either $(X,\Delta)$ is log canonical or $X$ is
      $\QQ$-factorial over $Z$ and the coefficients of $\Delta$ lie in $[0,1]$.
    \item $\dim(X) = 3$, $X$ is $\QQ$-factorial over $Z$, $(X,\Delta)$ is
      klt, and one of the following additional conditions holds:
\begin{enumerate}[label=$(\alph*)$,ref=\alph*]
  \item\label{thm:introrelativemmpcharptanbmpstww}
    $\dim(\pi(X)) > 0$ and none of the residue fields of $Z$
    at closed points (in cases $(\ref{setup:introalgebraicspaces})$ and
    $(\ref{setup:introformalqschemes})$), at rigid points in $Z$ (in cases 
    $(\ref{setup:introberkovichspaces})$ and
    $(\ref{setup:introrigidanalyticspaces})$), or at points in $\JG(Z)$ (in
    case $(\ref{setup:introadicspaces})$)
    are of characteristic $2$, $3$, or $5$.
  \item\label{thm:introrelativemmpcharpty} $\dim(\pi(X)) = 1$.
  \item\label{thm:introrelativemmpcharpkol}
    $\pi\colon X \to Z$ is a log resolution of a pair $(Z,\Gamma)$ where
    $\Gamma$ is a $\QQ$-Weil divisor such that $K_Z+\Gamma$ is $\RR$-Cartier
    for which $A$ below is a $\pi$-ample exceptional divisor.
\end{enumerate}
  \end{itemize}
  Let $A$ be a $\QQ$-invertible sheaf on $X$ such that the following conditions
  hold:
  \begin{enumerate}[label=$(\roman*)$,ref=\roman*]
    \item $A$ is $\pi$-ample.
    \item $K_X+\Delta+A$ is $\pi$-nef.
  \end{enumerate}
  Then, the relative minimal model program with scaling of $A$ over $Z$ exists.
  Moreover, we have the following properties.
  \begin{enumerate}[label=$(\arabic*)$,ref=\arabic*]
    \item
      The relative minimal model program with scaling of $A$ over $Z$ terminates after
      a finite sequence of flips and divisorial contractions over every affinoid
      subdomain $U \subseteq Z$ starting from $(\pi^{-1}(U),\Delta_{\vert
      \pi^{-1}(U)})$.
    \item The relative minimal model program with scaling of $A$ over $Z$ yields
      a commutative diagram
      \[
        \begin{tikzcd}[column sep=0.4em]
          (X,\Delta) \arrow[dashed]{rr}\arrow{dr}[swap]{\pi} & & (X_m,\Delta_m)
          \arrow{dl}{\pi_m}\\
          & Z
        \end{tikzcd}
      \]
      where $X \dashrightarrow X_m$ is a meromorphic map and over every affinoid
      subdomain $U \subseteq Z$, the morphism
      $\pi_m^{-1}(U) \to U$ is either the relative analytification of a minimal
      model over $U$ (when $(K_X+\Delta)_{\rvert\pi^{-1}(U)}$ is
      $\pi$-pseudoeffective) or the relative analytification of a Mori fibration
      over $U$ (when $(K_X+\Delta)_{\rvert\pi^{-1}(U)}$ is not
      $\pi$-pseudoeffective).
  \end{enumerate}
\end{customthm}
Here, $\JG(X)$ denotes the \textsl{Jacobson--Gelfand spectrum} of a Jacobson
adic space $X$ as defined in \cite{Lou} (see Definition \ref{def:jacobsongelfand}).
Note that even though termination (without scaling) of flips is known in the
situation of Theorem \ref{thm:introrelativemmpcharp} for
schemes that are quasi-projective over an excellent domain admitting a dualizing
complex \cite{Tan18,BMPSTWW,Sti}, the gluing
procedure described above requires scaling.
Thus, the relative minimal model program we
construct still uses scaling.\smallskip
\par For complex analytic spaces, where the relative minimal model program with
scaling is due to Fujino \cite{Fuj} and
Das--Hacon--P\u{a}un \cite{DHP}, Key point \ref{keypoint:scaling} allows us
to adapt these existing results to base spaces that are not necessarily Stein or
compact.
Compared to the results in \cite{Fuj,DHP}, our results hold for
all semianalytic germs $Z$ of complex analytic spaces (instead of Stein spaces
that may have to be replaced by smaller subsets at each step of the relative
minimal model program), but require
stronger assumptions on the scaling divisor $A$.
These stronger assumptions enable us to glue together
each step of the relative minimal model program that is constructed on an
affinoid cover.
By using the results in \cite{Fuj} as input, we obtain a version of these results
that do not require replacing the base by a smaller subset at each step.
See Theorem \ref{thm:complexmmpwithoutshrinking} and
compare \citeleft\citen{Fuj}\citemid Theorem 1.7, \S13, and \S22\citepunct
\citen{DHP}\citemid Theorem 1.4\citeright.\medskip
\par A key difference between the relative minimal model program with scaling in
\cite{BCHM10} and the relative minimal model program with scaling established in
this paper and in previous work of Villalobos-Paz \cite{VP} is that we can now
work with schemes that are not necessarily quasi-projective varieties.
This restriction in \cite{BCHM10} is necessary for two reasons: (1) existing
vanishing theorems and existing versions of the fundamental theorems of the
minimal model program require working with varieties and (2) applying Bertini
theorems globally on $X$ requires working with \emph{quasi-projective}
varieties.
We resolve these issues by using the following:
\begin{keypoint}
  We can work with spaces that are not quasi-projective varieties by using
  the vanishing theorems proved by the second author in \emph{\cite{Mur}} and by
  applying Bertini theorems locally over every local ring of the base space.
\end{keypoint}
Thus,
one surprising aspect of the proof of Theorem \ref{thm:introrelativemmp} is that
even after reducing to the special case $(\ref{setup:introalgebraicspaces})$
when $X$ and $Z$ are schemes or algebraic spaces, proving the necessary
vanishing theorems ultimately require one to leave the world of schemes and use
the Zariski--Riemann spaces from \cite{Nag63}.
See \cite{Mur}.
\par For Bertini theorems, the key idea is that Bertini theorems for relatively
generated invertible sheaves can be shown
locally over each local ring of the base space $Z$, as long as the local rings
have residue characteristic zero.
Previously, this was done for relatively very ample invertible sheaves in
\cite{BMPSTWW}.
We can then extend the divisors constructed over each local ring to an
affine cover using the excellence of $X$, and work over each member of this
affine cover separately.
See \S\ref{sect:bertini}.
These Bertini theorems are especially important when lifting sections from
subschemes in \S\ref{sect:cl12s3}, where Bertini theorems are used repeatedly to
perturb log regular pairs, and when running the relative minimal model program
with scaling in \S\ref{sect:mmpexistsandterminates}.
\subsection{Finite generation of relative adjoint rings}
One of the key results shown in \cite{BCHM10} to establish
Theorem \ref{thm:introrelativemmp}$(\ref{setup:introalgebraicspaces})$ for complex
varieties is the
finite generation of relative adjoint rings \cite[Theorem 1.2(3)]{BCHM10}.
We show the following finite generation result, following the approach of
Cascini--Lazi\'c \cite[Theorem A]{CL12} and Corti--Lazi\'c \cite[Theorem
2]{CL13} for complex varieties.
Case $(\ref{setup:introcomplexanalyticgerms})$ below gives a new proof of
\citeleft\citen{Fuj}\citemid Theorem F(1)\citepunct
\citen{DHP}\citemid Theorem 1.3\citeright\ (note that \cite{DHP} also uses the
strategy in \cite{CL12,CL13} in the complex analytic setting).
\begin{alphthm}[Finite generation of relative adjoint rings]\label{thm:introfinitegen}
  Fix notation as in the first paragraph of Theorem \ref{thm:introrelativemmp}.
Let $\Delta_i$ be effective $\mathbf Q$-Weil divisors on $X$ for $i \in
\{1,2,\ldots,\ell\}$ such that $K_X+\Delta_i$ is $\QQ$-Cartier and $(X,\Delta_i)$ is klt for each $i$. 
Let $A_i$ be $\pi$-nef $\QQ$-invertible sheaves for $i \in
\{1,2,\ldots,\ell\}$.
Assume that for each $i$, either $A_i$ is $\pi$-ample, or that there exists a rational number $c_i\in (-\infty,1]$ such that $c_iK_X+\Delta_i$ is $\QQ$-Cartier and $\pi$-big.
Then, the relative adjoint ring  
\[
  \bigoplus_{(m_1,m_2,\ldots,m_\ell) \in \NN^\ell} \pi_*\cO_X\Biggl(
  \Biggl\lfloor \sum_{i=1}^\ell m_i(K_X+\Delta_i+A_i)\Biggr\rfloor \Biggr)
\]
is of finite type over every affinoid subdomain in $Z$.
In particular, if $Z$ has a finite cover by affinoid subdomains, then the
relative adjoint ring is generated by finitely many summands.
\end{alphthm}
Theorem \ref{thm:introfinitegen} does not hold without the assumption
that the \(\Delta_i\) have rational coefficients.
For example, even if \(X = \bA^1_\CC\) and \(\pi\) is the identity map, the ring
\(\bigoplus_m \cO_X(\lfloor m\cdot rP \rfloor)\) is not finitely generated
when \(P \in \bA^1_\CC\) is a point and \(r\) is irrational.\medskip
\par An interesting aspect of our proof is that our version of \cite[Theorem
B]{CL12} (which states that $\mathcal{E}_A(V)$ is a rational polytope)
holds when $Z$ is a scheme of mixed characteristic. See Theorem \ref{thm:cl12b}.
This is because we can deduce it from \cite[Theorem B]{CL12} by passing to
generic fibers.
We note that Theorem \ref{thm:introfinitegen} in cases
$(\ref{setup:introformalqschemes})$,
$(\ref{setup:introcomplexanalyticgerms})$,
$(\ref{setup:introberkovichspaces})$,
$(\ref{setup:introrigidanalyticspaces})$, and
$(\ref{setup:introadicspaces})$
is not used to prove the corresponding
cases of Theorem \ref{thm:introrelativemmp}.
\subsection{Some aspects of the proofs of Theorems \ref{thm:introrelativemmp}
and \ref{thm:introfinitegen}}
\par As described in \S\ref{sect:keypoints},
Theorems \ref{thm:introrelativemmp} and \ref{thm:introfinitegen} unify
the aforementioned results in \cite{BCHM10,VP,Fuj,DHP} since we are able to
deduce them all from the case of excellent schemes.
There are several key new inputs compared to
\cite{KMM87,BCHM10,HM10,CL12,CL13}, which we summarize here.
\begin{enumerate}
  \item The Kawamata--Viehweg vanishing
    theorem for proper
    morphisms of schemes of equal characteristic zero, which was
    recently established by the second author in \cite{Mur}.
    In arbitrary dimension, the necessary vanishing theorems were previously
    only known for
    morphisms of varieties \cite{KMM87} and for morphisms of
    complex analytic spaces \cite{Nak87}.
  \item New, relative versions of Bertini theorems for globally generated
    invertible sheaves (see \S\ref{sect:bertini}).
    These relative Bertini theorems are necessary since the usual
    Bertini theorems for quasi-projective varieties do not apply.
    Similar Bertini theorems for very ample invertible sheaves
    were shown in \cite{BMPSTWW}.
  \item GAGA theorems for Grothendieck duality and
    dualizing complexes (see \S\ref{sect:dualityandgaga}).
    As mentioned above, these GAGA theorems are necessary to establish the
    minimal model program in other categories.
    As mentioned before, the case for formal schemes is proved in \cite{ATJLL99}.
  \item Uniqueness results for steps of the relative
    minimal model program with scaling (see
    \S\ref{sect:gluing}), which are variants of results in \cite{VP,EH}.
    These results show that steps of the relative minimal model program with
    scaling are compatible with base change along
    flat morphisms with geometrically normal fibers (see Remarks
    \ref{rem:amplemodelbasechange} and \ref{rem:mmpbasechange}).
    These gluing results are used to glue steps of the minimal model
    program together after constructing them over affinoid subdomains in
    $Z$.
\end{enumerate}
\par To prove Theorem \ref{thm:introrelativemmp}, we also need versions of 
the Basepoint-free, Contraction, Rationality, and Cone theorems for schemes and
algebraic spaces.
We give two proofs of these results: One by by adapting strategy in \cite{KMM87}
for complex varieties (see \S\ref{sect:bpfcontrratcone}), and another by
adapting the strategy in \cite{CL13} for complex varieties (see
\S\ref{sect:fundrevisit}).
We have included the results proved by adapting the strategy in \cite{KMM87}
because they apply more generally to divisorially log terminal (dlt) pairs, and this version
of the Rationality theorem (Theorem \ref{thm:rationality}) also yields
information on the denominators that can appear.
However, we will use some of our versions of the results in \cite{CL13} to deduce
termination with scaling.\medskip
\par Finally, we mention that one can consider other
generalizations of the minimal model program to other categories of
spaces.
For example, for complex analytic spaces (case
$(\ref{setup:introcomplexanalyticgerms})$), the minimal
model program for K\"ahler threefolds
\cite{CP97,Pet98,Pet01,HP15,HP16,CHP16,DO22,DH20,DOb} (see also \cite{DH}),
classes of K\"ahler fourfolds \cite{DHP}, and log surfaces in
Fujiki's class $\mathcal{C}$ \cite{Fuj21} are known.
For formal schemes (case $(\ref{setup:introformalqschemes})$),
Smith initiated the study of a
minimal model program for pseudo-proper formal schemes over a field in
\cite{Smi17}.
A major difficulty for this class of formal schemes is that Smith showed there are counterexamples
to Kodaira-type vanishing theorems
for smooth formal schemes that are pseudo-projective over fields of
characteristic zero \cite[Proposition 4.3.1]{Smi17}.
\subsection*{Outline}
This paper consists of six parts.
For readers who are primarily interested in our results for other
categories, Parts \ref{part:othercats} and \ref{part:additionalresults} can
largely be read independently from the previous parts as long as one accepts the
validity of Theorem \ref{thm:introrelativemmp} for schemes and refers
back to the necessary definitions and results earlier in the paper as needed.\medskip
\par In Part \ref{part:prelim}, we establish the necessary preliminaries for the
minimal model program for schemes and algebraic spaces.
Compared to the case of varieties, there are subtleties working with
divisors on algebraic spaces and having to do with $\QQ$-factoriality.
We also prove many fundamental results on relative nefness and bigness for
morphisms of algebraic spaces, for example the theorem of the base (Theorem
\ref{thm:RelNSFinite}) and Kleiman's criterion for relative ampleness (Theorem
\ref{thm:ampisinterior}), which we need to establish theorems of the
minimal model program for algebraic spaces in our setting.
\par We note that to prove Theorem \ref{thm:introrelativemmp}, it suffices to
prove Theorem \ref{thm:introrelativemmp}$(\ref{setup:introalgebraicspaces})$ for
schemes.
This is because one can deduce Theorem
\ref{thm:introrelativemmp}$(\ref{setup:introalgebraicspaces})$ for algebraic
spaces from the scheme-theoretic case using the framework in \cite{VP}, and
cases
$(\ref{setup:introformalqschemes})$,
$(\ref{setup:introcomplexanalyticgerms})$,
$(\ref{setup:introberkovichspaces})$,
$(\ref{setup:introrigidanalyticspaces})$, and
$(\ref{setup:introadicspaces})$
only use the scheme case of Theorem
\ref{thm:introrelativemmp}$(\ref{setup:introalgebraicspaces})$.
However, we have included the foundational results necessary to prove Theorem
\ref{thm:introrelativemmp} directly for algebraic spaces because we can prove
more general results on the relative minimal model program for
algebraic spaces by proceeding directly (see Theorems
\ref{rem:MMP} and \ref{thm:cl1365}) that we could not show using the strategy in
\cite{VP}.
Moreover, when verifying the necessary foundational results for schemes that we
could not locate in the literature,
we realized that we could prove the same statements for algebraic spaces.
We believe these statements to be of independent interest and hope they will be
useful for future reference.
Part \ref{part:prelim} also illustrates what foundational results would be necessary
to prove Theorem \ref{thm:introrelativemmp} directly in cases
$(\ref{setup:introformalqschemes})$,
$(\ref{setup:introberkovichspaces})$,
$(\ref{setup:introrigidanalyticspaces})$, and
$(\ref{setup:introadicspaces})$ (see \cite{Fuj,DHP} for case
$(\ref{setup:introcomplexanalyticgerms})$).\medskip
\par In Part \ref{part:bertiniandfund}, we prove our new relative versions of
Bertini theorems for schemes.
These theorems will become necessary later to perturb klt pairs without having
global Bertini theorems available as would be the case for quasi-projective
varieties over a
field.
We also show the fundamental theorems of the minimal model program (the
Basepoint-freeness, Contraction, Rationality, and Cone theorems) for algebraic
spaces adapting the strategy in \cite{KMM87} for complex varieties.
While we also prove dual versions of these theorems for klt pairs using the
method in \cite{CL13}
(see \S\ref{sect:fundrevisit}), we have
included these results because they hold more generally for divisorially log
terminal (dlt)
pairs, and the Rationality Theorem \ref{thm:rationality}
provides some more information about the
denominators that appear.\medskip
\par In Part \ref{part:fingen}, we prove Theorem \ref{thm:introfinitegen} for
schemes using the strategy of Cascini--Lazi\'c.
A key input is the version of the Kawamata--Viehweg vanishing theorem proved
by the second author \cite[Theorem A]{Mur}.
Because of the lack of Bertini theorems, however, we need to formulate many of
our results in terms of restriction maps on global sections instead of linear
systems as is done in \cite{CL12}.
This allows us to reduce to the case when the base scheme $Z$ is the spectrum of
an excellent local $\QQ$-algebra.
We conclude the part by proving finite generation for klt pairs
and giving alternative proofs of the Contraction, Rationality, and
Cone theorems by adapting the strategy in \cite{CL13} for complex varieties.
These results will be used in Part \ref{part:relativemmpforschemes} to prove
termination with scaling.\medskip
\par In Part \ref{part:relativemmpforschemes}, we establish the existence of
flips and termination with scaling for schemes and algebraic spaces, using
Theorem \ref{thm:introfinitegen}.
This completes the proof of Theorem
\ref{thm:introrelativemmp}$(\ref{setup:introalgebraicspaces})$.
We then give some applications of these results by showing that
$\QQ$-factorializations and terminalizations exist, which for simplicity we
prove only for schemes.\medskip
\par In Part \ref{part:othercats}, we setup the necessary preliminaries for
Theorem \ref{thm:introrelativemmp} in cases $(\ref{setup:introformalqschemes})$,
$(\ref{setup:introcomplexanalyticgerms})$,
$(\ref{setup:introberkovichspaces})$,
$(\ref{setup:introrigidanalyticspaces})$,
and $(\ref{setup:introadicspaces})$.
We then prove our GAGA-type results for dualizing complexes and Grothendieck
duality in \S\ref{sect:dualityandgaga}.
In \S\ref{sect:setupforothercats}, we set our notation for different categories
of spaces and check that the hypotheses in Theorems \ref{thm:introrelativemmp}
and \ref{thm:introfinitegen} are preserved under algebraization.
Finally, we prove Theorems \ref{thm:introrelativemmp},
\ref{thm:introrelativemmpcharp},
and \ref{thm:introfinitegen} in \S\ref{sect:mmpforothercats}.
The proof of Theorem \ref{thm:introrelativemmpcharp} utilizes recent progress on
the minimal model program for excellent schemes in dimensions $\le3$
\cite{Kaw94,Tan18,Kol21qfac,TY,BMPSTWW}.
\par We note that our assumptions on adic spaces in $(\ref{setup:introadicspaces})$
are necessary to even make sense of the normality and irreducibility assumptions
in Theorem \ref{thm:introrelativemmp}.
Normality and irreducibility of adic spaces locally of weakly finite type over a field
are defined in \cite{Man23}.
We will also use these assumptions to utilize excellence results from
\cite{Kie69,Con99,Duc09} in the proofs of
our statements on GAGA and Grothendieck duality.\medskip
\par Finally, in Part \ref{part:additionalresults}, we prove some additional
results in other categories utilizing the gluing techniques we developed to prove
Theorem \ref{thm:introrelativemmp}.
First, using as input the results in \cite{Fuj} for complex analytic
spaces and our methods in Part \ref{part:relativemmpforschemes} of this paper,
we show that one can run the relative minimal model program with
scaling for complex analytic spaces without shrinking the base space $Z$ at each
step (Theorem \ref{thm:complexmmpwithoutshrinking}).
Second, we discuss how the versions of the Basepoint-free and Contraction
theorems in this paper can be reformulated to avoid
the assumption that a dualizing complex exists on the base space $Z$ (Theorems
\ref{thm:bpfnoomega} and \ref{thm:contractionnoomega}).

\addtocontents{toc}{\protect\medskip}
\section*{Acknowledgments}
We are grateful to Dan Abramovich for helpful conversations about \cite{AT19},
to Jack Hall for helpful conversations on the GAGA theorems in \cite{Hal},
to James M\textsuperscript{c}Kernan for pointing out the reference \cite{Smi17},
and
to Peter Scholze for discussing his joint work with Clausen on condensed mathematics
\cite{CS19,CS22} with us (see Theorem \ref{thm:clausenscholze})
that is used in the proof of Theorem
\ref{thm:dualizingcomplexcompatadic}.
We are grateful to Makoto Enokizono and Kenta Hashizume for pointing out a
gap in the original gluing argument that appeared in \S\ref{sect:mmpforothercats}
and for subsequent discussions on their gluing
results from \cite{EH}.
We would like to thank
Olivier Benoist,
J\'anos Koll\'ar,
Joaqu\'in Moraga,
Santai Qu,
David Villalobos-Paz,
and
Chenyang Xu
for helpful discussions.
Finally, we thank the referees for their valuable comments and suggestions.

\addtocontents{toc}{\protect\medskip}
\section*{Notation and conventions}
All rings are commutative with identity, and all ring maps are unital.\medskip
\par For a ringed space or ringed site $X$, $\DD_{\coh}(X)$ denotes the derived
category of $\cO_X$-modules with coherent cohomology sheaves.
We can then define $\DD^+_{\coh}(X)$, $\DD^-_{\coh}(X)$, and $\DD^b_{\coh}(X)$
bounded-below, bounded-above, and bounded derived categories of $\cO_X$-modules
with coherent cohomology sheaves, respectively.
When the notion of quasi-coherent $\cO_X$-modules is defined, we define
$\DD_{\qc}(X)$, $\DD^+_{\qc}(X)$, $\DD^-_{\qc}(X)$, and $\DD^b_{\qc}(X)$ similarly.\medskip
\par Let $X$ be an algebraic space over a scheme $S$.
We say that a quasi-coherent sheaf $\cA$ of $\cO_X$-algebras is
\textsl{of finite type} if for every affine scheme $U = \Spec(R)$ \'etale over
$X$, we have $\cA_{\vert U} \cong \widetilde{A}$ where $A$ is an $R$-algebra of
finite type (see \citeleft\citen{EGAInew}\citemid (2.2.5)\citepunct
\citen{stacks-project}\citemid
\href{https://stacks.math.columbia.edu/tag/07V8}{Tag 07V8}\citeright).\medskip
\par A \textsl{non-trivially valued non-Archimedean field} is a topological field
whose topology is induced by a rank 1 valuation.
These are called \textsl{non-Archimedean fields} in \cite[Definition
1.1.3]{Hub96}.\medskip
\par For \(\kk \in \{\ZZ,\QQ,\RR\}\), we will use both the language of
\(\kk\)-invertible sheaves and \(\kk\)-Weil divisors as described in
\S\ref{subsect:divisors} for schemes and algebraic spaces and in
\S\ref{subsect:divqfac} otherwise. 
For algebraic spaces, we prefer to work with \(\kk\)-invertible
sheaves since there is no preexisting notion of Cartier divisors on algebraic
spaces in general.

\section*{Where assumptions are used}
While we will restate the assumptions used in each result, we point out where
the major assumptions are used in the paper.
These assumptions interact in a subtle way.
\begin{enumerate}[label=\((\arabic*)\),ref=\arabic*]
  \item \textbf{Dualizing complexes.}
    Dualizing complexes \(\omega_X^\bullet\) are used to define canonical
    sheaves \(\omega_X\), canonical divisors \(K_X\), and singularities of pairs
    \((X,\Delta)\) in \S\ref{sect:candiv}.
    Canonical divisors and singularities of pairs are necessary to make sense of
    the minimal model program.
    Canonical sheaves are necessary for the formulation of vanishing theorems using
    higher direct images in \cite[Theorem A]{Mur}, as opposed to the formulation
    using local cohomology modules of the form
    \[
      H^i_{f^{-1}(\{\fm\})}(\sL^{-1})
    \]
    in \cite[Theorem 8.2]{Mur}.
    \par See \S\ref{sect:noomega} for versions of the basepoint-freeness and
    contraction theorems that we can state and prove in the absence of dualizing
    complexes.
  \item\label{assum:eq} \textbf{Quasi-excellence.}
    We assume that our schemes, algebraic spaces, and formal schemes are
    quasi-excellent for multiple reasons.
    \begin{enumerate}[label=\((\alph*)\)]
      \item (Resolution of singularities)
        By \cite[Proposition 7.9.5]{EGAIV2}, if \(X\) is a locally
        Noetherian scheme for which every integral scheme \(Y\) finite over
        \(X\) has a resolution of singularities, then \(X\) is quasi-excellent.
        Grothendieck and Dieudonn\'e conjectured that all reduced locally
        Noetherian quasi-excellent schemes have resolutions of singularities
        \cite[Remark 7.9.6]{EGAIV2}.
        This conjecture is known in the following cases:
        \begin{enumerate}[label=\((\roman*)\)]
          \item (Equal characteristic zero)
            Resolutions of singularities exist for quasi-excellent schemes of
            equal characteristic zero by \cite{Hir64} (for schemes essentially
            of finite type over quasi-excellent local \(\QQ\)-algebras) and
            \cite{Tem08} (in general).
          \item (Dimensions \(\le 3\))
            Resolutions of singularities exist for quasi-excellent surfaces by
            \cite{Lip78} and for quasi-excellent threefolds by
            \cite{CP19,CJS20,BMPSTWW}.
        \end{enumerate}
      \item (Negativity Lemma) The Negativity Lemma \ref{lem:Negativity} and its
        consequences
        use
        quasi\hyph{}excellence because the proof ultimately reduces to the case of
        quasi-excellent surfaces, where one has resolutions of singularities by
        \cite{Lip78}.
      \item (Bertini theorems) In \S\ref{sect:bertini}, we prove Bertini
        theorems for proper morphisms \(\pi\colon X \to Z\) of locally
        Noetherian schemes of equal characteristic zero and
        \(\pi\)-generated invertible sheaves on \(X\).
        Quasi-excellence assumptions are not necessary when \(Z\) is the
        spectrum of a Noetherian local \(\QQ\)-algebra (Theorem
        \ref{thm:bertini} and Remark \ref{rem:bertinisnc}).
        When \(Z\) is not local, we extend divisors constructed over the
        local rings of \(Z\) to an open cover of \(Z\) using the quasi-excellent
        assumption (or more precisely, the J-2 condition which is part of the
        definition of quasi-excellence; see Definition
        \ref{def:excellent}\((\ref{def:excellentj2})\)).
        See Corollary \ref{cor:bertinionopencover}.
        Passing to an open cover in this way allows us to perturb klt pairs in
        Corollary \ref{cor:BertiniKLTOpenCover}.
        \par\hspace{\enumerateparindent}
        We use these Bertini theorems repeatedly throughout the paper, especially
        when lifting sections from subschemes in \S\ref{sect:cl12s3} and when
        running the relative minimal model program with scaling in
        \S\ref{sect:mmpexistsandterminates}.
    \end{enumerate}
  \item \textbf{Equal characteristic zero.}
    For most of this paper, we work with spaces of equal characteristic zero for
    multiple reasons.
    \begin{enumerate}[label=\((\alph*)\)]
      \item (Vanishing theorems)
        We heavily use the relative Kawamata--Viehweg vanishing theorem for
        proper morphisms of locally Noetherian schemes of equal characteristic
        zero, due to the second author \cite{Mur}.
        For complex varieties, this vanishing theorem is due to Kawamata
        \cite{Kaw82} and Viehweg \cite{Vie82} in the absolute setting and to
        Kawamata, Matsuda, and Matsuki \cite{KMM87} in the relative setting.
        These vanishing theorems are false in positive characteristic
        \cite{Ray78} and mixed characteristic \cite{Tot}.
      \item (Resolutions of singularities)
        As mentioned above in \((\ref{assum:eq})\), resolutions of singularities
        in dimensions \(\ge 4\) are known to exist only in equal characteristic
        zero.
      \item (Bertini theorems for relatively generated invertible sheaves)
        Bertini's theorem for globally generated invertible sheaves is false in
        positive and mixed characteristic \cite[p.\ 140]{Zar44}.
        We need Bertini's theorem for relatively generated invertible sheaves
        when lifting sections from subschemes in \S\ref{sect:cl12s3}.
        We only need the case for relatively very ample invertible sheaves
        in
        \S\ref{sect:mmpexistsandterminates} (see Lemma
        \ref{lem:AmpleIsGoodScaling}).
        Bertini's theorem for relatively very ample invertible sheaves
        is true in arbitrary characteristic \cite[Theorem 2.15]{BMPSTWW}.
    \end{enumerate}
  \item \textbf{Projectivity.}
    We work with projective morphisms \(\pi\colon X \to Z\) for two
    reasons.
    \begin{enumerate}[label=\((\alph*)\)]
      \item (Scaling) To make sense of the ``scaling'' part of the relative
        minimal model program with scaling in Theorem
        \ref{thm:introrelativemmp}, we need a relatively ample invertible sheaf.
        See Key point \ref{keypoint:scaling} above for more discussion on the
        importance of scaling for proving termination and gluing steps of the
        minimal model program.
      \item (GAGA) For formal schemes and analytic spaces, we need projectivity
        assumptions to ensure our morphisms are algebraizable.
        This allows us to reduce to the scheme-theoretic case using the GAGA-type
        theorems from \cite{Ser56,EGAIII1,Kop74,Ber93,Hub07,Poi10,AT19} and
        \S\ref{sect:dualityandgaga} of this paper.
    \end{enumerate}
\end{enumerate}

\begingroup
\makeatletter
\renewcommand{\@secnumfont}{\bfseries}
\part{Preliminaries for schemes and algebraic spaces}\label{part:prelim}
\makeatother
\endgroup
In this part, we establish preliminary definitions and results
that will be used throughout the paper.
For the reader's convenience, we have tried to provide references for
corresponding material in \cite{KMM87}, \cite{CL12}, and \cite{CL13}.
We use the definition of algebraic spaces over a scheme $S$ from
\cite[\href{https://stacks.math.columbia.edu/tag/025Y}{Tag
025Y}]{stacks-project}.\bigskip

\section{Quasi-excellence, excellence, and dualizing complexes}\label{sect:excellentschemes}
\subsection{Quasi-excellence and excellence}
We will mostly work with quasi-excellent or excellent schemes.
\begin{citeddef}[{\citeleft\citen{EGAIV2}\citemid D\'efinition 7.8.2 and
  (7.8.5)\citepunct \citen{Mat80}\citemid (34.A) Definition\citeright}]\label{def:excellent}
  Let $R$ be a ring.
  We say that $R$ is \textsl{excellent} if the
  following conditions are
  satisfied.
  \begin{enumerate}[label=$(\roman*)$,ref=\roman*]
    \item\label{def:excellentnoeth} $R$ is Noetherian.
    \item\label{def:excellentuc} $R$ is universally catenary.
    \item\label{def:excellentgring}
      $R$ is a \textsl{G-ring}, i.e., for every prime ideal $\fp \subseteq R$,
      the $\fp$-adic completion map $R_\fp \to \widehat{R}_\fp$ has
      geometrically regular fibers.
    \item\label{def:excellentj2}
      $R$ is \textsl{J-2}, i.e., for every $R$-algebra $S$ of finite type, the
      regular locus in $\Spec(S)$ is open.
  \end{enumerate}
  We say that $R$ is \textsl{quasi-excellent} if $(\ref{def:excellentnoeth})$,
  $(\ref{def:excellentgring})$, and $(\ref{def:excellentj2})$ are satisfied.
  A locally Noetherian scheme $X$ is \textsl{excellent} (resp.\
  \textsl{quasi-excellent}) if it admits an open
  affine covering $X = \bigcup_i
  \Spec(R_i)$ such that every $R_i$ is excellent (resp.\ quasi-excellent).
\end{citeddef}
Since quasi-excellence is an \'etale local property by
\cite[Theorem 32.2]{Mat89}, 
we can define quasi-excellence as follows.
\begin{definition}[see {\citeleft\citen{CT20}\citemid \S2.1\citeright}]
  Let $X$ be a locally Noetherian algebraic space over a scheme $S$.
  We say that $X$ is \textsl{quasi-excellent} if for
  every \'etale morphism $U \to X$ from a scheme $U$, the scheme $U$ is
  quasi-excellent.
\end{definition}
\subsection{Dualizing complexes}
We will need the notion of a dualizing complex to make sense of canonical
sheaves and divisors, which we will define in \S\ref{sect:candiv}.
\begin{citeddef}[{\citeleft\citen{Har66}\citemid Chapter V, Definition on p.\
  258\citepunct \citen{Con00}\citemid p.\ 118\citepunct
  \citen{stacks-project}\citemid
  \href{https://stacks.math.columbia.edu/tag/0A87}{Tag 0A87}\citeright}]
  \label{def:dualizingcomplexschemes}
  Let $X$ be a locally Noetherian scheme.
  A \textsl{dualizing complex} on $X$ is an object $\omega_X^\bullet$ in
  $\DD^b_{\mathrm{c}}(X)$ that has finite injective dimension, such that the
  natural morphism
  \[
    \id \longrightarrow
    \RRHHom_{\cO_X}\bigl(\RRHHom_{\cO_X}(-,\omega_X^\bullet),
    \omega_X^\bullet\bigr)
  \]
  of $\delta$-functors on $\DD_{\textup{c}}(X)$ is an isomorphism.
\end{citeddef}
\begin{remark}
  Locally Noetherian schemes admitting dualizing complexes have finite Krull
  dimension and are universally catenary
  \citeleft\citen{Har66}\citemid Chapter V, Corollary 7.2\citepunct
  \citen{stacks-project}\citemid
  \href{https://stacks.math.columbia.edu/tag/0A80}{Tag 0A80}\citeright.
  Thus, quasi-excellent schemes admitting dualizing complexes are excellent.
\end{remark}
\begin{remark}
  All excellent Henselian local rings admit a dualizing complex \cite[p.\
  289]{Hin93}.
  A recent result of the first author shows that every finite-dimensional
  quasi-excellent scheme has an \'etale cover that admits a dualizing complex
  \cite[Theorem 6.5]{Lyu}.
\end{remark}
\par We can define dualizing complexes for algebraic spaces \'etale-locally.
\begin{citeddef}[{\citeleft\citen{AB10}\citemid Definition 2.16\citepunct
  \citen{stacks-project}\citemid
  \href{https://stacks.math.columbia.edu/tag/0E4Z}{Tag 0E4Z}\citeright}]
  Let $X$ be a locally Noetherian algebraic space over a scheme $S$.
  A \textsl{dualizing complex} on $X$ is a complex $\omega_X^\bullet$ in
  $\DD^b_{\mathrm{qc}}(X)$ for which there exists a surjective \'etale morphism
  $U \to X$ from a scheme $U$ such that the pullback of $\omega_X^\bullet$ to
  $U$ is a dualizing complex
  on $U$ in the sense of Definition \ref{def:dualizingcomplexschemes}.
\end{citeddef}
We will frequently use the following fact:
\begin{lemma}[cf.\ {\citeleft\citen{Har66}\citemid (2) on p.\ 299\citepunct
  \citen{AB10}\citemid Proposition 2.18 and Remark on p.\ 14\citepunct
  \citen{stacks-project}\citemid
  \href{https://stacks.math.columbia.edu/tag/0AA3}{Tag
  0AA3}\citeright}]\label{lem:dualizingcomplexpullback}
  Let $f\colon X \to Y$ be a separated morphism of finite type between
  Noetherian algebraic spaces over a scheme $S$.
  Consider a Nagata compactification
  \[
    \begin{tikzcd}[column sep=0.75em]
      X \arrow[hook]{rr}\arrow{dr}[near start,swap]{f} & &
      \bar{X}\arrow{dl}[near start]{\vphantom{f}\smash{\bar{f}}}\\
      & Y
    \end{tikzcd}
  \]
  of $f$.
  If $\omega_Y^\bullet$ is a dualizing complex on $Y$, then
  \[
    f^!\omega_Y^\bullet \coloneqq \bigl( \bar{a}(\omega_Y^\bullet)
    \bigr)_{\vert X}
  \]
  is a dualizing complex on $X$, where $\bar{a}$ is the right adjoint of the
  derived pushforward $\RR \bar{f}_*$.
\end{lemma}
The right adjoint of the derived pushforward is constructed in
\cite[\href{https://stacks.math.columbia.edu/tag/0E55}{Tag
0E55}]{stacks-project}.
Nagata compactifications exist for separated morphisms of finite type between
quasi-compact quasi-separated algebraic spaces
\cite[Theorem 1.2.1]{CLO12} (see also \citeleft\citen{FK06}\citemid pp.\
355--356\citepunct \citen{Ryd}\citemid Theorem F\citeright).
\begin{proof}
  Let $U \to Y$ be an \'etale surjective morphism from a scheme $U$ such that
  the pullback of $\omega_Y^\bullet$ to $U$ is a dualizing complex.
  Next, we note that
  restriction and the right adjoint $a$ are compatible with \'etale base
  change by definition, where we use the fact that the right
  adjoint does not depend on whether we
  consider a scheme as an actual scheme or the algebraic space it represents
  by
  \cite[\href{https://stacks.math.columbia.edu/tag/0E6E}{Tag
  0E6E}]{stacks-project}.
  We therefore see that
  the pullback of $f^!$ to $U$ is the exceptional pullback for schemes
  constructed in \cite[\href{https://stacks.math.columbia.edu/tag/0A9Y}{Tag
  0A9Y}]{stacks-project}.
  The statement now follows from the scheme case (after replacing $U$ by an open
  affine cover) in
  \citeleft\citen{Har66}\citemid (2) on p.\ 299\citepunct
  \citen{stacks-project}\citemid
  \href{https://stacks.math.columbia.edu/tag/0AA3}{Tag 0AA3}\citeright.
\end{proof}
\section{Divisors and linear systems}
\label{sect:divisorsetc}
\subsection{Divisors}\label{subsect:divisors}
We will use the definition of the group $\Div(X)$ of Cartier
divisors for ringed spaces from \cite[D\'efinition 21.1.2]{EGAIV4}, and the
group $\WDiv(X)$ of Weil divisors for locally Noetherian schemes
from \cite[(21.6.2)]{EGAIV4}.
Here, we recall that a \textsl{Weil divisor} is a locally finite \(\ZZ\)-linear
combination of codimension 1 integral subschemes in \(X\).
\par See \cite[p.\ 204]{Kle79} for the definition of the sheaf $\sK_X$ of meromorphic
functions.
The group of Weil divisors is denoted by $\mathfrak{Z}^1(X)$ in
\cite[(21.6.2)]{EGAIV4} and by $\Div(X)$ in
\cite[\href{https://stacks.math.columbia.edu/tag/0ENJ}{Tag
0ENJ}]{stacks-project}.
The subgroup of principal Cartier divisors is denoted by $\Princ(X)$.
\par Instead of developing the theory of Cartier divisors and cycle maps for
algebraic spaces, we will only
work with the monoid of effective Cartier divisors $\Div^\eff(X)$
on algebraic spaces in the sense of
\cite[\href{https://stacks.math.columbia.edu/tag/083B}{Tag
083B}]{stacks-project} (denoted by $\operatorname{EffCart}(X)$ in
\cite[\href{https://stacks.math.columbia.edu/tag/0CPG}{Tag
0CPG}]{stacks-project}) and Weil divisors on integral locally Noetherian
algebraic spaces in the sense of
\cite[\href{https://stacks.math.columbia.edu/tag/0ENJ}{Tag
0ENJ}]{stacks-project}.
Note that the definition of Cartier divisors on algebraic spaces in
\cite[Chapter II, Definition 8.11]{Knu71} assumes the algebraic space is separated.\medskip
\par We now define Cartier and Weil divisors with $\QQ$- and $\RR$-coefficients.
\begin{definition}[see {\citeleft\citen{KMM87}\citemid Definitions
  0-1-3 and 0-1-8\citepunct \citen{BCHM10}\citemid Definition 3.1.1\citeright}]
  \label{def:kcartdiv}
  Let $X$ be a ringed space,
  and let $\kk \in \{\ZZ,\QQ,\RR\}$.
  A \textsl{$\kk$-Cartier divisor} on $X$ is an element of the group
  \begin{align*}
    \Div_\kk(X) &\coloneqq \Div(X) \otimes_\ZZ \kk.
    \intertext{If $X$ is a locally Noetherian scheme or an integral locally
    Noetherian algebraic space over a scheme $S$,
    a \textsl{$\kk$-Weil divisor} on $X$ is an
    element of the group}
    \WDiv_\kk(X) &\coloneqq \WDiv(X) \otimes_\ZZ \kk.
  \intertext{A $\kk$-Cartier divisor
  is \textsl{integral}
  if it lies in the image of the map}
    \Div(X) &\longrightarrow \Div_\kk(X)
    \intertext{and a $\kk$-Weil divisor is \textsl{integral} if it lies in the
    image of the map}
    \WDiv(X) &\longrightarrow \WDiv_\kk(X).
  \end{align*}
  A $\kk$-Cartier divisor (resp.\ \textsl{$\kk$-Weil divisor}) is
  \textsl{effective} if it can be written as a $\kk_{\ge0}$-linear combination
  of effective Cartier divisors (resp.\ effective Weil divisors).
  The set of effective $\kk$-Cartier (resp.\ $\kk$-Weil) divisors on $X$ is
  denoted $\Div_\kk^\eff(X)$ (resp.\ $\WDiv_\kk^\eff(X)$).
  We drop the prefix ``$\ZZ$-'' if $\kk = \ZZ$.
\par If $A = \sum_{i=1}^r a_iC_i$ is an
  $\RR$-Weil divisor on $X$ where the $C_i$ are distinct prime Weil divisors,
  then the \textsl{round-up} and \textsl{round-down} of $A$ are the Weil
  divisors
  \begin{align*}
    \lceil A \rceil &\coloneqq \sum_{i=1}^r \lceil a_i \rceil C_i \qquad\\
    \lfloor A \rfloor &\coloneqq \sum_{i=1}^r \lfloor a_i \rfloor C_i
  \intertext{respectively, and the \textsl{fractional part} of $A$ is}
    \{A\} &\coloneqq \sum_{i=1}^r \{a_i\} C_i,
    \intertext{where $\{a_i\} \coloneqq a_i - \lfloor a_i \rfloor$ is the
    fractional part of $a_i$ for every $i$.
    If $B = \sum_{i=1}^r b_iC_i$ is another $\RR$-Weil divisor on $X$, then we
    also set}
    A \wedge B &\coloneqq \sum_{i=1}^r \min\{a_i,b_i\}\,C_i.
  \end{align*}
\end{definition}
When $X$ is a locally Noetherian scheme, there is a commutative
diagram
\begin{equation}\label{eq:divwdivdiag}
  \begin{tikzcd}[baseline=(cyc.base)]
    \Div(X) \rar\dar{\cyc\vphantom{\QQ}}
    & \Div_\QQ(X) \rar[hook]\dar{\cyc\otimes_\ZZ\QQ}
    & \Div_\RR(X) \ar[d,"\cyc\otimes_\ZZ\RR"{name=cyc}]\\
    \WDiv(X) \rar[hook] & \WDiv_\QQ(X) \rar[hook] & \WDiv_\RR(X)
  \end{tikzcd}
\end{equation}
of Abelian groups, where the left vertical map is the \textsl{cycle map} from
\cite[(21.6.5.1)]{EGAIV4}, and the other vertical maps are obtained via
extension of scalars.
The cycle map preserves effective divisors \cite[Proposition
21.6.6]{EGAIV4}.
\begin{convention}
  Let $X$ be a locally Noetherian scheme.
  Then, the cycle map $\cyc$ is bijective if and only if $X$ is locally
  factorial \cite[Th\'eor\`eme
  21.6.9$(ii)$]{EGAIV4}.
  In this case, we can identify Cartier and Weil divisors, as well as their
  corresponding versions with $\QQ$- or $\RR$-coefficients.
  On such schemes, we omit the word ``Cartier'' or ``Weil.''
\end{convention}
\par Even if $X$ is not locally factorial, as long as $X$ is normal, we can pass
from Cartier divisors to Weil divisors.
\begin{definition}[see {\citeleft\citen{KMM87}\citemid Remark 0-1-6(2)\citepunct
  \citen{Laz04a}\citemid Remarks 1.1.4 and 1.3.8\citeright}]\label{def:kcartier}
  Let $X$ be a normal locally Noetherian scheme.
  Then, the cycle map
  \begin{align*}
    \cyc\colon& {\Div}(X) \longrightarrow \WDiv(X)
  \intertext{is injective \cite[Th\'eor\`eme 21.6.9$(i)$]{EGAIV4}, as
  are the maps
  \[
    \Div(X) \longrightarrow \Div_\kk(X)
  \]
  for $\kk \in \{\QQ,\RR\}$
  by the commutativity of the diagram
  \eqref{eq:divwdivdiag}.
  A Weil divisor $D$ \textsl{is Cartier} if $D$ lies in the image of $\Div(X)$
  under the cycle map $\cyc$.
  For $\kk \in \{\QQ,\RR\}$, a $\kk$-Weil divisor $D$ \textsl{is $\kk$-Cartier}
  if $D$ lies in the image of the map}
    \cyc \otimes_\ZZ \kk \colon& {\Div}_\kk(X) \longrightarrow \WDiv_\kk(X).
  \end{align*}
\end{definition}
\begin{convention}[see {\citeleft\citen{KMM87}\citemid Definition 0-1-7\citeright}]
  Let $X$ be a normal locally Noetherian scheme.
  We say that $X$ is \textsl{$\QQ$-factorial} if every $\QQ$-Weil divisor
  is $\QQ$-Cartier.
  In this case, we will say ``$\QQ$-divisor'' instead of ``$\QQ$-Cartier
  divisor'' or ``$\QQ$-Weil divisor.''
\end{convention}
\begin{remark}
  In the minimal model program, it is standard to say ``\(\QQ\)-divisor'' for a
  $\QQ$-Weil divisor that is not necessarily \(\QQ\)-Cartier.
  We avoid this terminology because the cycle map may not be
  injective when \(X\) is not necessarily normal.
  The terminology ``\(\QQ\)-Weil divisor'' appears for example in
  \cite[Definitions 16.1 and 16.2]{Cor92} to make sense of divisors with
  fractional coefficients on semi-log canonical schemes.
\end{remark}
To make analogous definitions for algebraic spaces, we will only work with Weil
divisors.
We recall that for ringed
spaces $X$, there is an exact sequence
\begin{align}
  0 \longrightarrow \Princ(X) \longrightarrow \Div(X)
  &\overset{l}{\longrightarrow} \Pic(X)\nonumber
\intertext{by \cite[Proposition 21.3.3$(i)$]{EGAIV4}.
For $\kk \in \{\QQ,\RR\}$, we will also consider its extension of scalars}
  0 \longrightarrow \Princ_\kk(X) \longrightarrow \Div_\kk(X)
  &\xrightarrow{\,l\otimes_\ZZ\kk\,} \Pic_\kk(X)\label{eq:divtopic}
\end{align}
to $\kk$, where
\begin{align*}
  \Princ_\kk(X) &\coloneqq \Princ(X) \otimes_\ZZ \kk,\\
  \Pic_\kk(X) &\coloneqq \Pic(X) \otimes_\ZZ \kk.
\intertext{For algebraic spaces $X$, we also have maps
\begin{equation}\label{eq:effcarttopic}
  \bigl(\Div^\eff(X)\bigr)^\gp_\kk \longrightarrow \Pic_\kk(X)
\end{equation}
for $\kk \in \{\ZZ,\QQ,\RR\}$ obtained
from \cite[\href{https://stacks.math.columbia.edu/tag/0CPG}{Tag
0CPG}]{stacks-project} via extension of scalars, where \((-)^\gp\) denotes the
Grothendieck group associated to a monoid and}
  \bigl(\Div^\eff(X)\bigr)^\gp_\kk &\coloneqq
  \bigl(\Div^\eff(X)\bigr)^\gp \otimes_\ZZ \kk.
\end{align*}
\begin{definition}[{see \citeleft\citen{FM}\citemid Definition 1.1\citepunct
  \citen{KMM87}\citemid Definition 0-1-3\citeright}]
  Let $X$ be a ringed site.
  For $\kk \in \{\QQ,\RR\}$, a \textsl{$\kk$-invertible sheaf} is an element of
  $\Pic_\kk(X)$.
  We will usually write $\Pic_\kk(X)$ additively, in which case we denote
  the invertible sheaves associated to elements $D \in \Pic_\ZZ(X) = \Pic(X)$
  and elements $D \in \Div_\ZZ(X)$ (for ringed spaces $X$) or $D \in
  \Div_\ZZ^\eff(X)$ (for algebraic spaces $X$) by $\cO_X(D)$.
  We say that $D,D' \in \Div_\kk(X)$ are
  \textsl{$\kk$-linearly equivalent} if their images in $\Pic_\kk(X)$
  are equal.
\end{definition}
When $X$ is a locally Noetherian scheme, these exact sequences fit into the
commutative diagram
\begin{equation}\label{eq:princses}
  \begin{tikzcd}
    0 \rar & \Princ_\kk(X) \rar\dar[equal] & \Div_\kk(X)
    \rar{l \otimes_\ZZ \kk}\dar{\cyc \otimes_\ZZ \kk}
    & \Pic_\kk(X) \dar\\
    & \Princ_\kk(X) \rar & \WDiv_\kk(X) \rar & \Cl_\kk(X) \rar & 0
  \end{tikzcd}
\end{equation}
for $\kk \in \{\ZZ,\QQ,\RR\}$
by definition of the divisor class group $\Cl(X)$ in \cite[(21.6.7)]{EGAIV4},
where
\[
  \Cl_\kk(X) \coloneqq \Cl(X) \otimes_\ZZ \kk.
\]
\begin{definition}[see {\cite[Definition 0-1-3]{KMM87}}]
  Let $X$ be a locally Noetherian scheme or an integral locally Noetherian
  algebraic space over a scheme $S$.
  For $\kk \in \{\QQ,\RR\}$, we say that $D,D' \in \WDiv_\kk(X)$ are
  \textsl{$\kk$-linearly equivalent} if their images in $\Cl_\kk(X)$
  are equal.
\end{definition}
We will need to know when the exact sequence in the top row of
\eqref{eq:princses} is also right exact.
\begin{remark}\label{rem:lbcomefromcartdiv}
  In the top exact sequence of \eqref{eq:princses}, the map \(l \otimes_\ZZ
  \kk\) is surjective in the following cases for
  $\kk = \ZZ$, and hence also for $\kk \in \{\QQ,\RR\}$ by flatness.
  \begin{enumerate}[label=$(\roman*)$]
    \item $X$ is a locally Noetherian scheme and $\Ass(\cO_X)$ is contained in
      an open affine subscheme of $X$ \cite[Proposition 21.3.4$(a)$]{EGAIV4}.
      This holds for example when $X$ is Noetherian and has an ample invertible
      sheaf, in particular when $X$ is quasi-projective over a Noetherian ring
      \cite[Corollaire 21.3.5]{EGAIV4}.
    \item $X$ is a reduced scheme whose set of irreducible components is locally
      finite \cite[Proposition 21.3.4$(b)$]{EGAIV4}.
  \end{enumerate}
\end{remark}
\begin{lemma}\label{lem:kcartierdefs}
  Let $X$ be a locally Noetherian scheme satisfying one of the hypotheses in
  Remark \ref{rem:lbcomefromcartdiv}.
  For $\kk \in \{\ZZ,\QQ,\RR\}$, a $\kk$-Weil divisor $D$ lies in the image
  of $\cyc \otimes_\ZZ \kk$ if and
  only if the class of $D$ in $\Cl_\kk(X)$ lies in the image of
  the map
  \[
    \Pic_\kk(X) \longrightarrow \Cl_\kk(X).
  \]
\end{lemma}
\begin{proof}
  The implication $\Rightarrow$ holds by the commutativity of the diagram in
  \eqref{eq:princses}.
  Conversely, suppose the class of $D$ in $\Cl_\kk(X)$ lies in
  the image of $\Pic_\kk(X)$.
  Since $l \otimes_\ZZ \kk$ is surjective, there exists an element
  $\tilde{D} \in \Div_\kk(X)$ such that
  \[
    (\cyc \otimes_\ZZ \kk)\bigl(\tilde{D}\bigr) \sim_\kk D.
  \]
  By the exactness of the bottom row in \eqref{eq:princses}, we therefore have
  an element $D' \in \Princ_\kk(X)$ such that
  \[
    (\cyc \otimes_\ZZ \kk)\bigl(\tilde{D} + D'\bigr) = D,
  \]
  and hence $D$ is $\kk$-Cartier.
\end{proof}
If $X$ is an integral locally Noetherian algebraic space, then
by \cite[\href{https://stacks.math.columbia.edu/tag/0ENV}{Tag
0ENV}]{stacks-project}, there is a map
\begin{equation}\label{eq:stacks0EPW}
  \Pic(X) \longrightarrow \Cl(X)
\end{equation}
that coincides with the corresponding map in \eqref{eq:princses} when $X$ is a
scheme.
We will use this map to define what it means for a $\kk$-Weil divisor to be
$\kk$-Cartier on an algebraic space.
\begin{definition}[see {\cite[Definition 1.3.4]{VPthesis}}]
  Let $X$ be an integral normal locally Noetherian algebraic space over a
  scheme $S$, in which case
  the map \eqref{eq:stacks0EPW} is injective
  \cite[\href{https://stacks.math.columbia.edu/tag/0EPX}{Tag
  0EPX}]{stacks-project}.
  A Weil divisor $D$ \textsl{is Cartier} if $D$ lies in the image of the map
  \eqref{eq:stacks0EPW}.
  For $\kk \in \{\QQ,\RR\}$, a $\kk$-Weil divisor $D$ \textsl{is $\kk$-Cartier}
  if $D$ lies in the image of the map
  \[
    \Pic_\kk(X) \longrightarrow \Cl_\kk(X)
  \]
  obtained from \eqref{eq:stacks0EPW} via extension of scalars.
  By Lemma \ref{lem:kcartierdefs}, this definition matches that in
  Definition \ref{def:kcartier} when $X$ is a scheme.
\end{definition}
\begin{convention}[see {\cite[Definition 1.3.4]{VPthesis}}]
  \label{convention:factorialityforspaces}
  Let $X$ be an integral normal locally Noetherian algebraic space over a scheme
  $S$.
  We say that $X$ is \textsl{locally factorial} (resp.\ is
  \textsl{$\QQ$-factorial})
  if $\Pic(X) \to \Cl(X)$ (resp.\ $\Pic_\QQ(X) \to \Cl_\QQ(X)$)
  is an isomorphism.
  In this case, we will say ``divisor'' (resp.\ $\QQ$-divisor) instead of
  ``Weil divisor'' (resp.\ ``$\QQ$-Weil divisor'').
\end{convention}
\begin{remark}
  Convention \ref{convention:factorialityforspaces} is chosen
  to work around the fact that the
  property of being locally factorial or $\QQ$-factorial is not
  \'etale local.
  See \citeleft\citen{Kaw88}\citemid p.\ 104\citepunct \citen{SGA2}\citemid
  Expos\'e XIII, note de l'\'editeur (15) on p.\ 150\citepunct
  \citen{BGS}\citemid p.\ 1\citeright.
\end{remark}

\subsection{Linear systems}
We now define
linear systems and their corresponding notions for $\QQ$- and
$\RR$-coefficients.
\begin{definition}[see
  {\citeleft\citen{KMM87}\citemid Definition p.\ 298\citepunct
  \citen{CL12}\citemid p.\ 2419\citepunct
  \citen{McK17}\citemid Definition 2.2\citeright}]
  \label{def:linearsystem}
  Let $X$ be a normal locally Noetherian scheme or an integral normal locally
  Noetherian algebraic space over a scheme $S$.
  We then define linear equivalence and $\kk$-linear equivalence for
  Weil divisors and $\kk$-Weil divisors using the cycle map and its extensions
  of scalars in \eqref{eq:divwdivdiag}.
  The \textsl{linear system} associated to a Weil divisor $D$ is
  \begin{align*}
    \lvert D \rvert &\coloneqq \Set[\big]{C \in \WDiv(X) \given C \ge 0\
    \text{and}\ C \sim D},
    \intertext{and for $\kk \in \{\QQ,\RR\}$, the \textsl{$\kk$-linear system}
    associated to a $\kk$-Weil divisor $D$ is}
    \lvert D \rvert_\kk &\coloneqq \Set[\big]{C \in \WDiv_\kk(X) \given C \ge 0\
    \text{and}\ C \sim_\kk D}.
  \end{align*}
\end{definition}
We can now state the main result that allows us to pass between sheaf-theoretic
language and the language of linear systems on schemes.
\begin{citedprop}[{\citeleft\citen{Har94}\citemid Proposition
  2.9\citepunct\citen{Har07}\citemid Remark 2.9\citeright}]\label{prop:har29}
Let $X$ be a normal Noetherian scheme, and
  consider a Weil divisor $D$ on $X$.
  Then, there is a bijection
  \[
    \lvert D \rvert \longleftrightarrow
    \biggl\{
      \begin{tabular}{@{}c@{}}
        nondegenerate global sections\\
        $s \in H^0\bigl(X,\cO_X(D)\bigr)$
      \end{tabular}
    \biggr\}\bigg/H^0(X,\cO_X^*).
  \]
\end{citedprop}
Here, $\cO_X(D)$ is the sheaf associated to the Weil divisor $D$, which can be
defined as $j_*\cO_U(D_{\vert U})$, where $U$
is the open subset where $D_{\vert U}$ is Cartier, and $j\colon U
\hookrightarrow X$ is the canonical open embedding (see
\cite[Definition on p.\ 301 and Proposition 2.7]{Har94}).
A global section $s \in H^0(X,\cO_X(D))$ is \textsl{nondegenerate} if it
is nonzero after localizing at the generic points of irreducible components of
$X$ \cite[Definition on p.\ 304]{Har94}.\medskip

\par We also prove the following result about the relationship between $\QQ$- and
$\RR$-linear systems of a $\QQ$-Weil divisor.
\begin{lemma}
  \label{lem:RatlIsDense}
  Let $X$ be a normal locally Noetherian scheme or an integral normal locally
  Noetherian algebraic space over a scheme $S$,
  and consider a $\QQ$-Weil divisor $D$ on
  $X$.
  Then, $\lvert D\rvert_{\mathbf Q}$ is dense in $\lvert D\rvert_{\mathbf R}$ in the following sense: 
For each $\sum a_iE_i\in \lvert D\rvert_{\mathbf R}$ where $a_i$ are real numbers and $E_i$ are prime divisors, 
there exist sequences of rational numbers $(a^j_i)_j$ such that
\[
  \lim_{j\to\infty} a_i^j=a_i \qquad \text{and} \qquad \sum_i a_i^jE_i\in \lvert D\rvert_{\mathbf Q}
\]
for all $i$.
\end{lemma}
\begin{proof}
We adapt the proofs of
\citeleft\citen{BCHM10}\citemid Lemma 3.5.3\citepunct \citen{CL12}\citemid Lemma
2.3\citeright.
Let
\[
  V=\mathbf{Q} \cdot D+\Span\{E_i\}\subseteq \WDiv_{\mathbf Q}(X),
\]
and let $V_0$ be the subspace of $V$ consisting of rational combinations of principal divisors. 
Then, $V_{\mathbf R} \coloneqq V \otimes_\QQ \RR$ is a (finite-dimensional)
subspace of $\WDiv_{\mathbf R}(X)$, and $(V_0)_{\mathbf R} \coloneqq V_0
\otimes_\QQ \RR$ is the subspace of $V_{\mathbf R}$ consisting of real combinations of principal divisors.
Let $\pi\colon V_{\mathbf R}\to V_{\mathbf R}/(V_0)_{\mathbf R}$ be the canonical projection map. The subset
\[
  \Set[\bigg]{\sum_i b_iE_i\in V_{\mathbf R} \given b_i\geq 0\ \text{and}\ 
  \pi\biggl(\sum_i b_iE_i\biggr)=\pi(D)}
\]
is cut out from $V$ by rational hyperplanes and half-spaces, and it contains the real point $\sum_i a_iE_i$.
The result now follows.
\end{proof}
\section{Positivity, the theorem of the base, cones,
and\texorpdfstring{\except{toc}{\\}}{} Kleiman's criterion
for ampleness}
\subsection{Relative positivity conditions}\label{sect:relativeampleness}
We define relative ampleness conditions for $\kk$-invertible sheaves and
$\kk$-Cartier divisors for $\kk \in
\{\ZZ,\QQ,\RR\}$.
\begin{definition}[see {\citeleft\citen{EGAII}\citemid D\'efinition
  4.6.1\citepunct \citen{KMM87}\citemid Definition
  0-1-4\citepunct \citen{BCHM10}\citemid Definition 3.1.1\citepunct
  \citen{CT20}\citemid \S2.1.1\citepunct
  \citen{FM}\citemid Definition 2.1\citepunct \citen{stacks-project}\citemid
  \href{https://stacks.math.columbia.edu/tag/0D31}{Tag 0D31}\citeright}]
  \label{def:relativeampleness}
  Let $\pi\colon X \to Z$ be a morphism of schemes (resp.\ algebraic spaces
  over a scheme $S$), and let $\sL$ be an invertible
  sheaf on $X$.
  \begin{enumerate}[label=$(\roman*)$,ref=\roman*]
    \item Suppose $\pi$ is quasi-compact (resp.\ representable).
      We say that $\sL$ is \textsl{$\pi$-ample} if 
      there exists an affine open cover $Z = \bigcup_i U_i$ such that
      $\sL\rvert_{\pi^{-1}(U_i)}$ is ample for all $i$ (resp.\ if for every
      morphism $Z' \to Z$ where $Z'$ is a scheme, the pullback of $\sL$ to $Z'
      \times_Z X$ is $\pi$-ample).
    \item We say that $\sL$ is \textsl{$\pi$-generated} if the
      adjunction morphism $\pi^*\pi_*\sL \to \sL$ is surjective.
    \item We say that $\sL$ is \textsl{$\pi$-semi-ample} if there exists an integer $n> 0$ such that $\sL^{\otimes n}$ is $\pi$-generated.
  \end{enumerate}
  When $X$ is a scheme,
  we can extend these definitions to Cartier divisors $L$ on $X$ by asking
  that their associated invertible sheaves $\cO_X(L)$ satisfy these conditions.
  \par Now suppose that $D$ is a $\kk$-invertible sheaf
  on $X$ for $\kk \in \{\QQ,\RR\}$.
  We say that $D$ is \textsl{$\pi$-ample}
  if $D$ is a
  finite nonzero $\kk_{>0}$-linear combination of $\pi$-ample
  invertible sheaves on $X$.
  We say that $D$ is \textsl{$\pi$-semi-ample} if $D$ is
  a finite $\kk_{\ge0}$-linear combination of $\pi$-semi-ample
  invertible sheaves on $X$.
  We extend these definitions to elements $D \in \Div_\kk(X)$ (resp.\
  $\Div_\kk^\eff(X)$) by asking that their images in $\Pic_\kk(X)$ satisfy these
  conditions.
\end{definition}

To define $\pi$-numerically trivial and $\pi$-nef
$\kk$-invertible sheaves or $\kk$-Cartier divisors, we recall
some background on intersection theory for proper morphisms.
Let $\pi\colon X \to Z$ be a proper morphism of
locally Noetherian algebraic spaces over a scheme $S$.
Let $z \in \lvert Z \rvert$ be a point, and consider a subspace $Y
\subseteq \pi^{-1}(z)$ that is closed in $\pi^{-1}(z)$.
We can consider \(Y\) as a scheme proper over \(\kappa(z)\) and
define the intersection
numbers
\[
  \bigl(\sL_1 \cdot \sL_2 \cdots \sL_m \cdot Y\bigr) \in \ZZ
\]
to be the coefficient of \(n_1n_2\cdots n_m\) in the numerical polynomial
\begin{align*}
  \MoveEqLeft[5]
  \chi\Bigl(Y,\sL_1^{\otimes n_1} \otimes_{\cO_X} \sL_2^{\otimes n_2}
  \otimes_{\cO_X} \cdots \otimes_{\cO_X} \sL_m^{\otimes n_m}\bigr\rvert_Y\Bigr)\\
  &\coloneqq \sum_{i=0}^{\dim(Y)} h^i\Bigl(Y,\sL_1^{\otimes n_1} \otimes_{\cO_X}
  \sL_2^{\otimes n_2} \otimes_{\cO_X} \cdots \otimes_{\cO_X}
  \sL_m^{\otimes n_m}\bigr\rvert_Y\Bigr)
\end{align*}
for invertible sheaves $\sL_i$ on $X$, where $m \ge \dim(Y)$.
See
\citeleft\citen{stacks-project}\citemid
\href{https://stacks.math.columbia.edu/tag/0EDF}{Tag 0EDF}\citeright.
By linearity
\citeleft\citen{stacks-project}\citemid
\href{https://stacks.math.columbia.edu/tag/0EDH}{Tag 0EDH}\citeright,
we can extend this definition to $\kk$-invertible sheaves
for $\kk\in\{\ZZ,\QQ,\RR\}$
(see \cite[Chapter VI, Definition-Corollary 2.7.4]{Kol96}).
When $X$ is a scheme, we can also extend this definition to $\kk$-Cartier
divisors using the group maps
\[
  l \otimes_\ZZ \kk\colon \Div_\kk(X) \longrightarrow \Pic_\kk(X)
\]
from \eqref{eq:divtopic}.
In this case, we denote the intersection product by
$(D_1 \cdot D_2 \cdots D_n \cdot Y)$, where $D_i \in
\Div_\kk(X)$ for all $i$.
\par We use this intersection product to define $\pi$-nef and
$\pi$-numerically trivial $\kk$-invertible sheaves or $\kk$-Cartier divisors.
\begin{definition}[see {\citeleft\citen{Kle66}\citemid pp.\ 334--335\citepunct
  \citen{KMM87}\citemid Definition 0-1-1\citepunct
  \citen{Kol90}\citemid p.\ 236\citepunct
  \citen{Kee03}\citemid Definition 2.9\citepunct
  \citen{BCHM10}\citemid Definition 3.1.1\citepunct
  \citen{CT20}\citemid \S2.1.1\citepunct
  \citen{VPthesis}\citemid Definition 1.3.8\citeright}]
  \label{def:RelNeronSeveri}
Let $\pi\colon X\to Z$ be a proper morphism of
algebraic spaces over a scheme $S$.
Let $\kk\in\{\ZZ,\QQ,\RR\}$.
\begin{enumerate}[label=$(\roman*)$]
  \item An element $D\in\Pic_{\kk}(X)$ is \textsl{$\pi$-nef} if, for every point
    $z \in \lvert Z \rvert$ and for every
    integral one-dimensional subspace $C \subseteq \pi^{-1}(z)$
    that is closed in $\pi^{-1}(z)$, we have $(D \cdot C)\geq
    0$.
    If $Z = \Spec(k)$ for a field $k$, we just say $D$ is \textsl{nef}.
  \item An element $D \in \Pic_\kk(X)$ is \textsl{$\pi$-numerically trivial} if
    both $D$ and $-D$ are $\pi$-nef.
    We denote by $N^1(X/Z)$ the quotient of $\Pic(X)$ by the subgroup of
    numerically trivial elements, and set
    \[
      N^1(X/Z)_\kk \coloneqq N^1(X/Z) \otimes_\ZZ \kk
    \]
    for $\kk \in \{\QQ,\RR\}$.
    If $Z = \Spec(k)$ for a field $k$, we just say $D$ is \textsl{numerically
    trivial}.
\end{enumerate}
If $X$ is a scheme, we extend these definitions to elements $D \in \Div_\kk(X)$
by asking that their images in $\Pic_\kk(X)$ satisfy these conditions.
By definition, this only depends on the class $[D] \in N^1(X/Z)_\kk$.
\end{definition}
We note that if $Z$ is not decent in the sense of 
\cite[\href{https://stacks.math.columbia.edu/tag/03I8}{Tag 03I8}]{stacks-project},
then the residue field of $z \in \lvert Z \rvert$ is not necessarily well-defined \cite[\href{https://stacks.math.columbia.edu/tag/02Z7}{Tag
02Z7}]{stacks-project}.
However, the condition that $(D \cdot C) \ge 0$ does not depend on the choice of
the representative $\Spec(K) \to Z$ of the point $z \in \lvert Z \rvert$ as
defined in \cite[\href{https://stacks.math.columbia.edu/tag/03BT}{Tag
03BT}]{stacks-project} by flat base change
\cite[\href{https://stacks.math.columbia.edu/tag/073K}{Tag
073K}]{stacks-project}.\medskip
\par We now prove some fundamental properties of nefness and numerical triviality.
Many of these results are known for proper morphisms of schemes or for algebraic
spaces that are proper over a field, but as far as we are aware they are new for
proper morphisms of algebraic spaces.
\begin{lemma}[cf.\ {\citeleft\citen{Kle66}\citemid Chapter I, \S4, Proposition
  1\citepunct \citen{Kee03}\citemid Lemma 2.17\citepunct
  \citen{CLM}\citemid Lemma 3.3\citeright}]\label{lem:nefpullback}
  Let $S$ be a scheme.
  Let
  \[
    \begin{tikzcd}[column sep=0.75em]
      X' \arrow{rr}{f}\arrow{dr}[near start,swap]{\pi'}
      & & X\arrow{dl}[near start]{\vphantom{\pi'}\smash{\pi}}\\
      & Z
    \end{tikzcd}
  \]
  be a commutative diagram of algebraic spaces over $S$, where $\pi$ and $\pi'$
  are proper.
  Let $D \in \Pic_\kk(X)$ for $\kk \in \{\ZZ,\QQ,\RR\}$.
  \begin{enumerate}[label=$(\roman*)$,ref=\roman*]
    \item\label{lem:nefpullbackalways}
      If $D$ is $\pi$-nef (resp.\ $\pi$-numerically trivial), then $f^*D$ is
      $\pi'$-nef (resp.\ $\pi'$-numerically trivial).
    \item\label{lem:nefpullbacksurjective}
      If $f$ is surjective and $f^*D$ is $\pi'$-nef (resp.\
      $\pi$-numerically trivial), then $D$ is $\pi$-nef (resp.\
      $\pi$-numerically trivial).
  \end{enumerate}
\end{lemma}
\begin{proof}
  By definition, it suffices to consider the nefness (resp.\ numerical
  triviality) of $D$ when $Z$ is the spectrum of a field.
  The statements for numerical triviality follow from those for nefness applied
  to $D$ and $-D$, and hence it suffices to show $(\ref{lem:nefpullbackalways})$
  and $(\ref{lem:nefpullbacksurjective})$ for nefness.
  \par For $(\ref{lem:nefpullbackalways})$, let $C' \subseteq X'$ be an integral
  one-dimensional closed subspace.
  By the projection formula
  \cite[\href{https://stacks.math.columbia.edu/tag/0EDJ}{Tag
  0EDJ}]{stacks-project}, we have
  \begin{align*}
    (f^*D \cdot C') &= \deg\bigl(C' \to f(C')\bigr)\bigl(D \cdot f(C')\bigr)
    \ge 0.
  \intertext{\endgraf For $(\ref{lem:nefpullbacksurjective})$, let $C \subseteq X$ be an
  be an integral one-dimensional closed subspace.
  By \cite[Lemma 3.2]{CLM}, there is an integral one-dimensional closed subspace
  $C' \subseteq X'$ such that $C = f(C')$.
  By the projection formula again
  \cite[\href{https://stacks.math.columbia.edu/tag/0EDJ}{Tag
  0EDJ}]{stacks-project}, we have}
    (D \cdot C) &= \bigl(\deg(C' \to C)\bigr)^{-1}(f^*D \cdot C') \ge 0.\qedhere
  \end{align*}
\end{proof}
We show that nefness and numerical triviality behave well under base
change.
\begin{lemma}[cf.\ {\citeleft\citen{Kee03}\citemid Lemma
  2.18\citeright}]\label{lem:nefbasechange}
  Let $S$ be a scheme.
  Consider a Cartesian diagram
  \[
    \begin{tikzcd}
      X' \rar{f}\dar[swap]{\pi'} & X\dar{\pi}\\
      Z' \rar{g} & Z
    \end{tikzcd}
  \]
  of algebraic spaces over $S$ where $\pi$ is proper.
  Let $D \in \Pic_\kk(X)$.
  \begin{enumerate}[label=$(\roman*)$,ref=\roman*]
    \item\label{lem:nefbasechangepullback}
      If $D$ is $\pi$-nef (resp.\ $\pi$-numerically trivial), then $f^*D$ is
      $\pi'$-nef (resp.\ $\pi'$-numerically trivial).
    \item\label{lem:nefbasechangeconverse}
      Suppose that $g$ is surjective.
      If $\pi^*D$ is $\pi'$-nef (resp.\ $\pi'$-numerically trivial), then $D$ is
      $\pi$-nef (resp.\ $\pi$-numerically trivial).
  \end{enumerate}
\end{lemma}
\begin{proof}
  As in the proof of Lemma \ref{lem:nefpullback}, it suffices to show the
  statement for nefness.
  By transitivity of fibers, it
  suffices to consider the case when $Z = \Spec(k)$ and $Z' = \Spec(k')$ for
  a field extension $k \subseteq k'$.\smallskip
  \par We first show $(\ref{lem:nefbasechangepullback})$.
  By the weak version of Chow's lemma in
  \cite[\href{https://stacks.math.columbia.edu/tag/089J}{Tag
  089J}]{stacks-project} (see Lemma \ref{lem:weakchow}),
  there exists a proper surjective morphism $\mu\colon
  Y \to X$ from a scheme $Y$ that is projective over $k$.
  We then consider the following Cartesian diagram:
  \[
    \begin{tikzcd}
      Y' \rar{f'}\dar[swap]{\mu'} & Y\dar{\mu}\\
      X' \rar{f} & X\mathrlap{.}
    \end{tikzcd}
  \]
  Then, we know that $\mu^*D$ is
  nef by Lemma \ref{lem:nefpullback}$(\ref{lem:nefpullbackalways})$.
  Now since $Y$ is a projective scheme over $k$, we know that choosing an
  ample invertible sheaf $A$ on $Y$, the $\RR$-invertible sheaf
  $\mu^*D + \varepsilon A$ is ample for every $\varepsilon >
  0$ by Kleiman's criterion for ampleness for projective schemes
  \cite[Chapter VI, Theorem 2.19]{Kol96}.
  Then,
  \[
    (\mu \circ f')^*D+\varepsilon\,f^{\prime*}A = (f \circ
    \mu')^*D+\varepsilon\,f^{\prime*}A
  \]
  nef for every $\varepsilon > 0$.
  Taking the limit as $\varepsilon \to 0$, we see that 
  $(\mu \circ f')^*D = (f \circ \mu')^*D$ is nef by
  \cite[Theorem 3.9]{Kee03}.
  Finally, we see that $f^*D$ is nef by Lemma
  \ref{lem:nefpullback}$(\ref{lem:nefpullbacksurjective})$.\smallskip
  \par For $(\ref{lem:nefbasechangeconverse})$, let $C \subseteq X$
  be an integral one-dimensional subspace.
  Let $C'_i$ be the irreducible components of $C' \coloneqq C \otimes_k k'$
  with reduced structure and
  geometric generic point $\bar{x}_i$, and let
  \[
    m_i = \Length_{\cO_{X \otimes_k k',\bar{x}_i}}\bigl( \cO_{C'_i,\bar{x}_i}
    \bigr).
  \]
  Then, we have
  \[
    (D \cdot C) = (f^*D \cdot C') = \sum_i m_i (f^*D \cdot C'_i) \ge 0
  \]
  where the first equality follows from flat base change
  \cite[\href{https://stacks.math.columbia.edu/tag/073K}{Tag
  073K}]{stacks-project}, the second equality is
  \cite[\href{https://stacks.math.columbia.edu/tag/0EDI}{Tag
  0EDI}]{stacks-project}, and the last inequality is by the assumption that
  $f^*D$ is nef.
\end{proof}

We note that nefness can be detected at closed points
$z \in \lvert Z \rvert$ under some additional assumptions.
Below, the decency assumption in $Z$ allows us to make sense of the residue
field $\kappa(z)$ at a point $z \in \lvert Z \rvert$.
See \cite[\href{https://stacks.math.columbia.edu/tag/0EMW}{Tag
0EMW}]{stacks-project} for the definition and see
\cite[\href{https://stacks.math.columbia.edu/tag/02Z7}{Tag
02Z7}]{stacks-project} for an example where the residue field at a point cannot
be defined.
\begin{lemma}[cf.\ {\citeleft\citen{Kee03}\citemid Lemma 2.18(1)\citepunct
  \citen{CT20}\citemid Lemma 2.6\citeright}]\label{lem:NefAgainstNonClosedContracted}
Let $\pi\colon X\to Z$ be a proper morphism of
algebraic spaces over a scheme $S$.
Suppose that $Z$ is quasi-compact and decent,
or that $Z$ is a locally
Noetherian scheme.
Let $D \in \Pic_\kk(X)$.
Then, the following conditions are equivalent.
\begin{enumerate}[label=$(\roman*)$,ref=\roman*]
  \item\label{lem:NefAgainstNonClosedContracteddef} $D$ is $\pi$-nef (resp.\
    $\pi$-numerically trivial).
  \item \label{lem:NefAgainstNonClosedContractedclosed}
    For every closed point $z \in \lvert Z \rvert$ and every integral
    one-dimensional subspace $C \subseteq \pi^{-1}(z)$ that is
    closed in $\pi^{-1}(z)$, we have $(D \cdot C) \ge 0$ (resp. $(D \cdot C) = 0$).
\end{enumerate}
\end{lemma}
\begin{proof}
  We have $(\ref{lem:NefAgainstNonClosedContracteddef})
  \Rightarrow
  (\ref{lem:NefAgainstNonClosedContractedclosed})$ by definition, and hence it
  suffices to show
  $(\ref{lem:NefAgainstNonClosedContractedclosed}) \Rightarrow
  (\ref{lem:NefAgainstNonClosedContracteddef})$.
  As in the proof of Lemma \ref{lem:nefpullback}, it suffices to show the
  statement for nefness.\smallskip
\par We want to show that for every $z \in \lvert Z \rvert$, the pullback
$D\rvert_{\pi^{-1}(z)}$ is nef over $K$.
We first show that $z$ specializes to a closed point $z_0 \in \lvert Z \rvert$.
If $Z$ is quasi-compact and decent, $\lvert Z \rvert$ is quasi-compact
and Kolmogorov \cite[\href{https://stacks.math.columbia.edu/tag/03K3}{Tag
03K3}]{stacks-project}, and hence every point $z \in \lvert Z \rvert$
specializes to a closed point \cite[(2.1.2)]{EGAInew}.
When $Z$ is a locally Noetherian scheme, then every point $z \in Z$ specializes
to a closed point as well
\cite[\href{https://stacks.math.columbia.edu/tag/02IL}{Tag
02IL}]{stacks-project}.
\par Now let $z \rightsquigarrow z_0$ be a specialization to a closed point in
$\lvert Z \rvert$, which exists by the previous paragraph.
By \cite[\href{https://stacks.math.columbia.edu/tag/0BBP}{Tag 0BBP} and
\href{https://stacks.math.columbia.edu/tag/03IL}{Tag
03IL}]{stacks-project}, there is an \'etale morphism $U \to Z$ from an affine
scheme $U$ with points $u \rightsquigarrow u_0$ mapping to $z \rightsquigarrow
z_0$ such that the field extension $\kappa(z_0) \hookrightarrow \kappa(u_0)$ is
an isomorphism.
We note that $\kappa(x) \subseteq \kappa(u)$ is a field extension, and hence by
Lemma \ref{lem:nefbasechange}$(\ref{lem:nefbasechangeconverse})$ it suffices to
show that the pullback of $D$ to $X \times_Z \Spec(\kappa(u))$ is nef over
$\Spec(\kappa(u))$.
By Lemma \ref{lem:nefpullback} and the weak version of Chow's lemma in
\cite[\href{https://stacks.math.columbia.edu/tag/089J}{Tag
089J}]{stacks-project} (see Lemma \ref{lem:weakchow}),
we may replace $X$ by a proper surjective cover $Y \to
X$ that is a projective scheme over $U$.
\par By \cite[Chapter VI, \S1, no.\ 2,
Corollary to Theorem 2]{BouCA}, we can find a valuation ring
$(R,\fm)$ and
a morphism $\Spec(R) \to U$ such that the generic point of $\Spec(R)$ maps to
$u$ and the closed point of $\Spec(R)$ maps to $u_0$, and such that the field
extension $\kappa(u) \hookrightarrow \Frac(R)$ is an isomorphism.
Let
\[
  C \subseteq X \times_Z \Spec\bigl(\Frac(R)\bigr)
\]
be an integral closed
one-dimensional subspace.
Taking the scheme-theoretic closure
\[
  \overline{C} \subseteq X \times_Z \Spec(R)
\]
of $C$ in $X \times_Z \Spec(R)$, we obtain
a flat family of
closed one-dimensional subspaces in $X \times_Z \Spec(R)$
over $R$ because the pushforward of the structure sheaf of $\overline{C}$ to
$\Spec(R)$
is torsion-free \citeleft\citen{EGAInew}\citemid Proposition 8.4.5\citepunct
\citen{BouCA}\citemid Chapter VI, \S4, no.\ 1, Lemma
1\citeright.
Since the residue field of $R$ is a field extension of $\kappa(z_0) \cong
\kappa(u_0)$, we see that the restriction of $D$ to $X \times_Z \Spec(R/\fm)$ is
nef over $\Spec(R/\fm)$ by Lemma
\ref{lem:nefbasechange}$(\ref{lem:nefbasechangepullback})$.
Thus, we have $(D' \cdot C) \geq 0$ by the
invariance of intersection numbers in flat families \cite[Proposition
B.18]{Kle05}.
\end{proof}
On the other hand, nefness cannot be checked at closed points without some
assumptions on \(Z\).
\begin{example}
  Let \(Z\) be an integral scheme without closed points, let \(X = \PP^1_Z\) with
  projection morphism \(\pi\colon X \to Z\), and let \(D = \cO(-1)\).
  Then, \(D\) is not \(\pi\)-nef over the generic point \(\eta\) of \(X\), but
  satisfies the condition in Lemma
  \ref{lem:NefAgainstNonClosedContracted}\((\ref{lem:NefAgainstNonClosedContractedclosed})\).
  \par For an explicit example of such a scheme \(Z\), we recall Schwede's
  construction in \cite[pp.\ 169--170]{Sch05}.
  Let \(k\) be a field, consider the ring
  \[
    A' = k[x_1,x_2,\ldots]\biggl[\frac{x_1}{x_2},\frac{x_1}{x_2^2},\ldots\biggr]
    \biggl[\frac{x_2}{x_3},\frac{x_2}{x_3^2},\ldots\biggr]\ldots,
  \]
  and set \(A = A'_{(x_1,x_2,\ldots)}\).
  The ring \(A\) is a valuation ring with order group \(G = \ZZ^{\oplus
  \omega}\), the direct sum of countably many copies of \(\ZZ\), considered
  with the lexicographic order.
  By \cite[Theorem 4.6]{Sch05}, the scheme
  \[
    Z = \Spec(A) - \bigl\{(x_1,x_2,\ldots)\bigr\}
  \]
  has no closed points.
\end{example}

Since nefness can be detected over closed points in many cases,
we define the following.
\begin{definition}[see {\citeleft\citen{Kle66}\citemid p.\ 335\citepunct
  \citen{KMM87}\citemid p.\ 291\citepunct
  \citen{Kee03}\citemid Definition 2.8\citepunct
  \citen{VPthesis}\citemid Definitions 1.3.19 and 1.3.20 and p.\ 15\citeright}]
  \label{def:intersectionWithContractedCurves}
Let $\pi\colon X\to Z$ be a proper morphism of
algebraic spaces over a scheme $S$, such that $Z$ is either quasi-compact
and decent,
or a locally Noetherian scheme.
A closed subspace $Y \subseteq X$ is \textsl{$\pi$-contracted} if $\pi(Y)$ is a
zero-dimensional (closed) subspace of $Z$.
A \textsl{$\pi$-contracted curve} is a $\pi$-contracted closed subspace that is
integral and of dimension one.
\par Now suppose that $X$ is quasi-compact.
We denote by $Z_1(X/Z)$ the free Abelian group generated by $\pi$-contracted
curves, and set
\begin{align*}
  Z_1(X/Z)_\kk &\coloneqq Z_1(X/Z) \otimes_\ZZ \kk
  \intertext{for
  $\kk\in\{\QQ,\RR\}$.
  An element $\beta\in Z_1(X/Z)_{\kk}$ is
  \textsl{$\pi$-numerically trivial} if $(D\cdot\beta)=0$ for all
  $D\in\Pic_\kk(X)$.
  We denote by $N_1(X/Z)$ the quotient of
  $Z_1(X/Z)$ by the subgroup of
  numerically trivial elements, and set}
  N_1(X/Z)_{\kk} &\coloneqq N_1(X/Z) \otimes_\ZZ \kk
\end{align*}
for
$\kk\in\{\QQ,\RR\}$.
\end{definition}
\subsection{Theorem of the base}
As in the absolute case, the modules $N^1(X/Z)_\kk$ and $N_1(X/Z)_\kk$ are
finitely generated.
This statement is called the theorem of the base.
This theorem allows us to define cones in $N^1(X/Z)_\kk$ and $N_1(X/Z)_\kk$
corresponding to the various positivity notions in
\S\ref{sect:relativeampleness}.\medskip
\par To prove the theorem of the base, we start with the following.
\begin{lemma}[cf.\ {\citeleft\citen{Kle66}\citemid Chapter IV, \S4, Proposition
  1\citepunct \citen{Kee03}\citemid Lemma
  2.20\citeright}]\label{lem:nefbasechangenotcartesian}
  Let $S$ be a scheme.
  Consider a commutative diagram
  \[
    \begin{tikzcd}
      X'\arrow{dr}{h}\arrow[bend right=15]{ddr}[swap]{\rho}
      \arrow[bend left=15]{drr}{f}\\
      & X \times_Z Z' \rar[swap]{g'}\dar{\vphantom{\pi}\smash{\pi'}} & X\dar{\pi}\\
      & Z' \rar{g} & Z
    \end{tikzcd}
  \]
  of algebraic spaces over $S$ where the square is Cartesian and
  $\pi$ and $\rho$ are proper.
  Let $D \in \Pic_\kk(X)$.
  \begin{enumerate}[label=$(\roman*)$,ref=\roman*]
    \item\label{lem:nefbasechangenotcartesianpullback}
      If $D$ is $\pi$-nef (resp.\ $\pi$-numerically trivial), then $f^*D$ is
      $\rho$-nef (resp.\ $\rho$-numerically trivial).
    \item\label{lem:nefbasechangenotcartesianconverse}
      Suppose that for every $z \in
      \lvert Z \rvert$ with representative $\Spec(K) \to Z$
      and every integral one-dimensional subspace $C \subseteq
      \pi^{-1}(z)$ that is closed in $\pi^{-1}(z)$, there exists a point $z' \in
      \lvert Z' \rvert$ with representative $\Spec(K') \to Z'$ mapping to $z$
      such that for every irreducible component
      \[
        C'_i \subseteq C' \coloneqq C \otimes_{K} K'
      \]
      with reduced structure,
      there exists an integral 
      one-dimensional subspace $C''_i \subseteq
      \rho^{-1}(z')$ that is closed in $\rho^{-1}(z')$ such that
      $h(C''_i) = C'_i$.
      If $f^*D$ is $\rho$-nef (resp.\ $\rho$-numerically trivial), then $D$ is
      $\pi$-nef (resp.\ $\pi$-numerically trivial).
    \item\label{lem:nefbasechangeconversenotcartesianclosed}
      Suppose that $Z$ either is quasi-compact and decent or is a locally Noetherian
      scheme.
      Suppose that for every $\pi$-contracted curve $C \subseteq X$, setting
      $z \in \lvert Z \rvert$ to be the image of $C$,
      there exists a point $z' \in \lvert Z' \rvert$ mapping to $z$
      such that for every irreducible component
      \[
        C'_i \subseteq C' \coloneqq C \otimes_{\kappa(z)} \kappa(z')
      \]
      with reduced structure,
      there exists an integral 
      one-dimensional subspace $C''_i \subseteq
      \rho^{-1}(z')$ that is closed in $\rho^{-1}(z')$ such that
      $h(C''_i) = C'_i$.
      If $f^*D$ is $\rho$-nef (resp.\ $\rho$-numerically trivial), then $D$ is
      $\pi$-nef (resp.\ $\pi$-numerically trivial).
  \end{enumerate}
\end{lemma}
\begin{remark}\label{rem:nefbasechange}
  The condition on curves in $(\ref{lem:nefbasechangenotcartesianconverse})$ and
  $(\ref{lem:nefbasechangeconversenotcartesianclosed})$ hold for example when $g = \id_S$ and
  $f$ is proper and surjective, which is the case proved in
  Lemma \ref{lem:nefpullback}, or when $g$ is surjective and $h = \id_{X'}$,
  which is the case proved in Lemma \ref{lem:nefbasechange}.
\end{remark}
\begin{proof}[Proof of Lemma \ref{lem:nefbasechangenotcartesian}]
  As in the proof of Lemma \ref{lem:nefpullback}, it suffices to show the
  statements for nefness.\smallskip
  \par For $(\ref{lem:nefbasechangenotcartesianpullback})$, we know that $g^{\prime*}D$ is
  $\pi'$-nef by Lemma
  \ref{lem:nefbasechange}$(\ref{lem:nefbasechangepullback})$.
  Thus, $f^*D$ is $\rho$-nef by Lemma
  \ref{lem:nefpullback}$(\ref{lem:nefpullbackalways})$.\smallskip
  \par For $(\ref{lem:nefbasechangenotcartesianconverse})$ (resp.\
  $(\ref{lem:nefbasechangeconversenotcartesianclosed})$),
  let $C \subseteq \pi^{-1}(z)$ be an integral one-dimensional subspace, where
  $z \in \lvert Z \rvert$ is a point (resp.\ a closed point).
  By definition (resp.\ by Lemma \ref{lem:NefAgainstNonClosedContracted}), it
  suffices to show that $(D \cdot C) \ge 0$.
  Let $C'_i$ be the irreducible components of $C'$ with reduced structure and
  geometric generic point $\bar{x}_i$, and let
  \begin{align*}
    m_i &= \Length_{\cO_{X \times_Z Z',\bar{x}_i}}\bigl( \cO_{C'_i,\bar{x}_i}
    \bigr).
  \intertext{Then, we have}
    (D \cdot C) = (g^{\prime*}D \cdot C')
    &= \sum_i m_i (g^{\prime*}D \cdot C'_i)\\
    &= \sum_i m_i \bigl( \deg(C''_i \to C'_i) \bigr)^{-1}
    (f^*D \cdot C''_i) \ge 0
  \end{align*}
  where the first equality follows from flat base change
  \cite[\href{https://stacks.math.columbia.edu/tag/073K}{Tag
  073K}]{stacks-project}, the second equality is
  \cite[\href{https://stacks.math.columbia.edu/tag/0EDI}{Tag
  0EDI}]{stacks-project}, the third equality is the projection formula
  \cite[\href{https://stacks.math.columbia.edu/tag/0EDJ}{Tag
  0EDJ}]{stacks-project}, and the last inequality is by the assumption that
  $\pi^*D$ is nef.
\end{proof}
We show that $N^1$ is compatible with pullbacks.
\begin{proposition}[cf.\ {\citeleft\citen{Kle66}\citemid Chapter IV, \S4,
  Proposition 1\citepunct
  \citen{Kee03}\citemid Lemma 3.1\citeright}]\label{lem:n1pullsback}
  Consider a commutative diagram
  \[
    \begin{tikzcd}
      X' \rar{f} \dar[swap]{\rho} & X \dar{\pi}\\
      Z' \rar{g} & Z
    \end{tikzcd}
  \]
  of algebraic spaces over a scheme $S$
  where $\pi$ and $\pi'$ are proper.
  \begin{enumerate}[label=$(\roman*)$,ref=\roman*]
    \item\label{lem:n1pullsbackmap}
      The pair $(f/g)$ induces a group map
      \[
        (f/g)^*\colon N^1(X/Z) \longrightarrow N^1(X'/Z').
      \]
    \item\label{lem:n1pullsinjective}
      The map $(f/g)^*$ is injective either if the condition in Lemma
      \ref{lem:nefbasechangenotcartesian}$(\ref{lem:nefbasechangenotcartesianconverse})$ holds, or if
      $Z$ is quasi-compact and decent or is a locally Noetherian scheme and the condition
      in Lemma
      \ref{lem:nefbasechangenotcartesian}$(\ref{lem:nefbasechangeconversenotcartesianclosed})$ holds.
  \end{enumerate}
\end{proposition}
\begin{proof}
  We first show $(\ref{lem:n1pullsbackmap})$.
  By \cite[\href{https://stacks.math.columbia.edu/tag/0B8P}{Tag
  0B8P}]{stacks-project}, pulling back invertible sheaves
  induces a map $\Pic(X) \to \Pic(X')$.
  It therefore suffices to show that the composition
  \[
    \Pic(X) \longrightarrow \Pic(X') \longrightarrow N^1(X'/Z')
  \]
  factors through $N^1(X/Z)$.
  This holds since $\pi$-numerically trivial elements pull back to
  $\rho$-numerically trivial elements by Lemma
  \ref{lem:nefbasechangenotcartesian}$(\ref{lem:nefbasechangenotcartesianpullback})$.
  \par For $(\ref{lem:n1pullsinjective})$, it suffices to note that if the
  pullback of $\sL \in \Pic(X)$ to $X'$ is $\rho$-numerically trivial, then
  $\sL$ is $\pi$-numerically trivial by Lemma
  \ref{lem:nefbasechangenotcartesian}$(\ref{lem:nefbasechangenotcartesianconverse})$ or Lemma
  \ref{lem:nefbasechangenotcartesian}$(\ref{lem:nefbasechangeconversenotcartesianclosed})$.
\end{proof}
We can now show the theorem of the base.
We note that Noetherian algebraic spaces are quasi-compact, quasi-separated, and
locally Noetherian (see the definition in
\cite[\href{https://stacks.math.columbia.edu/tag/03EA}{Tag
03EA}]{stacks-project}), and hence are automatically decent (see
\cite[\href{https://stacks.math.columbia.edu/tag/03I7}{Tag
03I7}]{stacks-project}).
\begin{theorem}[Theorem of the base; cf.\ {\citeleft\citen{Kle66}\citemid
  Chapter IV, \S4, Proposition 3\citepunct\citen{Kee03}\citemid Theorem
  3.6\citepunct \citen{Kee18}\citemid Theorem E2.2\citeright}]
  \label{thm:RelNSFinite}
  Let $\pi\colon X\to Z$ be a proper morphism of Noetherian algebraic
  spaces over a scheme $S$, and let
  $\kk\in\{\ZZ,\QQ,\RR\}$.
  Then, the $\kk$-modules $N^1(X/Z)_{\kk}$ and $N_1(X/Z)_{\kk}$ are finitely generated.
  Consequently, the intersection pairing 
  \[
  N^1(X/Z)_{\kk}\times N_1(X/Z)_{\kk}\longrightarrow \kk
  \]
  is a perfect pairing.
\end{theorem}
\begin{proof}
Since $N_1(X/Z)_{\kk}$ is a submodule of $\Hom_{\kk}(N^1(X/Z)_{\kk},\kk)$, it suffices to show $N^1(X/Z)_{\kk}$ finitely generated.
The cases $\kk = \QQ$ and $\kk = \RR$ follow from the case $\kk = \ZZ$ by
extending scalars.
The case when $Z$ is a scheme is proved in
\citeleft\citen{Kle66}\citemid Chapter IV,
\S4, Proposition 3\citepunct\citen{Kee03}\citemid Theorem
3.6\citepunct \citen{Kee18}\citemid Theorem E2.2\citeright.
It therefore suffices to consider the case when $Z$ is an algebraic space.
\par Let $Z' \to Z$ be an \'etale cover by a quasi-compact scheme $Z'$.
Note that $Z'$ is a Noetherian scheme.
We then consider the Cartesian diagram
\[
  \begin{tikzcd}
    X' \rar{f} \dar[swap]{\rho} & X \dar{\pi}\\
    Z' \rar{g} & Z\mathrlap{.}
  \end{tikzcd}
\]
By Proposition \ref{lem:n1pullsback} (see Remark \ref{rem:nefbasechange}), we
have an injection $N^1(X/Z) \hookrightarrow N^1(X'/Z')$.
Since $N^1(X'/Z')$ is finitely generated by the scheme case, we see that $N^1(X/Z)$
is finitely generated.
\end{proof}

\begin{remark}\label{rem:NefAgainstNonClosedContracted}
  With notation as in Definition \ref{def:intersectionWithContractedCurves},
  if $z\in \lvert Z \rvert$ is not closed, then a closed subspace $C$ of
  $\pi^{-1}(z)$ is not a closed subspace of $X$, and thus is not covered by
  Definition \ref{def:intersectionWithContractedCurves}.
However, if $\dim(C)=1$, the intersection number $(\sL\cdot C)$
is still well-defined and extends linearly to $\Div_{\kk}(X)$ for $\kk \in
\{\QQ,\RR\}$ as before (cf.\ the proof of Lemma
\ref{lem:NefAgainstNonClosedContracted}).
Consequently, if $D\in\Pic_{\kk}(X)$ for $\kk \in
\{\ZZ,\QQ,\RR\}$ and $C$ is a one-dimensional integral closed
subspace of $\pi^{-1}(z)$ for a point $z \in \lvert Z \rvert$, then $(D\cdot
C)=0$ whenever $[D]=0\in N^1(X/Z)_{\kk}$.
These subspaces $C \subseteq \pi^{-1}(z)$ define classes
\[
  [C]\in N_1(X/Z)_{\kk} = \Hom_{\kk}\bigl(N^1(X/Z)_{\kk},\kk\bigr),
\]
for $\kk\in\{\ZZ,\QQ,\RR\}$.
\end{remark}
\subsection{Cones and Kleiman's criterion for ampleness}
The theorem of the base allows us to define the relative ample and relative nef
cones for proper morphisms of Noetherian algebraic spaces.
\begin{definition}[see {\citeleft\citen{Kle66}\citemid p.\ 335\citepunct
  \citen{KMM87}\citemid p.\ 291\citepunct
  \citen{VPthesis}\citemid Definitions 1.3.21 and 1.3.24\citeright}]
  \label{def:AmpleConeAndNefCone}
Let $\pi\colon X\to Z$ be a proper morphism of Noetherian algebraic spaces over
a scheme $S$.
The \textsl{relative nef cone} is
\begin{align*}
  \Nef(X/Z)&\coloneqq\Set[\big]{[D]\in N^1(X/Z)_{\RR} \given D\in \Pic_\RR(X)\
  \text{is $\pi$-nef}},
  \intertext{and the \textsl{relative ample cone} is}
  \Amp(X/Z)&\coloneqq\Set[\big]{[D]\in N^1(X/Z)_{\RR} \given D\in \Pic_\RR(X)\
  \text{is $\pi$-ample}}.
\end{align*}

In the space $N_1(X/Z)_{\RR}$, we define the cone
$\NE(X/Z)$
to be the set of $\RR_{\geq 0}$-combinations of $\pi$-contracted curves, and let
$\NEbar(X/Z)$ be its closure.
By definition, it is clear that an $\RR$-invertible sheaf $D$ on $X$ is $\pi$-nef if and only if for all
$\gamma\in \NEbar(X/Z)$, we have $(D\cdot\gamma)\geq 0$.
For an $\RR$-invertible sheaf $D$ on $X$, we also define
\[
  \NEbar_{D\ge0}(X/Z) \coloneqq \Set[\big]{\gamma \in \NEbar(X/Z) \given (D \cdot
  \gamma) \ge 0}.
\]
Since $\NEbar(X/Z)$ is a closed convex subset of $N_1(X/Z)$, it is an
intersection of half-spaces.
Thus, we have
\begin{align}\label{NEBARisNEFgeq0}
  \NEbar(X/Z)=\Set[\bigg]{\gamma\in N_1(X/Z)_{\RR}\given
    \begin{tabular}{@{}c@{}}
      \((\beta\cdot\gamma)\ge0\) for all\\
      \(\beta\in \Nef(X/Z)\)
    \end{tabular}}.
\end{align}
\end{definition}

We now want to prove the relative version of Kleiman's criterion for ampleness
for proper morphisms of algebraic spaces.
We start with the following definition.
\begin{definition}[cf.\ {\citeleft\citen{Kle66}\citemid Chapter IV, \S4,
  Definition 1\citepunct \citen{Kee03}\citemid Definition
  3.8\citepunct \citen{FS11}\citemid Lemma 4.12\citeright}]
  \label{def:relquasidiv}
  Let $\pi\colon X \to Z$ be a proper morphism of Noetherian algebraic spaces
  over a scheme $S$.
  We say that $X$ is \textsl{relatively quasi-divisorial for $\pi$} if, for
  every $\pi$-contracted integral subspace $V$ of positive dimension, there
  exist an invertible sheaf $\sH$ on $X$ and a nonzero effective Cartier
  divisor $H$ on $V$ such that $\sH_{\vert V} \cong \cO_V(H)$.
\end{definition}
\begin{remark}
  With notation as in Definition \ref{def:relquasidiv}, $X$ is relatively
  quasi-divisorial for $\pi$ in the following cases:
  \begin{enumerate}[label=$(\roman*)$]
    \item When $\pi$ is projective (let $\sH$ be $\pi$-very ample in the sense
      of \cite[\S2.1.1]{CT20}; see \cite[p.\ 257]{Kee03}).
    \item When $X$ is a regular scheme, or more generally a $\QQ$-factorial
      scheme \cite[Chapter VI, Proof of Theorem 2.19]{Kol96}.
  \end{enumerate}
\end{remark}
We can now show that the ample cone is the interior of the nef cone.
\begin{theorem}[cf.\
  {\citeleft\citen{Kle66}\citemid Chapter IV, \S4, Theorem 2\citepunct
  \citen{Kee03}\citemid Theorem 3.9\citepunct \citen{Kee18}\citemid
  Theorem E2.2\citeright}]\label{thm:ampisinterior}
  Let $\pi\colon X \to Z$ be a proper morphism of Noetherian algebraic spaces
  over a scheme $S$.
  Then, we have
  \begin{align}
    \Amp(X/Z) &\subseteq \Int\bigl(\Nef(X/Z)\bigr).\label{eq:ampininterior}
    \intertext{If $X$ is relatively quasi-divisorial for $\pi$ and
    \(\Int(\Nef(X/Z))\) is nonempty, then we have}
    \Amp(X/Z) &= \Int\bigl(\Nef(X/Z)\bigr).\label{eq:ampequalsinterior}
  \end{align}
\end{theorem}
\begin{proof}
  We show $\Amp(X/Z) \subseteq \Nef(X/Z)$.
  Let $D \in \Amp(X/Z)$, and write $D = \sum_i a_iH_i$, where $H_i$ are
  $\pi$-ample invertible sheaves.
  We have $D \in \Nef(X/Z)$ since the restriction of each $H_i$ to the fibers of
  $\pi$ are ample, and hence nef by \cite[Proposition B.14]{Kle05}.
  \par For the statements involving interiors, as in
  \cite[Chapter IV, \S 1, Remarks 4 and 5]{Kle66}, the cone generated by
  $\Int(\Nef(X/Z)) \cap N^1(X/Z)$ is equal to $\Int(\Nef(X/Z))$, and hence it
  suffices to prove both statements for invertible sheaves $\sL$.
  Note that this reduction uses the fact that $N^1(X/Z)$ is finitely generated
  (Theorem \ref{thm:RelNSFinite}).
  Let $g\colon Z' \to Z$ be a surjective \'etale cover by a quasi-compact
  scheme $Z'$, and consider the
  associated Cartesian diagram
  \[
    \begin{tikzcd}
      X' \rar{f}\dar[swap]{\pi'} & X\dar{\pi}\\
      Z' \rar{g} & Z\mathrlap{.}
    \end{tikzcd}
  \]
  \par To show \eqref{eq:ampininterior}, let $\sL \in \Amp(X/Z)$.
  It suffices to show that for every $\sM \in \Pic(X)$, we have
  \[
    \sL^{\otimes m} \otimes_{\cO_X} \sM \in \Amp(X/Z)
  \]
  for $m \gg 0$.
  Since $\sL$ is $\pi$-ample, we know $X \to Z$ is representable, and hence
  $X'$ is a scheme.
  Since $f^*\sL$ is $\pi'$-ample, we know that $f^*\sL^{\otimes m} \otimes_{\cO_X}
  f^*\sM$ is $\pi'$-ample for all $m \gg 0$ by \cite[Corollaire
  4.6.12]{EGAII}.
  We therefore see that $\sL^{\otimes m} \otimes_{\cO_X} \sM$ is $\pi$-ample by
  \cite[\href{https://stacks.math.columbia.edu/tag/0D36}{Tag
  0D36}]{stacks-project}.\smallskip
  \par It remains to show \eqref{eq:ampequalsinterior} when $X$ is
  quasi-divisorial for $\pi$.
  Let $\sL \in \Int(\Nef(X/Z))$.
  It suffices to show that $f^*\sL$ is $\pi'$-ample and that $X'$ is a scheme by
  \cite[\href{https://stacks.math.columbia.edu/tag/0D36}{Tag
  0D36}]{stacks-project}.
  By \cite[\href{https://stacks.math.columbia.edu/tag/0D3A}{Tag
  0D3A}]{stacks-project} and the Nakai--Moishezon criterion for proper
  algebraic spaces over fields \citeleft\citen{PG85}\citemid Theorem
  1.4\citepunct \citen{Kol90}\citemid Theorem 3.11\citeright, it suffices to
  show that for every $\pi'$-contracted integral closed subspace $V \subseteq
  X'$ of dimension $d > 0$, we have $((f^*\sL)^d \cdot V) > 0$.
  \par We proceed by induction on $d$.
  Since $X$ is relatively quasi-divisorial for $\pi$, there exists $\sH \in
  \Pic(X)$ such that $\sH_{\vert f(V)} \cong \cO_{f(V)}(H)$ for some nonzero
  effective Cartier divisor $H$ on $X$, and hence $f^*\sH_{\vert V} \cong
  \cO_V(f^*H)$, where the pullback of $H$ is defined by
  \cite[\href{https://stacks.math.columbia.edu/tag/083Z}{Tag
  083Z}(1)]{stacks-project}.
  Since $\sL \in \Int(\Nef(X/Z))$, there exists $m > 0$ such that
  $\sL^{\otimes m} \otimes_{\cO_X} \sH^{-1}$ is $\pi$-nef, and hence
  \[
    f^*\sL^{\otimes m} \otimes_{\cO_{X'}} f^*\sH^{-1}
  \]
  is $\pi'$-nef by Lemma
  \ref{lem:nefbasechange}$(\ref{lem:nefbasechangepullback})$.
  We claim we have the following chain of equalities and inequalities:
  \begin{align*}
    (f^*\sL^d \cdot V) &= \frac{1}{m}\bigl( (f^*\sL)^{d-1} \cdot f^*\sL^{\otimes
    m} \cdot V\bigr)\\
    &= \frac{1}{m}\Bigl(\bigl( (f^*\sL)^{d-1} \cdot (f^*\sL^{\otimes
    m} \otimes_{\cO_{X'}} f^*\sH^{-1}) \cdot V\bigr) + \bigl( (f^*\sL)^{d-1} \cdot \sH
    \cdot V\bigr)\Bigr)\\
    &\ge \frac{1}{m}\bigl( (f^*\sL)^{d-1} \cdot f^*\sH \cdot
    V\bigr)\\
    &>0.
  \end{align*}
  The first two equalities follow from linearity of the intersection product
  \cite[\href{https://stacks.math.columbia.edu/tag/0EDH}{Tag
  0EDH}]{stacks-project}.
  To show the inequality in the third line, let $\mu\colon V' \to V$ be a finite
  surjective morphism from a scheme $V'$,
  which exists by \cite[\href{https://stacks.math.columbia.edu/tag/09YC}{Tag
  09YC}]{stacks-project}.
  Then, $(f_{\vert V} \circ \mu)^*\sL$ and
  \begin{align*}
    \bigl(f_{\vert V} \circ
    \mu\bigr)^*&\sL^{\otimes m} \otimes_{\cO_{V'}} \bigl(f_{\vert V}
    \circ \mu\bigr)^*\sH^{-1}
    \intertext{are nef on $V'$, and hence}
    \Bigl( (f^*\sL)^{d-1} \cdot{} &\bigl(f^*\sL^{\otimes
    m} \otimes_{\cO_{X'}} \sH^{-1}\bigr) \cdot V\Bigr) \ge 0
  \end{align*}
  by the projection formula
  \cite[\href{https://stacks.math.columbia.edu/tag/0EDJ}{Tag
  0EDJ}]{stacks-project} and \cite[Lemma 2.12]{Kee03}.
  For the last inequality, if $d = 1$, we see that 
  $V$ is a scheme by
  \cite[\href{https://stacks.math.columbia.edu/tag/0ADD}{Tag
  0ADD}]{stacks-project}, and hence
  \begin{align*}
    (&f^*\sH \cdot V) = \deg(f^*\sH) > 0
  \intertext{by \cite[\href{https://stacks.math.columbia.edu/tag/0B40}{Tag
  0B40}(2)]{stacks-project}.
  If $d \ge 2$, then we have}
    \bigl((f^*\sL)^{d-1} \cdot{} &f^*\sH
    \cdot f(V)\bigr) = \bigl((f^*\sL)^{d-1} \cdot f^*H\bigr) > 0
  \end{align*}
  by \cite[\href{https://stacks.math.columbia.edu/tag/0EDK}{Tag
  0EDK}]{stacks-project} and the inductive hypothesis.
\end{proof}

\begin{remark}\label{rem:ampspansn1}
    As seen in the proof of \eqref{eq:ampininterior},
    the ample cone is always open in $N^1(X/Z)_\RR$.
    In particular, if \(\Amp(X/Z)\) is nonempty, then
    the ample cone $\Amp(X/Z)$ $\RR$-linearly spans $N^1(X/Z)_\RR$.
\end{remark}

Next, we show that 
the relative ampleness of an $\RR$-Cartier divisor $D$ only depends on its class
$[D]$.
This is a consequence of the following relative version of Kleiman's criterion for ampleness stated
in terms of the cone $\NEbar(X/Z)$.
This result also implies that
$[D]\in\Amp(X/Z)$ if and only if $D$ is $\pi$-ample.
See \cite[Lemma 4.12]{FS11} for the case when $Z = \Spec(k)$, where $k$ is a
field.
See also \cite[Lemma 21]{Kol21def} and \cite[Corollary 1.4]{VP} for other
versions of Kleiman's criterion for algebraic spaces.
\begin{proposition}[see {\citeleft\citen{Kle66}\citemid Chapter IV, \S4,
  Proposition 4\citepunct \citen{FS11}\citemid Lemma 4.12\citeright}]
  \label{lem:AmpleIsPositiveOnNE} 
Let $\pi\colon X\to Z$ be a proper morphism of Noetherian algebraic spaces over
a scheme $S$.
Suppose that $X$ is relatively quasi-divisorial for $\pi$.
Then, $D\in\Pic_\RR(X)$ is $\pi$-ample if and only if for all nonzero
$\gamma\in \NEbar(X/Z)$, we have $(D\cdot\gamma)>0$.
\end{proposition}
\begin{proof}
  For $\Rightarrow$, we proceed by contradiction as in \cite[Chapter II,
  Proposition 4.8]{Kol96}.
  Suppose $(D\cdot\gamma)\le0$.
  Let $E \in \Pic(X)$ be such that $(E \cdot \gamma) < 0$.
  We have that $mD+E$ is $\pi$-ample for
  $m \gg 0$ by Theorem \ref{thm:ampisinterior}, and hence
  \[
    0 \le \bigl((mD+E) \cdot \gamma\bigr) = m(D \cdot \gamma) + (E \cdot \gamma)
    < 0,
  \]
  a contradiction.
  \par For $\Leftarrow$, by Theorem \ref{thm:ampisinterior},
  we need to show that $D \in \Int(\Nef(X/Z))$.
  We need to show that for arbitrary $D' \in \Pic_\kk(X)$, we have $mD+D' \in
  \Nef(X/Z)$ for all $m \gg 0$.
  We adapt the proof in \cite[Theorem 1.4.29]{Laz04a}.
  By Lemma \ref{lem:NefAgainstNonClosedContracted}, it suffices to show that
  there exists an $m$ such that $((mD+D') \cdot C) \ge 0$ for all
  $\pi$-contracted curves $C$.
  Consider the linear functionals
  \begin{align*}
    \phi_D\colon N_1(X/Z)_\RR &\longrightarrow \RR\\
    \phi_{D'}\colon N_1(X/Z)_\RR &\longrightarrow \RR
  \end{align*}
  defined by
  intersecting with $D$ and $D'$, respectively.
  Fix a norm $\lVert \cdot \rVert$ on $N_1(X/Z)_\RR$, and let
  \[
    S = \Set[\big]{\gamma \in N_1(X/Z)_\RR \given \lVert \gamma \rVert = 1}.
  \]
  Since $\NEbar(X/Z) \cap S$ is compact, there exists $\varepsilon \in
  \RR_{>0}$ such
  that $\phi_D(\gamma) \ge \varepsilon$ for all $\gamma \in \NEbar(X/Z) \cap S$.
  Similarly, there exists $\varepsilon' \in \RR$ such that $\phi_{D'}(\gamma)
  \ge
  \varepsilon'$ for all $\gamma \in \NEbar(X/Z) \cap S$.
  Thus, $(D \cdot C) \ge \varepsilon \cdot \lVert C \rVert$ and $(D' \cdot C)
  \ge \varepsilon' \cdot \lVert C \rVert$ for every
  $\pi$-contracted curve $C \subseteq X$.
  We then have
  \[
    \bigl((mD+E) \cdot C\bigr) = m(D \cdot C) + (E \cdot C) \ge (m\,\varepsilon
    + \varepsilon') \cdot \lVert C \rVert,
  \]
  and hence it suffices to choose $m \gg 0$ such that
  $m\,\varepsilon + \varepsilon' > 0$.
\end{proof}

Next, we consider the behavior of cones under localization on the base.
\begin{lemma}\label{lem:NumericalClassOpenSubset}
Let $\pi\colon X\to Z$ be a proper morphism of Noetherian algebraic spaces over
a scheme $S$.
Let $V$ be an open subspace of $Z$. Restriction of invertible sheaves gives a $\kk$-linear map
\begin{align}
  \Pic_{\kk}(X)&\longrightarrow \Pic_{\kk}\bigl(\pi^{-1}(V)\bigr)\nonumber
  \intertext{for $\kk \in \{\ZZ,\QQ,\RR\}$,
  and the construction in Remark \ref{rem:NefAgainstNonClosedContracted} gives a
  $\kk$-linear map}
  Z_1(\pi^{-1}(V)/V)_{\kk}&\longrightarrow N_1(X/Z)_{\kk}.\nonumber
\intertext{These maps are compatible with intersection products and thus give
$\kk$-linear maps}
  \begin{split}
    N^1(X/Z)_{\kk}&\longrightarrow N^1(\pi^{-1}(V)/V)_{\kk}\\
    N_1\bigl(\pi^{-1}(V)/V\bigr)_{\kk}&\longrightarrow N_1(X/Z)_{\kk}
  \end{split}\label{eq:NumericalClassOpenSubset}
\end{align}
that preserve $\Nef$, $\Amp$, and $\NEbar$.
\end{lemma}
\begin{proof}
That these maps are compatible with intersection products is a consequence of
the construction of $[C]$ as in 
Lemma \ref{lem:NefAgainstNonClosedContracted}
and Remark \ref{rem:NefAgainstNonClosedContracted}.
Therefore they induce the $\kk$-linear maps in
\eqref{eq:NumericalClassOpenSubset}.
Under these maps, $\Nef(X/Z)$ is mapped into $\Nef(\pi^{-1}(V)/V)$ by Lemma \ref{lem:NefAgainstNonClosedContracted},
and $\Amp(X/Z)$ is mapped into $\Amp(\pi^{-1}(V)/V)$ by definition, since a $\pi$-ample line bundle $\sL$ restricts to a $\pi_{|\pi^{-1}(V)}$-ample line bundle.
By \eqref{NEBARisNEFgeq0}, $\NEbar(\pi^{-1}(V)/V)$ is mapped into $\NEbar(X/Z)$.
\end{proof}
Finally, we will use the following terminology to describe our cones.
\begin{definition}[see {\cite[Definition 3-2-3]{KMM87}}]\label{def:SuppPlaneAndDualRay}
We say a subspace $W\subseteq N^1(X/Z)_{\RR}$ is a \textsl{supporting subspace of $\Nef(X/Z)$} if $W$ is the span of $W\cap\Nef(X/Z)$ and $W\cap \Amp(X/Z)=\emptyset$. We say a supporting subspace $W$ of $\Nef(X/Z)$ a \textsl{supporting hyperplane of $\Nef(X/Z)$} if
\[
  \dim W=\dim\bigl(N^1(X/Z)_{\RR}\bigr)-1.
\]

Let $W$ be a supporting subspace of $\Nef(X/Z)$.
The \textsl{extremal face dual to $W$} is
\[
  R=\Set[\big]{\gamma\in \NEbar(X/Z) \given (W\cdot\gamma)=0}.
\]
When $W$ is a supporting hyperplane, we call $R$ the \textsl{extremal ray dual to $W$}.

Note that $R$ is an extremal face of $\NEbar(X/Z)$ in the sense that if $\beta_1,\beta_2\in \NEbar(X/Z)$ satisfy $\beta_1+\beta_2\in R$, then $\beta_1,\beta_2\in R$.
\end{definition}

\begin{remark}\label{rem:DualToOneVector}
There always exist a single $[D_0]\in W\cap\Nef(X/Z)$ such that
\[R=\Set[\big]{\gamma\in \NEbar(X/Z) \given (D_0\cdot\gamma)=0}.\]
Indeed, by assumptions (and by Theorem \ref{thm:RelNSFinite}) $W$ is spanned by
several
\[
  [D_1],[D_2],\ldots,[D_n]\in\Nef(X/Z).
\]
Since \(D\cdot\gamma \ge 0\) for all \(D \in \Nef(X/Z)\) and \(\gamma \in
\NEbar(X/Z)\), we see that
\[
  D_0=D_1+D_2+\cdots +D_n
\]
is a valid choice.
If $W$ has a basis consisting of rational elements of $\Nef(X/Z)$, then by the above we may take $D_0$ rational. 
\end{remark}

\begin{remark}\label{rem:DualRay}
When $W$ is a supporting hyperplane, the extremal ray $R$ dual to $W$ is a ray in the $\RR$-vector space $N_1(X/Z)_{\RR}$. 
Indeed, $R\neq\{0\}$ by Proposition \ref{lem:AmpleIsPositiveOnNE}, and the span of $R$ has dimension at most one since $W$ has codimension one.
\end{remark}

\section{Relatively big \texorpdfstring{$\RR$}{R}-invertible sheaves}
In this section, we
define the ``birational'' variants of the relative ampleness conditions defined
in the previous section, i.e., relative bigness and relative pseudoeffectivity.
As far as we are aware, these results are new for algebraic spaces, even for
proper algebraic spaces over a field.
\subsection{A weak version of Chow's lemma}
We will use the following lemma repeatedly to reduce to the scheme case.
While it is a special case of the version of Chow's lemma in
\citeleft\citen{Knu71}\citemid Chapter IV, Theorem 3.1\citepunct
\citen{stacks-project}\citemid
\href{https://stacks.math.columbia.edu/tag/088U}{Tag
088U}\citeright, we will use this weak version of Chow's lemma because it is
easier to prove than the full statement.
\begin{lemma}[{see
  \cite[Proof of \href{https://stacks.math.columbia.edu/tag/0DN4}{Tag 0DN4}]{stacks-project}}]\label{lem:weakchow}
  Let \(A\) be a ring.
  Let \(X\) be an algebraic space over \(A\) that is separated and of finite
  type over \(A\).
  Then, there exists a generically finite
  proper surjective morphism \(X' \to X\) where \(X'\) is a
  scheme that is H-quasi-projective over \(A\).
\end{lemma}
\begin{proof}
  By the weak version of Chow's lemma in
  \cite[\href{https://stacks.math.columbia.edu/tag/089J}{Tag
  089J}]{stacks-project}, there exists a proper surjective morphism $\mu\colon
  X' \to X$ from a scheme $X'$ that is H-quasi-projective over $A$.
  By \cite[\href{https://stacks.math.columbia.edu/tag/0DMN}{Tag
  0DMN}]{stacks-project}, after replacing $X'$ by a closed integral subspace, we
  may assume that $\mu$ is generically finite.
\end{proof}
\subsection{Growth of cohomology and volume}
We will need the following estimate on the growth of cohomology of twists.
\begin{proposition}[cf.\ {\citeleft\citen{Deb01}\citemid Proposition
  1.31$(a)$\citeright}]
  \label{prop:deb131a}
  Let $X$ be a proper algebraic space over a field $k$ of dimension $d$, and
  let $\sL$ be an invertible sheaf on $X$.
      For every coherent sheaf $\sF$ on $X$, we have
      \begin{equation}\label{eq:dimcounttwists}
        h^i(X,\sF \otimes_{\cO_X} \sL^{\otimes m}) = O(m^d)
      \end{equation}
      for all $i$.
      Here, the dimension
      \[
        h^i(X,-) \coloneqq \dim_k\bigl(H^i(X,-)\bigr)
      \]
      is computed over \(k\).
\end{proposition}
\begin{proof}
  By d\'evissage \cite[\href{https://stacks.math.columbia.edu/tag/08AN}{Tag
  08AN}]{stacks-project}, it suffices to show the following:
  \begin{enumerate}[label=$(\alph*)$,ref=\alph*]
    \item\label{prop:dimcounttwistsdevissage1}
      For every short exact sequence
      \[
        0 \longrightarrow \sF_1 \longrightarrow \sF_2 \longrightarrow \sF_3
        \longrightarrow 0
      \]
      of coherent sheaves on $X$, if \eqref{eq:dimcounttwists}
      holds for two out
      of three of $\sF_1$, $\sF_2$, and $\sF_3$, then
      \eqref{eq:dimcounttwists}
      holds for the third.
    \item\label{prop:dimcounttwistsdevissage2}
      If \eqref{eq:dimcounttwists}
      holds for $\sF^{\oplus r}$ for some $r
      \ge 1$, then \eqref{eq:dimcounttwists}
      holds for $\sF$.
    \item\label{prop:dimcounttwistsdevissage3}
      For every integral closed subspace $\iota\colon V \hookrightarrow X$,
      there exists a
      coherent sheaf $\sG$ on $X$ whose scheme-theoretic support is $V$ such
      that \eqref{eq:dimcounttwists}
      holds for $\sG$.
  \end{enumerate}
  \par First, $(\ref{prop:dimcounttwistsdevissage1})$ follows from the
  inequalities
  \begin{alignat*}{3}
    h^i(X,\sF_1\otimes_{\cO_X}\sL^{\otimes m}) &{}\le{}&
    h^{i-1}(X,\sF_3\otimes_{\cO_X}\sL^{\otimes m}) &{}+{}&
    h^i(X,\sF_2\otimes_{\cO_X}\sL^{\otimes m})\\
    h^i(X,\sF_2\otimes_{\cO_X}\sL^{\otimes m}) &{}\le{}&
    h^i(X,\sF_1\otimes_{\cO_X}\sL^{\otimes m}) &{}+{}&
    h^i(X,\sF_3\otimes_{\cO_X}\sL^{\otimes m})\\
    h^i(X,\sF_3\otimes_{\cO_X}\sL^{\otimes m}) &{}\le{}&
    h^i(X,\sF_2\otimes_{\cO_X}\sL^{\otimes m}) &{}+{}&
    h^{i+1}(X,\sF_1\otimes_{\cO_X}\sL^{\otimes m})
  \end{alignat*}
  obtained by twisting the given exact sequence by $\sL^{\otimes m}$ and
  using the long exact sequence on sheaf cohomology.
  \par Second, $(\ref{prop:dimcounttwistsdevissage2})$ follows since
  \[
    h^i(X,\sF^{\oplus r} \otimes_{\cO_X} \sL^{\otimes m}) = r \cdot
    h^i(X,\sF \otimes_{\cO_X} \sL^{\otimes m}).
  \]
  \par Third, $(\ref{prop:dimcounttwistsdevissage3})$ follows from the scheme
  case of \eqref{eq:dimcounttwists} as follows.
  By the weak version of Chow's lemma (Lemma \ref{lem:weakchow}),
  there exists a generically finite proper
  surjective morphism $\mu\colon V' \to V$ from a scheme $V'$ that is a closed
  subscheme of $\PP^N_k$ for some $N$.
  Let $\cO_{V'}(n) = \cO_{\PP^N_k}(n)_{\vert V'}$.
  Choose $n > 0$ such that $R^p\mu_*\cO_{V'}(n) = 0$ for all $p > 0$
  \cite[\href{https://stacks.math.columbia.edu/tag/08AQ}{Tag
  08AQ}]{stacks-project}.
  We claim that $\sG = \iota_*\mu_*\cO_{V'}(n)$ satisfies
  \eqref{eq:dimcounttwists}.
  We have
  \begin{align*}
    h^i(X,\sG \otimes_{\cO_X} \sL^{\otimes m})
    &= h^i\Bigl(V',\cO_{V'}(n) \otimes_{\cO_{V'}} \mu^*\bigl(\sL^{\otimes
    m}_{\vert V}\bigr) \Bigr)\\
    &= O(m^{\dim(V)})
  \end{align*}
  by the projection formula
  \cite[\href{https://stacks.math.columbia.edu/tag/0944}{Tag
  0944}]{stacks-project}, the Leray spectral sequence
  \cite[\href{https://stacks.math.columbia.edu/tag/0733}{Tag
  0733}]{stacks-project}, and
  the scheme case of the proposition \cite[Proposition
  1.31$(a)$]{Deb01}.
\end{proof}
Next, we define volumes.
\begin{definition}[see {\cite[Definition 2.2.31]{Laz04a}}]
  Let $X$ be an integral proper algebraic space of dimension $d$ over a field
  $k$.
  The \textsl{volume} of an invertible sheaf $\sL$ on $X$ is
  \[
    \vol_X(\sL) \coloneqq \limsup_{m \to \infty}
    \frac{h^0(X,\sL^{\otimes m})}{m^d/d!},
  \]
  where the dimension
  \[
    h^0(X,-) \coloneqq \dim_k\bigl(H^0(X,-)\bigr)
  \]
  is computed over $k$.
\end{definition}
We also define the semigroup and the exponent of an invertible sheaf.
Following \cite{Iit71}, we use the convention that \(0 \notin \NN(\sL)\), in
contrast with \cite[Definition 2.4.7]{Laz04a}.
\begin{citeddef}[{\cite[pp.\ 356--357]{Iit71}}]
  Let \(X\) be a proper algebraic space over a field \(k\) and let \(\sL\) be an
  invertible sheaf on \(X\).
  The \textsl{semigroup} of \(\sL\) is
  \[
    \NN(\sL) \coloneqq \Set[\big]{m \ge 1 \given H^0(X,\sL^{\otimes m}) \ne 0}.
  \]
  The \textsl{exponent} \(m_0(\sL)\) of \(\sL\) is the greatest common divisor
  of all elements in \(\NN(\sL)\).
\end{citeddef}
We show that the volume is computed by a limit in some cases.
\begin{lemma}[cf.\ {\citeleft\citen{Laz04b}\citemid Example 11.4.7\citepunct
  \citen{Cut14}\citemid Theorem 8.1\citeright}]\label{lem:whenvollimit}
  Let $X$ be an integral proper algebraic space of dimension $d$ over a field
  $k$.
  Let \(\sL\) be an invertible sheaf on \(X\) and set \(m_0 \coloneqq
  m_0(\sL)\).
  Then, the limit
  \[
    \lim_{n \to \infty} \frac{h^0(X,\sL^{\otimes nm_0})}{(nm_0)^d/d!}
  \]
  exists.
  Moreover, if \(X\) is projective or normal, then \(m_0 = 1\), and
  \[
    \vol_X(\sL) = \lim_{m \to \infty}
    \frac{h^0(X,\sL^{\otimes m})}{m^d/d!}.
  \]
\end{lemma}
\begin{proof}
  By the weak version of Chow's lemma (Lemma \ref{lem:weakchow}),
  there exists a generically finite proper surjective morphism \(f\colon X' \to
  X\) from a projective variety \(X'\) over \(k\).
  We consider the graded linear system \(V_\bullet\) on \(X'\)
  defined by setting
  \[
    V_m \coloneqq \im\Bigl(H^0(X,\sL^{\otimes m}) \hooklongrightarrow
    H^0\bigl(X',f^*\sL^{\otimes m}\bigr)\Bigr).
  \]
  We then have
  \[
    \limsup_{n \to \infty} \frac{h^0(X,\sL^{\otimes nm_0})}{(nm_0)^d/d!}
    = \limsup_{n \to \infty} \frac{\dim_k(V_{nm_0})}{(nm_0)^d/d!}
  \]
  and the limit supremum on the right-hand side is a limit by \cite[Theorem
  8.1]{Cut14}.
  \par For the ``moreover'' statement, it remains to show that \(m_0 = 1\).
  The case when \(X\) is projective follows from \cite[Theorem 10.7]{Cut14}.
  For the case when \(X\) is normal, we use the full version of
  Chow's lemma \citeleft\citen{Knu71}\citemid Chapter IV, Theorem 3.1\citepunct
  \citen{stacks-project}\citemid \href{https://stacks.math.columbia.edu/tag/088U}{Tag
  088U}\citeright, which says that there exists a blowup \(f\colon X' \to
  X\) where \(X'\) is a projective variety over \(k\).
  Then, the pullback map
  \[
    H^0(X,\sL^{\otimes m}) \longrightarrow
    H^0\bigl(X',f^*\sL^{\otimes m}\bigr)
  \]
  is a bijection because the normality of \(X\) implies that \(\cO_X \to
  f_*\cO_{X'}\) is an isomorphism (see the proof of
    \cite[\href{https://stacks.math.columbia.edu/tag/0A19}{Tag
  0A19}(3)]{stacks-project}).
\end{proof}
We show that the volume behaves well with respect to
generically finite morphisms.
\begin{proposition}[cf.\ {\citeleft\citen{Hol}\citemid Lemma 4.3\citepunct
  \citen{Cut24}\citemid Lemma 2.9\citeright}]\label{prop:volumegenfin}
  Let $f\colon Y \to X$ be a surjective generically finite morphism of integral
  proper algebraic spaces over a field $k$.
  Consider an invertible sheaf $\sL$ on $X$.
  Then, we have
  \[
    \vol_Y(f^*\sL) = \deg(f) \cdot \vol_X(\sL).
  \]
\end{proposition}
\begin{proof}
  Since $f$ is generically finite, we know that $f_*\cO_Y$ has rank $r =
  \deg(f)$.
  Thus, there exists a dense open subspace $U \subseteq X$ such that
  $(f_*\cO_Y)_{\vert U} \cong \cO_U^{\oplus r}$, which yields an injection
  $f_*\cO_Y \hookrightarrow \sK_X^{\oplus r}$, where $\sK_X$ is the sheaf of
  meromorphic functions as defined in
  \cite[\href{https://stacks.math.columbia.edu/tag/0EN3}{Tag
  0EN3}]{stacks-project}.
  Consider the intersection $\sG = f_*\cO_Y \cap \cO_X^{\oplus r}$ as subsheaves
  of $\sK_X^{\oplus r}$, and the short exact sequences
  \[
    \begin{tikzcd}[sep=scriptsize]
      0 \rar & \sG \rar\dar[equal] & f_*\cO_Y \rar & \sG_1 \rar & 0\\
      0 \rar & \sG \rar & \cO_X^{\oplus r} \rar & \sG_2 \rar & 0\mathrlap{.}
    \end{tikzcd}
  \]
  Since $\sG_1$ and $\sG_2$ are supported in $X - U$, we see that
  \begin{align*}
    h^1(X,\sG_1 \otimes_{\cO_X} \sL^{\otimes m}) &= O(m^{d-1})\\
    h^1(X,\sG_2 \otimes_{\cO_X} \sL^{\otimes m}) &= O(m^{d-1})
  \end{align*}
  by Proposition \ref{prop:deb131a}.
  Twisting by $\sL^{\otimes m}$,
  the long exact sequence on sheaf cohomology and the projection formula
  \cite[\href{https://stacks.math.columbia.edu/tag/0944}{Tag
  0944}]{stacks-project} imply
  \begin{alignat*}{4}
    h^0(Y,f^*\sL^{\otimes m}) &{}-{}&
    h^0(X,\sG \otimes_{\cO_X} \sL^{\otimes m})
    &{}\le{}& 
    h^1(X,\sG_1 \otimes_{\cO_X} \sL^{\otimes m})
    &= O(m^{d-1}),\\
    r \cdot h^0(X,\sL^{\otimes m}) &{}-{}&
    h^0(X,\sG \otimes_{\cO_X} \sL^{\otimes m})
    &{}\le{}& 
    h^1(X,\sG_2 \otimes_{\cO_X} \sL^{\otimes m})
    &= O(m^{d-1}).
  \end{alignat*}
  We therefore see that $\vol_Y(f^*\sL) = r \cdot
  \vol_X(\sL)$.
\end{proof}
We also prove that volumes are homogeneous with respect to taking powers.
\begin{proposition}[cf.\ {\cite[Proposition
  2.2.35$(a)$]{Laz04a}}]\label{prop:volhomog}
  Let $X$ be an integral proper algebraic space of dimension $d$ over a field
  $k$.
  Then, for every integer $n > 0$, we have
  \[
    \vol_X(\sL^{\otimes n}) = n^d\vol_X(\sL).
  \]
\end{proposition}
\begin{proof}
  By the proof of \cite[Lemma 2.2.38]{Laz04a}, we have
  \begin{align*}
    \vol_X(\sL) &= \limsup_{m \to \infty}
    \frac{h^0(X,\sL^{\otimes m})}{m^d/d!}\\
    &= \limsup_{m \to \infty}
    \frac{h^0(X,\sL^{\otimes nm})}{(nm)^d/d!}\\
    &= \frac{1}{n^d} \cdot \vol_X(\sL^{\otimes n})
  \end{align*}
  for all integers \(n>0\).
\end{proof}

\subsection{Relatively big and pseudoeffective
\texorpdfstring{$\RR$}{R}-invertible sheaves}
We now define $\pi$-big
and $\pi$-pseudoeffective
$\kk$-invertible sheaves and $\kk$-Cartier divisors.
In the definition below, we recall if $X$ is an integral algebraic space, then
it is decent by definition
\cite[\href{https://stacks.math.columbia.edu/tag/0AD4}{Tag
0AD4}]{stacks-project}, and hence codimension zero points in $X$ are the same
thing as generic points of irreducible components in $\lvert X \rvert$
\cite[\href{https://stacks.math.columbia.edu/tag/0ABV}{Tag
0ABV}]{stacks-project}.
\begin{definition}[see {\citeleft\citen{Nak04}\citemid Chapter II, Definition
  5.16\citepunct \citen{BCHM10}\citemid Definition
  3.1.1(7)\citepunct \citen{Fuj14}\citemid Definition A.20\citeright}]
  \label{def:fbig}
  Let $\pi\colon X \to Z$ be a proper surjective morphism between integral
  algebraic spaces over a scheme $S$.
  Let $\eta \in \lvert Z \rvert$ be the generic point of $\lvert Z \rvert$.
  Let $\sL$ be an invertible sheaf on $X$.
  We say that $\sL$ is \textsl{$\pi$-big} if
  \begin{equation}\label{eq:bigvolcond}
    \vol_{X_\eta}\bigl(\sL_{\vert X_\eta}\bigr) \coloneqq
    \limsup_{m \to \infty} \frac{h^0
    \Bigl(X_\eta,\sL^{\otimes m}_{\vert X_\eta}
    \Bigr)}{m^{\dim(X_\eta)}/\bigl(\dim\bigl(X_\eta\bigr)\bigr)!} > 0,
  \end{equation}
  where $_\eta = (\pi^{-1}(\eta))$ is the generic fiber, and
  the dimension
  \[
    h^0 \Bigl(X_\eta,\sL^{\otimes m}_{\vert X_\eta} \Bigr)
    \coloneqq \dim_{\kappa(\eta)}\biggl(
      H^0 \Bigl(X_\eta,\sL^{\otimes m}_{\vert X_\eta} \Bigr)
    \biggr)
  \]
  is computed over \(\kappa(\eta)\).
  We note that
  $\lvert X_\eta \rvert$ is irreducible by \cite[Chapitre 0, Proposition
  2.1.13]{EGAInew}, and hence $X_\eta$ is integral.
  \par Now suppose $D$ is a $\kk$-invertible sheaf on $X$ for $\kk \in
  \{\QQ,\RR\}$.
  We say that $D$ is \textsl{$\pi$-big} if $D$ is a finite nonzero
  $\kk_{>0}$-linear combination of $\pi$-big invertible sheaves on $X$.
  If $Z = \Spec(k)$ for a field $k$, we just say that $\sL$ or $D$ is
  \textsl{big}.
  We use the same terminology for $\kk$-Cartier divisors when $X$ is a locally
  Noetherian scheme.
\end{definition}
\begin{remark}
  If $X_\eta$ is a scheme in Definition \ref{def:fbig}, the condition
  \eqref{eq:bigvolcond} holds if and only if for 
  $m \gg 0$, the rational map
  \[
    \begin{tikzcd}[column sep=large]
      X_\eta
      \rar[dashed]{\bigl\lvert \sL_{\vert X_\eta}^{\otimes m} \bigr\rvert}
      & \PP\Bigl(H^0\Bigl(X_\eta,\sL^{\otimes m}_{\vert X_\eta}
      \Bigr)\Bigr)
    \end{tikzcd}
  \]
  is generically finite onto its image by \cite[Theorem 8.2]{Cut14}.
\end{remark}
\begin{definition}[see {\citeleft\citen{Nak04}\citemid Chapter II, Definitions
  5.5 and 5.16\citepunct \citen{BCHM10}\citemid Definition
  3.1.1(9)\citeright}]
  Let $\pi\colon X \to Z$ be a proper surjective morphism between integral
  algebraic spaces over a scheme $S$.
  Let $D$ be a $\kk$-invertible sheaf on $X$ for $\kk \in \{\QQ,\RR\}$.
  We say that $D$ is \textsl{$\pi$-pseudoeffective} if the restriction
  $D_{\vert X_\eta}$ of $D$ to the generic fiber of $\pi$ is the limit of
  $\QQ$-invertible sheaves associated to effective $\QQ$-Cartier divisors under
  the map \eqref{eq:effcarttopic}.
  If $Z = \Spec(k)$ for a field $k$, we just say that $D$ is
  \textsl{pseudoeffective}.
\end{definition}
We now show a relative version of Kodaira's lemma.
\begin{lemma}[Relative Kodaira's lemma; cf.\ {\citeleft\citen{KMM87}\citemid
  Lemma 0-3-3 and Corollary 0-3-4\citepunct
  \citen{Fuj17}\citemid Lemma 2.1.27\citepunct
  \citen{CLM}\citemid Lemma 1.18\citeright}]\label{lem:BigIsAmpleAddEffective}
Let $\pi\colon X \to Z$ be a proper surjective morphism between integral
algebraic spaces over a scheme $S$.
Let $\sL$ be a $\pi$-big invertible sheaf on $X$.
Let $V \subseteq X$ be a proper closed subspace.
For infinitely many $m > 0$, we have
\[
  f_*\bigl( \cI_V \otimes_{\cO_X} \sL^{\otimes m} \bigr) \ne 0.
\]
If the generic fiber $X_\eta$ is projective or normal, then this holds
for all $m \gg 0$.
\end{lemma}
\begin{proof}
  By restricting to the generic fiber of $\pi$, it suffices to consider the case
  when $Z = \Spec(k)$ for a field $k$.
  \par Consider the short exact sequence
  \begin{align*}
    0 \longrightarrow \cI_V \otimes_{\cO_X} \sL^{\otimes m}
    \longrightarrow{}&
    \sL^{\otimes m} \longrightarrow \sL^{\otimes m}_{\vert V} \longrightarrow 0.
    \intertext{Taking global sections, we have the exact sequence}
    0 \longrightarrow H^0\bigl(X,\cI_V \otimes_{\cO_X} \sL^{\otimes m}\bigr)
    \longrightarrow{}& H^0\bigl(X,\sL^{\otimes m}\bigr) \longrightarrow
    H^0\Bigl(V,\sL^{\otimes m}_{\vert V}\Bigr).
    \intertext{Since $\sL$ is big, we see that}
    \dim_k\Bigl(H^0\bigl(X,\sL^{\otimes m}\bigr)\Bigr) >{}& {\dim_k
    \Bigl(H^0\Bigl(V,\sL^{\otimes m}_{\vert V}\Bigr)\Bigr)}
  \end{align*}
  for some $m$ by Proposition \ref{prop:deb131a}, and hence $H^0(X,\cI_V
  \otimes_{\cO_X} \sL^{\otimes m}) \ne 0$.
  \par The last statement when $X_\eta$ is a projective or normal
  holds because in this case,
  the limit supremum in \eqref{eq:bigvolcond} is a limit by Lemma
  \ref{lem:whenvollimit}.
\end{proof}
We obtain the following characterization of $\pi$-big $\kk$-invertible sheaves.
\begin{corollary}[cf.\ {\citeleft\citen{Laz04a}\citemid Corollary 2.2.7 and
  Proposition 2.2.22\citepunct
  \citen{Fuj17}\citemid Lemma 2.1.29\citeright}]\label{lem:kodairachar}
Let $\pi\colon X \to Z$ be a projective surjective morphism between integral
Noetherian schemes, such that $Z$ is affine.
Let $D$ be a $\kk$-invertible sheaf on $X$ for $\kk \in \{\QQ,\RR\}$.
The following are equivalent:
\begin{enumerate}[label=$(\roman*)$,ref=\roman*]
  \item\label{lem:kodairacharbig} $D$ is $\pi$-big.
  \item\label{lem:kodairacharapluse} We have $D = A+E$ in $\Pic_\kk(X)$ for
    $\kk$-invertible sheaves $A$
    and $E$ such that $A$ is a $\pi$-ample $\kk$-invertible sheaf
    and $E$ is the $\kk$-invertible sheaf
    associated to an effective $\kk$-Cartier divisor.
  \item\label{lem:kodairacharapluserational} We have $D = A+E$ in $\Pic_\kk(X)$
    for $\kk$-invertible sheaves $A$
    and $E$ such that $A$ is a $\pi$-ample $\kk$-invertible sheaf
    and $E$ is the $\kk$-invertible sheaf
    associated to an effective $\kk$-Cartier divisor, where $A$
    is in fact a $\QQ$-invertible sheaf.
  \item\label{lem:kodairacharapluserationaleff}
    We have $D = A+E$ in $\Pic_\kk(X)$
    for $\kk$-invertible sheaves $A$
    and $E$ such that $A$ is a $\pi$-ample $\kk$-invertible sheaf
    and $E$ is the $\kk$-invertible sheaf
    associated to an effective $\kk$-Cartier divisor, where $E$
    is in fact a $\QQ$-invertible sheaf.
\end{enumerate}
Moreover, if $D$ is $\pi$-big and $\pi$-nef, then writing $D = A+E$ as above,
we can make the coefficients on $E$ arbitrarily small without changing the
invertible sheaves that appear when expressing $E$ as a $\kk$-linear combination
of invertible sheaves.
\end{corollary}
\begin{proof}
  We first show $(\ref{lem:kodairacharbig}) \Rightarrow
  (\ref{lem:kodairacharapluse})$.
  Write $D = \sum_{i=1}^n a_iD_i$ for $a_i \in \kk_{>0}$.
  Let $A_0$ be a $\pi$-very ample effective Cartier divisor.
  Applying Lemma \ref{lem:BigIsAmpleAddEffective} to each invertible sheaf
  $\cO_X(D_i)$, we have
  \[
    H^0\bigl(X,\cO_X(m_iD_i-A_0)\bigr) \ne 0
  \]
  for some $m_i > 0$.
  We can then find an effective Cartier divisor $E_i \in \lvert m_iD_i - A_0
  \rvert$, and hence
  \[
    D = \sum_{i=1}^n a_iD_i \sim_\kk 
    \sum_{i=1}^n \frac{a_i}{m_i}A_0 + \sum_{i=1}^n \frac{a_i}{m_i}E_i.
  \]
  Setting $A = \sum_{i=1}^n \frac{a_i}{m_i}A_0$ and $E = \sum_{i=1}^n
  \frac{a_i}{m_i}E_i$, we are done.\smallskip
  \par Next, we show $(\ref{lem:kodairacharapluse}) \Rightarrow
  (\ref{lem:kodairacharapluserational})$ and 
  $(\ref{lem:kodairacharapluse}) \Rightarrow
  (\ref{lem:kodairacharapluserationaleff})$.
  If $\kk = \QQ$, there is nothing to show.
  If $A = \sum_{i=1}^m b_iA_i$ for $b_i \in \RR_{\ge0}$ and
  $E = \sum_{j=1}^n c_jE_j$ for $c_j \in \RR_{\ge0}$, then we can write
  \[
    D = \sum_{i=1}^m b_i'A_i + \sum_{j=1}^n (c_j-c_j')E_j
    + \sum_{i=1}^m (b_i-b_i')A_i + \sum_{j=1}^n c_j'E_j
  \]
  where $b_i',c_j' \in \QQ$.
  To obtain a decomposition $D = A+E$ where $A \in \Pic_\QQ(X)$, we
  choose $c_j = c_j'$ and choose $b_i'$ such that $0 \le b_i - b_i' \ll 1$.
  To obtain a decomposition $D = A+E$ where $E \in \Pic_\QQ(X)$, we
  choose $b_i = b_i'$ and choose $c_j'$ such that $\lvert c_j-c_j' \rvert \ll 1$
  and use the openness of the ample cone (Theorem
  \ref{thm:ampisinterior}).\smallskip
  \par Clearly $(\ref{lem:kodairacharapluserational}) \Rightarrow
  (\ref{lem:kodairacharapluse})$ and $(\ref{lem:kodairacharapluserationaleff})
  \Rightarrow (\ref{lem:kodairacharapluse})$.
  It therefore suffices to show 
  $(\ref{lem:kodairacharapluserational}) \Rightarrow
  (\ref{lem:kodairacharbig})$ to complete the proof.
  We first show the statement when $\kk = \QQ$.
  Writing $D = A+E$, we can clear denominators to reduce to the case when
  $D = A + E$ in $\Pic(X)$.
  In this case, we have
  \[
    H^0\bigl(X,\cO_X(mA)\bigr) \hooklongrightarrow
    H^0\bigl(X,\cO_X(mA+mE)\bigr) \cong H^0\bigl(X,\cO_X(mD)\bigr)
  \]
  for all $m > 0$, and hence the claim follows from asymptotic Riemann--Roch
  \cite[Chapter VI, Theorem
  2.15]{Kol96}.
  \par We now show $(\ref{lem:kodairacharapluserational}) \Rightarrow
  (\ref{lem:kodairacharbig})$ when $\kk = \RR$.
  Write $E = \sum_{j=1}^n c_jE_j$.
  We induce on $n$.
  If $n = 0$, there is nothing to show.
  If $n \ge 1$, write
  \[
    D = \biggl(A + \sum_{j=1}^{n-1} c_jE_j\biggr) + c_nE_n.
  \]
  By the inductive hypothesis, we know that $D' = A + \sum_{j=1}^{n-1}
  c_jE_j$ is $\pi$-big, and hence we can write $D' = \sum_{i=1}^m a_iD_i$ for
  $\pi$-big invertible sheaves $D_i$ and $a_i \in \RR_{>0}$.
  Choose $s_1,s_2 \in \QQ_{>0}$ such that $s_1 < c_n/a_m < s_2$ and $t
  \in [0,1]$ such that $c_n/a_m = ts_1 + (1-t)s_2$.
  We then have
  \begin{align*}
    D &= \sum_{i=1}^{m-1} a_iD_i + a_mD_m + c_nE_n\\
    &= \sum_{i=1}^{m-1} a_iD_i + a_m\biggl(D_m + \frac{c_n}{a_m}E_n \biggr)\\
    &= \sum_{i=1}^{m-1} a_iD_i + a_m\bigl( t(D_m + s_1E_n) + (1-t)(D_m+s_2E_n)
    \bigr).
  \end{align*}
  Since $D_m + s_1E_n$ and $D_m + s_2E_n$ are $\pi$-big by the implication
  $(\ref{lem:kodairacharapluserational}) \Rightarrow
  (\ref{lem:kodairacharbig})$ for $\kk = \QQ$, we see that
  $D$ is an $\RR_{>0}$-linear combination of $\pi$-big invertible
  sheaves.\smallskip
  \par Finally, if $D$ is $\pi$-nef and $\pi$-big, then $kD+A$ is $\pi$-ample
  for any positive integer $k$ by Theorem \ref{thm:ampisinterior}.
  If we have a decomposition $D = A+E$ as above, we then have
  \[
    D = \frac{1}{k+1}(kD+A) + \frac{1}{k+1}E.
  \]
  Replacing $A$ and $E$ by $\frac{1}{k+1}(kD+A)$ and $\frac{1}{k+1}E$,
  respectively, we can make the coefficients on $E$ arbitrarily small without
  changing the invertible sheaves that appear when writing $E$ as a $\kk$-linear
  combination of invertible sheaves.
\end{proof}
We show that bigness behaves well with respect to pulling back by generically finite morphisms.
\begin{lemma}[cf.\ {\citeleft\citen{Nak04}\citemid Chapter II, Lemma 5.6 and
  Remark on p.\ 69\citepunct \citen{Fuj14}\citemid Lemmas A.5 and
  A.18\citeright}]\label{lem:bigpullback}
  Let $S$ be a scheme.
  Let
  \[
    \begin{tikzcd}[column sep=0.75em]
      X' \arrow{rr}{f}\arrow{dr}[near start,swap]{\pi'}
      & & X\arrow{dl}[near start]{\vphantom{\pi'}\smash{\pi}}\\
      & Z
    \end{tikzcd}
  \]
  be a commutative diagram of integral algebraic spaces over $S$, where $\pi$
  and $\pi'$ are proper and $f$ is surjective.
  Let $D \in \Pic_\kk(X)$ for $\kk \in \{\ZZ,\QQ,\RR\}$.
  \begin{enumerate}[label=\((\roman*)\),ref=\roman*]
    \item\label{lem:bigpullbackbig} Suppose that \(f\) is genericaly finite.
      Then, $D$ is $\pi$-big if and only if $f^*D$
      is $\pi'$-big.
    \item\label{lem:bigpullbackpseff}
      If \(D\) is \(\pi\)-pseudoeffective, then $f^*D$
      is \(\pi'\)-pseudoeffective.
  \end{enumerate}
\end{lemma}
We start with the case when \(\kk \in \{\ZZ,\QQ\}\).
\begin{proof}[Proof of Lemma \ref{lem:bigpullback} when \(\kk \in \{\ZZ,\QQ\}\)]
  Replacing $Z$ by the spectrum of the generic point of \(\pi(X)\), we
  may assume that $Z =
  \Spec(k)$ for a field $k$.\smallskip
  \par We first show that if $D$ is big or pseudoeffective, then $f^*D$ is also.
  For bigness,
  working one term of $D$ at a time, it suffices to consider the case when $\kk
  = \ZZ$.
  The statement for bigness now follows from Proposition
  \ref{prop:volumegenfin}.
  The statement for pseudoeffectivity follows from taking limits,
  since the pullback of an effective $\QQ$-Cartier divisor is an effective $\QQ$-Cartier divisor.\smallskip
  \par We now show the direction \(\Leftarrow\) in \((\ref{lem:bigpullbackbig})\).
  If $D\in \Pic_\QQ(X)$,
  since the volume is homogeneous (Proposition \ref{prop:volhomog}),
  we can clear denominators and reduce to the case $D\in \Pic(X)$,
  and the statement follows from
  Proposition
  \ref{prop:volumegenfin}.
\end{proof}
To prove Lemma \ref{lem:bigpullback} for \(\RR\) coefficients, we need the
following lemma for \(\kk = \QQ\).
This lemma says that the sum of a $\pi$-big and
$\pi$-nef or $\pi$-pseudoeffective $\kk$-invertible sheaf is $\pi$-big.
\begin{lemma}\label{lem:bigplusnef}
Let $\pi\colon X \to Z$ be a proper surjective morphism between integral
algebraic spaces over a scheme $S$.
Let $D$ be a $\pi$-big $\kk$-invertible sheaf on $X$ for $\kk \in
\{\ZZ,\QQ,\RR\}$.
If $D'$ is a $\pi$-nef (resp.\ $\pi$-pseudoeffective) $\kk$-invertible sheaf on
$X$, then $D+D'$ is $\pi$-big.
\end{lemma}
\begin{proof}[Proof of Lemma \ref{lem:bigplusnef} when \(\kk \in \{\ZZ,\QQ\}\)]
  Replacing $Z$ by the spectrum of its generic point,
  we may assume that $Z =
  \Spec(k)$ for a field $k$.
  By the weak version of Chow's lemma (Lemma \ref{lem:weakchow}),
  there exists a generically finite morphism
  $\mu\colon X' \to X$ from a projective
  variety over $k$.
  We then see that $\mu^*D$ is big by Lemma \ref{lem:bigpullback}, and that
  $\mu^*D'$ is nef by Lemma \ref{lem:nefpullback} (resp.\
  pseudoeffective by Lemma \ref{lem:bigpullback} for \(\kk = \QQ\)).
  By Kodaira's lemma (Corollary \ref{lem:kodairachar}) applied to \(\mu^*D\) on
  \(X'\), we can write $\mu^*D
  = A+E$ in $\Pic_\kk(X')$ where $A$ is ample and $E$ is effective.
  Thus, we have
  \[
    \mu^*(D + D') = A + \mu^*D' + E.
  \]
  If $D'$ is nef, then $A+\mu^*D'$ is ample by Kleiman's criterion
  (Proposition \ref{lem:AmpleIsPositiveOnNE}), and hence $\mu^*(D+D')$ is big
  by Kodaira's lemma (Corollary \ref{lem:kodairachar}).
  If $D'$ is pseudoeffective, then $\mu^*D'$ can be written as a limit of
  effective $\QQ$-Cartier divisors $F_i$ as $i \to \infty$.
  Writing
  \[
    \mu^*(D+D') = A + (\mu^*D' - F_i) + F_i + E,
  \]
  we see that $A + (\mu^*D' - F_i)$ is ample for $i \gg 0$ by Theorem
  \ref{thm:ampisinterior}, and hence $\mu^*(D+D')$ is big by Kodaira's lemma
  (Corollary \ref{lem:kodairachar}).
  Finally, we conclude that $D+D'$ is big by Lemma \ref{lem:bigpullback} for
  \(\kk = \QQ\).
\end{proof}
We can now prove Lemmas \ref{lem:bigpullback} and \ref{lem:bigplusnef} for
\(\RR\) coefficients.
\begin{proof}[Proof of Lemma \ref{lem:bigpullback} when \(\kk = \RR\)]
  Replacing $Z$ by the spectrum of the generic point of \(\pi(X)\), we
  may assume that $Z =
  \Spec(k)$ for a field $k$.
  The proof of the implications \(\Rightarrow\) is the same as the \(\kk \in
  \{\ZZ,\QQ\}\) case.
  It therefore remains to show that for $D\in \Pic_\RR(X)$,
  if \(f\) is generically finite and \(f^*D\)
  is big, then \(D\) is big.\smallskip
  \par Write
  \begin{align*}
    D =& \sum_i a_i D_i,
  \intertext{where the $D_i$ are distinct elements of $\Pic_{\ZZ}(X)$
  and $a_i$ are real numbers.
  By the weak version of Chow's lemma (Lemma \ref{lem:weakchow}),
  there exists a generically finite morphism
  $\mu\colon X'' \to X'$ from a projective
  variety over $k$.
  Since $f^*D$ is big, the previous paragraph implies
  $\mu^*f^*D$ is big.
  By Kodaira's Lemma (Corollary \ref{lem:kodairachar}) 
  and the openness of the ample cone (Theorem \ref{thm:ampisinterior}) applied
  on \(X''\),
  we see that letting}
    D_0 =& \sum_i a_i'D_i
  \intertext{be a sufficiently close approximation
  of \(D\) such that \(a_i' \in \QQ\), the pullback $\mu^*f^*D_0$ is also big.
  Thus, the rational case of \((\ref{lem:bigpullbackbig})\) shown above implies
  \(D_0\) is big.
  Taking the limits \(a_i' \to a_i\), we see that \(D\) is pseudoeffective.
  Thus, we can write}
    D =& \sum_{j=1}^r b_jD_j'
  \intertext{where \(b_j \in \RR_{>0}\) and the \(D_j'\) are effective for every \(j\).
  Repeating the argument above, we know that for sufficiently close rational
  approximations \(b'_j < b_j\), the \(\QQ\)-invertible sheaf}
    D_0' =& \sum_{j=1}^r b_j'D_j'
  \intertext{is big.
  We can therefore write}
  D = D_0' \mathop{+}& \sum_{j=1}^r (b_j-b_j')D_j'
  \end{align*}
  as the sum of a big \(\QQ\)-invertible sheaf and an \(\RR\)-linear combination of
  effective invertible sheaves.
  \par We claim that \(D_0' + \sum_{j=1}^r (b_j-b_j')D_j'\) is big by induction
  on \(r\).
  If \(r = 0\), then there is nothing to show.
  If \(r > 0\), we adapt the proof in \cite[Lemma A.16]{Fuj14}.
  Let \(c_1,c_2 \in \QQ_{>0}\) such that \(c_1 < b_r - b_r' < c_2\).
  Let \(t \in (0,1)\) be a real number such that \(tc_1+(1-t)c_2 = b_r-b_r'\).
  Then, we have 
  \begin{align*}
    D_0' + \sum_{j=1}^{r} (b_j-b_j')D_j'
    \MoveEqLeft[5]=
    t\Biggl(\biggl(D_0' + \sum_{j=1}^{r-1} (b_j-b_j')D_j'\biggr) +
    c_1D_r'\Biggr)\\
    &+
    (1-t)\Biggl(\biggl(D_0' + \sum_{j=1}^{r-1} (b_j-b_j')D_j'\biggr) +
    c_2D_r'\Biggr).
  \end{align*}
  By Lemma \ref{lem:bigplusnef} for \(\kk = \QQ\), the two terms above are big.
  We therefore see that \(D_0' + \sum_{j=1}^{r} (b_j-b_j')D_j'\)
  is an \(\RR\)-linear combination of big \(\QQ\)-invertible sheaves, and is
  therefore big.
\end{proof}
\begin{proof}[Proof of Lemma \ref{lem:bigplusnef} when \(\kk = \RR\)]
  Now that we have established Lemma \ref{lem:bigpullback} for \(\RR\)
  coefficients,
  the proof of Lemma \ref{lem:bigplusnef} for \(\kk \in \{\ZZ,\QQ\}\) also
  appplies to the case \(\kk = \RR\).
\end{proof}
We want to show that 
bigness and pseudoeffectivity are well-behaved under birational transforms.
To do so, we need a suitable version of the negativity lemma.
See \cite[\href{https://stacks.math.columbia.edu/tag/0ED7}{Tag
0ED7}]{stacks-project} for the definition of universally catenary algebraic
spaces that appears in the version of the negativity lemma below.
\begin{lemma}[{Negativity Lemma; cf.\ \cite[Lemma 2.16]{BMPSTWW}}]\label{lem:Negativity}
  Let $h\colon X \to Y$ be a proper birational morphism of integral normal
  quasi-excellent Noetherian algebraic spaces over a scheme $S$
  that are universally catenary or
  have dualizing complexes.
  Let $B$ be a Weil divisor on $X$ such that $[B]$ is the class of an invertible
  sheaf $\sL$. 
  Assume that $\sL^{-1}$ is $h$-nef and that $h_*B$ is effective. 
  Then $B$ is effective.
\end{lemma}
\begin{proof}
  After replacing $Y$ by an \'etale cover $Y' \to Y$, we may reduce to the case of schemes.
  Note that $Y'$ is quasi-excellent by definition, and is moreover excellent
  either because $Y$ is universally catenary, or because $Y'$ has a dualizing
  complex.
  Nefness of $\sL^{-1}$ is preserved by Lemma
  \ref{lem:nefpullback}$(\ref{lem:nefpullbackalways})$.
  The effectivity of $B$ can be checked after flat base change.
\par When $Y$ is a scheme and $h$ is projective, this is \cite[Lemma 2.16]{BMPSTWW}. 
The general case follows from Chow's Lemma \cite[Th\'eor\`eme 5.6.1]{EGAII}, since we may pass to an affine open cover of $Y$ and pullback along a birational morphism $X'\to X$.
\end{proof}
We can now show that 
bigness and pseudoeffectivity are well-behaved under birational transforms.
\begin{lemma}[cf.\ {\cite[Chapter II, Lemma 5.6(1), Remark on p.\ 69]{Nak04}}]
  \label{lem:ContrPreservesBig}
  Let $S$ be a scheme.
  Let
  \[
    \begin{tikzcd}[column sep=0.75em]
      X \arrow{rr}{f}\arrow{dr}
      & & Y\arrow{dl}[near start]{g}\\
      & Z
    \end{tikzcd}
  \]
  be a commutative diagram of integral quasi-excellent algebraic spaces over $S$ where \(X\) and
  \(Y\) are normal, \(g\) is proper, and \(f\) is proper and birational.
Let $D$ be a $\QQ$-Weil divisor that is $\QQ$-Cartier on $X$ such that the birational transform $f_*D$ is $\QQ$-Cartier.
If $D$ is big over $Z$, so is $f_*D$.
The same statement holds for \(\RR\) coefficients if \(Y\) is \(\QQ\)-factorial,
or more generally, if \(f_*D_i\) is
\(\RR\)-Cartier for every invertible sheaf \(D_i\) appearing in \(D\).
\par Moroever, suppose that \(f\) is an isomorphism in codimension one.
If \(f_*D\) is big over \(Z\), then \(D\) is big over \(Z\).
\end{lemma}
\begin{proof}
By Definition \ref{def:fbig}, we may take the fiber over the generic point of
$(g \circ f)(X)$ to assume that $Z$ is the spectrum of a field.\smallskip
\par If $m\in\ZZ_{>0}$ is sufficiently divisible, then $mD$ and $m\,f_*D$ are Cartier. 
Then, the difference \(f^*f_*D-D\) is effective by the Negativity Lemma \ref{lem:Negativity}.
We therefore see that
\[
  f^*f_*D = (f^*f_*D-D) + D 
\]
is big by Lemma \ref{lem:bigplusnef}.
We conclude that \(f_*D\) is big by Lemma
\ref{lem:bigpullback}\((\ref{lem:bigpullbackbig})\).
For \(\RR\) coefficients, we can prove that \(f^*f_*D - D\) is effective by
applying the Negativity Lemma \ref{lem:Negativity} to each term \(D_i\)
appearing in \(D\).\smallskip
\par For the ``Moreover'' statement, we know that \(f^*f_*D=D\) because
\(f\) is an isomorphism in codimension one.
We conclude that $f_*D$ is big if and only if $D$ is big by Lemma
\ref{lem:bigpullback}\((\ref{lem:bigpullbackbig})\).
\end{proof}
For pseudoeffectivity, we have the following result.
The condition that \(Y\) has a \(g\)-big invertible sheaf in
\((\ref{lem:ContrPreservesPsEffPushForward})\) below holds if \(Y\) is
projective over \(Z\).
However, there exist normal complete toric varieties that do not have any big
invertible sheaves \cite[Example A.17]{Fuj14}.
\begin{lemma}[cf.\ {\cite[Chapter II, Lemma 5.6(2), Remark on p.\ 69]{Nak04}}]
  \label{lem:ContrPreservesPsEff}
  Let $S$ be a scheme.
  Let
  \[
    \begin{tikzcd}[column sep=0.75em]
      X \arrow{rr}{f}\arrow{dr}
      & & Y\arrow{dl}[near start]{g}\\
      & Z
    \end{tikzcd}
  \]
  be a commutative diagram of integral quasi-excellent algebraic spaces over $S$ where \(X\) and
  \(Y\) are normal, \(g\) is proper, and \(f\) is proper and birational.
Let \(\kk \in \{\QQ,\RR\}\).
Let $D$ be a $\kk$-Weil divisor that is $\kk$-Cartier on $X$ such that the birational transform $f_*D$ is $\kk$-Cartier.
\begin{enumerate}[label=\((\roman*)\),ref=\roman*]
  \item\label{lem:ContrPreservesPsEffPushForward}
    Suppose that one of the following conditions hold:
    \begin{itemize}
      \item \(\kk = \QQ\) and \(Y\) has a \(g\)-big invertible sheaf.
      \item \(\kk = \RR\) and \(Y\) is \(\QQ\)-factorial, or more generally,
        \(D\) is the limit of \(\RR\)-invertible sheaves \(F_i\) associated to
        effective \(\RR\)-Cartier divisors such
        that \(f_*F_i\) is \(\RR\)-Cartier for every \(i\).
    \end{itemize}
    If $D$ is pseudoeffective over $Z$, so is $f_*D$.
  \item\label{lem:ContrPreservesPsEffConverse}
    Let \(\kk = \RR\).
    Suppose that \(f\) is an isomorphism in codimension one.
    If \(f_*D\) is pseudoeffective over \(Z\), then \(D\) is pseudoeffective over \(Z\).
\end{enumerate}
\end{lemma}
\begin{proof}
By Definition \ref{def:fbig}, we may take the fiber over the generic point of
$(g \circ f)(X)$ to assume that $Z$ is the spectrum of a field.\smallskip
\par We first show \((\ref{lem:ContrPreservesPsEffPushForward})\) when \(\kk =
\QQ\) and \(Y\) has a big invertible sheaf.
For every big \(\QQ\)-Cartier divisor \(B\) on \(Y\), we have
\[
  f_*D+B=f_*(D+f^*B).
\]
The sum $D+f^*B$ is big by Lemma \ref{lem:bigplusnef} since $f^*B$ is big.
By Lemma \ref{lem:ContrPreservesBig}, $f_*D+B$ is big, and hence $f_*D$ is pseudoeffective.
\par We now show \((\ref{lem:ContrPreservesPsEffPushForward})\) when \(\kk =
\RR\) and \(D\) is the limit of \(\RR\)-invertible sheaves \(F_i\) associated to
effective \(\RR\)-Cartier divisors such
that \(f_*F_i\) is \(\RR\)-Cartier for every \(i\).
Then, \(f_*D\) is the limit of the \(f_*F_i\), and is therefore pseudoeffective.\smallskip
\par Finally, we show \((\ref{lem:ContrPreservesPsEffConverse})\).
Write \(f_*D\) as a limit of \(\RR\)-invertible sheaves \(F_i\) associated to
effective \(\RR\)-Cartier divisors.
Then, \(f^*f_*D\) is the limit of the \(f^*F_i\).
Since \(f\) is an isomorphism in codimension one, we have \(D = f^*f_*D\), and
hence \(D\) is pseudoeffective.
\end{proof}

\subsection{Linear systems and generic fibers}
Relative bigness and relative
pseudoeffectivity only depend on the generic fiber, and hence we
describe how linear systems behave when passing to the generic fiber of a
morphism.
\begin{lemma}[cf.\ {\cite[Lemma 3.2.1]{BCHM10}}]\label{lem:LinearSysLocalizes}
  Let $\pi\colon X\to Z$ be a proper surjective morphism of integral Noetherian
  schemes, where
  $X$ is normal and $Z$ is affine.
  Consider a point $z \in Z$, and set $R \coloneqq \cO_{Z,z}$ and $X_R \coloneqq
  X \times_Z \Spec(R)$.
  Let $D$ be a $\kk$-Weil divisor on $X$ and let $E$ an effective $\kk$-Weil divisor on
  $X_R$ such that $E_{\vert X_R} \sim_\kk D_{\vert X_R}$
  for some $\kk \in
  \{\ZZ,\QQ,\RR\}$.
  Then, there exists an effective $\kk$-Weil divisor $F$ on $X$ such that $F
  \sim_\kk D$ and
  $F_{\vert X_R}=E$.
\end{lemma}
\begin{proof}
  Let $E=\sum_{i=1}^n a_iE_i$ where $a_i\in\mathbf{k}$ and $E_i$ are prime divisors on $X_R$. There exist rational functions $f_1,f_2,\ldots,f_m$ on $X_R$ and numbers $b_1,b_2,\ldots,b_m\in\mathbf{k}$ such that
\begin{align*}
D_{|X_R}=\sum_{i=1}^n a_iE_i&+\sum_{j=1}^m b_j\prdiv_{X_R}(f_j).
\intertext{Since the function fields of $X$ and $X_R$ are the same, the functions $f_j$
define principal divisors $\prdiv_{X}(f_j)$ on $X$.
For each $i$, we also obtain a prime divisor $\overline{E}_i$ on $X$ as the closure of $E_i$. 
Let}
  D'=D-\sum_{i=1}^n a_i\overline{E}_i&-\sum_{j=1}^m b_j\prdiv_{X}(f_j).
\end{align*}
Then, $D'$ is a $\mathbf{k}$-linear combination of prime divisors that avoid $X_R$.
In other words, we have $D'=\sum_k c_kS_k$ where $(S_k)_{\vert X_R}=0$ and
$c_k\in\mathbf{k}$ for every $k$. 
If we can prove the result for $\sgn(c_k)S_k$ for each $k$ (and
$\mathbf{k}=\mathbf Z$) then we are done.

Let $\mathcal{F}=\mathcal O_X(\sgn(c_k)S_k)$.
By flat base change \cite[Proposition 1.4.15]{EGAIII1}, we have
\begin{align*}
  H^0(X,\mathcal{F})\otimes_{H^0(Z,\cO_Z)}R
  &= H^0(X_R,\mathcal{F}_R)\\
  &\simeq H^0(X_R,\cO_{X_R}).
\end{align*}
Since $H^0(X,\mathcal{F})$ is torsion-free as an $H^0(Z,\cO_Z)$-module
\cite[Proposition 8.4.5]{EGAInew},
there exists a section $s\in H^0(X,\mathcal{F})$ such that $s$ maps to a nonzero
section of $\mathcal{F}_R$.
We then have $\prdiv(s) \sim \sgn(c_k)S_k$ while $\prdiv(s)_{\vert X_R}=0$, and
hence we are done.
\end{proof}

\begin{corollary}\label{cor:EffIffEffAtGenFiber}
Let $\pi\colon X\to Z$ be a proper surjective morphism of integral
  Noetherian schemes with $X$ normal and $Z$ affine.
  Consider a point $z \in Z$, and set $R \coloneqq \cO_{Z,z}$ and $X_R \coloneqq
  X \times_Z \Spec(R)$.
  Let $D$ be a $\kk$-Weil divisor on $X$ where $\kk \in
  \{\ZZ,\QQ,\RR\}$. 
  Then $\lvert D \rvert_{\kk}\neq\emptyset$ if and only if $\lvert D_{\vert X_R}
  \rvert_{\kk}\neq\emptyset$.
\end{corollary}

\subsection{Relatively big \texorpdfstring{$\RR$}{R}-Weil divisors}
We now define $\pi$-bigness
for $\QQ$- or $\RR$-Weil divisors when \(\pi\) is projective (which is also
assumed when in \cite{CU15}), or more generally,
when the generic fiber \(X_\eta\) is projective.
\begin{definition}[cf.\ {\cite[Definition 2.16]{CU15}}]\label{def:BigIsAmpleAddEffective}
Let $\pi\colon X\to Y$ be a proper surjective morphism of integral locally Noetherian
algebraic spaces over a scheme $S$.
Let $X_\eta$ be
the generic fiber of $\pi$
and assume $X_\eta$ projective over $\kappa(\eta)$.
Let $D$ be a $\kk$-Weil divisor on $X$ where $\kk\in\{\QQ,\RR\}$. We say that $D$ is \emph{$\pi$-big} 
if $D_{\vert X_\eta} \sim_\kk A+E$ for an ample $\kk$-invertible sheaf $A$ on
$X_\eta$ and an effective $\kk$-Weil divisor $E$ on $X_\eta$.
\end{definition}

If $\pi$ is birational, then clearly every $\kk$-Weil divisor is
$\pi$-big.\medskip

Definition \ref{def:BigIsAmpleAddEffective} is equivalent to Definition \ref{def:fbig} for $\kk$-invertible sheaves or $\kk$-Cartier divisors.
The characterization below for \(Z\) affine and \(\pi\) projective
is the definition taken in
\cite[Definition 2.16]{CU15}.
\begin{lemma}\label{lem:BigWeilIsAmplePlusEffective}
  Let $\pi\colon X\to Z$ be a proper morphism of locally Noetherian
  schemes, such that $X$ is normal and $X_\eta$ is
  projective over $\kappa(\eta)$.
  Let $\kk\in\{\QQ,\RR\}$ and let $D$ be a $\kk$-Weil divisor on $X$.
  If $D$ is $\kk$-Cartier, $D$ is $\pi$-big in the sense of Definition \ref{def:BigIsAmpleAddEffective} if and only if $D$ is $\pi$-big in the sense of Definition \ref{def:fbig}.
  
  If $Z$ is affine and $\pi$ is projective, $D$ is $\pi$-big in the sense of Definition \ref{def:BigIsAmpleAddEffective} if and only if there exists a $\pi$-ample $\kk$-Cartier
divisor $A$ and an effective $\kk$-Weil divisor $E$ with $D\sim_\kk A+E$.
\end{lemma}
\begin{proof}
  The first statement follows from Corollary \ref{lem:kodairachar}.

  Now assume that $Z$ affine and $\pi$ is projective.
The implication $\Leftarrow$ is trivial, so we assume that $D$ is $\pi$-big in the sense of Definition \ref{def:BigIsAmpleAddEffective}. 
Let $A^\eta$ and $E^\eta$ be divisors on the generic fiber $X_\eta$ as in Definition \ref{def:BigIsAmpleAddEffective}. 
Let $H$ be a $\pi$-ample $\QQ$-Cartier divisor on $X$. After scaling, we may
assume $A^\eta-H_{|X_\eta}$ ample, so we see that
\[
  \bigl\lvert (D-H)_{\rvert X_\eta}\bigr\rvert_\kk\neq \emptyset.
\]
By Corollary
\ref{cor:EffIffEffAtGenFiber}, $\lvert D-H\rvert_\kk\neq \emptyset$, as desired.
\end{proof}

\section{Canonical sheaves, canonical divisors,\texorpdfstring{\except{toc}{\\}}{}
and singularities of pairs}\label{sect:candiv}
\subsection{Canonical sheaves and divisors}
We define canonical sheaves.
\begin{definition}[cf.\ {\citeleft\citen{KMM87}\citemid Remark
  0-2-2(2)\citepunct \citen{Cor92}\citemid
  (16.3.3)\citepunct \citen{Kov12}\citemid
  \S5\citeright}]\label{def:canonicalsheaf}
  Let $X$ be an equidimensional and connected locally Noetherian
  algebraic space over a scheme $S$.
  Suppose that $X$ has a dualizing complex $\omega_X^\bullet$.
  The \textsl{canonical sheaf} $\omega_X$ associated to $\omega_X^\bullet$
  is the cohomology sheaf of
  $\omega_X^\bullet$ in lowest cohomological degree.
\end{definition}
We can also often make sense of $\omega_X$ as a Weil divisor.
\begin{definition}[cf.\ {\citeleft\citen{KMM87}\citemid Remark
  0-2-2(2)\citepunct \citen{Cor92}\citemid (16.3.3)\citepunct
  \citen{Kov12}\citemid \S5\citeright}]\label{def:kx}
  Let $X$ be an equidimensional and connected locally Noetherian
  algebraic space over a scheme $S$.
  Suppose that $X$ has a dualizing complex $\omega_X^\bullet$ with associated
  canonical sheaf $\omega_X$.
  The sheaf $\omega_X$ is invertible on an open subspace $U \subseteq X$, since
  it is the complement of the closed subspace where
  \[
    \omega_X \otimes_{\cO_X}
    \HHom_{\cO_X}(\omega_X,\cO_X) \longrightarrow \cO_X
  \]
  is not an isomorphism by
  \cite[\href{https://stacks.math.columbia.edu/tag/0B8N}{Tag
  0B8N}]{stacks-project}.
  \par Now suppose that $X$ is integral and normal.
  Since $X$ is normal, $U$ contains all codimension one points of $X$.
  A \textsl{canonical divisor} $K_X$ on $X$ is a Weil divisor whose class in
  $\Cl(X)$ restricts to the image of $\omega_U$
  under the map $\Pic(U) \to
  \Cl(U)$ from \eqref{eq:stacks0EPW}.
\end{definition}
\begin{convention}\label{convention:kx}
  Let \(X\) be an equidimensional and connected locally Noetherian algebraic
  space over a scheme \(S\) and suppose that \(X\) has a dualizing complex
  \(\omega_X^\bullet\).
  We call the canonical divisor \(K_X\) constructed in Definition \ref{def:kx} 
  a \textsl{canonical divisor associated to \(\omega_X^\bullet\)}.
\end{convention}
\subsection{Singularities of pairs}
We can now define pairs and singularities of pairs in our setting.
\begin{definition}[see
  \citeleft\citen{Kol13}\citemid Definition 1.5 and (2.20)\citeright]
  \label{def:rpairs}
  Let $X$ be
  an integral normal locally
  Noetherian algebraic space over a scheme $S$.
  Suppose that $X$ has a
  dualizing complex $\omega_X^\bullet$ with associated canonical divisor $K_X$.
  Let $\kk \in \{\QQ,\RR\}$.
  A \textsl{$\kk$-pair} $(X,\Delta)$ is the combined data of $X$ together
  with an effective $\kk$-Weil divisor $\Delta$ such that $K_X + \Delta$ is
  $\kk$-Cartier.
\end{definition}
We will also use the following definition.
For algebraic spaces, we take the characterization in
\cite[\href{https://stacks.math.columbia.edu/tag/0BIA}{Tag
0BIA}(2)]{stacks-project} as our definition for a simple normal crossings
divisor.
\begin{definition}[see {\cite[p.\ 2418]{CL12}}]
  Let $(X,\Delta)$ be a $\kk$-pair for $\kk \in \{\QQ,\RR\}$.
  We say that
  $(X,\Delta)$ is \textsl{log regular} if $X$ is regular and $\Delta$
  has simple normal crossings support.
\end{definition}
For the definition below, we note that \cite{Kol13} works over a regular scheme
$B$ throughout (see \cite[Definition 1.5]{Kol13}), but this is not necessary for the
following definition to make sense, since we are assuming the existence of a
dualizing complex $\omega_X^\bullet$.
\begin{definition}[{see \citeleft\citen{KMM87}\citemid Definitions 0-2-6 and
  0-2-10\citepunct
  \citen{Kol13}\citemid Definitions
  2.4 and 2.8\citeright}]\label{def:singpairs}
  Let $(X,\Delta)$ be a $\kk$-pair for $\kk \in \{\QQ,\RR\}$.
  For a separated birational morphism $f\colon Y \to X$ of finite type
  from an integral normal locally Noetherian algebraic space $Y$ over $S$, we
  can write
  \[
    K_Y + f_*^{-1}\Delta \sim_\kk f^*(K_X+\Delta) + \sum_{\text{$f$-exceptional
    $E$}} a(E,X,\Delta)E
  \]
  for some $a(E,X,\Delta) \in \kk$,
  where the $E$ are $f$-exceptional prime
  Weil divisors and $f_*^{-1}\Delta$ is the birational transform of $\Delta$.
  \par For each $f$-exceptional
  prime Weil divisor $E$ on $Y$, the number $a(E,X,\Delta) \in
  \kk$ is called the \textsl{discrepancy of $E$} with respect to $(X,\Delta)$.
  For nonexceptional prime Weil divisors $D \subseteq X$, we set
  \begin{alignat*}{3}
    a(D,X,\Delta) &\coloneqq {-{\coeff_D(\Delta)}}.
    \intertext{If $f'\colon Y' \to X$ is another birational morphism and $E'
    \subseteq Y'$ is the birational transform of $E$, then}
    a(E,X,\Delta) &= a(E',X,\Delta),
  \end{alignat*}
  and hence the discrepancy of $E$ only depends on $E$ and not
  on $Y$.
  The \textsl{center} $\cent_X(E)$ of $E$ is the image of $E$ in $X$.
  \par Now suppose that $\Delta$ has coefficients in $[0,1]$.
  We say that $(X,\Delta)$ is
  \begin{center}
    \begin{tabular}{c@{}c@{}l}
      \textsl{terminal}
      & \rdelim\}{4}{*} \multirow{4}{*}{\;if $a(E,X,\Delta)$ is}
      \ldelim\{{4}{*} & $> 0$ for every exceptional $E$,\\
      \textsl{canonical} & & $\ge 0$ for every exceptional $E$,\\
      \textsl{klt} & & $> -1$ for every $E$,\\
      \textsl{dlt} & & $> -1$ for every $E$ such that $\cent_X(E)
      \subseteq \nonsnc(X,\Delta)$.
    \end{tabular}
  \end{center}
  Here, the divisors $E$ range over all prime Weil divisors on schemes $Y$
  birational over $X$ as above.
\end{definition}
We will also state some results using the notion of weakly log terminal
singularities from \cite{KMM87}, which is class of singularities of pairs that
is larger than the class of dlt singularities.
\begin{definition}[see {\cite[Definition 0-2-10]{KMM87}}]\label{def:wlt}
  Let $(X,\Delta)$ be a $\kk$-pair for $\kk \in \{\QQ,\RR\}$
  such that $X$ is quasi-excellent
  of equal characteristic zero and such that $\Delta$ has coefficients in
  $[0,1]$.
  We say that $(X,\Delta)$ is \textsl{weakly log terminal} if the following
  conditions hold:
  \begin{enumerate}[label=$(\roman*)$]
    \item There exists a resolution of singularities $f\colon Y \to X$ such that
      \[
        \Supp\bigl(f_*^{-1}\Delta\bigr) \cup \Exc(f)
      \]
      has normal crossings support (in the
      sense of \cite[\href{https://stacks.math.columbia.edu/tag/0BSF}{Tag
      0BSF}]{stacks-project}) and
      $a(E,X,\Delta) > -1$ for every $f$-exceptional $E$.
    \item There exists an $f$-ample invertible sheaf $\sH$ whose image in
      $\Cl(Y)$ is equal to the class of a Weil divisor whose support equals
      $\Exc(f)$.
  \end{enumerate}
\end{definition}
\begin{remark}[see {\citeleft\citen{Sza94}\citemid Divisorial log terminal
  theorem\citepunct
  \citen{Fuj17}\citemid Remark 2.3.22\citeright}]\label{rem:dltiswlt}
  Let $X$ be as in Definition \ref{def:wlt}.
  Since thrifty log resolutions exist in this setting by
  \cite[Theorems 1.1.6 and 1.1.13]{Tem18}, we see that dlt pairs are weakly
  log terminal.
\end{remark}
\begin{remark}\label{rem:singsetalelocal}
  Since terminal, canonical, and klt
  are \'etale-local conditions \cite[(2.14) and Proposition 2.15]{Kol13}, one
  can also define these notions for algebraic spaces by pulling back to an
  \'etale cover of $X$.
  Note that dlt is not an \'etale-local condition because of the simple normal
  crossing condition \cite[Warning on p.\ 47]{Kol13}.
\end{remark}
We will use the following lemma.
\begin{lemma}\label{lem:kltFacts}
  Let $(X,\Delta)$ be a $\kk$-pair for $\kk \in \{\QQ,\RR\}$,
and let $\Delta'$ be an effective $\kk$-Weil divisor on $X$.
Then, we have the following:
\begin{enumerate}[label=$(\roman*)$,ref=\roman*]
    \item \label{lem:kltSmaller}
    If $\Delta'$ is $\kk$-Cartier and $(X,\Delta+\Delta')$ is klt, then $(X,\Delta)$ is klt.
    
    \item \label{lem:kltConvex}
      Suppose $K_X+\Delta'$ is $\kk$-Cartier.
    If \((X,\Delta)\) and $(X,\Delta')$ are klt, then
    \[
      \bigl(X,t\Delta+(1-t)\Delta'\bigr)
    \]
    is klt for all
    $t\in [0,1] \cap \kk$.
    
    \item \label{lem:kltContinuous}
      Assume that $(X,\Delta)$ has a log resolution, that $(X,\Delta)$ is klt, and that $\Delta'$ is $\kk$-Cartier. 
    Then, for all sufficiently small $\varepsilon\in \kk_{>0}$, the pair
    $(X,\Delta+\varepsilon\Delta')$ is klt.
    
    \item \label{lem:kltContinuousConvex}
      Suppose $K_X+\Delta'$ is $\kk$-Cartier.
      Assume that $(X,\Delta)$ has a log resolution, that $(X,\Delta)$ is klt.
    Then, the pair
    \[
      \bigl(X,(1-\varepsilon)\Delta+\varepsilon\Delta'\bigr)
    \]
    is klt for all sufficiently small $\varepsilon\in \kk_{>0}$.
\end{enumerate}
\end{lemma}
\begin{proof}
  Items $(\ref{lem:kltSmaller})$ and $(\ref{lem:kltConvex})$ follow immediately from the definition of discrepancy. 
  For $(\ref{lem:kltContinuous})$ and  $(\ref{lem:kltContinuousConvex})$, it suffices to note
that klt-ness is detected by a single log resolution \cite[Corollary 2.13]{Kol13}. 
\end{proof}
\section{Base loci and restricted linear systems}\label{subsection:baseloci}
We define base loci and some of their asymptotic invariants, which we use to
define restricted linear systems.
\begin{definition}[see {\citeleft\citen{KMM87}\citemid p.\ 299\citepunct
  \citen{CL12}\citemid p.\ 2419\citepunct
  \citen{McK17}\citemid Definition 2.2\citeright}]
  Let $X$ be a normal locally Noetherian scheme or an integral normal locally
  Noetherian algebraic space over a scheme $S$.
  The \textsl{base locus} of a Weil divisor $D$ is the closed set
  \begin{align*}
    \Bs\lvert D \rvert &\coloneqq \bigcap_{D' \in \lvert D \rvert} \Supp(D').
    \intertext{We set $\Bs\lvert D \rvert = X$ if $\lvert D \rvert = \emptyset$.
    The \textsl{stable base locus} of an $\RR$-Weil divisor $D$ is the
    closed set}
    \SB(D) &\coloneqq \bigcap_{D' \in \lvert D \rvert_\RR} \Supp(D').
  \end{align*}
  We set $\SB(D) = X$ if $\lvert D \rvert_\RR = \emptyset$.
\end{definition}
We can now define restricted linear systems.
\begin{definition}[see {\citeleft\citen{ELMNP09}\citemid p.\ 612\citepunct
  \citen{CL12}\citemid p.\
  2420 and Definition 2.23\citeright}]\label{def:restrictedlinearsystem}
  Let $X$ be an algebraic space over a scheme $S$, and let $T \subseteq X$ be a
  closed subspace.
  For an invertible sheaf $\sL$ on $X$, we set
  \[
    H^0\bigl(X \vert T,\sL\bigr) \coloneqq \im\Bigl(
    H^0\bigl(X,\sL\bigr) \longrightarrow H^0\bigl(T,\sL_{\vert T}\bigr)\Bigr),
  \]
  which is denoted $\res_T(H^0(X,\sL))$ in \cite[Definition 2.23]{CL12}.
  \par Now suppose $X$ is a normal Noetherian scheme, $T$ is normal, and $D$ is a Cartier
  divisor intersecting $T$ properly.
  The \textsl{restricted linear system} $\lvert D \rvert_T$ is the subset of
  $\lvert D_{\vert T} \rvert$ corresponding to nondegenerate sections in
  $H^0(X \vert T,\cO_X(D))$ under the bijection in Proposition \ref{prop:har29}.
  The restriction map
  \[
    H^0\bigl(X,\cO_X(D)\bigr) \longrightarrow
    H^0\bigl(T,\cO_T(D_{\vert T})\bigr)
  \]
  induces a
  map $\lvert D \rvert \to \lvert D \rvert_T$ if $T$ is integral and $T
  \not\subseteq \Bs\lvert D \rvert$, since $H^0(X,\cO_X^*)$ maps to
  $H^0(T,\cO_T^*)$ and nondegenerate sections of $\cO_X(D)$ map to nondegenerate
  sections of $\cO_T(D_{\vert T})$.
\end{definition}
We now want to define the fixed and stable fixed parts of a linear system.
To do so, we need the following result, which shows that the definition of
$\SB(D)$ is compatible with the usual definition for $\QQ$-Cartier divisors in
\cite[Definition 2.1.20]{Laz04a}.
\begin{lemma}[see \citeleft\citen{BCHM10}\citemid Lemma 3.5.3\citepunct
  \citen{CL12}\citemid Lemma 2.3\citepunct \citen{McK17}\citemid Lemma
  2.4\citeright]\label{lem:cl23}
  Let $X$ be a normal locally Noetherian scheme or an integral normal locally
  Noetherian algebraic space over a scheme $S$.
  Consider a $\QQ$-Weil divisor $D$ 
  on $X$.
  Then, we have
  \[
    \SB(D) = \bigcap_{D' \in \lvert D \rvert_\QQ} \Supp(D').
  \]
\end{lemma}
\begin{proof}
This is immediate from Lemma \ref{lem:RatlIsDense}.
\end{proof}
Finally, we define fixed and mobile parts of linear systems, together with the
asymptotic variant of the fixed part.
\begin{definition}[{see \citeleft\citen{CL12}\citemid Definition
  2.5\citeright}]
  Let $X$ be a normal locally Noetherian scheme or an integral normal locally
  Noetherian algebraic space over a scheme $S$.
  Consider a Weil divisor $D$ on $X$.
  The \textsl{fixed part} $\Fix\lvert D \rvert$ of $D$ is the
  largest effective Weil divisor $F$ on $X$ such that $F \le D'$ for all $D'
  \in \lvert D \rvert$.
  We can then write
  \[
    \lvert D \rvert = \bigl\lvert \Mob(D) \bigr\rvert + \Fix\lvert D \rvert,
  \]
  where $\Mob(D)$ is the \textsl{mobile part} of $\lvert D \rvert$.
  If $T \subseteq X$ is a normal closed subscheme, we use the same definition for the
  restricted linear system $\lvert D \rvert_T$ to define the fixed part
  $\Fix\lvert D \rvert_T$. 
  \par Now consider a $\QQ$-Weil divisor $D$ on $X$.
  The \textsl{stable fixed part} of $D$ is
  \begin{align*}
    \FFix(D) &\coloneqq \liminf_{k \to \infty} \frac{1}{k} \Fix\lvert kD
    \rvert,\\
    \intertext{which by Lemma \ref{lem:cl23} is the divisorial part of the
    stable base locus $\SB(D)$.
    Similarly, we set}
    \FFix_T(D) &\coloneqq \liminf_{k \to \infty} \frac{1}{k} \Fix\lvert kD
    \rvert_T.
  \end{align*}
\end{definition}

\section{Convex sets in \texorpdfstring{$\Div_\RR(X)$}{Div\_R(X)} and relative
divisorial graded rings}\label{subsection:relativedivring}
We define some convex subsets of $\Div_\RR(X)$ associated to
finite-dimensional subspaces in $\Div_\RR(X)$, following \cite[\S2.1]{CL12}.
We will restrict to the scheme case in this section.
In the definition below,
$\cL(V)$ is a version of Shokurov's polytope $\mathcal{P}$ from
\cite[(1.3.2)]{Sho93} (see also \cite[First Main Theorem 6.2]{Sho96}),
and $\cE_A(V)$ is a version of Shokurov's polytope $\mathcal{M}$ from
\cite[Second Main Theorem 6.20]{Sho96}.
\begin{definition}[cf.\ {\cite[Definition 2.4]{CL12}}]\label{def:cl1224}
  Let $X$ be a regular locally Noetherian scheme with a dualizing complex
  $\omega_X^\bullet$.
  Denote by $K_X$ a canonical divisor on \(X\) associated to $\omega_X^\bullet$ (see
  Convention \ref{convention:kx}).
  Let $S_1,S_2,\ldots,S_p$ be distinct prime divisors on $X$ such that
  $(X,\sum_{i=1}^p S_i)$ is log regular.
  Let
  \[
    V = \sum_{i=1}^p \RR \cdot S_i \subseteq \Div_\RR(X),
  \]
  and let $A$ be a $\QQ$-divisor on $X$.
  We set
  \begin{align*}
    \cL(V) &\coloneqq \Set[\Big]{B = \sum b_iS_i \in V \given 0 \le b_i \le 1\
    \text{for all}\ i},\\
    \cE_A(V) &\coloneqq \Set[\big]{B \in \cL(V) \given \lvert K_X+A+B \rvert_\RR
    \ne \emptyset}.
\intertext{Let $S$ be a prime divisor on $X$ different from each $S_i$ such that 
$(X,S+\sum_{i=1}^p S_i)$ is log regular.
We set}
    \cB_A^S(V) &\coloneqq \Set[\big]{B \in \cL(V) \given S \not\subseteq
    \SB(K_X+S+A+B)}.
  \end{align*}
\end{definition}
We now define relative divisorial graded rings and establish some basic
properties about them, following \cite[\S2.4]{CL12}.
\begin{definition}[cf.\ {\citeleft\citen{KMM87}\citemid Definitions
  0-3-7 and 0-3-11\citepunct
  \citen{CL12}\citemid Definition 2.22\citepunct
  \citen{CL13}\citemid p.\ 620\citeright}]\label{def:cl12222}
  Let $\pi\colon X \to Z$ be a proper morphism of integral Noetherian schemes,
  where $X$ is regular and $Z$ is affine.
  Let $\cS \subseteq \Div_\QQ(X)$ be a finitely generated monoid.
  The \textsl{relative divisorial graded ring associated to $\cS$} is the
  $\cS$-graded $H^0(Z,\cO_Z)$-algebra
  \[
    R\bigl(X/Z;\cS\bigr) \coloneqq \bigoplus_{D \in \cS}
    H^0\bigl(X,\cO_X\bigl(\lfloor D \rfloor\bigr)\bigr).
  \]
  \par Now suppose that $Z$ has a dualizing complex $\omega_Z^\bullet$, and
  denote by $K_X$ a canonical divisor on \(X\) associated to $\omega_X^\bullet =
  \pi^!\omega_Z^\bullet$.
  If divisors $D_1,D_2,\ldots,D_\ell$ are generators of $\cS$ and if $D_i
  \sim_\QQ k_i(K_X+\Delta_i)$ for effective $\QQ$-divisors $\Delta_i$ and for
  $k_i \in \QQ_{\ge0}$, the algebra $R(X/Z;\cS)$ is called the \textsl{relative
  adjoint ring associated to $\cS$}, and the \textsl{relative adjoint ring
  associated to the sequence $D_1,D_2,\ldots,D_\ell$} is the $\NN^\ell$-graded
  $H^0(Z,\cO_Z)$-algebra
  \[
    R\bigl(X/Z;D_1,D_2,\ldots,D_\ell\bigr) \coloneqq
    \bigoplus_{(m_1,m_2,\ldots,m_\ell) \in \NN^\ell}
    H^0\Biggl(X,\cO_X\Biggl(\Biggl\lfloor \sum_{i=1}^\ell m_iD_i \Biggr\rfloor
    \Biggr)\Biggr).
  \]
  Note that there is a natural projection map $R(X/Z;D_1,D_2,\ldots,D_\ell) \to
  R(X/Z;\cS)$.
  The \textsl{support} of $R(X/Z;D_1,D_2,\ldots,D_\ell)$ is
  \[
    \Supp\Bigl(R\bigl(X/Z;D_1,D_2,\ldots,D_\ell\bigr)\Bigr)
    \coloneqq \biggl( \sum_{i=1}^\ell \RR_{\ge0} \cdot D_i \biggr)
    \cap \Div_\RR^\eff(X) \subseteq \Div_\RR(X).
  \]
  If $\cC \subseteq \Div_\RR(X)$ is a rational polyhedral cone, then
  Gordan's lemma \cite[\S1.2, Proposition 1]{Ful93} implies that $\cS = \cC \cap
  \Div(X)$ is a finitely generated monoid, and we define the \textsl{adjoint
  ring associated to $\cC$} to be
  \[
    R\bigl(X/Z;\cC\bigr) \coloneqq R\bigl(X/Z;\cS\bigr).
  \]
\end{definition}
\begin{definition}[cf.\ {\cite[Definition 2.23]{CL12}}]
  Let $\pi\colon X \to Z$ be a proper morphism of integral Noetherian schemes,
  where $X$ is regular and $Z$ is affine.
  Let $S$ be a regular prime divisor on $X$ and let $D$ be an effective divisor
  on $X$.
  Using Proposition \ref{prop:har29} (see also
  \cite[\href{https://stacks.math.columbia.edu/tag/01X0}{Tag
  01X0}]{stacks-project}), we fix $1_S \in H^0(X,\cO_X(S))$ such
  that $Z(1_S) = S$.
  Consider the exact sequence
  \begin{equation}\label{eq:restrictionexactseq}
    0 \longrightarrow H^0\bigl(X,\cO_X(D-S)\bigr)
    \longrightarrow
    H^0\bigl(X,\cO_X(D)\bigr) \xrightarrow{\,\rho_S\,}
    H^0\bigl(S,\cO_S(D)\bigr),
  \end{equation}
  where the middle map is obtained via twisting the map $\cO_X(-S)
  \hookrightarrow \cO_X$ corresponding to $1_S$ and applying global sections.
  For $\sigma \in H^0(X,\cO_X(D))$, we denote by $\sigma_{\vert S} \in
  H^0(X\vert S,\cO_X(D))$ the image of $\sigma$ under $\rho_S$, where
  $H^0(X\vert S,\cO_X(D))$ is the image of $\rho_S$ as defined in Definition
  \ref{def:restrictedlinearsystem}.
  \par If $\cS \subseteq \Div_\QQ(X)$ is a monoid generated by divisors
  $D_1,D_2,\ldots,D_\ell$, the \textsl{restriction of $R(X/Z;\cS)$ to $S$} is
  the $\cS$-graded $H^0(Z,\cO_Z)$-algebra
  \begin{align*}
    \res_S\bigl(R(X/Z;\cS)\bigr) &\coloneqq \bigoplus_{D \in \cS} H^0\bigl(X
    \vert S,\cO_X\bigl(\lfloor D \rfloor\bigr)\bigr),
    \intertext{and the \textsl{restriction of $R(X/Z;D_1,D_2,\ldots,D_\ell)$ to
    $S$} is the $\NN^\ell$-graded $H^0(Z,\cO_Z)$-algebra}
    \res_S\bigl(R(X/Z;D_1,D_2,\ldots,D_\ell)\bigr) &\coloneqq
    \bigoplus_{(m_1,m_2,\ldots,m_\ell) \in \NN^\ell}
    H^0\Biggl(X\vert S,\cO_X\Biggl(\Biggl\lfloor \sum_{i=1}^\ell m_iD_i
    \Biggr\rfloor \Biggr)\Biggr).
  \end{align*}
\end{definition}
We give two lemmas about finite generation of relative divisorial graded rings.
\begin{lemma}[cf.\ {\cite[Corollary 2.26]{CL12}}]\label{lem:cl12226}
  Let $\pi\colon X \to Z$ be a proper morphism of integral Noetherian schemes,
  where $X$ is regular and $Z$ is affine.
  Let $f\colon Y \to X$ be a proper birational morphism, where $Y$ is regular.
  Let $D_1,D_2,\ldots,D_\ell \in \Div_\QQ(X)$, let $D_1',D_2',\ldots,D_\ell' \in
  \Div_\QQ(X)$, and assume there exist positive rational numbers $r_i$ and
  $f$-exceptional $\QQ$-divisors $E_i \ge 0$ such that
  \[
    D_i' \sim_\QQ r_if^*D_i + E_i
  \]
  for every $i$.
  Then, the ring
  \begin{align*}
    R &= R\bigl(X/Z;D_1,D_2,\ldots,D_\ell\bigr)
    \intertext{is finitely generated over
    $H^0(Z,\cO_Z)$ if and only if the ring}
    R' &= R\bigl(Y/Z;D_1',D_2',\ldots,D_\ell'\bigr)
  \end{align*}
  is finitely generated over $H^0(Z,\cO_Z)$.
  Similarly, suppose $S$ is a regular prime divisor on $X$, and let $T =
  f_*^{-1}S$.
  Then, the ring $\res_S(R)$ is finitely generated over $H^0(Z,\cO_Z)$ if
  and only if the ring $\res_T(R')$ is finitely generated over $H^0(Z,\cO_Z)$.
\end{lemma}
\begin{proof}
  The proof of \cite[Corollary 2.26]{CL12} works after replacing absolute
  divisorial rings with relative divisorial graded rings.
  For completeness, we write down the proof below.
  \par Let $k$ be a positive integer such that for all $i$, we have that
  $kD_i$, $kr_iD_i'$, and $kE_i$ are all integral and
  \[
    kD_i' \sim kr_if^*D_i + kE_i.
  \]
  Then, the rings
  \begin{align*}
    R\bigl(X/Z&;kD_1,kD_2,\ldots,kD_\ell\bigr)\\
    R\bigl(Y/Z&;kD_1',kD_2',\ldots,kD_\ell'\bigr)
  \end{align*}
  are Veronese subrings of finite index in $R$ and $R'$, respectively, and both
  rings are isomorphic to
  \[
    R\bigl(Y/Z;kr_1f^*D_1 + kE_1,kr_2f^*D_2 + kE_2,\ldots,kr_\ell f^*D_\ell +
    kE_\ell\bigr).
  \]
  Similarly, the rings
  \begin{align*}
    \res_S\Bigl(R\bigl(X/Z&;kD_1,kD_2,\ldots,kD_\ell\bigr)\Bigr)\\
    \res_T\Bigl(R\bigl(Y/Z&;kD_1',kD_2',\ldots,kD_\ell'\bigr)\Bigr)
  \end{align*}
  are Veronese subrings of finite index in $\res_S(R)$ and $\res_T(R')$,
  respectively, and both rings are isomorphic to
  \[
    \res_T\Bigl(R\bigl(Y/Z;kr_1f^*D_1 + kE_1,kr_2f^*D_2 + kE_2,\ldots,kr_\ell
    f^*D_\ell + kE_\ell\bigr)\Bigr).
  \]
  In either case, the conclusion follows from
  \cite[Propositions 1.2.2 and 1.2.4]{ADHL15}.
\end{proof}
\begin{lemma}[cf.\ {\cite[Lemma 2.27]{CL12}}]\label{lem:cl12227}
  Let $\pi\colon X \to Z$ be a proper morphism of integral Noetherian schemes,
  where $X$ is regular and $Z$ is affine.
  Let $D_1,D_2,\ldots,D_\ell \in \Div_\QQ(X)$, and set
  \[
    \cC = \sum_{i=1}^\ell \RR_{\ge0} \cdot D_i \subseteq \Div_\RR(X).
  \]
  Then, we have the following:
  \begin{enumerate}[label=$(\roman*)$,ref=\roman*]
    \item\label{lem:cl12227i}
      If $R(X/Z;\cC)$ is finitely generated as an $H^0(Z,\cO_Z)$-algebra, then
      $R(X/Z;D_1,D_2,\ldots,D_\ell)$ is finitely generated as an
      $H^0(Z,\cO_Z)$-algebra.
    \item\label{lem:cl12227ii}
      Let $S$ be a regular prime divisor on $X$.
      If $\res_S(R(X/Z;\cC))$ is finitely generated as an
      $H^0(Z,\cO_Z)$-algebra, then
      \[
        \res_S\Bigl(R\bigl(X/Z;D_1,D_2,\ldots,D_\ell\bigr)\Bigr)
      \]
      is
      finitely generated as an $H^0(Z,\cO_Z)$-algebra.
  \end{enumerate}
\end{lemma}
\begin{proof}
  The proof of \cite[Lemma 2.27]{CL12} works after replacing absolute divisorial
  rings with relative divisorial graded rings.
  For completeness, we write down the proof below.
  \par Let $k$ be a positive integer such that $D_i' = k D_i \in \Div(X)$ for all
  $i$.
  The monoid
  \[
    \cS = \sum_{i=1}^\ell \NN \cdot D_i'
  \]
  is a submonoid of $\cC \cap
  \Div(X)$.
  If $R(X/Z;\cC)$ (resp.\ $\res_S(R(X/Z;\cC))$) is finitely generated, then
  $R(X/Z;\cS)$ (resp.\ $\res_S(R(X/Z;\cS))$) is also
  finitely generated by \cite[Proposition 1.2.2]{ADHL15}.
  Then, $R(X/Z;D_1',D_2',\ldots,D_\ell')$ (resp.\
  $\res_S(R(X/Z;D_1',D_2',\ldots,D_\ell'))$) is finitely generated by
  \cite[Proposition 1.2.6]{ADHL15}, which implies that
  $R(X/Z;D_1,D_2,\ldots,D_\ell)$ (resp.\
  $\res_S(R(X/Z;D_1,D_2,\ldots,D_\ell))$) is finitely generated by
  \cite[Proposition 1.2.4]{ADHL15}.
\end{proof}
\section{Asymptotic order of vanishing}\label{section:asymporder}
Following \cite[\S3 and \S8]{CL13}, we define the asymptotic order of vanishing in our
setting.
We will not need this in the proof of our analogue of \cite[Theorem B]{CL12},
since we are able to derive it from the result in \cite{CL12}.
On the other hand, we will need to use the asymptotic order of vanishing when
running the minimal model program, as in \cite{CL13}.\medskip
\par We work over an affine base and work with absolute linear systems as
in Definition \ref{def:linearsystem}.

\begin{definition}[see {\citeleft\citen{ELMNP06}\citemid p.\ 1713\citepunct
  \citen{CL13}\citemid p.\ 620\citeright}]
  \label{def:GeomVal}
Let $X$ be an integral normal separated scheme.
Let $v$ be a discrete valuation on the function field $K(X)$ of $X$ given by a
morphism $\Spec(R) \to X$, which is
uniquely determined by $v$ up to isomorphism. 
The \textsl{center} of $v$ is the image of the closed point of $\Spec(R)$.
We say $v$ is a \textsl{geometric valuation on $X$} if $v$ is given by the order
of vanishing at the generic point $\eta$ of a prime divisor $\Gamma$ on some
birational model $f\colon Y \to X$ of $X$.
In this case, the valuation is given by
the composition $\Spec(\cO_{Y,\eta}) \to Y \to X$.
\end{definition}
We now define the asymptotic order of vanishing for $\RR$-Weil divisors such
that $\lvert D\rvert_\mathbf{R}\neq \emptyset$.
When $D$ is a big $\RR$-Cartier divisor and $Z$ is a point, this notion
coincides with the invariant $v(\lVert D \rVert)$ defined in \cite[Definition
2.2]{ELMNP06}, and when $v$ is futhermore a geometric valuation given by a prime
divisor $\Gamma$, this notion coincides with the invariant $\sigma_\Gamma(D)$
from \cite[Chapter III, Definition 1.1]{Nak04}.
See also Remark \ref{rem:OIsNotSigma}.
\begin{definition}[see {\cite[p.\ 632]{CL13}; cf.\
  \citeleft\citen{ELMNP06}\citemid Lemma 3.3\citemid \citen{CDB13}\citemid
  Remark 2.16\citeright}]
  \label{def:AsymptoticOrder}
  Let $\pi\colon X \to Z$ be a proper morphism of integral Noetherian
  schemes, where $Z$ is affine.
  Let $D$ be an $\RR$-Weil divisor on $X$ such that
 $\lvert D\rvert_\mathbf{R}\neq \emptyset$. 
 For each discrete valuation $v$ on $K(X)$, the \textsl{asymptotic order of
 vanishing} of $D$ is
\[
  o_{v}(D)\coloneqq\inf_{E \in \lvert D \rvert_\RR}
  \bigl\{v(E)\bigr\}.
\]
For every positive real number $a$, we have $o_v(aD)=a\cdot o_v(D)$. 
For every pair of elements $D,D' \in \Div_\RR(X)$, we have
\[
  o_v(D+D')\leq o_v(D)+o_v(D')
\]
by \cite[Proposition 2.4]{ELMNP06}.
When $v$ comes from a prime divisor $S$ we write $o_S$ for $o_v$.
\end{definition}

\begin{remark}\label{rem:OIsNotSigma}
Let $D$ be an $\RR$-Weil divisor on a complex projective variety $X$.
If $\lvert D \rvert_{\RR}\neq\emptyset$, then $D$ is $\pi$-pseudoeffective.
However, the asymptotic order of vanishing $o_v(D)$ and the invariant $v(\lVert
D \rVert)$ defined in \cite{ELMNP06} are not necessarily equal.
See \cite[Remark 2.16]{CDB13}.
\end{remark}

\begingroup
\makeatletter
\renewcommand{\@secnumfont}{\bfseries}
\part{Bertini theorems and fundamental theorems of the MMP}
\label{part:bertiniandfund}
\makeatother
\endgroup
In this part, we prove our new relative versions of
Bertini theorems for schemes.
These theorems will become necessary later to perturb klt pairs without having
global Bertini theorems available as would be the case for quasi-projective
varieties over a
field.
We also show the fundamental theorems of the minimal model program (the
Basepoint-freeness, Contraction, Rationality, and Cone theorems) for algebraic
spaces adapting the strategy in \cite{KMM87} for complex varieties.\bigskip
\section{Bertini theorems}\label{sect:bertini}
As in the mixed characteristic case considered in \cite{BMPSTWW}, we will need
Bertini theorems that work for schemes that are finite type over a Noetherian
local domain of containing $\QQ$.
\begin{theorem}[cf.\ {\cite[Theorem 2.15]{BMPSTWW}}]\label{thm:bertini}
  Let $(R,\fm,k)$ be a Noetherian local domain containing $\QQ$.
  Fix an integer $N \ge 1$.
  Let $f\colon X \to \PP^N_R$ be a separated
  morphism of finite type from a regular Noetherian scheme $X$.
  Assume that every closed point of $X$ lies over the unique closed point of $\Spec(R)$.
  
  Let $T_0,T_1,\ldots,T_N$ be a basis of $H^0(\PP^N_R,\cO(1))$ as a free $R$-module.
  Then, there exists a nonempty Zariski open subset $W\subseteq \mathbf{A}^{N+1}_k$ with the following property:
  For all $a_0,a_1,\ldots,a_N\in R$, if
  \[
    (\bar{a}_0,\bar{a}_1,\ldots,\bar{a}_N)\in W(k),
  \]
  then the section
  \[
    h = a_0T_0+a_1T_1+\ldots+a_NT_N \in H^0\bigl(\PP^N_R,\cO(1)\bigr)
  \]
  is such that
  $f^{-1}(V(h))$ is regular.
\end{theorem}
\begin{proof}
  Denote by
  \[
    f_s\colon X_s \longrightarrow \PP^N_k
  \]
  the special fiber of \(f\) over the closed point \(\{s\} = \Spec(k)\) of \(\Spec(R)\).
  Choose a stratification $\{U_j\}_{j \in J}$ of $X_s$ by locally closed
  subschemes such that each $U_j$ is connected and regular.
  By Jouanolou's Bertini theorem \cite[Theorem 6.10(2)]{Jou83}, since $k$ is of
  characteristic zero, there exists a
  Zariski open subset $W \subseteq \mathbf{A}^{n+1}_k$ such that for all
  $(\bar{a}_0,\bar{a}_1,\ldots,\bar{a}_N) \in
  W(k)$, the section
  \[
    \bar{h} = \bar{a}_0T_0+\bar{a}_1T_1+\ldots+
    \bar{a}_NT_N \in H^0\bigl(\PP^N_k,\cO(1)\bigr)
  \]
  is such that $f_s^{-1}(V(\bar{h})) \cap U_j$ is regular for all $j$.
  \par We claim that this choice of $W$ satisfies the conclusion of the theorem.
  Since the regular locus is stable under generization, it suffices to show that
  $f^{-1}(V(h))$ is regular at every closed point $x \in f^{-1}(V(h))$.
  Let $0 \ne g \in \cO_{X,x}$ be the local equation defining
  $f^{-1}(V(h))$ at such a closed point $x$.
  By assumption, the image of $x$ in $\Spec(R)$
  is $\fm$, and hence there exists a member $U_j$ our stratification of $X_s$
  containing $x$.
  We now consider the image of $g$ under the composition
  \[
    \cO_{X,x}/\fm_x^2 \longrightarrow \cO_{X_s,x}/\fm_x^2 \longrightarrow
    \cO_{U_j,x}/\fm_x^2.
  \]
  By \cite[Chapitre 0, Proposition 17.1.7]{EGAIV1}, since $U_j$ and
  $f_s^{-1}(V(\bar{h})) \cap U_j$ are regular, we know that the image of $g$ in
  $\cO_{U_j,x}/\fm_x^2$ is nonzero.
  Thus, the image of $g$ in $\cO_{X,x}/\fm_x^2$ is also nonzero.
  Applying \cite[Chapitre 0, Proposition 17.1.7]{EGAIV1} again, we therefore see that
  $f^{-1}(V(h))$ is regular at $x$.
\end{proof}
\begin{remark}[cf.\ {\cite[Remark 2.16]{BMPSTWW}}]\label{rem:bertinisnc}
  Let
  $f\colon X \to \Spec(R)$ be a separated morphism of
  finite type mapping closed points to the unique closed point that
  factors through $\PP^N_R$ for some $N \ge 1$, and let $B$ be an
  effective divisor on $X$ with simple normal crossings.
  Applying Theorem \ref{thm:bertini} to $X$ and the finitely many strata of $B$,
  we obtain a divisor $H = g^{-1}(V(h))$ such that $(X,H+B)$ and $(H,B \cap H)$
  are log regular,
  where $g:X\to \PP^N_R$ is a factorization of $f$.
  We may also require $H$ to avoid finitely many given points, for example the generic points of the components of $B$.
  We will use this version of Bertini's theorem when working with linear systems
  associated to $f$-generated Cartier divisors.
\end{remark}
When $X$ is proper over a non-local base, we can still find semi-ample
regular divisors after passing to an affine open cover of the
base.
This provides a method to work around the fact that global Bertini theorems are
unavailable in our setting.
Below, a scheme is \textsl{J-2} if it admits an open affine covering $X =
\bigcup_i \Spec(R_i)$ such that every $R_i$ is J-2 in the sense of Definition
\ref{def:excellent}$(\ref{def:excellentj2})$ (see
\cite[\href{https://stacks.math.columbia.edu/tag/07R3}{Tag
07R3} and \href{https://stacks.math.columbia.edu/tag/07R4}{Tag
07R4}]{stacks-project}).
\begin{corollary}\label{cor:bertinionopencover}
  Let $R$ be a Noetherian domain containing $\QQ$.
  Fix an integer $N \ge 1$.
  Let
  \[
    \Set[\big]{f_i\colon X_i \to \PP^N_R}_{i}
  \]
  be a finite collection of closed separated
  morphisms of finite type 
  from regular Noetherian schemes $X_i$ that are J-2.
  Let $\Spec(R) = \bigcup_k V_k$ be a finite affine open cover of $\Spec(R)$.
  Then, there exists a finite affine open cover
  \[
    \Spec(R) = \bigcup_{j} U_j
  \]
  refining $\Spec(R) = \bigcup_k V_k$,
  such that for each $j$, there exists a section $h_j \in
  H^0(\PP^N_R,\cO(1))$ whose preimage $f_i^{-1}(V(h_j))$ is regular along the
  preimage of $U_j$ in $X_i$.
\end{corollary}
\begin{proof}
  For each prime ideal $\fp \subseteq R$, we can construct sections $h_\fp \in
  H^0(\PP^N_{R_\fp},\cO(1))$ such that the preimage of $V(h_\fp)$ in
  $X_{i} \otimes_R R_\fp$ is regular for every $i$ by Theorem \ref{thm:bertini}.
  Since $R$ is a domain, we can lift the sections $h_\fp$ to sections
  $\tilde{h}_\fp \in H^0(\PP^N_R,\cO(1))$ by clearing denominators.
  For each $\fp$ and $i$, denote by $\Sing(f_i^{-1}(V(\tilde{h}_\fp)))$ the
  singular locus of $f_i^{-1}(V(\tilde{h}_\fp))$, which is closed by the
  J-2 condition.
  Then, denoting by $\pi_i\colon X_i \to \Spec(R)$ the composition of $f_i$ with
  the projection morphism $\PP^N_R \to \Spec(R)$, we have
  \[
    \fp \in \Spec(R) - \pi_i\biggl( \bigcup_i
    \Sing\Bigl(f_i^{-1}\bigl(V(\tilde{h}_\fp)\bigr)\Bigr) \biggr)
  \]
  since $f_i^{-1}(V(\tilde{h}_\fp))$ is regular along the preimage of $\fp$ by
  construction, and hence
  \[
    \Spec(R) = \bigcup_{\fp \in \Spec(R)} \Biggl(\Spec(R) - \pi_i\biggl(
    \bigcup_i \Sing\Bigl(f_i^{-1}\bigl(V(\tilde{h}_\fp)\bigr)\Bigr)
    \biggr)\Biggr)
  \]
  is an open cover.
  Each of the members of this open cover contains an affine open $U_\fp$ such
  that $\fp \in U_\fp \subseteq V_k$ for some $k$,
  and since $\Spec(R)$ is quasi-compact, there is a finite
  subset $\{U_{\fp_j}\}\subseteq\{U_\fp\}$ that forms an affine open cover of
  $\Spec(R)$.
  Setting $U_j \coloneqq U_{\fp_j}$ and $h_j \coloneqq \tilde{h}_{\fp_j}$, we
  are done.
\end{proof}

Corollary \ref{cor:bertinionopencover} allows us to perturb klt pairs up to
replacing the base by an affine open cover.
This allows us to run inductive proofs like in the
classical setting for complex varieties after passing to an affine open cover of
the base.

\begin{corollary}\label{cor:BertiniKLTOpenCover}
Let $\pi\colon X \to Z$ be a proper morphism of excellent locally Noetherian
  schemes of equal characteristic zero.
  Suppose that $X$ is integral and normal, and that $Z$ has a dualizing complex $\omega_Z^\bullet$.
  Denote by $K_X$ a canonical divisor on $X$ associated to $\omega_X^\bullet =
  \pi^!\omega_Z^\bullet$.
  
  Let $(X,\Delta)$ be a klt $\kk$-pair for $\kk \in \{\QQ,\RR\}$.
  Let $A$ be a $\pi$-semi-ample $\kk$-Cartier divisor on $X$.
  Then, there exists an open covering $Z=\bigcup_a V_a$ and
  \[
    A_a\in \bigl\lvert A_{|\pi^{-1}(V_a)}\bigr\rvert_{\kk}
  \]
  such that $(\pi^{-1}(V_a),\Delta_{|\pi^{-1}(V_a)}+A_a)$ is klt.
\end{corollary}
\begin{proof}
The $\pi$-semi-ample $\kk$-Cartier divisor
$A$ is a $\kk_{\geq 0}$-linear combination of $\pi$-semi-ample Cartier divisors on
$X$, so it suffices to treat the case $A=rH$ where $r\in\kk$, $0<r<1$, and $H$
is a $\pi$-generated Cartier divisor.

We may assume $Z=\Spec(R)$ affine and integral.
Let $f\colon Y\to X$ be a log resolution of $(X,\Delta)$, which exists by
\cite[Theorem 2.3.6 and Lemma 4.2.4]{Tem08}.
Write
\[
  K_Y+\sum_E a(E)\,E\sim_{\kk}f^*(K_X+\Delta),
\]
where $a(E)\coloneqq a(E,X,\Delta)$ is the discrepancy.
The divisor $\Delta_Y\coloneqq \sum_E a(E)\,E$ is effective and satisfies $\lfloor\Delta_Y\rfloor=0$ since $(X,\Delta)$ is klt.
Since $H$ is $\pi$-generated,
it defines a morphism $h\colon X\to \PP^N_R$.

Applying Corollary \ref{cor:bertinionopencover}, there exists an affine open
cover \(Z = \bigcup_a V_a\) such that denoting by \(\pi_a \colon X_a \to V_a\) and
\(f_a\colon Y_a \to X_a\) the restrictions of \(\pi\) and \(f\) to \(V_a\) and
its preimages,
we can find
\[
H'_a\in \bigl\lvert H_{\rvert X_a} \bigr\rvert
\]
for every \(a\)
such that $f_a^*H'_a$ is reduced, does not share a component with $\Delta_{Y_a}
\coloneqq \Delta_{Y\rvert Y_a}$,
and is such that $(Y_a,\Delta_{Y_a}+f_a^*H'_a)$ is log regular. 
We have
\begin{gather*}
  A'_a\coloneqq rH'_a\in\bigl\lvert A_{\rvert X_a}\bigr\rvert_{\kk},\\
  f_{a*}\bigl(\Delta_{Y_a}+r\,f_a^*H'_a\bigr)=\Delta_{\rvert X_a}+A'_a,
\end{gather*}
and
\begin{align*}
  K_{Y_a}+\Delta_{Y_a}+r\,f_a^*H'_a &\sim_{\kk}f_a^*\bigl(K_{X_a}+\Delta_{\rvert
  X_a}+A'_a\bigr)
  \intertext{for every \(a\), and so}
  a\bigl(E,X_a,\Delta_{\rvert X_a}+A'_a\bigr)&=
  a\bigl(E,Y_a,\Delta_{Y_a}+r\,f_a^*H'_a\bigr)
\end{align*}
for all divisors $E$ over $X_a$ (cf.\
\cite[Lemma 2.30]{KM98}).
Since $r<1$, we have
\[
  \bigl\lfloor\Delta_{Y_a}+r\,f_a^*H'_a\bigr\rfloor=0
\]
and since $(Y_a,\Delta_{Y_a}+r\,f_a^*H'_a)$ is log regular, we see that
$(Y_a,\Delta_{Y_a}+r\,f_a^*H'_a)$ is klt for every \(a\) by
\cite[Corollary 2.11]{Kol13}.
Thus $(X_a,\Delta_{\rvert X_a}+A'_a)$ is klt for every \(a\), as desired.
\end{proof}
When $X$ is projective over an affine base, we can find ample
divisors avoiding finitely many points in $X$, even without passing to an affine
open cover of the base.
\begin{lemma}\label{lem:AmpleAvoid}
  Let $\pi\colon X \to Z$ be a projective morphism of integral Noetherian
  schemes, where $Z$ is affine.
  Let $\kk \in \{\ZZ,\QQ,\RR\}$.
For a $\pi$-ample $\kk$-Cartier divisor $A$ on $X$ and finitely many points $x_i\in X$, there exist a positive integer $n$ and a divisor $A'\in
\lvert nA\rvert_{\mathbf{k}}$ with $\mult_{x_i}(A')=0$ for all $i$.
\end{lemma}
\begin{proof}
  Since $\pi$-ample $\kk$-Cartier divisors are $\kk_{>0}$-linear combination of $\pi$-ample Cartier divisors, we may assume that $\kk=\ZZ$. 
  The statement now follows by the graded version of prime avoidance
  \cite[Chapter III, \S1, no.\ 4, Proposition 8]{BouCA}.
\end{proof}

\section{Basepoint-free, Contraction, Rationality, and Cone
theorems}\label{sect:bpfcontrratcone}
In this section, we prove that the Basepoint-free and Contraction, Rationality, and
Cone theorems
hold for projective morphisms of quasi-excellent algebraic
spaces of equal characteristic zero with
dualizing complexes by adapting the proofs in \cite{KMM87}.
Later, in \S\ref{sect:fundrevisit},
we will prove dual versions of these statements in the vein of
\cite{Kaw11} using our finite generation result (Theorem
\ref{thm:cl12a}), as is done for varieties in \cite{CL13}.\medskip
\par We have stated these results using the notion of weakly log terminal pairs (see
Definition \ref{def:wlt}). 
Dlt pairs are weakly log terminal by Remark \ref{rem:dltiswlt}.
\subsection{Basepoint-free theorem}
\par We start with the Basepoint-free theorem.
A version of the statement for schemes below appeared in
\cite[Proposition 2.48]{BMPSTWW}.
The statement for algebraic spaces when $Z = \Spec(k)$ for a field $k$ (resp.\
when $X$ is three-dimensional and of finite type over an algebraically closed
field of characteristic zero) is proved
in \cite[Basepoint-free theorem 1.4.4]{Kol91} (resp.\ \cite[Base Point Free
Theorem 6.16]{Sho96}).
\par We have included the statement for $\RR$-pairs to illustrate that for
schemes that are not necessarily quasi-projective over a field, one cannot simply
perturb boundary divisors directly at the beginning because we do not have
Bertini theorems available.
If $\pi$ is projective, one could instead replace $Z$ by an affine cover and use
an appropriate version of Corollary \ref{cor:BertiniKLTOpenCover}.
\begin{theorem}[Basepoint-free theorem; cf.\ {\cite[Theorem 3-1-1 and Remark
  3-1-2(1)]{KMM87}}]\label{thm:bpf}
  Let $\pi\colon X \to Z$ be a proper surjective morphism of integral
  quasi-excellent Noetherian algebraic spaces of equal characteristic
  zero over a scheme $S$.
  Suppose that $X$ is normal and that $Z$ admits a dualizing complex
  $\omega_Z^\bullet$.
  Denote by $K_X$ a canonical divisor on $X$ associated to $\omega_X^\bullet =
  \pi^!\omega_Z^\bullet$.
  \par Let $(X,\Delta)$ be an $\RR$-pair,
  and let $H \in
  \Pic(X)$ be $\pi$-nef.
  Suppose one of the following holds:
  \begin{enumerate}[label=$(\roman*)$,ref=\roman*]
    \item\label{thm:bpfdlt} $(X,\Delta)$ is dlt (or more generally, weakly log
      terminal) and
      $aH - (K_X+\Delta)$ is $\pi$-ample for some $a \in \ZZ_{>0}$.
    \item\label{thm:bpfklt}
      $(X,\Delta)$ is klt and $aH - (K_X+\Delta)$ is $\pi$-big and $\pi$-nef
      for some $a \in \ZZ_{>0}$.
  \end{enumerate}
  Then, there exists $m_0 \in \ZZ_{>0}$ such that $mH$ is $\pi$-generated
  for all $m \ge m_0$.
\end{theorem}
\begin{proof}
  After replacing $\pi$ by its Stein factorization
  \cite[\href{https://stacks.math.columbia.edu/tag/0A1B}{Tag
  0A1B}]{stacks-project}, we may assume that $Z$ is normal and that $\pi$ has
  geometrically connected fibers
  \cite[\href{https://stacks.math.columbia.edu/tag/0AYI}{Tag
  0AYI}]{stacks-project}.
  For $(\ref{thm:bpfklt})$, this does not change the $\pi$-bigness or the
  $\pi$-nefness of $aH -
  (K_X+\Delta)$ since it changes volumes and intersections on $\kappa(\eta)$ by
  the factor $[H^0(X_\eta,\cO_{X_\eta}) : \kappa(\eta)]$.\medskip
  \par We claim we may replace $Z$ by a scheme $Z'$ \'etale over $Z$.
  Let $Z' \to Z$ be an \'etale morphism where $Z'$ is a quasi-compact scheme,
  and consider the associated Cartesian diagram
  \[
    \begin{tikzcd}
      X' \rar{f'}\dar[swap]{\pi'} & X\dar{\pi}\\
      Z' \rar{f} & Z\mathrlap{.}
    \end{tikzcd}
  \]
  By flat base change \cite[\href{https://stacks.math.columbia.edu/tag/073K}{Tag
  073K}]{stacks-project}, it suffices to show that $m\,f^{\prime*}H$ is
  $\pi'$-generated for all $m \gg0$.
  Note that the assumptions on $(X,\Delta)$ are inherited by
  $(X',f^{\prime*}\Delta)$ by Remark \ref{rem:singsetalelocal}.
  Moreover, we have
  \[
    f^{\prime*}\bigl(aH - (K_X+\Delta)\bigr) = a\,f^{\prime*}H - (K_{X'}+
    f^{\prime*}\Delta),
  \]
  where $f^{\prime*}\Delta$ is the \'etale pullback of $\Delta$, since the
  formation of canonical divisors is compatible with \'etale base change (see
  the proof of Lemma \ref{lem:dualizingcomplexpullback}).
  This $\RR$-invertible sheaf is $\pi'$-nef by Lemma
  \ref{lem:nefbasechange}$(\ref{lem:nefbasechangepullback})$ and
  is $\pi'$-big by flat base change
  \cite[\href{https://stacks.math.columbia.edu/tag/073K}{Tag
  073K}]{stacks-project}.
  We can then replace $\pi$ by $\pi'$ to assume that $Z$ is a scheme.
  To assume that $X$ is integral, we work one connected component at a time and
  let $Z$ be the scheme theoretic images of these components.\medskip
  \par We now prove the theorem in the case $Z$ is a scheme.
  Let $f_1\colon Y_1 \to X$ be a projective log resolution of $(X,\Delta)$,
  where for $(\ref{thm:bpfdlt})$ we assume the hypotheses in Definition
  \ref{def:wlt},
  and for $(\ref{thm:bpfklt})$ we first apply Chow's lemma
  \cite[Th\'eor\`eme 5.6.1]{EGAII} then resolve using \cite[Theorem
  1.1.6]{Tem18} to assume that $Y_1$ is projective over $Z$.
  Then, we know that $f_1^*(aH-(K_X+\Delta))$ is $\pi$-big and $\pi$-nef by
  Lemmas \ref{lem:nefpullback} and \ref{lem:bigpullback}.
  By Kodaira's lemma (Corollary \ref{lem:kodairachar}), the $\RR$-divisor
  \[
    f_1^*\bigl(aH-(K_X+\Delta)\bigr) + \delta\,f_{1*}^{-1}\Delta - \sum_i
    \delta_{1i} G_i
  \]
  is $(\pi\circ f_1)$-ample for some $\delta,\delta_{1i} \in \RR$ with
  $0 < \delta \ll \min_{\delta_{1i} \ne 0}\{\delta_{1i}\} \ll 1$,
  where $\{G_i\}$ is a family of effective Cartier divisors on
  $X$ with normal crossings,
  $\Supp(\sum_i \delta_{1i} G_i)$ is $f_1$-exceptional, and
  \[
    K_{Y_1} + \delta\,f_{1*}^{-1}\Delta \sim_\RR
    f_1^*(K_X+\Delta) + \sum_i b_iG_i
  \]
  for $b_i \in \RR$ with $b_i > -1$.
  Let $C \coloneqq \sum_i (b_i-\delta_{1i})G_i$.
  After perturbing the $\delta_{1i}$ using Theorem \ref{thm:ampisinterior},
  we may assume that $C$ is a
  $\QQ$-divisor.
  Letting $\eta \in \lvert Z \rvert$ be the generic point, we can apply the
  Non-vanishing theorem \cite[Theorem 2-1-1]{KMM87} to a connected component
  of the geometric generic fiber $Y_{1\bar{\eta}}$ and the pullbacks of $f_1^*H$
  and $C$ to $Y_{1\bar{\eta}}$ to see that
  \[
    \bigl( (\pi \circ f_1)_*\cO_{Y_1}\bigl(m\,f_1^*H+\lceil C
    \rceil\bigr)\bigr)_\eta
    \cong H^0\bigl(Y_{1\eta},\cO_{Y_{1\eta}}(m\,f_1^*H_\eta+\lceil C_\eta
    \rceil\bigr)\bigr) \ne 0
  \]
  for $m \gg 0$
  by flat base change \cite[Proposition 1.4.15]{EGAIII1},
  since
  \begin{align*}
    a\,f_1^*H+C-K_{Y_1} &\sim_\RR
    a\,f_1^*H+C-\biggl(f_1^*(K_X+\Delta) + \sum_i b_iG_i -
    \delta\,f_{1*}^{-1}\Delta\biggr)\\
    &\sim_\RR
    f_1^*\bigl(aH-f_1^*(K_X+\Delta)\bigr) + \delta\,f_{1*}^{-1}\Delta
    - \sum_i \delta_{1i}G_i
  \end{align*}
  is $\pi$-ample.
  In particular, we have
  \[
    \pi_*\cO_X(mH) \cong (\pi \circ f_1)_*\cO_{Y_1}\bigl(m\,f_1^*H+\lceil C
    \rceil \bigr) \ne 0
  \]
  by the projection formula since $\lceil C \rceil$ is
  $g$-exceptional.\smallskip
  \par We now make the following claim:
  \begin{claim}\label{claim:pnbpf}
    For every prime number $p$,
    the divisor $p^nH$ is $\pi$-generated for $n \gg 0$.
  \end{claim}
  Showing Claim \ref{claim:pnbpf} would imply the theorem, since then the
  monoid of natural numbers $m \in \NN$ such that $mH$ is $\pi$-generated would
  contain all sufficiently large integers by \cite[Theorem 1.0.1]{RA05}.\smallskip
  \par Choose $n_0 > 0$ such that $\pi_*\cO_X(p^{n_0}H) \ne 0$ as above.
  If $p^{n_0}H$ is $\pi$-generated, there is nothing to show.
  We will therefore assume that $p^{n_0}H$ is not $\pi$-generated.
  \par First, let $f_1\colon Y_1 \to X$ be a projective log resolution of
  $(X,\Delta)$ as above.
  Taking successive blowups along regular centers (see
  \cite[Lemma 4.2.4]{Tem08}), there is a projective birational morphism $f_2\colon Y \to
  Y_1$ with a family of effective Cartier divisors $\{F_j\}$ with only
  \emph{simple} normal crossings such that setting $f \coloneqq f_1 \circ f_2$,
  the $\RR$-divisor
  \begin{align*}
    f_2^*\biggl( f_1^*\bigl(aH-(K_X+\Delta)\bigr) + \delta\,f_{1*}^{-1}\Delta -
    \sum_i \delta_{1i} G_i\biggr) &- \delta'A_2\\
    =
    f^*\bigl(aH-(K_X+\Delta)\bigr) + \delta\,f_2^*f_{1*}^{-1}\Delta &- \sum_j
    \delta_jF_j
  \end{align*}
  is $(\pi \circ f)$-ample for an $f_2$-exceptional $\RR$-divisor $A_2$ with $0
  < \delta' \ll \delta$, again using Kodaira's lemma (Corollary
  \ref{lem:kodairachar}).
  Moreover, we have
  \[
    K_Y +\delta\,f_2^*f_{1*}^{-1}\Delta \sim_\RR f^*(K_X+\Delta) + \sum_j a_jF_j
  \]
  for $a_j \in \RR$ with $a_j > -1$, and after possibly using \cite[Theorem
  1.1.6]{Tem18} to replace $f$ by a resolution that also resolves the
  $\pi$-base ideal of $\cO_X(p^{n_0}H)$, we have
  \[
    (\pi \circ f)^*(\pi \circ f)_*\cO_Y(f^*p^{n_0}H) \longtwoheadrightarrow
    \cO_Y\biggl(f^*p^{n_0}H - \sum_j r_jF_j\biggr) \subseteq \cO_Y(f^*p^{n_0}H)
  \]
  for some non-negative integers $r_j$ not all equal to zero.
  \par Next, since $0 < \delta' \ll \delta \ll
  \min_{\delta_{1i}\ne0}\{\delta_{1i}\} \ll 1$, we know that $a_j + 1 - \delta_j
  > 0$ for all $j$ by \cite[Corollary 2.11]{Kol13}.
  Set
  \[
    c \coloneqq \min_j \biggl\{ \frac{a_j+1-\delta_j}{r_j} \biggr\},
  \]
  where we set $(a_j+1-\delta_j)/r_j = \infty$ if $r_j = 0$.
  After possibly perturbing $A_2$ (and hence the $\delta_j$) slightly
  using Theorem \ref{thm:ampisinterior}, we may assume that the
  minimum $c$ is attained at a unique index $j$, which we relabel as $j = 0$,
  and that $a_j - \delta_j \in \QQ$ for all $j$.
  Set
  \begin{gather*}
    A \coloneqq \sum_{j \ne 0} (-cr_j + a_j - \delta_j)F_j,\\
    B \coloneqq F_0.
  \end{gather*}
  Then, the $\QQ$-divisor
  \begin{alignat*}{3}
    N &\coloneqq p^{n'}f^*H + A - B - K_Y\\
    &\sim_\RR c\biggl(f^*p^{n_0}H - \sum_j r_jF_j\biggr) &{}+{}& f^*\bigl(
    (p^{n'}-cp^{n_0})H - (K_X+\Delta)\bigr)\\
    &&{}+{}& \delta\,f_2^*f_{1*}^{-1}\Delta -
    \sum_j \delta_j F_j
  \end{alignat*}
  is $(\pi \circ f)$-ample for all $n' \in \NN$ such that $p^{n'} \ge
  cp^{n_0}+a$.
  Since
  \begin{align*}
    R^1(\pi\circ f)_*\cO_Y\bigl(p^{n'}f^*H &+ \lceil A \rceil - B \bigr) =
    R^1(\pi\circ f)_*\cO_Y\bigl(\lceil N \rceil + K_Y\bigr) = 0
  \intertext{by \cite[Theorem A]{Mur}, the morphism}
    (\pi \circ f)_*\cO_Y\bigl(p^{n'}f^*H &+ \lceil A \rceil\bigr) \longrightarrow
    (\pi \circ f)_*\cO_B\Bigl(\bigl(p^{n'}f^*H + \lceil A
    \rceil\bigr)_{\vert B}\Bigr)
  \end{align*}
  is surjective.
  Now by the Non-vanishing theorem \cite[Theorem 2-1-1]{KMM87}
  applied to a connected component of the
  geometric generic fiber $B_{\bar{\eta}}$ and the pullbacks of $p^{n'}f^*H$
  and $A$ to $B_{\bar{\eta}}$, we see that
  \[
    (\pi \circ f)_*\cO_B\Bigl(\bigl(p^{n'}f^*H + \lceil A
    \rceil\bigr)_{\vert B}\Bigr) \ne 0
  \]
  for $n' \gg 0$.
  Since $(\pi \circ f)_*\cO_Y(p^{n'}f^*H + \lceil A \rceil) \cong
  \pi_*\cO_X(p^{n'}H)$ by the projection formula and the fact that $\lceil A
  \rceil$ is $f$-exceptional, we have
  \[
    f(B) \not\subseteq \Supp\Bigl(\cok\bigl(\pi^*\pi_*\cO_X(p^{n'}H)
    \longrightarrow \cO_X(p^{n'}H)\bigr)\Bigr).
  \]
  Thus, we have
  \begin{align*}
    \MoveEqLeft[5]
    \Supp\Bigl(\cok\bigl(\pi^*\pi_*\cO_X(p^{n'}H)
    \longrightarrow \cO_X(p^{n'}H)\bigr)\Bigr)\\
    &\subsetneq \Supp\Bigl(\cok\bigl(\pi^*\pi_*\cO_X(p^{n_0}H)
    \longrightarrow \cO_X(p^{n_0}H)\bigr)\Bigr).
  \end{align*}
  By Noetherian induction, we therefore have
  \[
    \Supp\Bigl(\cok\bigl(\pi^*\pi_*\cO_X(p^{n}H)
    \longrightarrow \cO_X(p^{n}H)\bigr)\Bigr) = \emptyset,
  \]
  which is what we wanted to show in Claim \ref{claim:pnbpf}.
\end{proof}
\subsection{Contraction theorem}
Next, we consider the Contraction theorem.
Showing uniqueness of contraction morphisms is more involved than in the variety
case because we also need to consider integral one-dimensional closed subschemes
of non-closed fibers of $\pi$.
The following lemma fills this gap, and is pivotal when working with non-Jacobson schemes and
with algebraic spaces.
\begin{lemma}\label{lem:ContractsNonClosed}
  Let $Z$ be a Noetherian algebraic space over a scheme $S$
  and let $f\colon X\to Y$ and $f'\colon X\to Y'$
  be morphisms of proper algebraic spaces over $Z$.
  Suppose that for every integral one-dimensional closed subspace $C \subseteq X$ such
  that $f(C)$ is a point, we have that $f'(C)$ is a point.
  Then, for every $y\in \lvert Y \rvert$ and every connected component $W$ of $f^{-1}(y)$, we
  have that $f'(W)$ is a point.
\end{lemma}
\begin{proof}
  We fix the following notation for the structure morphisms of $X$, $Y$, and
  $Y'$:
  \[
    \begin{tikzcd}
      Y \arrow{dr}[swap]{h} & \lar[swap]{f} X \dar{\pi} \rar{f'} & Y'
      \arrow{dl}{h'}\\
      & Z\mathrlap{.}
    \end{tikzcd}
  \]
  Let $y \in \lvert Y \rvert$.
  It suffices to show that for each integral one-dimensional closed subspace
  $\Gamma$ of $f^{-1}(y)$, the image $f'(\Gamma)$ is a point.
  We may replace $X$ by the closure of $\Gamma$ equipped with the reduced
  induced structure, in which case $X$ is integral.
  After replacing $Y$, $Y'$, and $Z$ by the scheme-theoretic images of $X$, we
  may assume that $X$ maps surjectively onto $Y$, $Y'$, and $Z$, and that $Y$,
  $Y'$, and $Z$ are integral.
  In this case, we have $\pi^{-1}(\eta)=\Gamma$ where $\eta$ is the generic
  point of $Z$.
  \par Let $z\in \lvert Z \rvert$ be a closed point where the local ring of
  $Z$ at $z$ has minimal dimension $d$.
  We proceed by induction on $d$.
  If $d=0$ there is nothing to prove.
  If $d>0$, pick $\eta_1\in \lvert Z \rvert$ such that the local ring of $Z$ at $\eta_1$ is
  one-dimensional,
  $\eta_1\leadsto z$, and the dimension of $\overline{\{\eta_1\}}$ at $z$ is $<
  d$.
  By the inductive hypothesis, we see that the conclusion holds for the base
  change of $X$, $Y$, and $Y'$ to $\overline{\{\eta_1\}}$.
  The assumptions also hold for the base change of $X$, $Y$, and $Y'$ to an
  elementary \'etale neighborhood of $\eta_1$, and hence
  we may
  assume that $Z$ is an affine local scheme of dimension $1$.
  \par Since $f$ is surjective, we have $f(\Gamma)=h^{-1}(\eta)$, which
  means $h^{-1}(\eta)=\{y\}$ is (set-theoretically) a point.
  Thus $Y\to Z$ is generically finite, so $\dim(Y)\leq 1$. 
  Since $Y$ is integral,
  the closed fiber $h^{-1}(z)$ must be finite.
  Now each integral one-dimensional closed subscheme $C
  \subseteq X$ such that $\pi(C)$ is a point is also such that $f(C)$ is a
  point, and hence $f'(C)$ is a point by assumption.
  Thus $f'(\pi^{-1}(z))$ is finite, and this set is just
  $h^{\prime-1}(z)$.
  Therefore $h'$ is finite and we see that
  \[
    \dim\bigl(f'(\Gamma)\bigr) \leq \dim\bigl(h^{\prime-1}(\eta)\bigr)=0,
  \]
  as desired.
\end{proof}
We can now prove the Contraction theorem.
When $X$ is of finite type over an algebraically closed
field of characteristic zero, the case when $\dim(X) = 3$ is proved in
\cite[Contraction Theorem 6.15]{Sho96}, and the general
case follows from \cite[Theorem 2.6]{VP}.
We note that Lemma \ref{lem:ContractsNonClosed} is crucial in the proof of the
statement below to
allow us to characterize contractions in terms of contractions of closed
subspaces.
\begin{theorem}[Contraction theorem; cf.\ {\cite[Theorem 3-2-1]{KMM87}}]
  \label{thm:contraction}
  Let $\pi\colon X \to Z$ be a projective surjective morphism of integral
  quasi-excellent Noetherian algebraic spaces of equal characteristic zero over
  a scheme $S$.
  Suppose that $X$ is normal and that $Z$ admits a dualizing complex
  $\omega_Z^\bullet$.
  Denote by $K_X$ a canonical divisor on $X$ associated to $\omega_X^\bullet =
  \pi^!\omega_Z^\bullet$.
  \par Let
  $(X,\Delta)$ be a dlt (or more generally, weakly log terminal) $\RR$-pair,
  and let $H \in \Pic(X)$ be $\pi$-nef and such
  that
  \[
    F \coloneqq \bigl(H^\perp \cap \NEbar(X/Z)\bigr) - \{0\}
    \subseteq
    \Set[\big]{\beta \in N_1(X/Z) \given \bigl( (K_X+\Delta) \cdot \beta \bigr) < 0 },
  \]
  where $H^\perp \coloneqq \Set{\beta \in N_1(X/Z) \given (H \cdot \beta) = 0}$.
  Then, the morphism $\varphi$ in the Stein factorization
  \[
    X \overset{\varphi}{\longrightarrow} Y
    \longrightarrow \Proj_Z\Biggl( \bigoplus_{m=0}^\infty
    \pi_*\cO_X(mH)\biggr)
  \]
  is a projective and surjective morphism to
  an integral normal quasi-excellent Noetherian algebraic space \(Y\)
  projective over $Z$.
  The morphism \(\varphi\) satisfies the following properties:
  \begin{enumerate}[label=$(\roman*)$,ref=\roman*]
    \item\label{thm:contractioni}
      For every integral one-dimensional closed subspace $C \subseteq X$ such that
      $\pi(C)$ is a point, the image $\varphi(C)$ is a point if
      and only if $(H \cdot C) = 0$, i.e., if and only if $[C] \in F$.
    \item\label{thm:contractionii}
      $\cO_Y \to \varphi_*\cO_X$ is an isomorphism.
  \end{enumerate}
  \par Moreover, consider a projective surjective morphism \(\varphi'\colon X \to Y'\)
  fitting into the commutative diagram
  \[
    \begin{tikzcd}[column sep=tiny]
      X \arrow{rr}{\varphi'}\arrow{dr}[swap]{\pi} & & Y' \arrow{dl}{\sigma'}\\
      & Z
    \end{tikzcd}
  \]
  where \(Y'\) is an integral normal quasi-excellent Noetherian algebraic
  space projective over \(Z\).
  Suppose that \(\varphi'\) satisfies properties $(\ref{thm:contractioni})$ and
  $(\ref{thm:contractionii})$.
  Then, \(\varphi'\) is isomorphic to \(\varphi\) over \(Z\), and \(\varphi'\) satisfies the
  following additional property:
  \begin{enumerate}[resume,label=$(\roman*)$,ref=\roman*]
    \item\label{thm:contractioniii}
      $H = \varphi^{\prime*}A$ for some $\sigma'$-ample $A \in \Pic(Y)$.
  \end{enumerate}
\end{theorem}
\begin{proof}
  By Kleiman's criterion for $\pi$-ampleness (Proposition
  \ref{lem:AmpleIsPositiveOnNE}), there exists $a \in \NN$ such that $aH -
  (K_X+\Delta)$ is $\pi$-ample.
  Thus, by the Basepoint-free theorem \ref{thm:bpf}, we know that $mH$ is
  $\pi$-generated for $m \gg 0$.
  \par We claim that the relative section ring
  \begin{align*}
    R\bigl(X/Z;H\bigr) \coloneqq{}& \bigoplus_{m=0}^\infty
    \pi_*\bigl(\cO_X(mH)\bigr)
    \intertext{is an $\cO_Z$-algebra of finite type.
    It suffices to show that for every affine scheme $U = \Spec(R)$ \'etale over
    $X$, the pullback of $R(X/Z;H)$ is an $R$-algebra of finite type.
    By flat base change
    \cite[\href{https://stacks.math.columbia.edu/tag/073K}{Tag
    073K}]{stacks-project}, we note that}
    R\bigl(X/Z;H\bigr)_{\vert U} \cong{}& \bigoplus_{m=0}^\infty
    H^0\bigl(U,\cO_U(mH_{\vert U})\bigr).
  \end{align*}
  Base changing along the morphism $U \to Z$, we reduce to the case when $Z$ is
  an affine scheme.
  We can also replace $\pi$ by its Stein factorization \cite[Th\'eor\`eme
  4.3.1]{EGAIII1} to assume that $H^0(X,\cO_X) = R$.
  \par Since $mH$ is globally generated, we have a surjection
  \[
    H^0\bigl(X,\cO_X(mH)\bigr) \otimes_R \cO_X \longtwoheadrightarrow \cO_X(mH),
  \]
  which induces a morphism
  \[
    \psi_m\colon
    X \xrightarrow{\lvert mH \rvert}
    \PP_Z\Bigl(H^0\bigl(X,\cO_X(mH)\bigr)\Bigr) \eqqcolon \PP_m
  \]
  such that $\psi_m^*\cO_{\PP_m}(1) \cong \cO_X(mH)$.
  Let $\phi_m\colon X \to Y_m$ be the Stein factorization of $\psi_m$, and
  denote by $\cO_{Y_m}(1)$ the pullback of $\cO_{\PP_m}(1)$ to $Y_m$.
  By the projection formula, we know that
  \[
    R(X;mH) \coloneqq
    \bigoplus_{m'=0}^\infty H^0\bigl(X,\cO_X(mm'H)\bigr) \cong
    \bigoplus_{m'=0}^\infty H^0\bigl(Y_m,\cO_{Y_m}(m')\bigr).
  \]
  Since the right-hand side is a finitely generated $R$-algebra by
  \cite[Proposition 2.3.4$(ii)$]{EGAIII1}, we see that $R(X;H)$ is a finitely
  generated $R$-algebra by \cite[Proposition 1.2.2]{ADHL15}.
  \par We now claim the morphism $\varphi$ in the Stein factorization
  \[
    X \overset{\varphi}{\longrightarrow} Y
    \longrightarrow \Proj_Z\Biggl( \bigoplus_{m=0}^\infty
    \pi_*\cO_X(mH)\biggr)
  \]
  satisfies $(\ref{thm:contractioni})$ and
  $(\ref{thm:contractionii})$,
  where the composition is the natural morphism from
  \cite[\href{https://stacks.math.columbia.edu/tag/0D2Z}{Tag
  0D2Z}]{stacks-project}.
  $(\ref{thm:contractioni})$ holds by the projection formula for
  intersection products
  \cite[\href{https://stacks.math.columbia.edu/tag/0EDJ}{Tag
  0EDJ}]{stacks-project}, and
  $(\ref{thm:contractionii})$ holds by construction of the Stein factorization
  in \cite[\href{https://stacks.math.columbia.edu/tag/0A1B}{Tag
  0A1B}]{stacks-project}.
  \par Next, we show that $(\ref{thm:contractioni})$ and
  $(\ref{thm:contractionii})$ characterize $\varphi$ after pulling back along
  every \'etale morphism $U \to Z$ from a scheme $U$.
  In this case,
  by Lemma \ref{lem:ContractsNonClosed}, $(\ref{thm:contractioni})$
  characterizes $\varphi$ topologically.
  The isomorphism $\cO_Y \overset{\sim}{\to} \varphi_*\cO_X$ characterizes
  $\varphi$ as a morphism of ringed spaces.
  \par Finally, we show that $(\ref{thm:contractioniii})$ holds for $\varphi$ as
  defined above.
  We have
  \[
    \psi_{m+1}^*\cO_{\PP_{m+1}}(1) \otimes_{\cO_X} \psi_m^*\cO_{\PP_m}(-1) \cong
    \cO_X\bigl((m+1)H - mH) = \cO_X(H).
  \]
  Since the respective Stein factorizations $\phi_m\colon X \to Y_m$ and $\phi_{m+1}\colon
  X \to Y_{m+1}$ of $\psi_m$ and $\psi_{m+1}$ 
  satisfy $(\ref{thm:contractioni})$ and
  $(\ref{thm:contractionii})$, they are both isomorphic to $\varphi$.
  Thus, setting
  \[
    \cO_Y(A) = \cO_{\PP_{m+1}}(1)_{\vert Y} \otimes_{\cO_Y}
    \cO_{\PP_m}(-1)_{\vert Y},
  \]
  we see $\cO_X(H)=\varphi^*\cO_X(A)$.
  Finally, since $\cO_X(mH)\cong\varphi^*\cO_{\PP_m}(1)_{\vert Y}$,
  we see that $\cO_Y(mA)\cong\cO_{\PP_m}(1)_{\vert Y}$ by $(\ref{thm:contractionii})$,
  so $A$ is ample.
\end{proof}
\begin{remark}
  Suppose $X$ is a scheme.
  Then, since both $X$ and $Y$ are normal, the condition in
  $(\ref{thm:contractionii})$ holds if and only if $K(Y)$ is algebraically
  closed in $K(X)$, which holds if and only if the fibers of $\varphi$ are
  geometrically connected by
  \cite[Remarque 4.3.4 and Corollaire 4.3.12]{EGAIII1}.
\end{remark}
We use Theorem \ref{thm:contraction} to define extremal faces and extremal rays.
\begin{definition}[cf.\ {\cite[Definition 3-2-3]{KMM87}}]\label{def:contraction}
  Fix notation as in Theorem \ref{thm:contraction}.
  Since $\varphi$ is characterized by properties which only depend on $F$ and
  not on $H$, we call $\varphi$ the \textsl{contraction} of $F$.
  If $H$ is a $\pi$-nef \(\ZZ\)-invertible sheaf on $X$ such that $F = (H^\perp \cap
  \NEbar(X/Z)) - \{0\}$, we say that $H$ is a \textsl{supporting function} of
  $F$.
  We then say that $F$ is an \textsl{extremal face of $\NEbar(X/Z)$ for
  $(X,\Delta)$} (or \textsl{for $K_X+\Delta$}).
  If $\dim_\RR(F) = 1$, we say that $F$ is an \textsl{extremal ray}.
\end{definition}
\begin{definition}\label{def:GoodContrOfR}
  Fix notation as in Definition \ref{def:contraction}.
We say a contraction $f\colon X\to Y$ is \textsl{small} if the exceptional locus of $f$ is of codimension at least 2 in $X$.
In particular, $f$ is birational when $X$ is integral.
\par Let $R\subseteq \NEbar(X/Z)$ be an extremal face. We say that a contraction $f\colon X\to Y$ is \textsl{a
contraction of $R$}, if a $\pi$-contracted curve $C$ is $f$-contracted when and only when $[C]\in R$.
A contraction of $R$ is an isomorphism if and only if $R$ does not contain the class of any $\pi$-contracted curve. 
If $f$ is not an isomorphism and $R$ is a ray, then $R=\RR_{\geq 0}\cdot[C]$ for any $f$-contracted curve $C$.
Therefore we see $R=\NEbar(X/Y)$.

We say a contraction $f\colon X\to Y$ of an extremal ray $R$ is \textsl{good} if, for all
$\sL\in\Pic(X)_{\QQ}$, we have
$(\sL\cdot R)=0$ if and only if there exists an element $\sK\in\Pic(Y)_{\QQ}$ such that $\sL=f^*\sK\in \Pic(X)_{\QQ}$.
In this case $\NE(X/Y)\subseteq R$ canonically, and $\sL\in \Pic(X)_{\QQ}$ is $f$-ample if $(\sL\cdot R)>0$.
In general, when $Y$ is projective over $Z$, we have $\NEbar(X/Y)= R$; see the proof of \cite[Lemma 3-2-4]{KMM87}. 

For a good contraction $f$ of an extremal ray $R$, we always have
\[
\dim (N^1(Y/Z)_{\RR})=\dim (N^1(X/Z)_{\RR})-1.
\]
See \cite[Lemma 3-2-5]{KMM87} and its proof.
\end{definition}
\subsection{Rationality theorem}
We now consider the Rationality theorem.
\begin{theorem}[Rationality theorem; cf.\ {\citeleft\citen{KMM87}\citemid
  Theorem 4-1-1\citepunct \citen{KM98}\citemid Theorem 3.5\citeright}]\label{thm:rationality}
  Let $\pi\colon X \to Z$ be a proper surjective morphism of integral
  quasi-excellent Noetherian algebraic spaces of equal characteristic zero over
  a scheme $S$.
  Suppose that $X$ is normal and that $Z$ admits a dualizing complex
  $\omega_Z^\bullet$.
  Denote by $K_X$ a canonical divisor on $X$ associated to $\omega_X^\bullet =
  \pi^!\omega_Z^\bullet$.
  \par Let $(X,\Delta)$ be a $\QQ$-pair, and let $H \in \Pic(X)$ such that one
  of the following holds:
  \begin{enumerate}[label=$(\roman*)$,ref=\roman*]
    \item\label{thm:ratdlt} $(X,\Delta)$ is dlt (or more generally, weakly log
      terminal) and $H$ is
      $\pi$-ample.
    \item\label{thm:ratklt}
      $(X,\Delta)$ is klt and $H$ is $\pi$-big and $\pi$-nef.
  \end{enumerate}
  If $K_X+\Delta$ is not $\pi$-nef, then
  \begin{align*}
    r &\coloneqq \max\Set[\big]{t \in \RR \given H+t(K_X+\Delta)\ \text{is
    $\pi$-nef}}
  \intertext{is a rational number.
  Moreover, expressing $r/a = u/v$ with $u,v \in \ZZ_{>0}$ and $(u,v) = 1$, we
  have $v \le a(b+1)$, where}
    a &\coloneqq \min\Set[\big]{e \in \ZZ_{>0} \given e(K_X+\Delta)\ \text{is
    Cartier}},\\
    b &\coloneqq \max_{\substack{z \in Z\\\text{closed}}}
    \Set[\big]{\dim_{\kappa(z)}\bigl(\pi^{-1}(z)\bigr)}.
  \end{align*}
\end{theorem}
\begin{proof}
  We claim we may replace $Z$ by a scheme $Z'$ \'etale over $Z$.
  Let $f\colon Z' \to Z$ be a surjective \'etale morphism where $Z'$ is a
  quasi-compact scheme, and consider the associated Cartesian diagram
  \[
    \begin{tikzcd}
      X' \rar{f'}\dar[swap]{\pi'} & X\dar{\pi}\\
      Z' \rar{f} & Z\mathrlap{.}
    \end{tikzcd}
  \]
  As in the proof of Theorem \ref{thm:bpf}, the conditions on $(X,\Delta)$ are preserved.
  Since $f$ is surjective, nefness is invariant under base change by Lemma
  \ref{lem:nefbasechange}.
  The number $b$ is invariant because $f$ is quasi-finite.
  The number $a$ is invariant because of the definition of $\Pic(X)$.\medskip
  \par We now prove the theorem when $Z$ is a scheme.
  We will derive a contradiction assuming that either $r \notin \QQ$, or that
  $r \in \QQ$ and $v > a(b+1)$.\smallskip
  \par We first claim that we may assume that $H$ is $\pi$-generated and that $H
  - (K_X+\Delta)$ is $\pi$-ample in case $(\ref{thm:ratdlt})$, and $\pi$-big and
  $\pi$-nef in case $(\ref{thm:ratklt})$.
  Let $c$ be sufficiently large such that $a < cr$ and $(c,v) = 1$.
  We then see that
  \[
    cH + a(K_X+\Delta)
  \]
  is $\pi$-nef since $a < cr$.
  Moreover, we claim that
  \[
    cH + (a-1)(K_X+\Delta) = \frac{c}{a}H
    + \frac{a-1}{a}\bigl(cH + a(K_X+\Delta)\bigr)
  \]
  is $\pi$-ample in case $(\ref{thm:ratdlt})$, and $\pi$-big and $\pi$-nef in
  case $(\ref{thm:ratklt})$.
  Case $(\ref{thm:ratdlt})$ is clear from Theorem \ref{thm:ampisinterior}, since
  it is the sum of a $\pi$-ample and a $\pi$-nef $\QQ$-invertible sheaf.
  In case $(\ref{thm:ratklt})$, we see that $cH + (a-1)(K_X+\Delta)$ is
  $\pi$-nef since it is the sum of two $\pi$-nef $\QQ$-invertible sheaves, and
  is $\pi$-big by Lemma \ref{lem:bigplusnef} since it is the sum of a $\pi$-big
  and a $\pi$-nef $\QQ$-invertible sheaf.
  Since $cH+(a-1)(K_X+\Delta)$ is $\pi$-big and $\pi$-nef, the Basepoint-free
  theorem \ref{thm:bpf} implies
  \[
    H' \coloneqq n\bigl(cH+a(K_X+\Delta)\bigr)
  \]
  is $\pi$-generated for $n \gg 0$.
  We moreover choose $n$ such that $(nc,v) = 1$.
  Setting
  \[
    r' \coloneqq \max\Set[\big]{t \in \RR\given H' + t(K_X+\Delta)\ \text{is
    $\pi$-nef}},
  \]
  we have $r'/a = ncr/a - n$.
  Thus, we have $r \in \QQ$ if and only if $r' \in \QQ$.
  In this case, writing $r'/a = u'/v'$ with $u',v' \in \NN$ and $(u',v') = 1$,
  we have $v = v'$ by the choice of $c$ and $n$.
  We therefore also have $v \le a(b+1)$ if and only if $v' \le a(b+1)$.
  We can therefore replace $H$ by $H'$ to assume that $H$ is
  $\pi$-generated.
  We also know that
  \[
    H' - (K_X+\Delta) = (n-1)\bigl(cH+a(K_X+\Delta)\bigr) + cH+(a-1)(K_X+\Delta)
  \]
  is $\pi$-ample in case $(\ref{thm:ratdlt})$, and $\pi$-big and $\pi$-nef in
  case $(\ref{thm:ratklt})$ by the same argument as above.\smallskip
  \par We can now proceed as in the proof of \cite[Theorem
  4-1-1]{KMM87} starting at \cite[Paragraph 2 on p.\ 324]{KMM87} with the
  following changes:
  \begin{itemize}
    \item In \cite[Paragraph 2 on p.\ 324]{KMM87},
      we can apply \cite[Lemma 4-1-2]{KMM87} to each connected component of
      the geometric generic fiber of $\pi \circ f$.
      This comment also applies to \cite[Bottom of p.\ 325]{KMM87}, where we can
      apply \cite[Lemma 4-1-2]{KMM87} to each connected component of
      the geometric generic fiber of $\pi \circ f$ restricted to \(B\) defined
      in \cite[Bottom of p.\ 324]{KMM87}.
    \item The necessary log resolutions in \cite[Paragraph 3 on p.\ 324]{KMM87}
      can be constructed as in the proof of Theorem \ref{thm:bpf}.
    \item In \cite[p.\ 325]{KMM87}, the Kawamata--Viehweg vanishing theorem \cite[Theorem
      1-2-3]{KMM87} should be replaced by \cite[Theorem A]{Mur}, and the
      Basepoint-free theorem \cite[Theorem 3-1-1]{KMM87} should be replaced
      by the Basepoint-free
      theorem \ref{thm:bpf}.\qedhere
  \end{itemize}
\end{proof}
\subsection{Cone theorem}
Finally, we consider the Cone theorem.
\begin{theorem}[Cone theorem; cf.\ {\cite[Theorem
  4-2-1]{KMM87}}]\label{thm:kmmcone}
  Let $\pi\colon X \to Z$ be a projective surjective morphism of integral
  quasi-excellent Noetherian algebraic spaces of equal characteristic zero over
  a scheme $S$.
  Suppose that $X$ is normal and that $Z$ admits a dualizing complex
  $\omega_Z^\bullet$.
  Denote by $K_X$ a canonical divisor on $X$ associated to $\omega_X^\bullet =
  \pi^!\omega_Z^\bullet$.
  \par Let $(X,\Delta)$ be a dlt (or more generally, weakly log terminal)
  $\QQ$-pair.
  Then,
  \[
    \NEbar(X/Z) = \NEbar_{K_X+\Delta\ge0}(X/Z) + \sum_j R_j,
  \]
  where $R_j$ are extremal rays of $\NEbar(X/Z)$ for $(X,\Delta)$.
  Moreover, if $C_j \subseteq X$ is an integral closed subscheme such that $R_j
  = \RR_{\ge0} \cdot [C_j]$, then for every $\pi$-ample $A\in \Pic(X)$,
  expressing
  \[
    \frac{(A\cdot C_j)}{a((K_X+\Delta)\cdot C_j)} = - \frac{u_j}{v_j}
  \]
  with $u_j,v_j \in \ZZ_{>0}$ and $(u_j,v_j) = 1$, we have $v_j \le a(b+1)$,
  where
  \begin{align*}
    a &\coloneqq \min\Set[\big]{e \in \ZZ_{>0} \given e(K_X+\Delta)\ \text{is
    Cartier}},\\
    b &\coloneqq \max_{\substack{z \in Z\\\text{closed}}}
    \Set[\big]{\dim_{\kappa(z)}\bigl(\pi^{-1}(z)\bigr)}.
  \end{align*}
  In particular, the $R_j$ are discrete in the half space
  \[
    \Set[\big]{\beta \in N_1(X/Z) \given \bigl( (K_X+\Delta) \cdot \beta \bigr) < 0}.
  \]
\end{theorem}
\begin{proof}
  The proof of \cite[Theorem 4-2-1]{KMM87} applies with the following changes:
  \begin{itemize}
    \item The proof of \cite[Lemma 4-2-2]{KMM87} applies in our setting.
      First, the Contraction theorem \cite[Theorem
      3-2-1]{KMM87} should be replaced with our Contraction theorem \ref{thm:contraction}.
      Second, the proof of \cite[Lemma 3-2-4]{KMM87} works in our setting since
      Kleiman's criterion for ampleness holds (Proposition
      \ref{lem:AmpleIsPositiveOnNE}).
      Third, the proof of \cite[Lemma 3-2-5]{KMM87} works in our setting by
      replacing the Basepoint-free theorem \cite[Theorem 3-1-1]{KMM87} with our
      Basepoint-free theorem \ref{thm:bpf}.
    \item In \cite[Step 1 on p.\ 327]{KMM87}, the Rationality theorem
      \cite[Theorem 4-1-1]{KMM87} should be replaced by our Rationality theorem
      \ref{thm:rationality}.
    \item In \cite[Step 2 on p.\ 327]{KMM87}, the preliminary result \cite[Lemma
      4-2-2]{KMM87} holds by the first item in this list.
    \item In \cite[Step 3 on p.\ 328]{KMM87}, the Contraction theorem \cite[Theorem
      3-2-1]{KMM87} should be replaced with our Contraction theorem
      \ref{thm:contraction}, and the Rationality theorem
      \cite[Theorem 4-1-1]{KMM87} should be replaced by our Rationality theorem
      \ref{thm:rationality}.
    \item In \cite[Step 4 on p.\ 328]{KMM87}, there exist a finite basis of
      \(N^1(X/Z)\) consisting of numerical classes of \(\pi\)-ample invertible
      sheaves by Remark \ref{rem:ampspansn1}.\qedhere
  \end{itemize}
\end{proof}

\begingroup
\makeatletter
\renewcommand{\@secnumfont}{\bfseries}
\part{Finite generation of relative adjoint rings}\label{part:fingen}
\makeatother
\endgroup
In this part, we prove Theorem \ref{thm:introfinitegen} for schemes and
algebraic spaces by adapting the strategy in \cite{CL12} that was used for
complex varieties.
We then prove dual versions of the Rationality, Cone, and Contraction theorems
in the vein of \cite{Kaw11} using our finite generation result (Theorem
\ref{thm:introfinitegen}), as is done for varieties in \cite{CL13}.
These versions of these results will be used later when showing termination with
scaling.\medskip
\par To summarize the general structure of this section, the main results and
the logical relationship between them are as follows:
\[
  \begin{tikzcd}[ampersand replacement=\&,column sep=-6em,row sep=2.4em]
    \mathllap{\begin{tabular}{@{}c@{}}
      \(\cE_A(V)\) is a rational\\
      polytope (Theorem \ref{thm:cl12b})
    \end{tabular}}
    \arrow[Rightarrow,end anchor=north,
    start anchor={[yshift=-0.25em]east},in=90,out=-5]{dr}
    \& \&
    \mathrlap{\begin{tabular}{@{}c@{}}
      Lifting sections\\
      (Theorem \ref{thm:cl1234})
    \end{tabular}}
    \arrow[Rightarrow,end anchor=north,
    start anchor={[yshift=-0.25em]west},in=90,out=185]{dl}
    \\
    \& \begin{tabular}{@{}c@{}}
      \(\cB_A^S(V)\) is a rational\\
      polytope (Theorem \ref{thm:cl1243})
    \end{tabular} \dar[Rightarrow]\\
    \& \begin{tabular}{@{}c@{}}
      Finite generation of relative adjoint rings\\
      for log regular pairs (Theorem \ref{thm:cl12a})
    \end{tabular} \dar[Rightarrow]\\
    \& \begin{tabular}{@{}c@{}}
      Finite generation of relative adjoint rings\\
      for klt pairs (Theorem \ref{thm:cl1332})
    \end{tabular} \dar[Rightarrow]\\
    \& \begin{tabular}{@{}c@{}}
      Dual versions of the Rationality, Cone,\\
      and Contraction theorems (Theorem \ref{thm:cl1342})
    \end{tabular}
  \end{tikzcd}\medskip
\]
\section{Statements of theorems}\label{sect:cl12statements}
\par We state our version of \cite[Theorem A]{CL12}, which is very close to
the original.
\begin{theorem}[{cf.\ \cite[Theorem A]{CL12}}]\label{thm:cl12a}
  Let $\pi\colon X \to Z$ be a projective morphism of integral Noetherian
  excellent schemes of equal characteristic zero,
  such that $X$ is regular of dimension $n$ and such
  that $Z$ is affine and has a dualizing complex $\omega_Z^\bullet$.
  Denote by $K_X$ a canonical divisor on $X$ associated to $\omega_X^\bullet =
  \pi^!\omega_Z^\bullet$.
  \par Let $B_1,B_2,\ldots,B_k$ be $\QQ$-divisors on $X$ such that $\lfloor B_i
  \rfloor = 0$ for all $i$, and such that $\sum_{i=1}^k B_i$ has simple normal
  crossings support.
  Let $A$ be a $\pi$-ample $\QQ$-divisor on $X$, and set $D_i = K_X+A+B_i$
  for every $i$.
  Then, the relative adjoint ring
  \[
    R\bigl(X/Z;D_1,D_2,\ldots,D_k\bigr)
    = \bigoplus_{(m_1,m_2,\ldots,m_k) \in \NN^k}
    H^0\bigl(X,\cO_X\bigl( \lfloor m_1D_1 + m_2D_2 + \cdots + m_kD_k\rfloor
    \bigr)\bigr)
  \]
  is finitely generated as an $H^0(Z,\cO_Z)$-algebra.
\end{theorem}

As in \cite{CL12}, we will prove Theorem \ref{thm:cl12a} by induction.
We will prove Theorem \ref{thm:cl12a} as part of Theorem \ref{thm:proofofcl12a} below.
In order to facilitate the induction, we adopt the following:
\begin{convention}
  In this paper, we write ``Theorem $\text{\ref{thm:cl12a}}_{n}$ holds'' 
  to mean ``Theorem
  \ref{thm:cl12a} holds when $\dim(X) = n$.'' 
\end{convention}
Next, we state our version of \cite[Theorem B]{CL12}.
Note that $Z$ does not necessarily have to be an excellent scheme of equal
characteristic zero in this statement.
\begin{theorem}[cf.\ {\cite[Theorem B]{CL12}}]\label{thm:cl12b}
  Let $\pi\colon X \to Z$ be a projective morphism of integral Noetherian
  schemes, such that $X$ is regular of dimension $n$ and such that $Z$ is
  affine and has a dualizing complex $\omega_Z^\bullet$.
  Assume that the function field of $X$ has characteristic zero. 
  \par Let $S_1,S_2,\ldots,S_p$ be distinct prime divisors on $X$ such that
  $(X,\sum_{i=1}^p S_i)$ is log regular, 
  and consider a $\pi$-ample $\QQ$-divisor $A$ on $X$.
  Then, setting
  \begin{align*}
    V &= \sum_{i=1}^p \RR \cdot S_i \subseteq \Div_\RR(X)\\
    \cL(V) = \biggl\{B &= \sum_{i=1}^p b_iS_i \in V
    \nonscript\:\bigg\vert\allowbreak\nonscript\:\mathopen{}
    0 \le b_i \le 1\ \text{for all}\ i\biggr\}
    \intertext{the set}
    \cE_A(V) = \bigl\{B &\in \cL(V)
    \nonscript\:\big\vert\allowbreak\nonscript\:\mathopen{}
    \lvert K_X+A+B \rvert_\RR \ne \emptyset\bigr\}
  \end{align*}
  is a rational polytope.
\end{theorem}
In \cite{CL12}, Cascini and Lazi\'c prove \cite[Theorems A and B]{CL12}
simultaneously by induction on $n$.
We will deduce Theorem \ref{thm:cl12b} directly from their work, which yields
this possibly mixed characteristic version of \cite[Theorem B]{CL12}.
We will prove Theorem \ref{thm:cl12b} at the end of \S\ref{section:eavpolytope}.

\section{\texorpdfstring{$\cE_A(V)$}{E\_A(V)}
is a rational polytope}\label{section:eavpolytope}
The goal of this section is to prove Theorem \ref{thm:cl12b}, which is our
version of \cite[Theorem B]{CL12}.
We can reduce Theorem \ref{thm:cl12b} to \cite[Theorem B]{CL12}.
To this end, we show some localization results for some asymptotic
loci of divisors.
Among those, only Corollary
  \ref{cor:LociLocalize}$(\ref{cor:EAVlocalize})$ is used in this section;
  other results will be needed later.\medskip

\par The following two results are quick corollaries of Lemma \ref{lem:LinearSysLocalizes} and Corollary \ref{cor:EffIffEffAtGenFiber}, so we call them corollaries.

\begin{corollary}\label{cor:BBsFixLocalize}
  Let $\pi\colon X\to Z$ be a projective surjective morphism of integral Noetherian
  schemes such that $X$ is regular and $Z$ is affine.
  Consider a point $z \in Z$.
  Set $R \coloneqq \cO_{Z,z}$ and
  \begin{align*}
    X_R \coloneqq{}& X \times_Z \Spec(R).
    \intertext{For divisors \(D\) on \(X\), we have}
    \Bs\bigl\lvert D_{\rvert X_R} \bigr\rvert ={}& \Bs\lvert D
    \rvert \times_Z \Spec(R),\\
    \Fix\bigl\lvert D_{\rvert X_R} \bigr\rvert ={}& \Fix\lvert D
    \rvert \times_Z \Spec(R).
  \intertext{For \(\RR\)-divisors \(D\) on \(X\),
  we have}
    \SB\bigl( D_{\rvert X_R} \bigr) ={}& \SB(D)
    \times_Z \Spec(R),\\
    \FFix\bigl( D_{\rvert X_R} \bigr)={}& \FFix(D)
    \times_Z \Spec(R).
  \end{align*}
\end{corollary}
\begin{proof}
  In all cases, the inclusion \(\subseteq\) holds trivially.
  The other inclusion follows from Lemma \ref{lem:LinearSysLocalizes}.
\end{proof}

\begin{corollary}\label{cor:LociLocalize}
Let $\pi\colon X\to Z$ be a projective surjective morphism of integral Noetherian
  schemes with $X$ regular and $Z$ affine with a dualizing complex $\omega^{\bullet}_Z$.
  Consider a point $z \in Z$, and set $R \coloneqq \cO_{Z,z}$ and $X_R \coloneqq
  X \times_Z \Spec(R)$.
  \par Let $S_1,S_2,\ldots,S_p$ be distinct prime divisors on $X$ such that
  $(X,\sum_{i=1}^p S_i)$ is log regular.
  Renumber the $S_i$ so that there exists $a \in \{1,2,\ldots,p\}$ such that $z
  \in \pi(S_i)$ for all $i \le a$ while $z \notin \pi(S_i)$ for all $i \ge a+1$.
  Let
  \[
    V_R=\sum_{i\leq a}{\mathbf R} \cdot (S_i)_R \subseteq \Div_\RR(X_R),
  \]
  and consider a $\pi$-ample $\QQ$-divisor $A$ on $X$.
  Define $\mathcal L(V_R)$ as in Definition \ref{def:cl1224} for the morphism
  $X_R \to \Spec(R)$, and identify $\mathcal L(V)$ with $\mathcal L(V_R)\times
  [0,1]^{p-a}$.
  \begin{enumerate}[label=$(\roman*)$,ref=\roman*]
    \item\label{cor:EAVlocalize}
      Define $\mathcal E_{A_R}(V_R)$ as in Definition \ref{def:cl1224} for
      the morphism $X_R \to \Spec(R)$.
      We then have
      \[
        \mathcal E_{A}(V)=\mathcal E_{A_R}(V_R)\times [0,1]^{p-a}.
      \]
    \item\label{cor:BSAVlocalize}
      Let $S$ be a prime divisor on $X$ distinct from the $S_i$ such that
      $(X,S+\sum_{i=1}^p S_i)$ is log regular and $z \in \pi(S)$.
      Define $\mathcal B^{S_R}_{A_R}(V_R)$ as in Definition \ref{def:cl1224}
      for the morphism $X_R \to \Spec(R)$.
      We then have
      \[
        \mathcal B^S_{A}(V)=\mathcal B^{S_R}_{A_R}(V_R)\times [0,1]^{p-a}.
      \]
  \end{enumerate}
\end{corollary}
\begin{proof}
  Follow immediately from Corollary \ref{cor:EffIffEffAtGenFiber} and Corollary \ref{cor:BBsFixLocalize}, respectively.
\end{proof}

We remark that the objects considered above also behave well with respect to field extensions.
This is mostly trivial with $\QQ$-coefficients, but we take extra caution here because we need to deal with $\RR$-coefficients.
We only record the results neccessary to the proof of our Theorem
\ref{thm:cl12b}; therefore we restrict our attention to $\lvert \,\cdot\,
\rvert_{\RR}$ and $\cE_A(V)$, whereas similar results hold for $\SB(\,\cdot\,)$, $\cB^S_A(V)$, etc.

\begin{lemma}\label{lem:LinearSystemFieldExtn}
Let $k$ be a field, and let $X$ be a normal geometrically connected scheme of finite type over $k$.
Let $L/k$ be a separable field extension. 
Let $D$ be an $\RR$-Weil divisor on $X$ and $D_L$ its pullback to $X_L$.
Then $\lvert D \rvert_{\RR}\neq \emptyset$ if and only if $\lvert D_L \rvert_{\RR}\neq \emptyset$.
\end{lemma}
\begin{proof}
We denote by $K(-)$ the function field of an integral scheme.

Assume $\lvert D \rvert_{\RR}\neq \emptyset$, so $D=E+\sum_i a_i\,\prdiv_X(f_i)$ where $E$ is an effective $\RR$-Weil divisor and $f_i\in K(X)^\times$.
Then, $D_L=E_L+\sum_i a_i\,\prdiv_{X_L}(f_i)$, and thus $\lvert D_L \rvert_{\RR}\neq \emptyset$.

Conversely, assume $\lvert D_L \rvert_{\RR}\neq \emptyset$, so there exist an
effective $\RR$-Weil divisor $F$ on $X_L$ and $g_j\in K(X_L)^\times$ with
$D_L=F+\sum_j b_j\,\prdiv_{X_L}(g_j)$.
There exists a finitely generated subextension $L'/k$ of $L/k$ such that $F$ is
the pullback of an effective $F'$ on $X_{L'}$ and all $g_j\in K(X_{L'})^\times$,
so $D_{L'}=F'+\sum_j b_j\,\prdiv_{X_{L'}}(g_j)$.
Therefore we may assume $L/k$ of finite type, and since $L/k$ is separable, $L$ is the function field of an integral smooth $k$-algebra $S$; see for example \cite[\href{https://stacks.math.columbia.edu/tag/00TV}{Tag 00TV}]{stacks-project}.

Now $K(X_L)=K(X_S)$, so we have the divisor $\sum_j b_j\,\prdiv_{X_{S}}(g_j)$.
After possibly replacing $S$ by a localization, we have an effective $\RR$-Weil
divisor $\mathfrak{F}$ on $X_S$ with $D_S=\mathfrak{F}+\sum_j b_j\,\prdiv_{X_{S}}(g_j)$.
Therefore, for a suitable maximal ideal $\mathfrak{m}$ of $S$, we have a well-defined effective $\RR$-divisor  $\mathfrak{F}_{S/\mathfrak{m}}$
and well-defined elements $\overline{g_j}\in K(X_{S/\mathfrak{m}})^\times$ 
such that
\[
  D_{S/\mathfrak{m}}=\mathfrak{F}_{S/\mathfrak{m}}+\sum_j
  b_j\,\prdiv_{X_{S/\mathfrak{m}}}(\overline{g_j}).
\]
The degree $d$ of $S/\mathfrak{m}$ over $k$ is finite, thus $h\colon X_{S/\mathfrak{m}}\to X$ is finite flat of degree $d$. 
Thus the proper pushforward $h_*\colon \WDiv_{\RR}(X_{S/\mathfrak{m}})\to \WDiv_{\RR}(X)$ satisfies $h_*D_L=dD$, so we have
\[
  D=\frac{1}{d}h_*(\mathfrak{F}_{S/\mathfrak{m}})+\frac{1}{d}\sum_j
  b_j\,h_*(\prdiv_{X_{S/\mathfrak{m}}}\bigl(\overline{g_j})\bigr).
\]
Since $\mathfrak{F}_{S/\mathfrak{m}}$ is effective, so is $h_*(\mathfrak{F}_{S/\mathfrak{m}})$; 
and if $\operatorname{Norm}$ is the norm function for the field extension
$K(X_{S/\mathfrak{m}})/K(X)$, then
$h_*(\prdiv_{X_{S/\mathfrak{m}}}(\overline{g_j}))=\prdiv_{X}(\operatorname{Norm}(\overline{g_j}))$.
Therefore $\frac{1}{d}h_*(\mathfrak{F}_{S/\mathfrak{m}})\in\lvert D\rvert_{\RR}$ and $\lvert
D\rvert_{\RR}\neq\emptyset$ as desired.
\end{proof}
\begin{lemma}\label{lem:EA(V)FIELDEXTENSION}
Let $k$ be a field and let $X$ be a scheme of finite type over $k$.
Let $S_1,S_2,\ldots,S_p$ be distinct prime divisors on $X$ such that
  $(X,\sum_{i=1}^p S_i)$ is log regular.
  
Let $L/k$ be a separable extension of fields. 
Let $T_{i1},T_{i2},\ldots, T_{iq_i}$ be all the irreducible components of
$(S_i)_L\subseteq X_L$, so $(X_L,\sum_{i=1}^p \sum_{j=1}^{q_i} T_{ij})$ is log
regular, and consider $V$ and $\cL(V)$ as defined in Definition \ref{def:cl1224}. 
Set
\[
W=\sum_i\sum_{j\leq q_i}\RR\cdot T_{ij}\subseteq \Div_{\RR}(X_L),
\] 
so there is a canonical injective linear map $\varphi\colon V\to W$ sending $S_i$ to $\sum_{j=1}^{q_i} T_{ij}$.

Let $A$ be an ample $\QQ$-divisor on $X$, so $A_L$ is an ample $\QQ$-divisor on $X_L$.
Then, with notation as in Definition \ref{def:cl1224}, we have
\[
  \varphi\bigl(\cE_A(V)\bigr)=\cE_{A_L}(W)\cap \varphi(V).
\]
\end{lemma}
\begin{proof}
Let $B\in \cL(V)$. 
Then $\varphi(B)=B_L$, since $L/k$ is separable. 
Since $\lvert K_X+A+B \rvert_{\RR}\neq \emptyset$, Lemma
\ref{lem:LinearSystemFieldExtn} implies $\lvert
K_{X_L}+A_L+\varphi(B)\rvert_{\RR}\neq \emptyset$.
Thus, $\varphi(\cE_A(V))\subseteq\cE_{A_L}(W)\cap \varphi(V)$.

Conversely, let $C\in \cE_{A_L}(W)\cap \varphi(V)$, so $C=\varphi(B)$ for some $B\in V$.
It is clear that $B\in\cL(V)$, and that $\lvert K_{X_L}+A_L+\varphi(B) \rvert_{\RR}\neq \emptyset$ by the definition of $\cE_{A_L}(W)$.
By Lemma \ref{lem:LinearSystemFieldExtn}, we conclude that $B\in \cE_A(V)$, as desired.
\end{proof}

With these results, we conclude that our Theorem \ref{thm:cl12b}
follows from \cite[Theorem B]{CL12}. 
\begin{proof}[Proof of Theorem \ref{thm:cl12b}]
  Since the $\RR$-linear system $\lvert K_X+A+B\rvert_\RR$ does not change
  when replacing $\pi\colon X \to Z$ by its Stein factorization,
  we may assume that $\pi$ is surjective with geometrically connected fibers.
  Let $K$ be the function field of $Z$. By Corollary
  \ref{cor:LociLocalize}$(\ref{cor:EAVlocalize})$, we may assume $Z=\Spec(K)$. 
If $K=\mathbf C$ this is exactly \cite[Theorem B]{CL12}, 
therefore we get the result from the Lefschetz Principle and Lemma \ref{lem:EA(V)FIELDEXTENSION}. 
\end{proof}

\section{Lifting sections}\label{sect:cl12s3}
The main result in this section is Theorem \ref{thm:cl1234}.
This result is a version of 
Cascini and Lazi\'c's lifting theorem
\cite[Theorem 3.4]{CL12}, which in turn is a version of Hacon and
M\textsuperscript{c}Kernan's lifting theorem \cite[Theorem 6.3]{HM10}.
To prove these results for schemes, we require the version of the Kawamata--Viehweg
vanishing theorem
for proper morphisms of schemes of equal characteristic zero proved by the second
author \cite[Theorem A]{Mur}.
In this context, log resolutions exist by \cite{Tem08,Tem12,Tem18}.
\par One additional difficulty unique to our situation
is the lack of Bertini theorems.
To use our version of Bertini theorems over local domains (Theorem
\ref{thm:bertini} and Remark \ref{rem:bertinisnc}), we need to rephrase
everything in terms of restriction maps on global sections and then reduce to
the case when we work over the spectrum of an excellent local $\QQ$-algebra
using flat base change.\medskip
\par We prove each result in \cite[\S3]{CL12}.
When the proof is not too different from that in \cite{CL12}, we indicate how
the proof therein can be adapted.
We start with the following consequences of the Kawamata--Viehweg
vanishing theorem for schemes of equal characteristic zero \cite[Theorem A]{Mur}.
\begin{lemma}[{cf.\ \cite[Lemma 3.1]{CL12}}]\label{lem:cl1231}
  Let $\pi\colon X \to Z$ be a proper morphism of integral Noetherian
  schemes of equal characteristic zero
  such that $X$ is regular of dimension $n$ and such that $Z$ is
  affine with a dualizing complex $\omega_Z^\bullet$.
  Denote by $K_X$ a canonical divisor on \(X\) associated to $\omega_X^\bullet =
  \pi^!\omega_Z^\bullet$.
  \par Let $B$ be an effective $\QQ$-divisor on $X$ such that $(X,B)$ is log
  regular and $\lfloor B \rfloor = 0$.
  Let $A$ be a $\pi$-nef and $\pi$-big $\QQ$-divisor.
  \begin{enumerate}[label=$(\roman*)$,ref=\roman*]
    \item\label{lem:cl1231i}
      Let $S \subseteq X$ be a prime divisor such that $S \not\subseteq
      \Supp(B)$.
      Consider a divisor $G$ on $X$ such that
      \[
        G \sim_\QQ K_X+S+A+B.
      \]
      Then, the restriction map
      \[
        H^0\bigl(X,\cO_X(G)\bigr) \longrightarrow H^0\bigl(S,\cO_S(G)\bigr)
      \]
      is surjective.
      In particular, we have $\lvert G_{\vert S} \rvert = \lvert G \rvert_S$.
    \item\label{lem:cl1231ii}
      Let $f\colon X \to Y$ be a birational morphism of integral excellent
      Noetherian schemes of equal characteristic zero such that the diagram
      \[
        \begin{tikzcd}[column sep=tiny]
          X \arrow{rr}{f}\arrow{dr}[swap]{\pi} & & Y\arrow{dl}\\
          & Z
        \end{tikzcd}
      \]
      commutes, where $Y \to Z$ is projective.
      Let $U \subseteq X$ be an open subset such that $f_{\vert U}$ is an
      isomorphism and such that $U$ intersects at most one irreducible component
      of $B$.
      Let $H'$ be a Cartier divisor on $Y$ that is very ample over $Z$, and let
      $H = f^*H'$.
      If $F$ is a divisor on $X$ such that
      \[
        F \sim_\QQ K_X+(n+1)H+A+B,
      \]
      then $\cO_X(F)$ is $\pi$-generated at every point of $U$.
      In particular, $\lvert F \rvert$ is basepoint-free
      at every point of $U$.
  \end{enumerate}
\end{lemma}
\begin{proof}
  The ``in particular'' statements follow from Proposition \ref{prop:har29},
  and hence it suffices to show the sheaf-theoretic
  statements in $(\ref{lem:cl1231i})$ and $(\ref{lem:cl1231ii})$.
  By flat base change, it suffices to show each statement after replacing $Z$
  with $\Spec(\cO_{Z,z})$ for every point $z \in Z$.
  This will allow us to use our version of the Bertini theorem (Theorem
  \ref{thm:bertini} and Remark \ref{rem:bertinisnc}).
  \par For $(\ref{lem:cl1231i})$, we consider the exact sequence
  \[
    0 \longrightarrow \cO_X(G-S) \longrightarrow
    \cO_X(G) \longrightarrow
    \cO_S(G) \longrightarrow 0.
  \]
  By Kawamata--Viehweg vanishing \cite[Theorem A]{Mur}, we have
  $H^1(X,\cO_X(G-S)) = 0$, and hence $H^0(X,\cO_X(G))
  \to H^0(S,\cO_S(G))$ is surjective.
  \par For $(\ref{lem:cl1231ii})$, we induce on $n = \dim(X)$.
  The case when $n = 0$ holds because in this case $X$ is affine.
  Now suppose $n > 0$.
  Since the locus where $\pi^*\pi_*\cO_X(F) \to \cO_X(F)$ is not surjective is
  closed, it suffices to show that for every closed point $x \in U$, the
  morphism $\pi^*\pi_*\cO_X(F) \to \cO_X(F)$ is surjective at $x$.
  We claim there exists a divisor $T \sim H$ such that $T$ is regular and passes
  through $x$.
  Consider the blowup $\mu\colon X' \to X$ of $X$ at $x$ with exceptional
  divisor $E$, and consider the divisor $\mu^*H - E$.
  The sheaf $\cO_{X'}(\mu^*H - E)$ is $(\pi \circ \mu)$-generated, and hence we
  can apply Theorem \ref{thm:bertini} and Remark \ref{rem:bertinisnc} to produce
  a divisor $T' \sim \mu^*H - E$ on $X'$ that is regular and intersects $E$ and
  the preimage of $B$ in $X'$ transversely, that also maps birationally onto its
  image in $Y$.
  The image of $T'$ in $X$ is then a divisor $T \sim H$ that is regular and
  passes through $x$ that intersects $B$ transversely, and hence $(X,T+B)$ is
  log regular.
  Since
  \[
    F_{\vert T} \sim_\QQ K_T + n\,H_{\vert T} + A_{\vert T} + B_{\vert T},
  \]
  by the inductive hypothesis we know that $\cO_T(F_{\vert T})$ is
  $\pi_{\vert T}$-generated at every point of $U \cap T$ (we note that $T$ may
  decompose into finitely many connected components, but the conclusion of
  $(\ref{lem:cl1231ii})$ still holds by working with each component
  separately).
  We now have the commutative diagram
  \[
    \begin{tikzcd}
      & \pi^*\pi_*\bigl(\cO_X(F-T)\bigr) \rar\dar
      & \pi^*\pi_*\bigl(\cO_X(F)\bigr) \rar\dar
      & (\pi_{\vert T})^*(\pi_{\vert T})_*
      \bigl(\cO_T(F_{\vert T})\bigr) \rar\dar & 0\\
      0 \rar & \cO_X(F-T) \rar
      & \cO_X(F) \rar
      & \cO_T(F_{\vert T}) \rar & 0
    \end{tikzcd}
  \]
  with exact rows,
  where the vertical arrows come from the counit of the adjunction $f^* \dashv
  f_*$, and the top row is exact since $H^1(X,\cO_X(F-T)) = 0$ by the
  Kawamata--Viehweg vanishing theorem \cite[Theorem A]{Mur}.
  By the inductive hypothesis, we see that the right vertical arrow is
  surjective at $x$.
  By the NAK lemma \cite[Theorem 2.3]{Mat89}, this implies that the middle
  vertical arrow is also surjective at $x$.
\end{proof}
\begin{lemma}[cf.\ {\cite[Lemma 3.2]{CL12}}]\label{lem:cl1232}
  Let $\pi\colon X \to Z$ be a proper morphism of integral Noetherian
  schemes of equal characteristic zero
  such that $X$ is regular and such that $Z$ is affine and
  excellent with a dualizing complex $\omega_Z^\bullet$.
  Denote by $K_X$ a canonical divisor on \(X\) associated to $\omega_X^\bullet =
  \pi^!\omega_Z^\bullet$.
  \par Let $S$ be a regular prime divisor on $X$ and let $B$ be an effective
  $\QQ$-divisor on $X$ such that $S \not\subseteq
  \Supp(B)$.
  Let $A$ be a $\pi$-nef and $\pi$-big $\QQ$-divisor on $X$.
  Assume that $D$ is a divisor on $X$ such that
  \[
    D \sim_\QQ K_X+S+A+B,
  \]
  and let $\sigma \in H^0(S,\cO_S(D_{\vert S}))$ be a nonzero global section
  with corresponding divisor $\Sigma$.
  Let $\Phi$ be an effective $\QQ$-divisor on $S$ such that the $\QQ$-pair $(S,\Phi)$
  is klt and such that $B_{\vert S} \le \Sigma+\Phi$.
  Then, $\sigma \in H^0(X \vert S,\cO_X(D))$.
  In particular, we have $\Sigma \in \lvert D \rvert_S$.
\end{lemma}
\begin{proof}
  The ``in particular'' statement follows from Proposition \ref{prop:har29},
  and hence it suffices to show the module-theoretic statement.

  By \cite[Theorem 2.3.6 and Lemma 4.2.4]{Tem08},
we get a log resolution $f\colon Y\to X$ of $(X,S+B)$. Write $T=f_*^{-1}S$.
Let $K_Y$ be the unique canonical divisor such that $K_Y-f^*K_X$ is $f$-exceptional. 
Then there are $f$-exceptional divisors $\Theta\geq 0$ and $E\geq 0$ on $Y$ with no common components such that
\begin{align*}
K_Y + T + \Theta &= f^*(K_X + S) + E\in \Div(Y).
\intertext{Let $g=f_{\vert R}\colon T\to S$. 
Restrict the corresponding invertible sheaves to $T$, we see that there exist canonical divisors $K_S$ of $S$ and $K_T$ of $T$ such that}
K_T + \Theta_{\vert T} &= g^*K_S + E_{\vert T}\in \Div(T).
\intertext{Therefore,}
K_T + \Theta_{\vert T} + g^*\Phi &= g^*(K_S +\Phi) + E_{\vert T}\in \Div_{\QQ}(T).
\end{align*}
Since $\Theta_{\vert T}$ and $E_{\vert T}$ are $g$-exceptional, the coefficients of $E_{\vert T}-\Theta_{\vert T}-g^*\Phi$ are the discrepancies of the klt pair $(S,\Phi)$, thus are greater than $-1$.
Therefore
\begin{equation}\label{MinusPhi32}
\lceil-g^*\Phi\rceil\geq \Theta_{\vert T}-E_{\vert T}.
\end{equation}
Now, by assumption $\Sigma\geq B_{\vert S}-\Phi$, so $g^*\Sigma\geq g^*B_{\vert S}-g^*\Phi$, 
thus $g^*\Sigma\geq \lceil -g^*\Phi\rceil+\lfloor g^*(B_{\vert S})\rfloor$ as $\Sigma$ is an integral divisor. 
Combining with the inequality \eqref{MinusPhi32}, we get
\begin{equation}\label{SigmaIneq32}
g^*\Sigma\geq \Theta_{\vert T}+\lfloor g^*(B_{\vert S})\rfloor-E_{\vert T}.
\end{equation}
Let $\Gamma=\Theta+f^*B\in \Div_{\QQ}(Y)$, so that $T\not\subseteq \Supp(\Gamma)$, $\Gamma$ and $E$ have no common components, and we have
\[
K_Y + T + \Gamma = f^*(K_X + S + B) + E\in \Div_{\QQ}(Y).
\]
Let $C = \Gamma -E$ and
\begin{equation}\label{DefineG32}
    G = f^*D -\lfloor C\rfloor = f^*D -\lfloor\Gamma\rfloor + E.
\end{equation}
Then, the $\QQ$-divisor
\[G -\bigl(K_Y + T + \{C\}\bigr) \sim_{\mathbf Q} f^*(K_X + S + A + B)-(K_Y + T + C) =
f^*A\]
is $(\pi\circ f)$-nef and $(\pi\circ f)$-big, and Lemma \ref{lem:cl1231}$(\ref{lem:cl1231i})$ implies that
\begin{equation}\label{SectionLiftG32}
    H^0\bigl(T,\cO_T(G_{\vert T})\bigr)=H^0\bigl(Y|T,\cO_Y(G)\bigr).
\end{equation}

We let $g=f_{\vert T}\colon T\to S$ and consider the composition
\[
  \cO_T\longrightarrow \cO_T(E_{\vert T})\longrightarrow
  \cO_T\bigl(E_{\vert T}+g^*(D_{\vert S})\bigr)
\]
where the second map is defined by $g^*\sigma$.
This gives a section 
\[
  \sigma'\in H^0\bigl(T,\cO_T(E_{\vert T}+g^*(D_{\vert S})\bigr)
  \]
  with divisor $E_{\vert T}+g^*\Sigma$. 
  By \eqref{SigmaIneq32}, $E_{\vert T}+g^*\Sigma\geq \Theta_{\vert T}+\lfloor g^*(B_{\vert S})\rfloor=\lfloor \Gamma\rfloor_{\vert T}$,
so the section $\sigma'$ comes from a section
\[
  \tau\in H^0\bigl(T,\cO_T\bigl(E_{\vert T}+g^*(D_{\vert S})-\lfloor
  \Gamma\rfloor_{\vert T}\bigr)\bigr)=H^0\bigl(T,\cO_T(G_{\vert T})\bigr),
\]
where the last equality holds by the definition of $G$ in
\eqref{DefineG32}. Therefore by \eqref{SectionLiftG32}, $\tau$ lifts to $\tilde\tau\in H^0(Y,\cO_Y(G))$, which in turn gives rise to an element
\[
  \rho\in H^0\bigl(Y,\cO_Y(G+\lfloor
  \Gamma\rfloor)\bigr)=H^0\bigl(Y,\cO_Y(f^*D+E)\bigr).
\]
By construction, we have $\rho_{\vert T}=\sigma'$. Since $E$ is
$f$-exceptional, pushing forward we see that $\sigma\in H^0(X|S,\cO_X(D))$ as
desired.
\end{proof}
\begin{lemma}[cf.\ {\cite[Lemma 3.3]{CL12}}]\label{lem:cl1233}
  Let $\pi\colon X \to \Spec(R)$ be a projective morphism of integral
  Noetherian schemes of equal characteristic zero
  such that $X$ is regular and $Z$ is affine and excellent with a dualizing complex $\omega_Z^\bullet$.
  Denote by $K_X$ a canonical divisor on \(X\) associated to $\omega_X^\bullet =
  \pi^!\omega_Z^\bullet$.
  \par Let $S$ be a prime divisor on $X$, let $B$ be an effective
  $\QQ$-divisor on $X$, and let $D$ be an effective $\QQ$-divisor on $X$
  such that the $\QQ$-pair $(X,S+B+D)$ is log regular, $S \not\subseteq
  \Supp(B)$, $\lfloor B \rfloor = 0$, and $D$ and $S+B$ have no common
  components.
  Let $P$ be a $\pi$-nef $\QQ$-divisor, and set $\Delta = S+B+P$.
  Assume that
  \[
    K_X+\Delta \sim_\QQ D.
  \]
  Let $k$ be a positive integer such that the divisors $kP$ and $kB$ are
  integral, and write $\Omega = (B+P)_{\vert S}$.
  Then, there is a $\pi$-very ample divisor $H$ on $X$ such that, for all
sections $\sigma \in H^0(S,\cO_S(k(K_S+\Omega)))$ and $u\in H^0(S,\cO_S(H_{\vert S})$ and all positive integers $l$, we have
  \[\sigma^lu\in H^0\bigl(X|S,\cO_X\bigl(lk(K_X+\Delta)+H\bigr)\bigr).\]
In particular, if $\Sigma$ (resp.\ $U$) is the divisor of $\sigma$ (resp.\ $u$),
we have $l\Sigma+U\in \lvert lk(K_X+\Delta)+H\rvert_S.$
\end{lemma}
\begin{proof}
For each $m \geq 0$, let $l_m = \bigl\lfloor\frac{m}{k}\bigr\rfloor$, let $r_m = m - l_m k \in \{0, 1,\ldots, k -1\}$, define
$B_m = \lceil mB\rceil - \lceil (m-1)B\rceil$, and set $P_m = kP$ if $r_m = 0$ and $P_m = 0$ otherwise. Let
\begin{equation}\label{DefineDm33}
    D_m =\sum_{i=1}^m(K_X + S + P_i + B_i) = m(K_X + S) + l_m kP + \lceil mB\rceil,
\end{equation}
and note that $D_m$ is integral and
\begin{equation}\label{DmANDDrm33}
D_m = l_m k(K_X + \Delta) + D_{r_m}.
\end{equation}
We choose a suitable $\pi$-very ample divisor $H$ as follows.
First, we choose an arbitrary $\pi$-very ample divisor $H'$ on $X$.
      Then, there exists an integer $n > 0$ such that
      $\cO_X(nH'+D_j)$ is $\pi$-generated for every $j \in \{0,1,\ldots,k-1\}$
      by \cite[Proposition 2.6.8$(i)$]{EGAII}.
      Now $\cO_X((n+m)H'+D_j)$ is $\pi$-very ample for every $j \in
      \{0,1,\ldots,k-1\}$ and every integer $m > 0$ by \cite[Proposition
      4.4.8]{EGAII}.
      Finally, by relative Serre vanishing \cite[Th\'eor\`eme 2.2.1$(ii)$]{EGAIII1},
      choosing $m$ large enough and setting $H = (n+m)H'$, we have
      $H^1(X, \cO_X(D_k + H-S))=0$.
      Therefore, our $H$ satisfies
      \begin{equation}\label{Casek33}
      H^0\bigl(X|S,\cO_X(D_k + H)\bigr) = H^0\bigl(S,\cO_S( (D_k+H)_{\vert
      S})\bigr)
      \end{equation}
      and $\cO_X(H+D_j)$ is $\pi$-very ample for every $j \in
      \{0,1,\ldots,k-1\}$.
\par We claim the following. For all $m\geq k$ and all sections $u_m\in
H^0(S,\cO_S((D_{r_m}+H)_{\vert S}))$, we have
\[\sigma^{l_m}u_m\in H^0\bigl(X|S,\cO_X(D_m+H)\bigr).\]
The case $r_m=0$ is what we want. 
The claim is local, so after replacing $Z$ with $\Spec(\cO_{Z,z})$ for every
point $z \in Z$, we may assume that $Z$ is local, in which case we may use
our version of the Bertini theorem (Theorem
\ref{thm:bertini} and Remark \ref{rem:bertinisnc}).
\par We prove the claim by induction on $m$. The case $m = k$ is covered by \eqref{Casek33}. Now
let $m > k$, and pick a small positive rational number $\delta$ such that $D_{r_m-1} + H + \delta B_m$ is $\pi$-ample. 
Note that $0 \leq B_m \leq \lceil B\rceil$, that $(X, S + B + D)$ is log regular, and that
$D$ and $S + B$ have no common components. Thus, there exists a small positive rational number $\varepsilon$ such that, if we define
\begin{equation}\label{DefineF33}
     F = (1 - \varepsilon\delta)B_m + l_{m-1}k\varepsilon D,
\end{equation}
then $(X, S + F)$ is log regular, $\lfloor F\rfloor = 0$ and $S \not\subseteq
\Supp(F)$. In particular, by Theorem \ref{thm:bertini} and Remark \ref{rem:bertinisnc} applied to $S\to Z$, there exists an element
$W$ of the $\pi$-generated (in fact $\pi$-very ample) linear system
$\lvert(D_{r_m-1} + H)_{|S}\rvert$ such that $W$ is reduced, does not share a component with $F_{|S}$, and that $(S,W+F_{|S})$ is log regular. Thus, if we let
\begin{equation}\label{DefinePhi33}
     \Phi = F_{|S} + (1 -\varepsilon)W,
\end{equation} 
then $(S, \Phi)$ is klt. 
By induction, there is a divisor $\Theta\in \lvert D_{m-1} + H\rvert$ whose
support does not contain $S$ and $\Theta_{|S} = l_{m-1}\Sigma + W$.
Note that the statement is about sections, but we get divisors from sections.
\par Denoting $C = (1 -\varepsilon)\Theta + F$, by \eqref{DefineF33} we have
\begin{equation}\label{DefineC33}
     C \sim_{\mathbf Q} (1 -\varepsilon)(D_{m-1} + H) + (1 - \varepsilon\delta)B_m + l_{m-1}k\varepsilon D,
\end{equation} 
and \eqref{DefinePhi33} yields
\begin{equation}\label{RestrictC33}
     C_{|S} = (1 -\varepsilon)\Theta_{|S} + F_{|S} = (1-\varepsilon)l_{m-1}\Sigma
     +\Phi \leq \bigl(l_m\Sigma + \prdiv(u_m)\bigr) + \Phi.
\end{equation}  
By the choice of $\delta$ and since $P_m= kP$ or $0$ is $\pi$-nef, the $\QQ$-divisor
\[ A = \varepsilon (D_{r_{m-1}} + H + \delta B_m) + P_m\]
is $\pi$-ample. Then by \eqref{DefineDm33}, \eqref{DmANDDrm33}, and
\eqref{DefineC33}, we have
\begin{align*}
    D_m + H &= K_X + S + D_{m-1} + B_m + P_m + H\\
&= K_X + S + (1-\varepsilon)D_{m-1} + l_{m-1}k\varepsilon(K_X + \Delta) + \varepsilon D_{r_{m-1}} + B_m + P_m + H\\
&= K_X + S + A + (1-\varepsilon)D_{m-1} + l_{m-1}k\varepsilon D + (1 -\varepsilon\delta)B_m + (1 -\varepsilon)H\\
&\sim_{\mathbf Q} K_X + S + A + C,
\end{align*}
and thus $\sigma^{l_m}u\in H^0(X|S, \cO_X(D_{r_{m}} + H))$ by \eqref{RestrictC33} and Lemma \ref{lem:cl1232}.
\end{proof}
\begin{theorem}[{cf.\ \cite[Theorem 3.4]{CL12}}]\label{thm:cl1234}
  Let $\pi\colon X \to Z$ be a projective morphism of integral
  Noetherian schemes of equal characteristic zero
  such that $X$ is regular and such that $Z$ is affine
  and excellent with a dualizing complex $\omega_Z^\bullet$.
  Denote by $K_X$ a canonical divisor on \(X\) associated to $\omega_X^\bullet =
  \pi^!\omega_Z^\bullet$.
  \par Let $S$ be a prime divisor on $X$ and let $B$ be an effective
  $\QQ$-divisor on $X$ such that $(X,S+B)$ is log regular, $S \not\subseteq
  \Supp(B)$, and $\lfloor B \rfloor = 0$.
  Let $A$ be a $\pi$-ample $\QQ$-divisor on $X$, and set $\Delta = S + A + B$.
  Let $C$ be an effective $\QQ$-divisor on $S$ such that $(S,C)$ is canonical,
  and let $m$ be a positive integer such that $mA$, $mB$, and $mC$ are
  integral.
  \par Assume there exists a positive integer $q > 0$ such that $qA$ is
  $\pi$-very ample, and we have
  \begin{gather*}
    S \not\subseteq \Bs\Bigl\lvert
    qm\Bigl(K_X+\Delta+\frac{1}{m}A\Bigr)\Bigr\rvert\\
    C \le B_{\vert S} - B_{\vert S} \wedge \frac{1}{qm} \Fix\Bigl\lvert
    qm\Bigl(K_X+\Delta+\frac{1}{m}A \Bigr) \Bigr\rvert_S
  \end{gather*}
  where $K_X$ is a canonical divisor on \(X\) associated to $\omega_X^\bullet =
  \pi^!\omega_Z^\bullet$.
  Then, for every nonzero global section $\sigma \in H^0(S,\cO_S(m(K_S+A_{\vert
  S}+C)))$, the image of $\sigma$ under the map
  \[
    H^0\bigl(S,\cO_S\bigl(m(K_S+A_{\vert S}+C)\bigr)\bigr) \xrightarrow{\cdot
    m(B_{\vert S} - C)} H^0\bigl(S,\cO_S\bigl(m(K_X+\Delta)_{\vert
    S}\bigr)\bigr)
  \]
  lies in $H^0(X\vert S,\cO_X(m(K_X+\Delta)))$.
  In particular, we have
  \[
    \bigl\lvert m(K_S+A_{\vert S}+C) \bigr\rvert + m(B_{\vert S}
    - C) \subseteq \bigl\lvert m(K_X+\Delta)\bigr\rvert_S,
  \]
  and if $\lvert m(K_S+A_{\vert S}+C) \rvert \ne \emptyset$, then
  $\lvert m(K_X+\Delta)\rvert_S \ne \emptyset$, and
  \[
    \Fix\bigl\lvert m(K_S+A_{\vert S}+C) \bigr\rvert + m(
    B_{\vert S} - C ) \ge \Fix\bigl\lvert m(K_X+\Delta)\bigr\rvert_S \ge m
    \FFix_S(K_X+\Delta).
  \]
\end{theorem}
\begin{proof}
  The ``in particular'' statements follow from Proposition \ref{prop:har29},
  and hence it suffices to show the module-theoretic statement.
  By flat base change and \cite[Chapter II, \S3, no.\ 3,
  Corollary 1 to
  Theorem 1]{BouCA}, it suffices to show the statement after replacing $Z$ with
  $\Spec(\cO_{Z,z})$ for every point $z \in Z$.
We may therefore assume $Z$ is local, in which case we may use
our version of the Bertini theorem (Theorem
\ref{thm:bertini} and Remark \ref{rem:bertinisnc}).
\par By \cite[Chapter I, \S3, Main Theorem I$(n)$]{Hir64}, 
we can find a simultaneous log resolution $f\colon Y\to X$ of $(X,S\cup
\Supp(B))$ and the base ideal $\fb(\lvert qm(K_X+\Delta+\frac{1}{m}A)\rvert)$. 
Then, for some choice of the canonical divisor $K_Y$, there are ${\mathbf Q}$-divisors
$B', E \geq 0$ on $Y$ with no common components, such that $E$ is $f$-exceptional and
\begin{align*}
  K_Y + T + B'&= f^*(K_X + S + B) + E,
  \intertext{where $T = f_*^{-1}S$.
  Note that this implies}
  K_T + B'_{|T} &= g^*(K_S + B_{|S}) + E_{|T}
\end{align*}
where $g=f_{\vert T}\colon T\to S$ and $K_T$ and $K_S$ are some choices of canonical divisors of $T$ and $S$ respectively.
Since $(Y, T + B' + E)$ is log regular and $B'$ and $E$ do not have common components, it follows that $B'_{\vert T}$
and $E_{|T}$ do not have common components.
In
particular, $E_{|T}$ is g-exceptional and $g_*B'_{|T} = B_{|S}$. 
Let $\Gamma = T + f^*A + B'$, and define
\[F_q =\frac{1}{qm}
\Fix \Bigl|qm\Bigl(K_Y + \Gamma + \frac{1}{m}f^*A\Bigr)\Bigr|.\]
We notice that $qm(K_Y + \Gamma + \frac{1}{m}f^*A)=f^*(qm(K_X+\Delta+\frac{1}{m}A))+qmE$ and that $E$ is $f$-exceptional.
Therefore, $\fb(\lvert qm(K_Y + \Gamma + \frac{1}{m}f^*A)\rvert)$ is the product of $\mathcal O_Y(-qmE)$ and $\fb(\lvert f^*(qm(K_X+\Delta+\frac{1}{m}A))\rvert)$, the latter being equal $f^*\fb(\lvert qm(K_X+\Delta+\frac{1}{m}A)\rvert)$.
Since we resolved $\fb(\lvert qm(K_X+\Delta+\frac{1}{m}A)\rvert)$, its pullback is an invertible ideal, hence so is $\fb(\lvert qm(K_Y + \Gamma + \frac{1}{m}f^*A)\rvert)$.
Therefore the mobile part
\[
  \Mob\Bigl(qm\Bigl(K_Y + \Gamma + \frac{1}{m}f^*A\Bigr)\Bigr)=qm\Bigl(K_Y+\Gamma
  +\frac{1}{m}f^*A-F_q\Bigr)
\]
is $(\pi\circ f)$-generated.  
By Theorem \ref{thm:bertini} and Remark \ref{rem:bertinisnc}, we may take $D^\circ\in \lvert K_Y+\Gamma
+\frac{1}{m}f^*A-F_q\rvert_{\mathbf Q}$ such that 
$(Y,T+B'+F_q+D^\circ)$ is log regular and that $D^\circ$ does not contain any component of $T+B'$.
Now define
\[B'_q = B'- B' \wedge F_q, \qquad \Gamma_q = T + B'_q + f^*A, \qquad
D=D^\circ+F_q-B' \wedge F_q.\]
Then, \[ D\sim_{\mathbf Q}K_Y+\Gamma_q+\frac{1}{m}f^*A,\]
the pair $(Y,T+B'_q+D)$ is log regular, and $D$ does not contain any component of $T+B'_q$.

Let $g=f_{\vert T}\colon T\to S$ and $C'=g^{-1}_*C$. We claim that $C'\leq B'_{q|T}$. Assuming the claim, let us show how it implies the theorem.
By Lemma \ref{lem:cl1233}, there exists a $\pi$-very ample divisor $H$ on $Y$
such that for all divisors $\Sigma'\in\lvert
K_T+(B'_q+(1+\frac{1}{m})f^*A)_{\vert T}\rvert$ and $U\in \lvert H_{\vert
T}\rvert$ and for all positive integers $p$, we have
\[p\Sigma'+U\in \Bigl|pqm\Bigl(K_X+\Delta+\frac{1}{m}A\Bigr)+H\Bigr|_T.\]
Since $f$ is constructed as a blowup of $X$ along regular centers, there exists
an effective $f$-exceptional divisor $G$ such that $-G$ is $f$-ample.
After possibly replacing $G$ by a small rational multiple, we therefore see that
$f^*A-G$ is ample, and $\lfloor B' + \frac{1}{m}G \rfloor = 0$, in which case
$(T,(B'+\frac{1}{m}G)_{\vert T})$ is klt. 
Now, we choose a positive integer $k$ so large such that for $l=kq$ the $\mathbf Q$-divisor
\begin{displaymath}
    A_0=\frac{1}{m}(f^*A-G)-\frac{m-1}{ml}H
\end{displaymath}
$(\pi\circ f)$-ample. This is possible because $f^*A-G$ is $(\pi\circ f)$-ample. 
By Theorem \ref{thm:bertini} and Remark \ref{rem:bertinisnc}, we may find
reduced divisors $W_1\in \lvert q(f^*A)_{|T}\rvert$ and $W_2\in\lvert
H_{|T}\rvert$ such that $(B'+\frac{1}{m}G)_{|T}, W_1$ and $W_2$ share no common components and that
$(T,(B'+\frac{1}{m}G)_{|T}+W_1+W_2)$ is log regular. For $W=kW_1+W_2$ and
\[
  \Phi=B'_{q|T}+\frac{1}{m}G_{|T}+\frac{1}{l}W=B'_{q|T}+\frac{1}{m}G_{|T}+\frac{1}{q}W_1+\frac{1}{l}W_2,
\]
the pair $(T,\Phi)$ is klt, since $\lfloor B'+\frac{1}{m}G\rfloor=0$.
Now the proof of \cite[Theorem 3.4]{CL12} applies verbatim, except \cite[Lemma
3.2]{CL12} should be replaced by Lemma \ref{lem:cl1232}.

It remains to verify the claim $C'\leq B'_{q|T}$. This is also identical to the
corresponding part of the proof of \cite[Theorem 3.4]{CL12}, except for the
word change ``basepoint-free'' to ``$(\pi\circ f)$-generated.''
\end{proof}
As in \cite{CL12}, we immediately obtain the following version of the lifting
theorem of Hacon and M\textsuperscript{c}Kernan \cite[Theorem 6.3]{HM10}.
\begin{corollary}[{cf.\ \citeleft\citen{CL12}\citemid Corollary 3.5\citeright}]
  \label{cor:cl1235}
  Let $\pi\colon X \to Z$ be a projective morphism of integral
  Noetherian schemes of equal characteristic zero
  such that $X$ is regular and such that $Z$ is affine
  and excellent with a dualizing complex $\omega_Z^\bullet$.
  Denote by $K_X$ a canonical divisor on \(X\) associated to $\omega_X^\bullet =
  \pi^!\omega_Z^\bullet$.
  \par Let $S$ be a prime divisor on $X$ and let $B$ be an effective
  $\QQ$-divisor on $X$ such that $(X,S+B)$ is log regular, $S \not\subseteq
  \Supp(B)$, and $\lfloor B \rfloor = 0$.
  Suppose that $(S,B_{\vert S})$ is canonical.
  Let $A$ be a $\pi$-ample $\QQ$-divisor on $X$, and set $\Delta = S + A + B$.
  Let $m$ be a positive integer such that $mA$ and $mB$ are integral and such
  that $S \not\subseteq \Bs\lvert m(K_X+\Delta)\rvert$.
  Set
  \[
    \Phi_m = B_{\vert S} - B_{\vert S} \wedge \frac{1}{m} \Fix\bigl\lvert
    m(K_X+\Delta) \bigr\rvert_S.
  \]
  Then, we have
  \[
    \bigl\lvert m(K_S+A_{\vert S}+\Phi_m) \bigr\rvert + m(B_{\vert S}
    - \Phi_m) \subseteq \bigl\lvert m(K_X+\Delta)\bigr\rvert_S.
  \]
\end{corollary}
\begin{proof}
  The proof of \cite[Corollary 3.5]{CL12} applies after replacing \cite[Theorem
  3.4]{CL12} with our Theorem \ref{thm:cl1234}.
\end{proof}
\begin{lemma}[cf.\ {\cite[Lemma 3.6]{CL12}}]\label{lem:cl1236}
  Let $\pi\colon X \to Z$ be a projective morphism of integral
  Noetherian schemes of equal characteristic zero
  such that $X$ is regular and such that $Z$ is affine
  and excellent with a dualizing complex $\omega_Z^\bullet$.
  Denote by $K_X$ a canonical divisor on \(X\) associated to $\omega_X^\bullet =
  \pi^!\omega_Z^\bullet$.
  \par Let $S$ be a regular prime divisor on $X$, let $D$ be a $\QQ$-divisor
  on $X$ such that $S \not\subseteq \SB(D)$, and let $A$ be a $\pi$-ample
  $\QQ$-divisor.
  Then, we have
  \[
    \frac{1}{q} \Fix\bigl\lvert q(D+A) \bigr\rvert_S \le \FFix_S(D)
  \]
  for all sufficiently divisible positive integers $q$.
\end{lemma}
\begin{proof}
The proof of \cite[Lemma 3.6]{CL12} carries word by word with the following changes:
\begin{itemize}
  \item All instances of the words ``ample'' and ``very ample'' become ``$\pi$-ample''
    and ``$\pi$-very ample,'' respectively.
  \item The sentence ``In particular, if $V\in \lvert F\rvert$ is a general element,
  then $P\not\subseteq \Supp f_*V$'' becomes ``In particular, for some $V\in
  \lvert F\rvert$ we have $P\not\subseteq \Supp f_*V$.''
\item The reference \cite[Lemma 3.1]{CL12} should be replaced by Lemma
  \ref{lem:cl1231}.
\end{itemize}
We note that the $\QQ$-divisor $D'$ in the proof of \cite[Lemma 3.6]{CL12} does
not come from Bertini's theorem, since the existence of a $\QQ$-divisor $D'
\sim_\QQ D$ satisfying $S \not\subseteq\Supp(D')$ and $\mult_P(D'_{\vert S}) <
1/q$ follows from the definition of $\FFix_S(D)$.
\end{proof}
\section{\texorpdfstring{$\cB_A^S(V)$}{B\_A\textasciicircum S(V)}
is a rational polytope}
Following \cite[\S4]{CL12}, we prove that the set $\cB_A^S(V)$ defined in
Definition \ref{def:cl1224} is a rational polytope.
Given the work we have done in \S\ref{sect:cl12s3}, the proof in
\cite[\S4]{CL12} applies almost verbatim.\medskip
\par We replace \cite[Setup 4.1]{CL12} with the following setup.
In the rest of this section, we write ``Setup $\text{\ref{setup:cl1241}}_n$'' to
mean ``Setup \ref{setup:cl1241} when $\dim(X) = n$.''
We have only written down the notation from \cite[Setup 4.1]{CL12} that will be
used in the statements in the rest of this section.
\begin{setup}[cf.\ {\cite[Setup 4.1]{CL12}}]\label{setup:cl1241}
  Let $\pi\colon X \to Z$ be a projective morphism of integral Noetherian
  schemes of equal characteristic zero, such that $X$ is regular of dimension $n$ and such that $Z$ is
  affine and excellent and has a dualizing complex $\omega_Z^\bullet$.
  Let $S,S_1,S_2,\ldots,S_p$ be distinct prime divisors on $X$ such that
  $(X,S+\sum_{i=1}^p S_i)$ is log regular.
  We assume that Theorem $\text{\ref{thm:cl12a}}_{n-1}$ holds.
  Note that we have already shown that Theorem \ref{thm:cl12b} holds.
  \par Consider a $\pi$-ample $\QQ$-divisor $A$ on $X$.
  Let 
  \[
    V = \sum_{i=1}^p \RR \cdot S_i \subseteq \Div_\RR(X),
  \]
  and let $W \subseteq \Div_\RR(S)$ be the subspace spanned by the
  components of $\sum_i (S_i)_{\vert S}$.
  By Theorem $\text{\ref{thm:cl12b}}$, the set
  \[
    \cE_{A_{\vert S}}(W)=\Set[\big]{E\in\mathcal L(W) \given \lvert K_S+A_{\vert
    S}+E\rvert_{\RR}\neq \emptyset}
  \]
  is a rational polytope. 
  If $E_1,E_2,\ldots,E_d$ are its extreme points, then the ring
  \[
    R\bigl(S/Z;K_{S}+A_{\vert S}+E_1,K_{S}+A_{\vert S}+E_2,\ldots,K_{S}+A_{\vert
    S}+E_d\bigr)
  \]
  is finitely generated as a $H^0(Z,\cO_Z)$-algebra by Theorem
  $\textup{\ref{thm:cl12a}}_{n-1}$.
  Therefore, if we set
  \[
    \FF(E) = \FFix(K_{S}+A_{\vert S}+E)
  \]
  for a $\QQ$-divisor $E \in \cE_{A_{\vert S}}(W)$, then
  \cite[Lemma 2.28]{CL12} 
  implies that $\FF$ extends to a rational piecewise
  affine function on $\cE_{A_{\vert S}}(W)$, and there exists a
  positive integer $k$ such that
  \[
    \FF(E) = \frac{1}{m} \Fix\bigl\lvert m(K_{S}+
    A_{\vert S}+E)\bigr\rvert
  \]
  for every $E \in \cE_{A_{\vert S}}(W)$ and every $m \in \NN$
  such that $mA/k$ and $mE/k$ are integral.
  \par For a $\QQ$-divisor $B \in \cB_{A}^{S}(V)$, set
  \[
    \FF_{S}(B) = \FFix_{S}(K_X+S+A+B),
  \]
  and for every positive integer $m$ such that $mA$ and $mB$ are integral and
  $S \not\subseteq \Bs\lvert m(K_X+S+A+B)\rvert$,
  denote
  \[
    \Phi_m(B) = B_{\vert S} - B_{\vert S} \wedge \frac{1}{m} \Fix\bigl\lvert
    m(K_X+S+A+B)\bigr\rvert_S.
  \]
  Let $\Phi(B) = B_{\vert S} - B_{\vert S} \wedge \FF_S(B)$, where we note that
  $\Phi(B) = \limsup_{m \to \infty} \Phi_m(B)$.
\end{setup}
The analogue of the main result in \cite[\S4]{CL12} is the following:
\begin{theorem}[cf.\ {\cite[Theorem 4.3]{CL12}}]\label{thm:cl1243}
  Let the assumptions of Setup $\text{\ref{setup:cl1241}}_n$ hold.
  Let $\cG$ be a rational polytope contained in the interior of $\cL(V)$, and
  assume that $(S,G_{\vert S})$ is terminal for every $G \in \cG$.
  Denote $\cP = \cG \cap \cB_A^S(V)$.
  We then have the following:
  \begin{enumerate}[label=$(\roman*)$,ref=\roman*]
    \item $\cP$ is a rational polytope.
    \item $\Phi$ extends to a rational piecewise affine function on $\cP$, and
      there exists a positive integer $\ell$ such that $\Phi(P) = \Phi_m(P)$ for
      every $P \in \cP$ and every positive integer $m$ such that $mP/\ell$ is
      integral.
  \end{enumerate}
\end{theorem}
\begin{proof}
  We work through the proofs of \cite[Lemma 4.2]{CL12}, \cite[Lemma 4.4]{CL12},
  and \cite[Theorem 4.3]{CL12}.
  Throughout, \cite[Theorem 3.4]{CL12} and \cite[Lemma 3.6]{CL12} should
  be replaced by our Theorem \ref{thm:cl1234} and Lemma \ref{lem:cl1236},
  respectively.
  \par The proof of \cite[Lemma 4.2]{CL12} works with no changes.
  The proof of \cite[Lemma 4.4]{CL12} works with the following changes:
  \begin{itemize}
    \item In Step 2, the rational number $0 < \varepsilon \ll 1$ should be
      chosen such that the divisors $D+A/4$ and $\varepsilon(K_X+S+A+B)+A/4$ are
      $\pi$-ample.
    \item In the first paragraph of \cite[p.\ 2442]{CL12}, the divisors
      \[
        H = \Gamma - B_\delta + \frac{1}{4m}A \qquad \text{and} \qquad
        G = \frac{\varepsilon}{m} (K_X+S+A+B_\delta) +\frac{1}{4m}A
      \]
      are $\pi$-ample.
  \end{itemize}
  The proof of \cite[Theorem 4.3]{CL12} works with no changes.
\end{proof}
As a result, we obtain the following corollary.
\begin{corollary}[cf.\ {\cite[Corollary 4.6]{CL12}}]\label{cor:cl1246}
  Assume 
  Theorem $\text{\ref{thm:cl12a}}_{n-1}$ holds. 
  Let $\pi\colon X \to Z$ be a projective morphism of integral Noetherian
  schemes of equal characteristic zero, such that $X$ is regular of dimension $n$ and such that $Z$ is
  affine and excellent with a dualizing complex $\omega_Z^\bullet$.
  Let $S,S_1,S_2,\ldots,S_p$ be distinct prime divisors on $X$ such that
  $(X,S+\sum_{i=1}^p S_i)$ is log regular.
  \par Let 
  \[
    V = \sum_{i=1}^p \RR \cdot S_i \subseteq \Div_\RR(X),
  \]
  and let $A$ be a $\pi$-ample $\QQ$-divisor on $X$. Then, $\cB^S_A(V)$ is a rational polytope
  and
  \[
    \cB^S_A(V) = \Set{B \in \cL(V) \given \sigma_S(K_X+S+A+B) =
    0}.
  \]
\end{corollary}
\begin{proof}
  The proof of \cite[Corollary 4.6]{CL12} applies with the following changes:
  \begin{itemize}
    \item In the first paragraph, \cite[Theorem 4.3]{CL12} should be replaced by
      our Theorem \ref{thm:cl1243}.
    \item In the second paragraph, \cite[Lemma 2.2]{CL12} holds for the pair
      $(X,S+B^G)$ since log resolutions exist \cite[Theorem 1.1.6]{Tem18}, and the
      proof of \cite[Proposition 2.36(1)]{KM98} works in this setting as well.
      Later, we choose $f^*A^G - F$ to be $(\pi\circ f)$-ample, where if $F$ is small
      enough, then $(T,(C+F)_{\vert T})$ is terminal. 
      Here, the choice of $F$ is exactly like the choice of $G$ in the proof of 
      Theorem \ref{thm:cl1234}, which works since Temkin's log resolutions are
      constructed by blowing up regular centers (see also \cite[Claim
      8.1]{Kol21qfac}).\qedhere
  \end{itemize}
\end{proof}

\section{Finite generation}
In this section, we prove Theorem $\textup{\ref{thm:cl12a}}_n$ assuming Theorem
$\textup{\ref{thm:cl12a}}_{n-1}$.
Again, we note that we have already shown Theorem \ref{thm:cl12b}.
\begin{lemma}[cf.\ {\cite[Lemma 6.1]{CL12}}]\label{lem:cl1261}
  Let $\pi\colon X \to Z$ be a proper morphism of integral
  Noetherian schemes such that $X$ is regular and such that $Z$ is affine.
  \par Let $S_1,S_2,\ldots,S_p$ be distinct prime divisors on $X$ such that
  $(X,\sum_{i=1}^p S_i)$ is log regular.
  Let
  \[
    \cC \subseteq \sum_{i=1}^p \RR_{\ge0}\cdot S_i \subseteq \Div_\RR(X)
  \]
  be a rational polyhedral cone, and let $\cC = \bigcup_{j=1}^q \cC_j$ be a
  rational polyhedral decomposition.
  Set $\cS = \cC \cap \Div(X)$ and $\cS_j = \cC_j \cap \Div(X)$ for all $j$.
  Assume the following:
  \begin{enumerate}[label=$(\roman*)$,ref=\roman*]
    \item\label{lem:cl1261condi}
      There exits a real number $M > 0$ such that if $\sum_i \alpha_iS_i \in
      \cC_j$ for some $j$ and for some $\alpha_i \in \NN$ where $\sum_i \alpha_i
      \ge M$, then $\sum_i \alpha_i S_i - S_j \in \cC$; and
    \item\label{lem:cl1261condii}
      The ring $\res_{S_j}(R(X/Z;\cS_j))$ is finitely generated as a $H^0(X
      \vert S_j,\cO_{S_j})$-algebra for every $j \in \{1,2,\ldots,p\}$.
  \end{enumerate}
  Then, the relative divisorial ring $R(X/Z;\cS)$ is finitely generated as an
  $H^0(Z,\cO_Z)$-algebra.
\end{lemma}
\begin{proof}
  After replacing $\pi\colon X \to Z$ by its Stein factorization 
  \cite[Th\'eor\`eme 4.3.1]{EGAIII1}, we may assume that $H^0(Z,\cO_Z)$
  is the degree zero piece of $R(X/Z;\cS)$.
  We now follow the proof of \cite[Lemma 6.1]{CL12}.
  For every $i \in \{1,2,\ldots,p\}$, we use Proposition \ref{prop:har29} to
  choose sections $\sigma_i \in H^0(X,\cO_X(S_i))$ such that $\prdiv(\sigma_i) =
  S_i$.
  Let $\fR \subseteq R(X/Z;S_1,S_2,\ldots,S_p)$ be the
  $H^0(Z,\cO_Z)$-subalgebra generated by $R(X/Z;\cS)$ and
  $\sigma_1,\sigma_2,\ldots,\sigma_p$.
  Note that $\fR$ is graded by $\sum_{i=1}^p \NN \cdot S_i \subseteq
  \Div(X)$.
  By \cite[Proposition 1.2.2]{ADHL15}, since $R(X/Z;\cS)$ is a Veronese subring of
  $\fR$, it suffices to show that $\fR$ is finitely generated
  as an $H^0(Z,\cO_Z)$-algebra.
  \par For each $\alpha = (\alpha_1,\alpha_2,\ldots,\alpha_p) \in \NN^p$, set
  $D_\alpha = \sum_i \alpha_i S_i$ and $\deg(\alpha) = \sum_i \alpha_i$, and for
  a section $\sigma \in H^0(X,\cO_X(D_\alpha))$, set $\deg(\sigma) =
  \deg(\alpha)$.
  By $(\ref{lem:cl1261condii})$, for each $j \in \{1,2,\ldots,p\}$, there exists
  a finite set $\cH_j \subseteq R(X/Z;\cS_j)$ such that $\res_{S_j}(
  R(X/Z;\cS_j))$ is generated by the set
  \[
    \Set[\big]{\sigma_{\vert S_j} \given \sigma \in \cH_j}
  \]
  over $H^0(X \vert S_j,\cO_X)$.
  Since the $H^0(Z,\cO_Z)$-module $H^0(X,\cO_X(D_\alpha))$ is finitely generated
  for every $\alpha \in \NN^p$, there is a finite set $\cH \subseteq
  R(X/Z;\cS_j)$ such that
  \[
    \{\sigma_1,\sigma_2,\ldots,\sigma_p\} \cup \cH_1 \cup \cH_2 \cup \cdots \cup
    \cH_p \subseteq \cH
  \]
  and such that
  \[
    H^0\bigl(X,\cO_X(D_\alpha)\bigr) \subseteq \bigl(H^0(Z,\cO_Z)\bigr)[\cH]
  \]
  inside of $\fR$ for all $\alpha \in \NN^p$ with
  $D_\alpha \in \cS$ and $\deg(\alpha) \le M$, where
  $(H^0(Z,\cO_Z))[\cH] \subseteq \fR$ holds by definition of $\fR$ and
  $\cH$.
  To show that $\fR$ is finitely generated as an $H^0(Z,\cO_Z)$-algebra, it
  therefore suffices to show that $\fR \subseteq (H^0(Z,\cO_Z))[\cH]$.
  \par Let $\chi \in \fR$.
  By definition of $\fR$, we can write
  \[
    \chi = \sum_i \sigma_1^{\lambda_{1,i}}\sigma_2^{\lambda_{2,i}} \cdots
    \sigma_p^{\lambda_{p,i}} \chi_i,
  \]
  where $\chi_i \in H^0(X,\cO_X(D_{\alpha_i}))$ for some $D_{\alpha_i} \in \cS$
  and $\lambda_{j,i} \in \NN$.
  It therefore suffices to show that $\chi_i \in (H^0(Z,\cO_Z))[\cH]$.
  After replacing $\chi$ by $\chi_i$, we may assume that $\chi \in
  H^0(X,\cO_X(D_\alpha))$ for some $D_\alpha \in \cS$.
  We induce on $\deg(\chi)$.
  If $\deg(\chi) \le M$, then $\chi \in (H^0(Z,\cO_Z))[\cH]$ by the definition
  of $\cH$ in the previous paragraph.
  Now suppose $\deg(\chi) > M$.
  Then, there exists $j \in \{1,2,\ldots,p\}$ such that $D_\alpha \in \cS_j$,
  and hence there exist $\theta_1,\theta_2,\ldots,\theta_z \in \cH$ and a
  polynomial $\varphi \in (H^0(Z,\cO_Z))[X_1,X_2,\ldots,X_z]$ such that
  \[
    \chi_{\vert S_j} = \varphi\bigl(\theta_{1\vert S_j},\theta_{2\vert
    S_j},\ldots,\theta_{z\vert S_j}\bigr).
  \]
  By the exact sequence
  \[
    0 \longrightarrow H^0\bigl(X,\cO_X(D_\alpha - S_j)\bigr)
    \xrightarrow{\sigma_j\cdot}
    H^0\bigl(X,\cO_X(D_\alpha)\bigr) \longrightarrow
    H^0\bigl(S_j,\cO_{S_j}(D_\alpha)\bigr),
  \]
  we therefore obtain
  \[
    \chi - \varphi(\theta_1,\theta_2,\ldots,\theta_z) = \sigma_j \cdot \chi'
  \]
  for some $\chi' \in H^0(X,\cO_X(D_\alpha-S_j))$.
  Since $D_\alpha - S_j \in \cS$ by $(\ref{lem:cl1261condi})$ and since
  $\deg(\chi') < \deg(\chi)$, by the inductive hypotheses we see that $\chi' \in
  (H^0(Z,\cO_Z))[\cH]$.
  Thus, we have
  \[
    \chi = \sigma_j \cdot \chi' + \varphi(\theta_1,\theta_2,\ldots,\theta_z) \in
    \bigl(H^0(Z,\cO_Z)\bigr)[\cH]
  \]
  as desired.
\end{proof}
\begin{lemma}[cf.\ {\cite[Lemma 6.2]{CL12}}]\label{lem:cl1262}
  Assume Theorem $\text{\ref{thm:cl12a}}_{n-1}$ holds.
  Let $\pi\colon X \to Z$ be a projective morphism of integral
  Noetherian schemes of equal characteristic zero such that $X$ is regular of dimension $n$
  and such that $Z$ is affine
  and excellent with a dualizing complex $\omega_Z^\bullet$.
  Denote by $K_X$ a canonical divisor on \(X\) associated to $\omega_X^\bullet =
  \pi^!\omega_Z^\bullet$.
  Let $S,S_1,S_2,\ldots,S_p$ be distinct prime divisors on $X$ such that
  $(X,S+\sum_{i=1}^p S_i)$ is log regular.
  \par Let
  \[
    V = \sum_{i=1}^p \RR \cdot S_i \subseteq \Div_\RR(X),
  \]
  let $A$ be a $\pi$-ample $\QQ$-divisor on $X$, and let $B_1,B_2,\ldots,B_m \in
  \cE_{S+A}(V)$ be $\QQ$-divisors.
  Set $D_i = K_X+S+A+B_i$.
  Then, the ring
  \[
    \res_S\bigl(R\bigl(X/Z;D_1,D_2,\ldots,D_m\bigr)\bigr)
  \]
  is finitely generated as an $H^0(Z,\cO_Z)$-algebra.
\end{lemma}
\begin{proof}
  Following the proof of \cite[Lemma 6.2]{CL12},
  we first prove the lemma under the additional assumption that the $B_i$ lie in
  the interior of $\cL(V)$, and that the pairs $(S,B_{i\vert S})$ are all
  terminal.
  This part of the proof of \cite[Lemma 6.2]{CL12} applies with the following
  changes:
  \begin{itemize}
    \item In the second paragraph, \cite[Lemma 2.27]{CL12} should be replaced by
      our Lemma \ref{lem:cl12227}.
    \item In the third and fourth paragraphs, \cite[Setup 4.1]{CL12} and
      \cite[Theorem 4.3]{CL12} should be replaced by our Setup \ref{setup:cl1241}
      and Theorem \ref{thm:cl1243}, respectively.
    \item In the fourth paragraph, \cite[Corollary 3.5]{CL12} and \cite[Theorem
      $\textup{A}_{n-1}$]{CL12} should be replaced by
      our Corollary \ref{cor:cl1235} and Theorem
      $\text{\ref{thm:cl12a}}_{n-1}$, respectively.
  \end{itemize}
  \par We now prove the general case of the lemma.
  For every $i$, we choose a $\QQ$-divisor $G_i \in V$ such that $A - G_i$ is
  $\pi$-ample and such that $B_i+G_i$ is in the interior of $\cL(V)$.
  Let $A'$ be a $\pi$-ample $\QQ$-divisor such that every $A-G_i-A'$ is also
  ample.
  We claim that there exists a finite open affine cover $Z = \bigcup_j U_j$ and
  effective $\QQ$-divisors $A_{ij} \sim_\QQ A-G_i-A'$ such that
  setting $X_j = \pi^{-1}(U_j)$, we have the following:
  \begin{enumerate}[label=$(\roman*)$,ref=\roman*]
    \item For every $j$, $\lfloor A_{i\vert X_j} \rfloor = 0$;
    \item For every $j$, the pair $(X,S+\sum_{i=1}^p S_i + \sum_{i=1}^m
      A_{ij})$ is log regular along $X_j$; and
    \item For every $j$, the support of $\sum_{i=1}^m A_{ij\vert X_j}$ does not
      contain any of the divisors $S_{\vert X_j},S_{1 \vert X_j},\ldots,$ $S_{p
      \vert X_j}$.
  \end{enumerate}
  We induce on $m$.
  The case $m=0$ follows by assumption.
  Now suppose $m > 0$.
  By the inductive hypothesis, there exists a finite affine open cover
  $Z = \bigcup_k V_k$ and $\pi$-ample $\QQ$-divisors $B_{ik} \sim_\QQ A-G_1-A'$
  for $i \in \{1,2,\ldots,m-1\}$ such that for every $k$, setting $X_k =
  \pi^{-1}(V_k)$, we have $\lfloor A_{i\vert X_k} \rfloor = 0$, the pair
  $(X,S+\sum_{i=1}^p S_i + \sum_{i=1}^{m-1} B_{ik})$ is log regular along $X_k$,
  and the support of $\sum_{i=1}^{m-1} B_{ik\vert X_j}$ does not
  contain any of the divisors $S_{\vert X_k},S_{1 \vert X_k},\ldots,S_{p
  \vert X_k}$.
  We can now apply Corollary \ref{cor:bertinionopencover} to the strata of the
  pair $(X,S+\sum_{i=1}^p S_i + \sum_{i=1}^{m-1} B_{ik})$ to construct a finite
  affine open cover $Z = \bigcup_j U_j$ refining $Z = \bigcup_k V_k$ and
  effective $\QQ$-divisors $A_{mj} \sim_\QQ A-G_i-A'$ satisfying the
  requirements above.
  Finally, by \cite[Corollaire 6.3.9]{EGAInew} and flat base change, to show that
  $\res_S(R(X/Z;D_1,D_2,\ldots,D_m))$ is finitely generated as an
  $H^0(Z,\cO_Z)$-algebra, it suffices to show that
  \[
    \res_{S \vert X_j}\bigl(R\bigl(X_j/Z_j;D_{1 \vert X_j},D_{2 \vert X_j},
    \ldots,D_{m \vert X_j}\bigr)\bigr)
  \]
  is finitely generated as an $H^0(U_j,\cO_{U_j})$-algebra for every $j$.
  Replacing $\pi\colon X \to Z$ by $\pi_{\vert X_j}\colon X_j \to U_j$, we may
  assume that the open affine cover $Z = \bigcup_j U_j$ has only one member.
  We now proceed as in the proof of \cite[Lemma 6.2]{CL12} with the following
  changes in the last paragraph:
  \begin{itemize}
    \item In the first line, \cite[Lemma 2.2]{CL12} holds for the pair
      $(X,S+B)$ since log resolutions exist \cite[Theorem 1.1.6]{Tem18}, and the
      proof of \cite[Proposition 2.36(1)]{KM98} works in this setting as well.
    \item Later, the $\QQ$-divisor $A^\circ$ is $\pi$-ample.
    \item In the last line, \cite[Corollary 2.26]{CL12} should be replaced by our
      Lemma \ref{lem:cl12226}.\qedhere
  \end{itemize}
\end{proof}
\begin{theorem}[cf.\ {\cite[Theorem 6.3]{CL12}}]\label{thm:proofofcl12a}
  Theorem $\text{\ref{thm:cl12a}}_{n-1}$ implies Theorem
  $\text{\ref{thm:cl12a}}_{n}$.
  Thus, Theorem \ref{thm:cl12a} holds.
\end{theorem}
\begin{proof}
  The proof of \cite[Theorem 6.3]{CL12} applies with the following changes:
  \begin{itemize}
    \item In (69), the words ``log smooth'' should be replaced by
      ``log regular.''
    \item Throughout, the references to \cite[Corollary 2.26]{CL12} and \cite[Lemma
      2.27]{CL12} should be replaced by references to our Lemmas
      \ref{lem:cl12226} and \ref{lem:cl12227}, respectively.
    \item After $(iii)$ on p.\ 2463, \cite[Lemma 6.1]{CL12} should
      be replaced by our Lemma \ref{lem:cl1261}.
    \item At the bottom of p.\ 2464, \cite[Lemma 6.2]{CL12} should be replaced
      by our Lemma \ref{lem:cl1262}.
    \item In the second paragraph on p.\ 2465, \cite[Theorem $\text{B}_n$]{CL12}
      should be replaced by our Theorem \ref{thm:cl12b}, which we have already
      shown holds when $\dim(X)$ is arbitrary.
    \item In the last paragraph, the log resolution $f\colon Y \to X$ exists by
      \cite[Theorem 1.1.6]{Tem18}.
      Later, we choose $A^\circ = f^*A-H$ to be $(\pi \circ f)$-ample and
      $C_i^\circ = C_i+H$ such that $\lfloor C_i^\circ \rfloor = 0$ for all $i$,
      where the choice of $H$ is exactly like the choice of $G$ in the proof of
      Theorem \ref{thm:cl1234}, which works since Temkin's log resolutions are
      constructed by blowing up regular centers (see also \cite[Claim
      8.1]{Kol21qfac}).
  \end{itemize}
  \par Finally, to show Theorem \ref{thm:cl12a}, we need to prove the base case
  when $\dim(X) = 0$.
  Let $m$ be an integer such that $mD_1,mD_2,\ldots,mD_k$ are integral.
  Then, $R(X/Z;mD_1,mD_2,\ldots,mD_k)$ is finitely generated over
  $H^0(Z,\cO_Z)$, since it is isomorphic to a polynomial ring with variables
  $x_1,x_2,\ldots,x_k$ corresponding to $mD_1,mD_2,\ldots,mD_k$ in the direct
  sum decomposition in Definition \ref{def:cl12222}.
  Finally, $R(X/Z;D_1,D_2,\ldots,D_k)$ contains $R(X/Z;mD_1,mD_2,\ldots,mD_k)$
  as a Veronese subring of finite index, and hence $R(X/Z;D_1,D_2,\ldots,D_k)$
  is finitely generated by \cite[Proposition 1.2.2]{ADHL15}.
\end{proof}
 \section{Finite generation for klt pairs}\label{sect:cl13s3}

In this section, we prove finite generation of relative adjoint rings for klt
pairs, adapting corresponding results in \cite[\S3]{CL13} to our setting.
We also adapt other results from \cite[\S3]{CL13}, which
will be used in the proofs of other theorems but are of independent interest as
well.
In contrast to previous sections in Part \ref{part:fingen}, where we worked with
log regular pairs, we work with normal schemes and klt pairs.
We will frequently use the continuity of kltness (Lemma
\ref{lem:kltFacts}$(\ref{lem:kltContinuous})$) in this and the following sections.
We sometimes do not explicitly refer to the lemma and just say ``by continuity.''
We note that log resolutions exist for quasi-excellent schemes of equal characteristic
zero by \cite[Theorem
2.3.6 and Lemma 4.2.4]{Tem08}, and thus the lemma is applicable.
\begin{lemma}[{cf.\ \cite[Lemma 1]{CL13}}]\label{lem:cl1331}
  Let $\pi\colon X \to Z$ be a projective morphism of integral Noetherian schemes with $Z$ affine.
  Let $D_1, D_2,\ldots, D_\ell$ be $\QQ$-Cartier divisors on $X$. 
  The ring
  \[
    R = R\bigl(X/Z; D_1,D_2, \ldots, D_\ell\bigr)
  \]
  is finitely generated over $H^0(Z,\cO_Z)$ if and only
if one of its Veronese subrings of finite index is finitely generated over $H^0(Z,\cO_Z)$. In particular,
if $D'_i \sim_{\mathbf Q} e_iD_i$ for some $e_i\in\QQ_{>0}$
and if $R$ is finitely generated over $H^0(Z,\cO_Z)$, 
then the ring $R' = R(X/Z; D'_1,D'_2,\ldots,  D'_\ell)$ is
finitely generated over $H^0(Z,\cO_Z)$.
\end{lemma}
\begin{proof}
If $D'_i \sim_{\mathbf Q} e_iD_i$, then  $R'$ and $R$ have isomorphic Veronese
subrings of finite index, hence the ``in particular'' statement. The principal
statement follows from \cite[Propositions 1.2.2 and 1.2.4]{ADHL15}.
\end{proof}

We also notice the following fact.
\begin{lemma}\label{lem:BigDeltaIsDeltaPlusH}
Let $\pi\colon X \to Z$ be a projective morphism of integral Noetherian
schemes of equal characteristic zero, such that $X$ is normal and such that $Z$ is affine, excellent, and has a dualizing complex $\omega_Z^\bullet$.
  Denote by $K_X$ a canonical divisor on $X$ associated to $\omega_X^\bullet =
  \pi^!\omega_Z^\bullet$.

Let $\Delta$ be an effective $\mathbf Q$-Weil divisor on $X$ such that
$K_X+\Delta$ is $\QQ$-Cartier and $(X,\Delta)$ is klt.
Assume that there exists a rational number $c\in (-\infty,1]$ such that $cK_X+\Delta$ is $\QQ$-Cartier and $\pi$-big. 
Then there exists a rational number $e>0$, an effective $\mathbf Q$-Weil divisor $\Gamma$ on $X$ such that  
$(X,\Gamma)$ is klt, and a $\pi$-ample $\QQ$-Cartier divisor $A$ such that $K_X+\Delta\sim_{\QQ}e(K_X+\Gamma+A)$.
\end{lemma}
\begin{proof}
  By Kodaira's lemma (Corollary \ref{lem:kodairachar}), there exist
a $\pi$-ample $\QQ$-Cartier divisor $H$ 
and an effective $\QQ$-Weil divisor $E$ 
such that $cK_X+\Delta \sim_\QQ H+E$.
For a sufficiently
small $\varepsilon\in \QQ_{>0}$, we have
\begin{align*}
K_X+\Delta&\sim_\QQ (1-c)K_X+(1-\varepsilon)(cK_X+\Delta)+\varepsilon(H+E)\\
&=(1-c\varepsilon )K_X+(1-\varepsilon)\Delta+\varepsilon E+\varepsilon H\\
&=(1-c\varepsilon)\left(K_X+\frac{1-\varepsilon}{1-c\varepsilon}\Delta+\frac{\varepsilon}{1-c\varepsilon} E+\frac{\varepsilon}{1-c\varepsilon} H\right).
\end{align*}
By Lemma \ref{lem:kltFacts}$(\ref{lem:kltContinuousConvex})$ when $c<1$ (with $\Delta'$ there defined to be $\frac{1}{1-c}E$) and Lemma \ref{lem:kltFacts}$(\ref{lem:kltContinuous})$ when $c=1$, for sufficiently
small $\varepsilon\in \QQ_{>0}$, setting
$\Gamma=\frac{1-\varepsilon}{1-c\varepsilon}\Delta+\frac{\varepsilon}{1-c\varepsilon} E$,
the pair $(X,\Gamma)$ is klt. 
We may thus fix such an $\varepsilon$ and set $e=1-c\varepsilon,A=\frac{\varepsilon}{1-c\varepsilon} H$ to conclude.
\end{proof}
\begin{theorem}[{cf.\ \cite[Theorem 2]{CL13}}]\label{thm:cl1332}
Let $\pi\colon X \to Z$ be a projective morphism of integral Noetherian
schemes of equal characteristic zero, such that $X$ is normal and such that $Z$ is
affine and excellent and has a dualizing complex $\omega_Z^\bullet$.
  Denote by $K_X$ a canonical divisor on $X$ associated to $\omega_X^\bullet =
  \pi^!\omega_Z^\bullet$.
  
Let $\Delta_i$ be effective $\mathbf Q$-Weil divisors on $X$ for $i \in
\{1,2,\ldots,\ell\}$ such that $K_X+\Delta_i$ is $\QQ$-Cartier and $(X,\Delta_i)$ is klt for each $i$. 
Let $A_i$ be $\pi$-nef $\QQ$-Cartier divisors for $i \in
\{1,2,\ldots,\ell\}$.
Assume that for each $i$, either $A_i$ is $\pi$-ample, or that there exists a rational number $c_i\in (-\infty,1]$ such that $c_iK_X+\Delta_i$ is $\QQ$-Cartier and $\pi$-big.
Then 
    the relative adjoint ring  
      \[
      R\bigl(X/Z;K_X+\Delta_1+A_1,K_X+\Delta_2+A_2,\ldots,K_X+\Delta_\ell+A_\ell\bigr)
      \]
    is a finitely generated $H^0(Z,\cO_Z)$-algebra.
\end{theorem}
\begin{proof}
If there exists a rational number $c_i\in (-\infty,1]$ such that $c_iK_X+\Delta_i$ is $\QQ$-Cartier and $\pi$-big, 
then by Lemma \ref{lem:BigDeltaIsDeltaPlusH} we may write $K_X+\Delta_i\sim_{\QQ}e_i(K_X+\Theta_i+H_i)$ where $e_i\in\QQ_{>0},$
$H_i$ is $\QQ$-Cartier and $\pi$-ample, 
and $\Theta_i$ is effective with $(X,\Theta_i)$ klt.
Thus $K_X+\Delta_i+A_i\sim_{\QQ}e_i(K_X+\Theta_i+H_i+\frac{1}{e_i}A_i)$ and $H_i+\frac{1}{e_i}A_i$ is $\pi$-ample.
By Lemma \ref{lem:cl1331} we see that we may assume $A_i$ 
$\pi$-ample for all $i$.

Let $f\colon Y\to X$ be a log resolution of $(X,\sum_i\Delta_i)$, which exists
by \cite[Theorem 1.1.6]{Tem18}. 
Since Temkin's log resolutions are constructed by blowing up regular centers, we may assume that there exists an $f$-exceptional effective Cartier divisor $F$ such that $-F$ is $f$-ample (see also \cite[Claim 8.1]{Kol21qfac}).
Take a $\pi$-ample $\QQ$-Cartier divisor $A$ on $X$ such that $A_i-A$ are all $\pi$-ample.
Write 
\[
f^*(K_X+\Delta_i)+E_i\sim_{\QQ} K_Y+\Gamma_i
\]
where $E_i\geq 0$ is $f$-exceptional, all coefficients of $\Gamma_i$ are in $(0,1)$, and $E_i$ and $\Gamma_i$ do not share common components.
This is possible since $\Delta_i\geq 0$ and $(X,\Delta_i)$ is klt. 
By Lemma \ref{lem:cl12226}, it suffices to show
\[
R=R\bigl(Y/Z;K_Y+\Gamma_1+f^*A_1,K_Y+\Gamma_2+f^*A_2,\ldots,K_Y+\Gamma_\ell+f^*A_\ell\bigr)
\]
is finitely generated.

Let $r\in\QQ_{>0}$ be sufficiently small such that $H\coloneqq f^*A-rF$ is $(\pi\circ f)$-ample and such that all coefficients of $\Gamma'_i\coloneqq \Gamma_i+rF$ are less than 1. 
Let $H_i=f^*(A_i-A)$, which is $(\pi\circ f)$-semi-ample by our choice.
Then, we have
\[
R=R\bigl(Y/Z;K_Y+\Gamma'_1+H_1+H,K_Y+\Gamma'_2+H_2+H,\ldots,K_Y+\Gamma'_\ell+H_\ell+H\bigr).
\]
Let $q$ be a positive integer such that every $qH_i$ is integral and $(\pi\circ f)$-generated, and such that all coefficients of $\Gamma'_i$ are less than $1-\frac{1}{q}$.
By Corollary \ref{cor:bertinionopencover}, after replacing $Z$ by the scheme
theoretic image of $\pi$ (thus making it integral) and passing to an affine open
cover (allowed by \cite[Corollaire 6.3.9]{EGAInew} and flat base change), we may assume that there exists $H_i'\in \lvert qH_i \rvert$ such that
$H_i'$ is regular and such that $\sum_i H_i'+\sum_i\Gamma_i$ has simple normal
crossings support.
Since all coefficients of $\Gamma'_i$ are less than $1-\frac{1}{q}$, all coefficients of $\Gamma'_i+\frac{1}{q}H_i'$ are less than 1, regardless of possible shared components.
Therefore, the relative adjoint ring
\[
R\Bigl(Y/Z;K_Y+\Gamma'_1+\frac{1}{q}H_1'+H,K_Y+\Gamma'_2+\frac{1}{q}H_2'+H,\ldots,K_Y+\Gamma'_\ell+\frac{1}{q}H_\ell'+H\Bigr)
\]
is finitely generated by Theorem \ref{thm:cl12a}.
Since $\frac{1}{q}H_i'\sim_{\QQ}H_i$, Lemma \ref{lem:cl1331} gives the finite generation of $R$.
\end{proof}
We therefore obtain Theorem \ref{thm:introfinitegen} for algebraic spaces where
the base is no longer affine.
\begin{theorem}\label{thm:finitegenerationalgspaces}
  Let $\pi\colon X \to Z$ be a proper morphism of integral
  quasi-excellent locally Noetherian algebraic spaces of equal characteristic
  zero over a scheme $S$.
  Suppose that $X$ is normal and that $Z$ admits a dualizing complex
  $\omega_Z^\bullet$.
  Denote by $K_X$ a canonical divisor on $X$ associated to $\omega_X^\bullet =
  \pi^!\omega_Z^\bullet$.
  \par 
Let $\Delta_i$ be effective $\mathbf Q$-Weil divisors on $X$ for $i \in
\{1,2,\ldots,\ell\}$ such that $K_X+\Delta_i$ is $\QQ$-Cartier and $(X,\Delta_i)$ is klt for each $i$. 
Let $A_i$ be $\pi$-nef $\QQ$-invertible sheaves for $i \in
\{1,2,\ldots,\ell\}$.
Assume that for each $i$, either $A_i$ is $\pi$-ample, or that there exists a rational number $c_i\in (-\infty,1]$ such that $c_iK_X+\Delta_i$ is $\QQ$-Cartier and $\pi$-big.
Then, the relative adjoint ring  
\[
  \bigoplus_{(m_1,m_2,\ldots,m_\ell) \in \NN^\ell} \pi_*\cO_X\Biggl(
  \Biggl\lfloor \sum_{i=1}^\ell m_i(K_X+\Delta_i+A_i)\Biggr\rfloor \Biggr)
\]
is an $\cO_Z$-algebra locally of finite type.
\end{theorem}
\begin{proof}
  By definition and flat base change
  \cite[\href{https://stacks.math.columbia.edu/tag/073K}{Tag
  073K}]{stacks-project}, we can pullback along \'etale morphisms from affine schemes
  $\Spec(R) \to Z$ to reduce to the case proved in Theorem \ref{thm:cl1332}.
\end{proof}

For later use, we prove some other consequences of finite generation, adapting
the proofs from \cite{CL13} for complex varieties.
See Definition \ref{def:cl12222} for the definition of the support \(\Supp(R)\)
of a relative adjoint ring that appears in the statement below.
\begin{theorem}[{cf.\ \cite[Theorem 3]{CL13}}]\label{thm:cl1335}
Let $\pi\colon X \to Z$ be a projective morphism of integral Noetherian schemes
such that $Z$ is affine.

Let $D_1,D_2,\ldots, D_\ell$ be $\QQ$-Cartier divisors on $X$. Assume that the ring
\[
  R = R\bigl(X/Z; D_1,D_2,\ldots, D_\ell\bigr)
\]
is finitely generated over $H^0(Z,\cO_Z)$,
and let
\[
  \begin{tikzcd}[row sep=0,column sep=1.475em]
    \mathllap{\mathbf{D}\colon}\RR^\ell \rar & \Div_\RR(X)\\
    (\lambda_1,\lambda_2,\ldots, \lambda_\ell) \rar[mapsto]
    & \displaystyle \sum_{i=1}^\ell \lambda_iD_i
  \end{tikzcd}
\]
be the \textsl{tautological map} from \emph{\cite[p.\ 620]{CL13}}. 
We then have the following:
\begin{enumerate}[label=$(\roman*)$,ref=\roman*]
    \item\label{thm:cl13351} The support $\Supp(R)$
      of $R$ is a rational polyhedral cone.
    \item\label{thm:cl13352} Suppose that $\Supp(R)$ contains a $\pi$-big
      $\RR$-Cartier divisor. If $D \in
      \sum_i {\RR}_{\ge0}D_i$ is $\pi$-pseudo effective, then $D \in\Supp(R)$.

    \item\label{thm:cl13353}
      There is a finite rational polyhedral subdivision $\Supp(R)=\bigsqcup_i \cC_i $ such that $o_v$ is a linear function on $\cC_i$
for every geometric valuation $v$ of $X$. 
Furthermore,
there is a coarsest subdivision with this property in the sense that, if $i$ and
$j$ are distinct, there is at least one geometric valuation $v$ of $X$ such that
(the linear extensions of) $(o_v)_{|\cC_i}$ and $(o_v)_{|\cC_j}$ are different.
    \item\label{thm:cl13354}
    There is a finite index subgroup $L \subseteq \ZZ^\ell$
such that for all $\mathbf{n} \in \NN^\ell\cap L$, if $\mathbf{D}(\mathbf{n}) \in \Supp(R)$, then
\[
o_v\bigl(\mathbf{D}(\mathbf{n})\bigr)=\inf_{E\in\lvert\mathbf{D}(\mathbf{n})\rvert}\big\{\mult_v(E)\bigr\}
\]
for all geometric valuations $v$ of $X$.
\end{enumerate}
\end{theorem}
\begin{proof}
The proof of \cite[Theorem 3]{CL13} carries verbatim here, noting that the
external reference \cite[Proposition 4.7]{ELMNP06} holds for arbitrary Noetherian schemes.
\end{proof}

In the next result, for the same reason as the case of Theorem \ref{thm:cl12b}, we do not need to assume from the outset that $Z$ is of equal characteristic zero.

\begin{corollary}[{cf.\ \cite[Corollary 1]{CL13}}]\label{cor:cl1336}
  Let $\pi\colon X \to Z$ be a projective morphism of integral Noetherian
schemes, such that $X$ is normal and such that $Z$ is
affine and has a dualizing complex $\omega_Z^\bullet$.
  Denote by $K_X$ a canonical divisor on $X$ associated to $\omega_X^\bullet =
  \pi^!\omega_Z^\bullet$.
  Assume that the function field of $X$ has characateristic zero.

Let $\Delta$ be an effective $\QQ$-Weil divisor on $X$ such that $K_X+\Delta$ is $\QQ$-Cartier and $(X,\Delta)$ is klt. Let $A$ be a $\pi$-nef $\QQ$-Cartier divisor on $X$. 

Assume that either $A$ is $\pi$-ample or $\Delta$ is $\pi$-big, and assume that $K_X+\Delta+A$ is $\pi$-pseudoeffective.
Then $\lvert K_X+\Delta+A\rvert_{\QQ}\neq\emptyset$.

\end{corollary}
\begin{proof}
We may assume $\pi$ surjective.
Let $\eta$ be the generic point of $Z$.
We know (Definition \ref{def:fbig}) that $K_X+\Delta+A+H$ is $\pi$-big for all
$\pi$-ample Cartier divisors $H$ on $X$.
Since there exists such an $H$, it follows that $K_{X_\eta}+\Delta_{|X_\eta}+A_{|X_\eta}+H$
is $\pi_{|X_\eta}$-big for all
$\pi_{|X_\eta}$-ample Cartier divisors $H$ on $X_\eta$,
so $K_{X_\eta}+\Delta_{|X_\eta}+A_{|X_\eta}$ is $\pi_{|X_\eta}$-pseudoeffective.
By Corollary \ref{cor:EffIffEffAtGenFiber}, it suffices to show $|K_{X_\eta}+\Delta_{|X_\eta}+A_{|X_\eta}|_{\QQ}\neq\emptyset$,
so we may replace $Z$ by the spectrum of its function field and assume that $Z$ is an excellent scheme of equal characteristic zero.

Let $H$ be a $\pi$-ample Cartier divisor on $X$.
  By Theorem \ref{thm:cl1332}, 
the adjoint ring $R=R(X/Z;K_X+\Delta+A,K_X+\Delta+A+H)$ is finitely generated over $H^0(Z,\cO_Z)$. 
Its support contains the $\pi$-big $\QQ$-Cartier divisor $K_X+\Delta+A+H$.
Thus, Theorem \ref{thm:cl1335}$(\ref{thm:cl13352})$ applies and shows  
$K_X+\Delta+A\in\Supp(R)$, i.e., $\lvert K_X+\Delta+A \rvert_{\RR}\neq\emptyset$.
By Lemma \ref{lem:RatlIsDense} we are done.
\end{proof}
\begin{lemma}[{cf.\ \cite[Lemma 3]{CL13}}]\label{lem:cl1337}
Let $\pi\colon X \to Z$ be a projective morphism of integral Noetherian schemes with $Z$ affine.
Let $D$ be a $\QQ$-Cartier divisor on $X$.
We then have the following:
\begin{enumerate}[label=$(\roman*)$,ref=\roman*]
    \item\label{lem:cl13371} If $D$ is $\pi$-semi-ample, then $o_v(D)=0$ for all geometric valuations $v$ of $X$.
    \item\label{lem:cl13372} Assume that there exist $\QQ$-Cartier divisors
      $D_1,D_2,\ldots, D_\ell$ on $X$ such that the ring
      \[
        R = R\bigl(X/Z; D_1,D_2,\ldots, D_\ell\bigr)
      \]
      is finitely generated over $H^0(Z,\cO_Z)$, and suppose $D\in \Supp(R)$.
    If $o_v(D)=0$ for all geometric valuations $v$ of $X$, then $D$ is $\pi$-semi-ample.
\end{enumerate}

\end{lemma}
\begin{proof}
Assume $D$ is $\pi$-semi-ample. 
Then, $\sL\coloneqq\cO_X(pD)$ is a $\pi$-generated line bundle for some $p>0$. 
Since $Z$ is affine, for each geometric valuation $v$ of $X$, there exists a section $s$ of $\sL$ that avoids the center of $v$. 
Then $\frac{1}{p}\prdiv(s)\in \lvert D\rvert_{\QQ}$ has order zero with respect
to $v$, and thus $o_v(D)=0$. 

Now suppose the assumptions in $(\ref{lem:cl13372})$ hold and suppose $o_v(D)=0$ for all geometric valuations $v$ on $X$. 
By Theorem \ref{thm:cl1335}$(\ref{thm:cl13354})$, there exists a positive
integer $p$ such that $pD$ Cartier and such that
\[
  o_v(pD)=\inf_{E\in \lvert pD\rvert}\bigl\{\mult_v(E)\bigr\}
\]
for all geometric valuations $v$ on $X$.
Since $o_v(pD)=p\cdot o_v(D)=0$, we see that the center of $v$ is not in
$\Bs\lvert pD\rvert$.
Since each closed point of $X$ is the center of a geometric valuation (unless
$\dim(X)=0$, in which case the result is trivially true), we see that $\Bs\lvert
pD\rvert=\emptyset$ and hence $pD$ is $\pi$-generated.
\end{proof}

\begin{corollary}[{cf.\ \cite[Corollary 2]{CL13}}]\label{cor:cl1338}
  Let $\pi\colon X \to Z$ be a projective morphism of integral Noetherian
schemes of equal characteristic zero, such that $X$ is normal and such that $Z$ is excellent and has a dualizing complex $\omega_Z^\bullet$.
  Denote by $K_X$ a canonical divisor on $X$ associated to $\omega_X^\bullet =
  \pi^!\omega_Z^\bullet$.

Let $\Delta$ be an effective $\QQ$-Weil divisor on $X$ such that $K_X+\Delta$ is $\QQ$-Cartier and $(X,\Delta)$ is klt. Let $A$ be a $\pi$-nef $\QQ$-Cartier divisor on $X$.

Assume that either $A$ is $\pi$-ample or $\Delta$ is $\pi$-big.
If $K_X+\Delta+A$ is $\pi$-nef, then it is $\pi$-semi-ample.
\end{corollary}
\begin{proof}
Being $\pi$-semi-ample is local on the base, so we may assume $Z$ affine.

Let $H$ be a $\pi$-ample Cartier divisor on $X$.
  By Theorem \ref{thm:cl1332}, the adjoint ring $R=R(X/Z;K_X+\Delta+A,K_X+\Delta+A+H)$ is finitely generated over $H^0(Z,\cO_Z)$. 
  By Corollary \ref{cor:cl1336}, we have $\lvert
  K_X+\Delta+A\rvert_{\QQ}\neq\emptyset$, and hence $\lvert K_X+\Delta+A+H\rvert_{\QQ}\neq\emptyset$.
Therefore 
\[
  \Supp(R)\supseteq \RR_{\geq 0}\cdot(K_X+\Delta+A)+\RR_{\geq 0}\cdot(K_X+\Delta+A+H).
\]
Since $K_X+\Delta+A$ is $\pi$-nef, we see $K_X+\Delta+A+\varepsilon H$ is $\pi$-ample for all $\varepsilon\in \QQ_{>0}$. 
Therefore, for each geometric valuation $v$ of $X$ and each $\varepsilon\in \QQ_{>0}$, we have $o_v(K_X+\Delta+A+\varepsilon H)=0$. 
Since $o_v$ is continuous on $\Supp(R)$ by Theorem \ref{thm:cl1335}$(\ref{thm:cl13353})$, we see that $o_v(K_X+\Delta+A)=0$ as well. 
By Lemma \ref{lem:cl1337}$(\ref{lem:cl13372})$, we conclude that $K_X+\Delta+A$ is $\pi$-semi-ample.
\end{proof}

\section{Rationality, Cone, and Contraction theorems
revisited}\label{sect:fundrevisit}
We now prove the rationality, cone, and contraction theorems, modeled after
Kawamata's reformulation \cite{Kaw11} of the statements that appear
in \cite{KMM87}.\medskip
\par We start with the following preliminary result.
\begin{lemma}[{cf.\ \cite[Corollary 3]{CL13}}]\label{lem:cl1339}
  Let $\pi\colon X \to Z$ be a projective morphism of integral Noetherian schemes with $Z$ affine.
Let $D_1,D_2,\ldots, D_\ell$ be $\QQ$-Cartier divisors on $X$. 
Let
\[
  \varphi\colon \sum_{i=1}^\ell \RR\cdot D_i\longrightarrow N^1(X/Z)_{\RR}
\]
be the natural projection map. 
Assume that the ring $R = R(X/Z; D_1,D_2,\ldots, D_\ell)$ is finitely generated over $H^0(Z,\cO_Z)$.
Let $\Supp(R)=\bigsqcup_j \cC_j$ be a finite rational polyhedral subdivision such that $o_v$ is a linear function on $\cC_j$
for every geometric valuation $v$ of $X$, as in Theorem
\ref{thm:cl1335}$(\ref{thm:cl13353})$.

Fix an index $k$. 
Assume that $\cC_k\cap \varphi^{-1}(\Amp(X/Z))\neq\emptyset$. 
Then $\cC_k\subseteq \varphi^{-1}(\Nef(X/Z))$. 
If additionally the decomposition $\Supp(R)=\bigsqcup_j \cC_j$ is the coarest
subdivision satisfying the hypotheses above, then $\cC_k=\Supp(R)\cap
\varphi^{-1}(\Nef(X/Z))$, in which case $\cC_k$ is convex.
\end{lemma}  
\begin{proof}
  Note that by Theorem \ref{thm:cl1335}$(\ref{thm:cl13353})$, all asymptotic order functions $o_v$ are identically
zero on $\cC_k$, because they are identically zero on the subset $\cC_k\cap
\varphi^{-1}(\Amp(X/Z))$, which is nonempty and open in the relative topology of
$\cC_k$.
By Lemma \ref{lem:cl1337}$(\ref{lem:cl13372})$, all rational members of $\cC_k$ are $\pi$-semiample, thus $\pi$-nef, and thus all members of $\cC_k$ are $\pi$-nef since rational members are dense in the rational polyhedron $\cC_k$.

Now suppose that the decomposition $\Supp(R)=\bigsqcup_j \cC_j$ is coarsest in
the sense stated above.
Since all asymptotic order functions $o_v$ are identically
zero on every cell $\cC_j$ that touches $\varphi^{-1}(\Amp(X/Z))$, if the
decomposition is coarsest then $\varphi^{-1}(\Amp(X/Z))\subseteq \cC_k$. 
Since $\varphi^{-1}(\Amp(X/Z))\neq \emptyset$, every $\pi$-nef member of
$\Supp(R)$ is a limit of elements of $\varphi^{-1}(\Amp(X/Z))$, and is therefore
contained in the closed subset $\cC_k$.
Since the other inclusion is already established, we conclude that
\[\cC_k=\Supp(R)\cap \varphi^{-1}\bigl(\Nef(X/Z)\bigr).\]
The statement that $\cC_k$ is convex follows from the fact that both $\Supp(R)$ and $\Nef(X/Z)$ are convex.
\end{proof} 

\begin{theorem}[{cf.\ \cite[Theorem 4]{CL13}}]\label{thm:cl1342}
Let $\pi\colon X \to Z$ be a projective morphism of integral Noetherian
schemes of equal characteristic zero, such that $X$ is normal and such that $Z$ is excellent and has a dualizing complex $\omega_Z^\bullet$.

Let $\mathcal{A}=\mathcal{A}(X/Z)$ be the set of classes $\mathbf{u}\in N^1(X/Z)_{\RR}$ that satisfies the following condition. 
There exists an open covering $Z=\cup_a V_a$ such that for each index $a$, there exists
a $\mathbf Q$-Weil divisor $\Delta_a\geq 0$ on $\pi^{-1}(V_a)$ with $K_{\pi^{-1}(V_a)}+\Delta_a$ $\QQ$-Cartier and $(\pi^{-1}(V_a),\Delta_a)$ klt,
a positive real number $c_a$,
and a class $\mathbf{w}_a\in \Amp(\pi^{-1}(V_a)/V_a)$ such that the restriction
of $\mathbf{u}$ to $N^1(\pi^{-1}(V_a)/V_a)$ (Lemma \ref{lem:NumericalClassOpenSubset}) equals to $c_a[K_{\pi^{-1}(V_a)}+\Delta_a]+\mathbf{w}_a$. 

Let $V^\circ=\mathcal{A}\cap\partial\Nef(X/Z)$. 
We then have the following:
\begin{enumerate}[label=$(\roman*)$,ref=\roman*]
\item \label{thm:cl1342Precise} Let $\mathbf{u}\in \mathcal{A}\cap\Nef(X/Z)$. 
There exists a closed convex rational polytope $P$ containing $\mathbf{u}$ in its interior such that
$P\cap \Nef(X/Z)$ is a closed convex rational polytope with nonempty interior.

\item \label{thm:cl1342Formal} For $P$ as in $(\ref{thm:cl1342Precise})$, let $F_1,F_2,\ldots,F_m$ be all the codimension one faces of $P\cap \Nef(X/Z)$ that intersects the interior of $P$. 
Then each $F_i$ span a supporting hyperplane (Definition \ref{def:AmpleConeAndNefCone}) of $\Nef(X/Z)$,
and $\Int(P)\cap \partial \Nef(X/Z)=\Int(P)\cap (F_1\cup F_2\cup\cdots\cup F_m)$.

\item \label{thm:cl13421} Every compact subset of $V^\circ$ is contained in a finite union of supporting hyperplanes.

\item \label{thm:cl13422} Let $D$ be a $\QQ$-Cartier divisor on $X$ such that $[D]\in \mathcal{A}\cap\Nef(X/Z)$. Then $D$ is $\pi$-semi-ample.
\end{enumerate}
\end{theorem}
\begin{remark}
We do not require any compatibility of the divisors $\Delta_a$ and classes $\mathbf{w}_a$ in the definition of $\mathcal{A}$.

Since $\mathcal{A}\cap\Nef(X/Z)\subseteq \Amp(X/Z)\cup V^\circ$, item $(\ref{thm:cl1342Precise})$ (resp. $(\ref{thm:cl13422})$) is only nontrivial for those $\mathbf{u}$ (resp. $[D]$) in $V^\circ$. However, $\mathcal{A}\cap\Nef(X/Z)$ behaves better when we pass to an open cover of $Z$.
\end{remark}
\begin{proof}
Since $\Amp(X/Z)$ is open and convex, it is clear that $\mathcal{A}$ is open and convex and that
\[
\mathcal{A}\cap N^1(X/Z)_{\QQ}=\Set[\big]{a\mathbf{v}+\mathbf{w}\given a\in\QQ_{>0},\ \mathbf{w}\in\Amp(X/Z)\cap N^1(X/Z)_{\QQ}}.
\]
We first prove $(\ref{thm:cl1342Precise})$. By the definition of $\mathcal{A}$, we can find a finite affine cover $V_1,\ldots,V_t$ of $Z$,  
a $\mathbf Q$-Weil divisor $\Delta_a\geq 0$ on $\pi^{-1}(V_a)$ with $K_{\pi^{-1}(V_a)}+\Delta_a$ $\QQ$-Cartier and $(\pi^{-1}(V_a),\Delta_a)$ klt,
a positive real number $c_a$,
and a class $\mathbf{w}_a\in \Amp(\pi^{-1}(V_a)/V_a)$ such that the restriction of $\mathbf{u}$ to $N^1(\pi^{-1}(V_a)/V_a)$ equals to $c_a[K_{\pi^{-1}(V_a)}+\Delta_a]+\mathbf{w}_a$. 

We use the notations $\rho_a:N^1(X/Z)_{\RR}\to N^1(\pi^{-1}(V_a)/V_a)_{\RR}$ for restriction of divisors.
Assume for each $a$ we have a rational polytope $P_a$ in $N^1(\pi^{-1}(V_a)/V_a)_{\RR}$ for $\rho_a(\mathbf{u})$ that fulfills $(\ref{thm:cl1342Precise})$. 
If $P_0$ is any closed convex rational polytope containing $\mathbf{u}$ in its interior, so is $P:=P_0\cap \rho_1^{-1}(P_1)\cap\ldots\cap\rho_t^{-1}(P_t)$, 
and since $\Nef(X/Z)=\cap_a\rho_a^{-1}\Nef(\pi^{-1}(V_a)/V_a)$ (by definition
and Lemma \ref{lem:NumericalClassOpenSubset}), we see that
\[
P\cap \Nef(X/Z)=P_0\cap \rho_1^{-1}\left(P_1\cap\Nef(\pi^{-1}(V_1)/V_1)\right)\cap\ldots\cap\rho_t^{-1}\left(P_t\cap\Nef(\pi^{-1}(V_t)/V_t)\right)
\]
is a closed convex rational polytope.
Since $P$ contains $\mathbf{u}\in\Nef(X/Z)$ in its interior, $\mathrm{int}(P)\cap \Amp(X/Z)\neq\emptyset$, thus $P\cap \Nef(X/Z)$ has nonempty interior.

Thus we may assume $Z$ affine, that there exists a $\mathbf Q$-Weil divisor $\Delta\geq 0$ on $X$ with $K_X+\Delta$ $\QQ$-Cartier and $(X,\Delta)$ klt,
and that $\mathbf{u}$ lies in the subset
\begin{align*}
  \mathcal{A}_0{}\coloneqq{}&\Set[\big]{c[K_X+\Delta]+\mathbf{w}\given c\in\RR_{>0},\ \mathbf{w}\in\Amp(X/Z)}
  \intertext{of $N^1(X/Z)_{\RR}$.
  It is easy to see that $\mathcal{A}_0$ is open and convex, and that}
  \mathcal{A}_0\cap N^1(X/Z)_{\QQ}{}={}&\Set[\big]{c[K_X+\Delta]+\mathbf{w}\given c\in\QQ_{>0},\ \mathbf{w}\in\Amp(X/Z)\cap N^1(X/Z)_{\QQ}}.
\end{align*}

For a sufficiently small closed convex rational polytope $P$ 
whose interior $\Int(P)$ contains $\mathbf{u}$, we have $P\subseteq
\mathcal{A}_0$.
Notice again that $  P\cap\Amp(X/Z)\neq\emptyset$, as $P$ contains $\mathbf{u}\in\Nef(X/Z)$ in its interior.
Each vertex of $P$ has the form $c[K_X+\Delta]+\mathbf{w}=c([K_X+\Delta]+c^{-1}\mathbf{w})$ where $c\in\QQ_{>0}$ and $\mathbf{w}\in\mathrm{Amp}(X/Z)$ rational.
Therefore, we can find $\ell\in\ZZ_{>0}$, $c_i\in \QQ_{>0}$ and $\pi$-ample $\QQ$-Cartier divisors $A_i\ (i \in \{1,2,\ldots,\ell\})$, such that $c_i[K_X+\Delta+A_i]\ (i \in \{1,2,\ldots,\ell\})$ are the vertices of $P$. Write $D_i=K_X+\Delta+A_i$.
\par Consider the adjoint ring 
\[
R=R\bigl(X/Z;D_1,D_2,\ldots,D_\ell\bigr),
\]
which is finitely generated by Theorem \ref{thm:cl1332}.
Every element $\mathbf{x}\in P$ is a convex combination of the classes $c_i[D_i]$,
and thus is a $\RR_{\geq 0}$-combination of the classes $[D_i]$. 
In particular, $\Supp(R)$ contains a $\pi$-ample divisor since $P\cap\Amp(X/Z)\neq\emptyset$.
By Theorem \ref{thm:cl1335}$(\ref{thm:cl13352})$, we see that every element
$\mathrm{x}\in P\cap\Nef(X/Z)$ is the class of an element of $\Supp(R)$. 
In other words, if $\varphi$ is the
canonical map from Lemma \ref{lem:cl1339}, we have $\Nef(X/Z)\cap
P\subseteq \varphi(\Supp(R))$.

Let $\Supp(R)=\bigsqcup_j \cC_j$ be the coarest finite rational polyhedral subdivision such that $o_v$ is a linear function on $\cC_j$
for every geometric valuation $v$ of $X$, as in Theorem
\ref{thm:cl1335}$(\ref{thm:cl13353})$.
Since $\Supp(R)$ contains a $\pi$-ample divisor, 
there exists an index $k$ with $\cC_k\cap \varphi^{-1}(\Amp(X/Z))\neq\emptyset$. 
By Lemma \ref{lem:cl1339}, the set $\cC_k= \varphi^{-1}(\Nef(X/Z))$ is convex, 
and $\varphi(\cC_k)=P\cap \Nef(X/Z)$, as desired.\smallskip

We now show $(\ref{thm:cl1342Formal})$.
Let $W_i$ be the linear span of $F_i$. 
To show $W_i$ is a supporting hyperplane of $\Nef(X/Z)$, it suffices to show $W_i\cap \Amp(X/Z)=\emptyset$.
However, since $F_i$ is convex and contained in $\Nef(X/Z)$, we see that $W_i\cap \Amp(X/Z)\neq\emptyset$ will imply $F_i\cap \Amp(X/Z)\neq\emptyset$, which is impossible since $F_i$ is a face of $P\cap\Nef(X/Z)$, so $F_i\subseteq \partial\Nef(X/Z)$.
\par This argument also tells us that
\[
  \Int(P)\cap \partial \Nef(X/Z)\supseteq\Int(P)\cap (F_1\cup F_2\cup\cdots\cup F_m).
\]
Conversely, if $\mathrm{x}\in \Int(P)\cap \partial \Nef(X/Z)$, it is in the
boundary of $P\cap\Nef(X/Z)$, and thus is contained in some $F_i$.
Therefore we get the identity of sets.\smallskip

Since $(\ref{thm:cl13421})$ follows immediately from $(\ref{thm:cl1342Formal})$, it remains to show $(\ref{thm:cl13422})$.
By the discussion above, upon passing to a (finite) affine open covering of $Z$,
we may assume that there exists a $\mathbf Q$-Weil divisor $\Delta\geq 0$ on $X$ with $K_X+\Delta$ $\QQ$-Cartier and $(X,\Delta)$ klt,
and our divisor $D$ satisfies $[D]=c[K_X+\Delta]+\mathrm{w}$ for some $c\in \QQ_{>0}$ and $\mathrm{w}\in\Amp(X/Z)$.
Therefore the $\QQ$-Cartier divisor $A\coloneqq c^{-1}D-K_X-\Delta$ is $\pi$-ample.
We have that $K_X+\Delta+A=c^{-1}D$ is $\pi$-nef, since $[D]\in \Nef(X/Z)$. 
By Corollary \ref{cor:cl1338}, we see that $K_X+\Delta+A$
is $\pi$-semi-ample and hence so is $D$.
\end{proof}
With uniqueness we can prove the following result.
\begin{lemma}[{cf.\ \cite[p.\ 85, Step 9]{KM98}}]\label{lem:CL13S4ProducesGoodContractions}
Let $\pi\colon X \to Z$ be a projective morphism of integral Noetherian
schemes of equal characteristic zero, such that $X$ is normal and such that $Z$ is excellent and has a dualizing complex $\omega_Z^\bullet$.
Let $W$ be a supporting hyperplane spanned by a face $F$ as in Theorem
\ref{thm:cl1342}\kern1pt$(\ref{thm:cl1342Formal})$.
Then, the extremal ray $R$ dual to $W$ (Definition \ref{rem:DualRay}) has a good contraction with target $Y$ projective over $Z$.
\end{lemma}
\begin{proof}
By Theorem \ref{thm:cl1342}$(\ref{thm:cl1342Precise})$, it is clear that $W$ has a basis consisting of rational members of $\Nef(X/Z)$, and that there exists $[D_1]\in W_{\QQ}\cap\mathcal{A}$. 

By Remark \ref{rem:DualToOneVector}, there exists a rational member $[D_0]$ of $W\cap \Nef(X/Z)$ such that 
\[
  R=\Set[\big]{\gamma\in \NEbar(X/Z) \given (D_0\cdot\gamma)=0}.
\]
Since $(\Nef(X/Z)\cdot\NEbar(X/Z))\geq 0$, 
it is clear that
\[
  R=\Set[\big]{\gamma\in \NEbar(X/Z) \given (D_1+\varepsilon D_0\cdot\gamma)=0} 
\]
for all  $\varepsilon\in\RR_{>0}$, 
and we know that for $\varepsilon$ rational and sufficiently small,
$D_1+\varepsilon D_0$ is $\pi$-semi-ample, by Theorem
\ref{thm:cl1342}$(\ref{thm:cl13422})$.

Fix such an $\varepsilon$ an fix an $m\in\ZZ_{>0}$ such that $D_2\coloneqq 
m(D_1+\varepsilon D_0)$ is integral and $\pi$-generated.
Then $\lvert D_2\rvert$ defines a morphism $X\to \PP_Z(\pi_*\cO_X(D_2))$, and we
denote by $f\colon X\to Y$ the Stein factorization of this morphism.
Then $Y$ is projective over $Z$, $f$ is proper, $f_*\cO_X=\cO_Y$, and $D_2\sim f^*A$ for some Cartier divisor $A$ on $Y$ ample over $Z$.

Since $(D_2\cdot R)=0$, $D_2$ is not $\pi$-ample, so $f$ is not an isomorphism and thus there exists an $f$-contracted curve $C$.
Now $(D_2\cdot C)=0$, so $[C]\in R$ and $R=\RR_{\geq 0}[C]$. 
In particular, for each $\QQ$-Cartier divisor $E$ on $Y$, we have $(f^*E\cdot R)=0$.

Conversely, let $D\in \Div_{\QQ}(X)$ be such that $(D\cdot R)=0$. 
Then $D$ and $D_2$ both induce a linear functional on the real vector space
$U\coloneqq N_1(X/Z)_{\RR}/\RR [C]$. 
The image $\cC$ of $\NEbar(X/Z)$ in $U$ is a compact cone and $D_2$ maps $\cC-\{0\}$ to $\RR_{>0}$.
By local compactness, for sufficiently small $\sigma\in \QQ_{>0}$,
$D_2+\sigma D$ maps $\cC-\{0\}$ to $\RR_{>0}$ as well.
Thus $D_3\coloneqq D_2+\sigma D$ is $\pi$-nef,
\[
  R=\Set[\big]{\gamma\in \NEbar(X/Z) \given (D_3\cdot\gamma)=0},
\]
and $[D_3]\in W$ since $W$ is the subspace dual to $\RR[C]$.
Decreasing $\sigma$, we may assume $[D_3]\in \mathcal{A}$, so $D_3$ is $\pi$-semi-ample by Theorem \ref{thm:cl1342}$(\ref{thm:cl13422})$.
By the same argument as that for $D_1+\varepsilon D_0$, a multiple of $D_3$ is $\pi$-generated and is pulled back from a contraction $f'\colon X\to Y'$ of $R$.
However, by uniqueness of contraction (Theorem \ref{thm:contraction}), this implies that $D_3\sim_{\QQ}f^*E_3$ for some $E_3\in\Div_{\QQ}(Y)$.
Thus $D\sim_{\QQ}f^*(\sigma^{-1}(E_3-A))$, as desired.
\end{proof}

\begingroup
\makeatletter
\renewcommand{\@secnumfont}{\bfseries}
\part{The relative MMP with scaling for schemes and algebraic
spaces}\label{part:relativemmpforschemes}
\makeatother
\endgroup
In this part, we establish the existence of flips and termination with scaling
for schemes and algebraic spaces using Theorem \ref{thm:introfinitegen}.
This completes the proof of Theorem
\ref{thm:introrelativemmp}$(\ref{setup:introalgebraicspaces})$.
We then give some applications of these results by showing that
$\QQ$-factorializations and terminalizations exist, which for simplicity we
prove only for schemes.\bigskip
\section{Birational contractions and \texorpdfstring{$\QQ$}{Q}-factoriality}
We setup the necessary preliminaries for birational contractions.
We characterize the types of contractions that are possible as outputs of
Theorem \ref{thm:contraction}.
\begin{lemma}[cf.\ {\cite[Proposition 2.5]{KM98}}]
\label{lem:Contraction3Types}
Let $\pi\colon X\to Z$ be a projective surjective morphism of integral quasi-excellent
Noetherian algebraic spaces over a scheme $S$
with $X$ normal and $\QQ$-factorial. 
Let $R\subseteq \NEbar(X/Z)$ be an extremal ray. 
Let $f\colon X\to Y$ be a contraction of $R$ over $Z$. 

Then, exactly one of the following holds.
\begin{enumerate}[label=$(\roman*)$,ref=\roman*]

\item \label{lem:ContractionMori}
$\dim X>\dim Y$.

\item \label{lem:ContractionDivisorial}
$f$ is birational, $\mathrm{Ex}(f)\subseteq X$ is a prime divisor.

\item \label{lem:ContractionFlip}
$f$ is birational, $\mathrm{Ex}(f)\subseteq X$ is of codimension $\geq 2$;
i.e., $f$ is small (Definition \ref{def:GoodContrOfR}).

\end{enumerate}
\end{lemma}
\begin{proof}
It suffices to prove that if $f$ is birational and $\mathrm{Ex}(f)\subseteq X$ contains a prime divisor $E$, then $\mathrm{Ex}(f)=E$. 
Fix $n\in\ZZ_{>0}$ such that $[nE]$ is the Weil divisor class associated to an
invertible sheaf.

Assume not. 
Then there exists a point $\zeta\in Y$, not necessarily closed, such that
\[
  \pi^{-1}(\zeta)=\mathrm{Ex}(f)\cap \pi^{-1}(\zeta)\supsetneq E\cap \pi^{-1}(\zeta).
\]
By Zariski's Main Theorem, each irreducible component of $\pi^{-1}(\zeta)$ is positive-dimensional, and at least one of such is not contained in $E\cap\pi^{-1}(\zeta).$
Therefore there exists a one-dimensional integral closed subspace $C$ of $\pi^{-1}(\zeta)$ that is not contained in $E\cap \pi^{-1}(\zeta)$,
so $(nE\cdot C)\geq 0$, where
we use Remark \ref{rem:NefAgainstNonClosedContracted} to make sense of
this intersection number.

Since $\pi$ is projective, the class $[C]$ defined using
Remark
\ref{rem:NefAgainstNonClosedContracted}
is nonzero, and it belongs to $\NEbar(X/Y)$ by Lemma \ref{lem:NumericalClassOpenSubset}.
As noted in Definition \ref{def:GoodContrOfR}, we have $R=\RR_{\geq 0}[C]$, thus $nE$ is $f$-nef.
Applying Lemma \ref{lem:Negativity} to the divisor $B=-nE$, we get a contradiction.
\end{proof}

\begin{lemma}[cf.\ \citeleft\citen{KMM87}\citemid Lemma 5-1-5 and Proposition
  5-1-6\citepunct \citen{KM98}\citemid Corollary 3.18\citeright]\label{lem:ContractionQfac}
  In cases $(\ref{lem:ContractionMori})$ and $(\ref{lem:ContractionDivisorial})$ in Lemma \ref{lem:Contraction3Types}, if $f$ is a good contraction 
  then $Y$ is $\QQ$-factorial.
\end{lemma}
\begin{proof}
In case $(\ref{lem:ContractionDivisorial})$, the proof is identical to the proof
of \cite[Corollary 3.18]{KM98}.\smallskip

The proof of case $(\ref{lem:ContractionMori})$ is also very similar to
\cite[Corollary 3.18]{KM98}, except for the fact that we must work with points
\(y \in Y\) that are not necessarily closed.
We provide a complete proof for case $(\ref{lem:ContractionMori})$ below.
Let $Y^\circ$ be the regular locus of $Y,$
which is open since $Y$ is quasi-excellent.
The complement of \(Y^\circ\) in \(Y\) is of codimension at least 2 since $Y$ is normal.
Let $B$ be a prime divisor on $Y$.
Then, $B\cap Y^\circ$ is a prime divisor on $Y^\circ$ and is Cartier since $Y^\circ$ is regular.
Therefore, $f^{-1}(B\cap Y^\circ)$ is an effective Cartier divisor on
$f^{-1}(Y^\circ)$.
We let $D$ be its closure in $X$.
The class of $D$ is the class associated to a $\QQ$-invertible sheaf $\tilde{D}$
since $X$ is $\QQ$-factorial.
Take $y\in Y^\circ$ not in $B$ (which is not necessarily closed in \(Y\)),
and consider an integral curve $C\subseteq f^{-1}(y)$.
As in the proof of Lemma \ref{lem:Contraction3Types},
$C$ defines a class $[C]\in N_1(X/Y)_{\RR}$ (by Remark
\ref{rem:NefAgainstNonClosedContracted})
and $R=\RR_{\geq 0}[C]$.
Since $D\cap f^{-1}(Y^\circ)=f^{-1}(B\cap Y^\circ)$, $(\tilde{D}\cdot C)=0$. 
Thus $(\tilde{D}\cdot R)=0$ and $\tilde{D}\sim_{\QQ}f^*E$ for some $E\in\Pic_{\QQ}(Y)$ as $f$ is a good contraction.

Take $m\in\ZZ_{>0}$ such that $mE$ is integral and $m\tilde{D}\sim f^*(mE)$.
Then there exists a global section $s$ of $\cO_X(f^*(mE))$ with $\mathrm{div}(s)=mD$.
Since $f_*\cO_X=\cO_Y$, we have a well-defined global section $f_*(s)$ of
$\cO_Y(mE)$ with $f^{-1}\mathrm{div}(f_*(s))=mD$.
Thus by construction $\mathrm{div}(f_*(s))\cap Y^\circ=mB\cap Y^\circ$ and $\mathrm{div}(f_*(s))=mB$ since the complement of $Y^\circ$ is of codimension at least 2.
Thus $mB\sim mE$ is Cartier and $B$ is $\QQ$-Cartier.
\end{proof}

\section{Existence of flips}
In this section, we show that flips exist.
To do so, we first define flips.
\begin{definition}[cf.\ {\cite[p.\ 335]{KMM87}}]\label{def:FLIPS}
  Let $\pi\colon X \to Z$ be a projective surjective morphism of integral
  quasi-excellent Noetherian algebraic spaces over
  a scheme $S$.
  Suppose that $X$ is normal and that $Z$ admits a dualizing complex
  $\omega_Z^\bullet$.
  Denote by $K_X$ a canonical divisor on $X$ associated to $\omega_X^\bullet =
  \pi^!\omega_Z^\bullet$.

Let $\Delta\geq 0$ be a $\QQ$-divisor on $X$ such that $K_X+\Delta$ is $\QQ$-Cartier and that $(X,\Delta)$ is klt.
Let $f\colon X\to Y$ be a small birational contraction over $Z$ (Definition \ref{def:GoodContrOfR}) such that $-(K_X+\Delta)$ is $f$-ample. 
A \textsl{flip} of $f$ is a proper birational morphism $f^+\colon X^+\to Y$ with the following properties.

\begin{enumerate}[label=$(\roman*)$,ref=\roman*]
    \item\label{def:FLIPSnormal} $X^+$ is normal (and integral).
    \item\label{def:FLIPSsmall}
    The morphism $f^+$ is a small contraction.
    \item\label{def:FLIPSample} $K_{X^+}+\Delta^+$ is $\QQ$-Cartier and $f^+$-ample where $\Delta^+$ is the strict transform of $\Delta$.
  \end{enumerate}
  
Note that since $\mathrm{Ex}(f^+)\subseteq X^+$ is of codimension $\geq 2$, the
strict transform operation $D\mapsto D^+$ induces an isomorphism $\WDiv_{\kk}(X)\cong \WDiv_{\kk}(X^+)\ (\kk=\ZZ,\QQ$ or $\RR$) that preserves principal divisors and maps $K_X$ to a canonical divisor of $X^+$.
Moreover, $f_*\cO_X(D)=f^+_*\cO_{X^+}(D^+)$ for all $D\in \Div(X)$.
See for example \cite[Lemma 4(3)]{CL13}.

A birational map $h\colon X\dashrightarrow X'$ of algebraic spaces over $Z$ is called a
\textsl{flip} of the pair $(X,\Delta)$ if 
$h$ is isomorphic to the birational map $(f^+)^{-1}\circ f:X\dashrightarrow X^+$ for some $f,X^+$ as above.
\end{definition}
We can now show flips exist.
The case for complex quasi-projective varieties is \cite[Corollary
1.4.1]{BCHM10} (cf.\ \cite[Theorem 5]{CL13}).
When $X$ is of finite type over an algebraically
closed field of characteristic zero, the three-dimensional case is proved in
\cite[Log Flip Theorem 6.13]{Sho96} and the general case follows from
\cite[Theorem 2.6]{VP}.
\begin{theorem}\label{thm:FLIP}
  Let $\pi\colon X \to Z$ be a projective surjective morphism of integral
  quasi-excellent Noetherian algebraic spaces over
  a scheme $S$.
  Suppose that $X$ is normal and that $Z$ admits a dualizing complex
  $\omega_Z^\bullet$.
  Denote by $K_X$ a canonical divisor on $X$ associated to $\omega_X^\bullet =
  \pi^!\omega_Z^\bullet$.

Let $\Delta\geq 0$ be a $\QQ$-Weil divisor on $X$ such that $(X,\Delta)$ is klt.
Let $f\colon X\to Y$ be a small contraction over $Z$ such that $-(K_X+\Delta)$ is $f$-ample.
Then the following hold.

\begin{enumerate}[label=$(\roman*)$,ref=\roman*]
    \item\label{thm:FLIPunique} A flip of $f$ is unique up to unique isomorphism.
    \item\label{thm:FLIPisProj} If $Z$ is of equal characteristic zero, then the quasi-coherent $\cO_Y$-algebra 
\[
  \mathcal{A}\coloneqq\bigoplus_{m=0}^\infty f_*\cO_X\bigl(\lfloor
  m(K_X+\Delta)\rfloor\bigr)
\]
is of finite type and $\PProj(\mathcal{A})$ is a flip of $f$.
  \end{enumerate}
\end{theorem}
\begin{proof}
It is clear that if $f^+\colon X^+\to Y$ is a flip, then
$X^+\cong\PProj(\mathcal{A})$ (see the proof of \cite[Lemma 6.2]{KM98}).
Thus it suffices to show $(\ref{thm:FLIPisProj})$.
Since $f$ is birational, $\Delta$ is $f$-big, thus Theorem
\ref{thm:finitegenerationalgspaces} applies to show $\mathcal{A}$ is of
finite type.
Thus $X^+\coloneqq\PProj(\mathcal{A})$ is locally projective, in particular proper, over $Y$, and
$X^+$ is normal and birational to $Y$ since $X$ is.
The proof of the properties $(\ref{def:FLIPSsmall})$ and
$(\ref{def:FLIPSample})$ as
in Definition \ref{def:FLIPS} is the same as the proof of \cite[Proposition 5-1-11(2)]{KMM87}.
\end{proof}

\begin{lemma}\label{lem:FLIPQfacAndN1same}
Notations and assumptions in Theorem \ref{thm:FLIP}. 
Assume further that $f$ is a good contraction of some extremal ray $R\subseteq \NEbar(X/Z)$. 

Let $X^+$ be a flip of $f$.
Then the following hold.

\begin{enumerate}[label=$(\roman*)$,ref=\roman*]
    \item\label{lem:FLIPQfac} 
    $X^+$ is $\QQ$-factorial if $X$ is.
    
    \item\label{lem:FLIPNefPullback} 
    If $D\in \Pic_{\QQ}(X)$ is $\pi$-nef and satisfies $(D\cdot R)=0$, then
    $D\sim_{\QQ}f^*E$ for some $E\in\Pic_{\QQ}(Y)$ nef over $Z$.
    
    \item\label{lem:FLIPN1Same}
    $D\mapsto D^+$ induces an isomorphism $N^1(X/Z)_{\RR}\cong N^1(X^+/Z)_{\RR}$.
  \end{enumerate}
\end{lemma}
\begin{proof}
Let $D$ be a $\QQ$-invertible sheaf on $X$.
Since $R$ is a ray, there exists $a\in \QQ$ such that
\begin{align}
  \bigl(D+a(K_X&+\Delta)\cdot R\bigr)=0.\nonumber
\intertext{Since $f$ is a good contraction, $D+a(K_X+\Delta)\sim_{\QQ}f^*E$ for some $E\in
\Pic_{\QQ}(Y)$.
Thus}
  D^++a\bigl(K_{X^+}&+\Delta^+\bigr)\sim_{\QQ}f^{+*}E.\label{eq:FLIPQfacAndN1same}
\end{align}
Since $K_{X^+}+\Delta^+$ is $\QQ$-Cartier (Definition
\ref{def:FLIPS}$(\ref{def:FLIPSample})$), we see that $D^+$ is $\QQ$-Cartier. 
Since every $\QQ$-Weil divisor on $X^+$ is of the form $D^+$ for some $\QQ$-Weil divisor $D$ on $X$, we see that $X^+$ is $\QQ$-factorial if $X$ is.

If $D\in \Pic_{\QQ}(X)$ is $\pi$-nef and satisfies $(D\cdot R)=0$, then $a=0$, and $D\sim_{\QQ}f^*E$ for some $E\in\Pic_{\QQ}(Y)$. 
For each $(Y\to Z)$-contracted curve $C$, there exists a $\pi$-contracted curve $C'$ such that $C'$ maps finite surjectively to $C$.
We know $(D\cdot C')\geq 0$, thus $(E\cdot C)\geq 0$ and $E$ is nef over $Z$ since $C$ was arbitrary.

Now $[D^+]=[f^{+*}E]$ is nef over $Z$.
If $[D]=0\in N^1(X/Z)_{\RR}$, then we get $[(-D)^+]=-[D^+]$ nef over $Z$ as well, so $[D^+]=0$.
This shows that $D\mapsto D^+$ induces a linear map $N^1(X/Z)_{\RR}\to N^1(X^+/Z)_{\RR}$, which is automatically surjective.
If $[D^+]=0\in N^1(X^+/Z)_{\RR}$, from the equation \eqref{eq:FLIPQfacAndN1same} and the fact $K_{X^+}+\Delta^+$ ample over $Y$ we see that $a=0$, 
so by the same argument we get $[E]=0\in N^1(Y/Z)_{\RR}$ and thus $[D]=[f^*E]=0\in N^1(X/Z)_{\RR}$.
Thus the linear map $N^1(X/Z)_{\RR}\to N^1(X^+/Z)_{\RR}$ is an isomorphism.
\end{proof}

\begin{lemma}[cf. {\cite[Lemma 3.38]{KM98}}]\label{lem:DiscrepancyNonIncreasing}
  Let $Y$ be a quasi-excellent integral Noetherian algebraic space over a scheme
  $S$.
  Let $X$ and $X'$ be algebraic spaces projective over $Y$ that are integral, normal, and
  birational to $Y$.
  Suppose that $Y$ admits a dualizing complex
  $\omega_Y^\bullet$.
  Denote by $K_X$ and $K_{X'}$ canonical divisors on $X$ and $X'$ associated to
  the exceptional pullbacks of $\omega_Y^\bullet$.

Let $\Delta\geq 0$ be a $\mathbf Q$-Weil divisor on $X$ such that $K_X+\Delta$ is $\QQ$-Cartier. 
Let $\Delta'\geq 0$ be the birational transform of $\Delta$ on $X'$ 
and assume that $K_{X'}+\Delta'$ is $\QQ$-Cartier. 
Assume that the following hold.
\begin{enumerate}[label=$(\roman*)$,ref=\roman*]
    
      \item\label{lem:ConditionKNegative} $-(K_X+\Delta)$ is nef over $Y$.
      
      \item\label{lem:ConditionKPositive} $K_{X'}+\Delta'$ is nef over $Y$.
  \end{enumerate}
  
Then for all divisors $E$ over $Y$, $a(E,X,\Delta)\leq a(E,X',\Delta')$, and if at least one of $K_X+\Delta$ and $K_{X'}+\Delta'$ is not numerically trivial over $Y$, then strict inequality holds for at least one such $E$.
\end{lemma}
\begin{proof}
Consider a commutative diagram
\[
\begin{tikzcd}
  W\rar{g'}\dar[swap]{g} &X'\dar\\
X\rar &Y
\end{tikzcd}
\]
where $g,g'$ are birational and $W$ is integral and normal.
We write
\begin{align*}
  K_W\sim_{\QQ} g^*(K_X+\Delta)&+\sum_F a(F,X,\Delta)F
  \intertext{and}
  K_W\sim_{\QQ} {g'}^*(K_{X'}+\Delta')&+\sum_F a(F,X',\Delta')F
\end{align*}
as usual, so
\[    
    {g'}^*(K_{X'}+\Delta')-g^*(K_X+\Delta)\sim_{\QQ} \sum_F \left(
      a(F,X,\Delta)-a(F,X',\Delta')
    \right) F.
\]
By our assumptions $(\ref{lem:ConditionKNegative})$ and $(\ref{lem:ConditionKPositive})$, ${g'}^*(K_{X'}+\Delta')-g^*(K_X+\Delta)$ is nef. 
On the other hand, since $\Delta'$ is the birational transform of $\Delta$,
$B\coloneqq-\sum_F\left( a(F,X,\Delta)-a(F,X',\Delta')
    \right) F$ is exceptional over $Y$.
Therefore Lemma \ref{lem:Negativity} applies and shows that $B$ is effective,
i.e., $a(F,X,\Delta)\leq a(F,X',\Delta')$ for all $F$ in the sum.

Now for each divisor $E$ over $Y$, we may always find a diagram as above such
that $E$ occurs as a prime divisor on $W$, so $a(E,X,\Delta)\leq a(E,X',\Delta')$.

If at least one of $K_X+\Delta$ and $K_{X'}+\Delta'$ is not numerically trivial over $Y$, then $B$ is not numerically trivial over $Y$ so
strict inequality must hold for some $F$. 
\end{proof}
For the following two statements, we assume the existence
of a flip $(X^+,\Delta^+)$ of \(f\) to make the statement of these corollaries
characteristic-free.
Of course, if \(Z\) is of equal characteristic zero, then flips exist by Theorem
\ref{thm:FLIP}.

\begin{corollary}\label{cor:MMPpreservesKLT}
  Let $\pi\colon X \to Z$ be a projective surjective morphism of integral
  quasi-excellent Noetherian algebraic spaces over
  a scheme $S$.
  Suppose that $X$ is normal and that $Z$ admits a dualizing complex
  $\omega_Z^\bullet$.
  Denote by $K_X$ a canonical divisor on $X$ associated to $\omega_X^\bullet =
  \pi^!\omega_Z^\bullet$.

Let $\Delta\geq 0$ be a $\QQ$-Weil divisor on $X$ such that $K_X+\Delta$ is
$\QQ$-Cartier and that $(X,\Delta)$ is klt (resp.\ terminal).
Let $f\colon X\to Y$ be a birational contraction over $Z$ such that $-(K_X+\Delta)$ is $f$-ample.
Then the followings hold.
\begin{enumerate}[label=$(\roman*)$,ref=\roman*]

    \item\label{lem:ContrPreservesKLT} 
      If $K_Y+f_*\Delta$ is $\QQ$-Cartier (resp.\ $K_Y+f_*\Delta$ is
      $\QQ$-Cartier and $\Exc(f) \not\subseteq \Supp(\Delta)$), then
      $(Y,f_*\Delta)$ is klt (resp.\ terminal).
    
      \item\label{lem:FlipPreservesKLT} 
      Assume that $f$ is small and assume that a flip $(X^+,\Delta^+)$ of $f$ exists.
      Then $(X^+,\Delta^+)$ is klt (resp.\ terminal).
  \end{enumerate}

\end{corollary}
\begin{proof}
Immediate from definitions and Lemma \ref{lem:DiscrepancyNonIncreasing}.
\end{proof}

\begin{corollary}\label{cor:MMPpreservesKeffective}
  Let $\pi\colon X \to Z$ be a projective surjective morphism of integral
  quasi-excellent Noetherian algebraic spaces over
  a scheme $S$.
  Suppose that $X$ is normal and that $Z$ admits a dualizing complex
  $\omega_Z^\bullet$.
  Denote by $K_X$ a canonical divisor on $X$ associated to $\omega_X^\bullet =
  \pi^!\omega_Z^\bullet$.

Let $\Delta\geq 0$ be a $\QQ$-Weil divisor on $X$ such that $K_X+\Delta$ is
$\QQ$-Cartier and that $(X,\Delta)$ is klt (resp.\ terminal).
Let $f\colon X\to Y$ be a birational contraction over $Z$ such that $-(K_X+\Delta)$ is $f$-ample.
Then, the following hold.
\begin{enumerate}[label=$(\roman*)$,ref=\roman*]

    \item\label{lem:ContrPreservesKeffective} 
    Assume that $K_Y+f_*\Delta$ is $\QQ$-Cartier.
    Then $K_Y+f_*\Delta$ is pseudoeffective over $Z$ if and only if $K_X+\Delta$ is.
    
      \item\label{lem:FlipPreservesKeffective} 
      Assume that $f$ is small and assume that a flip $(X^+,\Delta^+)$ of $f$ exists.
      Then $K_{X^+}+\Delta^+$ is pseudoeffective over $Z$ if and if $K_X+\Delta$ is.
  \end{enumerate}

\end{corollary}
\begin{proof}
In either case we may replace $Z$ by the Stein factorization of $Y\to Z$. 
Furthermore, it is clear that taking generic fiber of $Z$ preserves assumptions
and conclusions (see Definitions \ref{def:fbig}, \ref{def:GoodContrOfR},
\ref{def:FLIPS}), so we may assume that $Z$ is the spectrum of a
field, in which case \(X\) is a projective variety.\smallskip

We now prove $(\ref{lem:ContrPreservesKeffective})$.
If \(K_X+\Delta\) is pseudoeffective, then \(K_Y+f_*\Delta\) is pseudoeffective
by Lemma
\ref{lem:ContrPreservesPsEff}\((\ref{lem:ContrPreservesPsEffPushForward})\).
It remains to show that if \(K_Y+f_*\Delta\) is pseudoeffective, then
\(K_X+\Delta\) is pseudoeffective.
This follows since, by 
Lemma \ref{lem:DiscrepancyNonIncreasing} and its proof, we have
\[
K_X+\Delta\sim_{\QQ}f^*(K_Y+f_*\Delta)+E
\]
where $E$ is an effective exceptional $\QQ$-Cartier divisor.
\par Item $(\ref{lem:FlipPreservesKeffective})$ follows immediately from Lemma \ref{lem:ContrPreservesPsEff}. 
\end{proof}
\begin{corollary}\label{cor:MMPnotCycleBack}
  Let $\pi\colon X \to Z$ be a projective surjective morphism of integral
  quasi-excellent Noetherian algebraic spaces over
  a scheme $S$.
  Suppose that $X$ is normal and that $Z$ admits a dualizing complex
  $\omega_Z^\bullet$.
  Denote by $K_X$ a canonical divisor on $X$ associated to $\omega_X^\bullet =
  \pi^!\omega_Z^\bullet$.

Let $\Delta\geq 0$ be a $\QQ$-Weil divisor on $X$ such that $K_X+\Delta$ is $\QQ$-Cartier.
Let $m\in \ZZ_{>0}$ and
\[
  X \coloneqq X_1 \overset{f_1}{\dashrightarrow} X_2
  \overset{f_2}{\dashrightarrow} \cdots \overset{f_{m-1}}{\dashrightarrow} X_m
\]
be a sequence of birational maps over $Z$ such that each $X_i$ is normal.
Let $\Delta_i$ be the birational transform of $\Delta$ on $X_i$ and assume that $K_{X_i}+\Delta_i$ is $\QQ$-Cartier for all $i\leq m$.

Assume that for each $i<m$, either $f_i$ is a morphism and a contraction with $-(K_{X_i}+\Delta_i)$ $f_i$-ample, 
or that $f_i$ is a flip of the pair $(X_i,\Delta_i)$; 
and assume that there exists an index $i_0<m$ such that $f_{i_0}$ is not an isomorphism.
Then the composition $X\dashrightarrow X_m$ is not an isomorphism.
\end{corollary}
\begin{proof}
By Lemma \ref{lem:DiscrepancyNonIncreasing}, there exists a divisor $E$ over $X_{i_0}$ such that
\begin{align*}
  a(E,X_{i_0},\Delta_{i_0})&<a(E,X_{i_0+1},\Delta_{i_0+1}).
  \intertext{This divisor defines a divisor over each $X_i$ and we have}
  a(E,X,\Delta)&\leq a(E,X_{i_0},\Delta_{i_0})\\
  a(E,X_{i_0+1},\Delta_{i_0+1})&\leq a(E,X_m,\Delta_m)
\end{align*}
by the same lemma.
Thus $a(E,X,\Delta)<a(E,X_m,\Delta_m)$ and $X\dashrightarrow X_m$ is not an isomorphism.
\end{proof}

We check that contractions and flips behave well when we pass to an open subset of the base $Z$.
The assumption on Picard groups below is satisfied when, for example, $X$ is integral, normal, and $\QQ$-factorial.

\begin{lemma}\label{lem:ContSmallerOpen}
  Let $\pi\colon X \to Z$ be a projective surjective morphism of integral
  quasi-excellent Noetherian algebraic spaces over
  a scheme $S$.
  Suppose that $X$ is normal and that $Z$ admits a dualizing complex
  $\omega_Z^\bullet$.
  Denote by $K_X$ a canonical divisor on $X$ associated to $\omega_X^\bullet =
  \pi^!\omega_Z^\bullet$.

Let $R\subseteq \NEbar(X/Z)$ be an extremal ray and let $f\colon X\to Y$ be a contraction of $R$.
Let $W$ be an open subspace of $Z$ and denote by $\square_W$ the base change of a
$Z$-space or a $Z$-morphism $\square$ to $W$.

Assume that $\Pic(X)_{\QQ}\to \Pic(X_W)_{\QQ}$ is surjective and that $f_W\colon
X_W\to Y_W$ is not an isomorphism.
Then $f_W$ is a contraction of an extremal ray $R^W\subseteq \NEbar(X_W/W)$.
Moreover, if $f$ is a good contraction of $R$, $f_W$ is also a good contraction of $R^W$.
\end{lemma}
\begin{proof}
Since $f_W\colon X_W\to Y_W$ is not an isomorphism, there exists a closed point $z\in W$ such that $\mathrm{Ex}(f)\subset X$ intersects the fiber $f^{-1}(z)$.
In particular, $R^W\coloneqq\NEbar(X_W/Y_W)$ is nontrivial.
Since $\Pic(X)_{\QQ}\to \Pic(X_W)_{\QQ}$ is surjective, $N^1(X/Z)_{\RR}\to N^1(X_W/W)_{\RR}$ is also surjective, thus
the canonical map $N_1(X_W/W)_{\RR}\to N_1(X/Z)_{\RR}$ is injective.
By Lemma \ref{lem:NumericalClassOpenSubset}, $R^W$ is sent into $\NEbar(X/Y)$, 
which equals to $R$ as noticed in Definition \ref{def:GoodContrOfR}.
Thus $R^W$ is a ray, and it is clear that $f_W$ is the contraction of $R^W$.

Now assume that $f$ is a good contraction and let $\sL^W$ be an element in $\Pic(X_W)_{\QQ}$.
If we can write $\sL^W=f^*(\sK^W)\in\Pic(X_W)_{\QQ}$ for some $\sK^W\in\Pic(Y_W)_{\QQ}$, then $(\sL^W\cdot R^W)=0$ by the definition of $R^W$.
Thus it suffices to show the converse.

Since $\Pic(X)_{\QQ}\to \Pic(X_W)_{\QQ}$ is surjective,
there exists $\sL\in\Pic(X)_{\QQ}$ such that $\sL_{|X_W}=\sL^W$.
Now, if $(\sL^W\cdot R^W)=0$, then $(\sL\cdot R)=0$ since $R$ is a ray, and thus there exists $\sK\in \Pic(Y)_{\QQ}$ such that $\sL=f^*\sK\in\Pic(X)_{\QQ}$.
Thus $\sL^W=f^*(\sK_{|Y_W})\in\Pic(X_W)_{\QQ}$, as desired.
\end{proof}

We now prove two lemmas that are important to the proof of termination. 
The first one is about the asymptotic order of vanishing (Definition \ref{def:AsymptoticOrder}).

\begin{lemma}
\label{lem:cl1352}
Let $\pi_i\colon X_i \to Z\ (i=1,2)$ be two proper morphisms of Noetherian
schemes, such that $X_1$ and $X_2$ are integral and normal and $Z$ is affine.
Let $g\colon X_1\dashrightarrow X_2$ be a birational map over $Z$ that is an isomorphism in codimension 1.

Let $v$ be a geometric valuation on $X_1$ (Definition \ref{def:GeomVal}).
Then $v$ induces canonically a geometric valuation $g_*v$ on $X_2$, and for each $\RR$-Weil divisor $D$ on $X_1$ with $|D|_\RR\neq\emptyset$, we have $|g_*D|_\RR\neq\emptyset$ and $o_v(D)=o_{g_*v}(g_*D)$.
\end{lemma}
\begin{proof}
By definition, $v$ is given by a prime divisor $\Gamma$ in a scheme $Y$ birational and proper over $X_1$.
By taking a resolution of the composition $Y\to X_1\dashrightarrow X_2$ we find $g_*v$.

It is clear that for each effective $\RR$-Weil divisor $E$ on $X_1$, we have $v(E)=g_*v(g_*E)$.
If $D$ is an $\RR$-Weil divisor on $X_1$ with $|D|_\RR\neq\emptyset$,
$g_*$ induces a bijection $|D|_\RR\to |g_*D|_\RR$, 
thus by definition $o_v(D)=o_{g_*v}(g_*D)$.
\end{proof}

The second is about a sufficient condition for a birational map to be a morphism.

\begin{lemma}[cf.\ {\citeleft\citen{KM98}\citemid Lemma 6.39\citepunct
  \citen{CL13}\citemid Lemma 6\citeright}]\label{lem:cl1353}
  Let $\pi_i\colon X_i \to Z$ for $i \in \{1,2\}$ be two proper morphisms of excellent Noetherian
schemes, such that $X_1$ and $X_2$ are integral and normal. 
Let $g\colon X_1\dashrightarrow X_2$ be a birational map over $Z$ that is an isomorphism in codimension 1.

Assume that there exists a $\pi_1$-ample effective $\QQ$-Cartier divisor $A$ on
$X_1$ such that the birational transform $B\coloneqq g_*A$ is $\QQ$-Cartier and $\pi_2$-nef.
Then $g^{-1}$ is a morphism.
\end{lemma}
\begin{proof}
  By taking the normalization of the fiber product $X_1 \times_Z X_2$, there
  exists an integral normal scheme $W$ with proper birational morphisms
  $h_i\colon W\to X_i$ for $i \in \{1,2\}$ such that $g=h_2\circ h_1^{-1}$ as rational maps. 
Since 
$B=g_*A$ 
and since $g$ is an isomorphism in codimension 1,
the $h_1$-exceptional divisors are exactly the $h_2$-exceptional divisors and we can write
$h_2^*B+E=h_1^*A+F$
where $E,F$ are 
$h_1$-exceptional divisors. 
Since $B$ is $\pi_2$-nef, $h_2^*B$ is $h_1$-nef and thus so is
$F-E=h_2^*B-h_1^*A$.
By the Negativity Lemma \ref{lem:Negativity} we see $E-F$ is effective, and by the same reason $F-E$ is effective.
Thus $h_2^*B=h_1^*A$.

Since $A$ is $\pi_1$-ample, we see that every $h_2$-contracted curve on $W$ must be $h_1$-contracted.
By Lemma \ref{lem:ContractsNonClosed}, we see that every fiber of $h_2$ is mapped to a point under $h_1$, 
so there exists a continuous map of topological spaces $u\colon X_2\to X_1$ compatible with $h_1$ and $h_2$.
Since $\cO_{X_i}=h_{i*}\cO_W$, this continuous map upgrades to a morphism of schemes and is the inverse of $g$ as a rational map.
\end{proof}

\section{Existence and termination of the relative MMP with scaling}
\label{sect:mmpexistsandterminates}

In this section, following \cite{CL13}, we prove the termination of MMP under suitable assumptions.

\begin{definition}[cf.\ {\cite[Definition 6.1]{CL13}}]\label{def:cl1361}
  Let $\pi\colon X \to Z$ be a projective surjective morphism of integral
  quasi-excellent Noetherian algebraic spaces of equal characteristic zero over
  a scheme $S$.
  Suppose that $X$ is normal and that $Z$ admits a dualizing complex
  $\omega_Z^\bullet$.
  Denote by $K_X$ a canonical divisor on $X$ associated to $\omega_X^\bullet =
  \pi^!\omega_Z^\bullet$.
  
Let $\Delta\geq 0$ be a $\mathbf Q$-Weil divisor on $X$ such that $K_X+\Delta$ is $\QQ$-Cartier and such that $(X,\Delta)$ is klt. 
For a $\QQ$-invertible sheaf $D$ on $X$, the \textsl{$\pi$-nef threshold of the
pair $(X,\Delta)$ with respect to $D$} is
\[
\lambda(X/Z,\Delta,D)\coloneqq\inf\Set{t\in\RR_{\geq 0}\given K_X+\Delta+tD\text{\ is\ }\pi\text{-nef}}\in \RR_{\geq 0}\cup \{\infty\}.
\]
\end{definition}

We now introduce a concept for the scaling divisor similar to that in \cite{CL13}.
Note that in item $(\ref{lem:cl1362condition1})$ we need to pass to an open covering of the base, since we do not assume $Z$ affine.
Even if $Z$ was affine, we still need to pass to an open covering since we do not have a global Bertini theorem.

\begin{definition}\label{def:GoodScalingDivisors}
  Let $\pi\colon X \to Z$ be a projective surjective morphism of integral
  quasi-excellent Noetherian algebraic spaces of equal characteristic zero over
  a scheme $S$.
  Suppose that $X$ is normal and that $Z$ admits a dualizing complex
  $\omega_Z^\bullet$.
  Denote by $K_X$ a canonical divisor on $X$ associated to $\omega_X^\bullet =
  \pi^!\omega_Z^\bullet$.
  
Let $\Delta\geq 0$ be a $\mathbf Q$-Weil divisor on $X$ such that $K_X+\Delta$ is $\QQ$-Cartier and such that $(X,\Delta)$ is klt. 
We say a $\QQ$-invertible sheaf $A$ on $X$ is a \textsl{good scaling divisor for the pair $(X,\Delta)$}, if 
the following conditions hold.
\begin{enumerate}[label=$(\roman*)$,ref=\roman*]
    \item\label{lem:Abig}
    $A$ is $\pi$-big.
    
    \item\label{lem:K+Anef}
    $K_X+\Delta+A$ is $\pi$-nef.
    
      \item\label{lem:cl1362condition1} There exists an \'etale covering
        $Z=\bigcup_a
        V_a$ and $\QQ$-Weil divisors $A_a\in |A_{|\pi^{-1}(V_a)}|_{\QQ}$ such that
        $(\pi^{-1}(V_a),\Delta_{|\pi^{-1}(V_a)}+A_a)$
        is klt.
  \end{enumerate}
It is clear that base change to an open subset of the base preserves this property.
\end{definition}

The following lemma tells us that it is always possible to find a good scaling divisor.

\begin{lemma}\label{lem:AmpleIsGoodScaling}
  Let $\pi\colon X \to Z$ be a projective surjective morphism of integral
  quasi-excellent Noetherian algebraic spaces of equal characteristic zero over
  a scheme $S$.
  Suppose that $X$ is normal and that $Z$ admits a dualizing complex
  $\omega_Z^\bullet$.
  Denote by $K_X$ a canonical divisor on $X$ associated to $\omega_X^\bullet =
  \pi^!\omega_Z^\bullet$.
  
Let $\Delta\geq 0$ be a $\mathbf Q$-Weil divisor on $X$ such that $K_X+\Delta$ is $\QQ$-Cartier and such that $(X,\Delta)$ is klt.
Let $A$ be a $\pi$-ample $\QQ$-invertible sheaf on $X$ such that $K_X+\Delta+A$ is $\pi$-nef. 
Then $A$ is a good scaling divisor for the pair $(X,\Delta)$.
Moreover, if $Z$ is a scheme, then the cover in $(\ref{lem:cl1362condition1})$
can be chosen to be an affine cover.
\end{lemma}
\begin{proof}
Items $(\ref{lem:Abig})$ and $(\ref{lem:K+Anef})$ in Definition \ref{def:GoodScalingDivisors} are clear. $(\ref{lem:cl1362condition1})$ follows from Corollary \ref{cor:BertiniKLTOpenCover}
after passing to an \'etale cover by affine schemes.
When $Z$ is a scheme, we can instead choose an open cover by affine schemes.
\end{proof}

We now prepare to prove the existence of the minimal model program with scaling.
We start with the following definition, which is a version of a condition stated
in Theorem \ref{thm:cl1342} for algebraic spaces.

\begin{definition}\label{def:cl13thm4cond}
Let $\pi\colon X \to Z$ be a projective surjective morphism of Noetherian
algebraic spaces of equal characteristic zero over a scheme $S$, such that $X$ is integral and normal and such that $Z$ is quasi-excellent and has a dualizing complex $\omega_Z^\bullet$.
  Denote by $K_X$ a canonical divisor on $X$ associated to $\omega_X^\bullet =
  \pi^!\omega_Z^\bullet$.
We define 
$\mathcal{A}=\mathcal{A}(X/Z)$ to be the set of classes $\mathbf{u}\in N^1(X/Z)_{\RR}$ that satisfies the following condition: 
\begin{quote}
  There exists an \'etale covering $Z=\bigcup_a V_a$ such that for each index $a$, there exists
  a $\mathbf Q$-Weil divisor $\Delta_a\geq 0$ on $\pi^{-1}(V_a)$ with $K_{\pi^{-1}(V_a)}+\Delta_a$ $\QQ$-Cartier and $(\pi^{-1}(V_a),\Delta_a)$ klt,
  a positive real number $c_a$,
  and a class $\mathbf{w}_a\in \Amp(\pi^{-1}(V_a)/V_a)$ such that the restriction
  of $\mathbf{u}$ to $N^1(\pi^{-1}(V_a)/V_a)$ (Lemma \ref{lem:NumericalClassOpenSubset}) equals to $c_a[K_{\pi^{-1}(V_a)}+\Delta_a]+\mathbf{w}_a$. 
\end{quote}
\end{definition}

\begin{lemma}\label{lem:K+Delta+lambdaAInmathcalA}
Let $\pi\colon X \to Z$ be a projective surjective morphism of Noetherian
algebraic spaces of equal characteristic zero over a scheme $S$, such that $X$ is integral and normal and such that $Z$ is quasi-excellent and has a dualizing complex $\omega_Z^\bullet$.
  Denote by $K_X$ a canonical divisor on $X$ associated to $\omega_X^\bullet =
  \pi^!\omega_Z^\bullet$.
  
Let $\Delta\geq 0$ be a $\mathbf Q$-Weil divisor on $X$ such that $K_X+\Delta$ is $\QQ$-Cartier and such that $(X,\Delta)$ is klt. 
  Assume that $K_X +\Delta$ is not $\pi$-nef, and let $A$ be a good scaling divisor for the pair $(X,\Delta)$.  
  Let $\lambda\in [0,1]\subseteq \RR$.
  Then, the class $\mathbf{u}\coloneqq[K_X+\Delta+\lambda A]$ belongs to the set $\mathcal{A}$ as in
  Definition \ref{def:cl13thm4cond},
  and we can further require that the numbers $c_a=1$.
\end{lemma}
\begin{proof}
Passing to an affine \'etale covering of $Z$, we may assume that $Z$ is an
affine scheme, and that 
\begin{align}\label{Originalcl1362condition1}
    A\geq 0\qquad \text{and}\qquad (X,\Delta+A)\ \text{is klt}.
\end{align}

Write $A= H+E$, where $H$ is a $\pi$-ample $\QQ$-Cartier divisor and $E\geq 0$. 
This is possible by Lemma \ref{lem:BigWeilIsAmplePlusEffective}.
Choose $\varepsilon\in\QQ_{>0}$ such that $\varepsilon<\lambda$ and that 
$(X,A+\Delta+\varepsilon E)$ klt, which is possible by Lemma \ref{lem:kltFacts}$(\ref{lem:kltContinuous})$, since log resolutions exist for excellent $\QQ$-schemes
\cite[Theorem 2.3.6 and Lemma 4.2.4]{Tem08}; and we choose $\delta\in\RR_{>0}$ such that $\lambda-\varepsilon-\delta\in\QQ_{>0}$ and that $\varepsilon H+\delta A$ is $\pi$-ample.
Set
\[
  \Delta'=\Delta+(\lambda-\varepsilon-\delta)A+\varepsilon E
\]
and $H'=\varepsilon H+\delta A$. 
Then, by our choice (and Lemma \ref{lem:kltFacts}$(\ref{lem:kltSmaller})$), $H'$ is a $\pi$-ample $\RR$-divisor, $\Delta'$ is an effective $\QQ$-Weil divisor with $K_X+\Delta'$ $\QQ$-Cartier and $(X,\Delta')$ klt
and we have
\[
K_X+\Delta+\lambda A=K_X+\Delta+\varepsilon E+(\lambda-\varepsilon)A+\varepsilon H=K_X+\Delta'+H',
\]
as desired.
\end{proof}

\begin{lemma}[cf.\ {\citeleft\citen{KM98}\citemid \S3.1\citepunct
  \citen{CL13}\citemid Lemma 8\citeright}]\label{lem:cl1362}
Let $\pi\colon X \to Z$ be a projective surjective morphism of Noetherian
algebraic spaces of equal characteristic zero over a scheme $S$,
such that $X$ is integral and normal and such that $Z$ is quasi-excellent and has a dualizing complex $\omega_Z^\bullet$.
  Denote by $K_X$ a canonical divisor on $X$ associated to $\omega_X^\bullet =
  \pi^!\omega_Z^\bullet$.
  
Let $\Delta\geq 0$ be a $\mathbf Q$-Weil divisor on $X$ such that $K_X+\Delta$ is $\QQ$-Cartier and such that $(X,\Delta)$ is klt. 
  Assume that $K_X +\Delta$ is not $\pi$-nef, and let $A$ be a good scaling divisor for the pair $(X,\Delta)$.  
  Let $\lambda = \lambda(X/Z, \Delta, A)$ be the $\pi$-nef threshold.
  Then, $\lambda\in\QQ_{>0}$, 
  and there exists an extremal ray $R\subseteq\overline{\mathrm{NE}}(X/Z)$ with a good contraction with target projective over $Z$, and satisfies $(K_X +\Delta +\lambda A)\cdot R = 0$ and $(K_X +\Delta )\cdot R-\{0\} < 0.$
\end{lemma}
\begin{proof}
By Lemma \ref{lem:K+Delta+lambdaAInmathcalA}, $\mathbf{u}\coloneqq[K_X+\Delta+\lambda A]$ belongs to the set $\mathcal{A}$ as in
Definition \ref{def:cl13thm4cond}.
By the definition of $\lambda$, $\mathbf{u}\in\partial\Nef(X/Z)$, so
we can apply the Cone Theorem \ref{thm:kmmcone} (or Theorem
\ref{thm:cl1342}$(\ref{thm:cl1342Precise})(\ref{thm:cl1342Formal})$ in the
scheme case) to conclude that
there exist finitely many rational supporting hyperplanes $W_1,\ldots,W_m$ of $\mathrm{Nef}(X/Z)$  
cutting out closed half-spaces $W_1^+,\ldots,W_m^+$ such that, 
for some small open rational polytope $P$ containing $\mathbf{u}$,
\[
  P\cap \operatorname{Nef}(X/Z)=P\cap(W_1^+\cup\ldots\cup W_m^+).
\]
Since the spaces $W_i$ are rational, 
it is now clear that $\lambda\in\QQ$ by the definition, and $\lambda\in\QQ_{>0}$ since $K_X+\Delta$ is not $\pi$-nef.

Finally, we show the existence of a desired ray $R$.
Shrinking $P$ if necessary, we may assume $\mathbf{u}\in W_i$ for all $i$. Since $\mathbf{u}-\sigma [A]\not\in \mathrm{Nef}(X/Z)$ for all $\sigma\in (0,\lambda)$ by the definition of $\lambda$, we see that $-[A]\not\in W_i^+$ for some $i$. 
We may thus take $R$ to be the extremal ray dual to $W_i$, see Definition \ref{def:SuppPlaneAndDualRay}. 
Then $R$ is an extremal ray and $(K_X+\Delta+\lambda A)\cdot R=0$ since $\mathbf{u}\in W_i$.
Since $-[A]\not\in W_i^+$, we have $A\cdot R>0$, so $(K_X+\Delta)\cdot R<0$.
The fact $R$ has a good contraction with projective target follows from Lemma \ref{lem:CL13S4ProducesGoodContractions}. \end{proof}

We can now prove the existence of the relative minimal model program with
scaling.
By Lemma \ref{lem:AmpleIsGoodScaling}, this implies the existence part of
Theorem \ref{thm:introrelativemmp}$(\ref{setup:introalgebraicspaces})$.
\begin{theorem}\label{rem:MMP}
Let $\pi\colon X \to Z$ be a projective surjective morphism of Noetherian
algebraic spaces of equal characteristic zero over a scheme $S$,
such that $X$ is integral and normal and such that $Z$ is quasi-excellent and
has a dualizing complex $\omega_Z^\bullet$. 
  Denote by $K_X$ a canonical divisor on $X$ associated to $\omega_X^\bullet =
  \pi^!\omega_Z^\bullet$.
  \par Suppose $X$ is $\QQ$-factorial and let $\Delta$ be a $\QQ$-divisor such that
  $(X,\Delta)$ is klt.
  Let $A$ be a good scaling divisor for $(X,\Delta)$.
  Then, the relative minimal model program with scaling of $A$ over $Z$ exists.
\end{theorem}
\begin{proof}
  First, 
  find a ray $R$ as in Lemma \ref{lem:cl1362}, and let $h\colon X\to Y$ be a good contraction of $R$.

If $\dim Y<\dim X$, we do nothing further and say that the minimal model program of $(X,\Delta)$ over $Z$ with the scaling of $A$ terminates with a Mori fibration.

Otherwise $f$ is birational. 
By Lemma \ref{lem:Contraction3Types}, $\mathrm{Ex}(h)\subseteq X$ is either a prime divisor, in which case we let  $X'=Y$, or is of codimension $\geq 2$, in which case we let $X'$ be a flip of $h$, which exists (Theorem \ref{thm:FLIP}$(\ref{thm:FLIPisProj})$).

Denote by $h'$ and $\pi'$ the maps from $X'$ to $Y$ and $Z$ respectively.
Let $K_{X'}$ be the birational transform of $K_X$ on $X'$, which is a canonical divisor of $X'$; let $\Delta'$ and $A'$ be the birational transforms of $\Delta$ and $A$ respectively.

We note that $Y$ is projective over $Z$, hence so is $X'$; 
$X'$ is integral and normal, see Definitions \ref{def:GoodContrOfR} and \ref{def:FLIPS}; 
$X'$ is $\QQ$-factorial, see Lemmas \ref{lem:ContractionQfac} and \ref{lem:FLIPQfacAndN1same}.
The pair $(X',\Delta')$ is klt by Corollary \ref{cor:MMPpreservesKLT}.
\smallskip

We now verify that $\lambda A'$ is a good scaling divisor for the pair $(X',\Delta')$.
Since $A$ is $\pi$-big, so is $\lambda A$, and we see that $\lambda A'$ is $\pi'$-big from Lemma \ref{lem:ContrPreservesBig}.
We know that $K_X+\Delta+\lambda A\sim_{\QQ}h^*E$ for some effective $\QQ$-Cartier divisor $E$ on $Y$, so $K_{X'}+\Delta'+\lambda A'\sim_{\QQ}{h'}^*E$ and therefore $K_{X'}+\Delta'+\lambda A'$ is $\pi'$-nef.

It remains to verify $(\ref{lem:cl1362condition1})$ in Definition \ref{def:GoodScalingDivisors}.
Notice that birational transform preserves $\QQ$-linear equivalence,
thus after passing to an \'etale covering of $Z$,
we may assume $A\geq 0$ and $(X,\Delta+A)$ klt, and thus $(X,\Delta+\lambda A)$ is klt.
By construction, $K_X+\Delta+\lambda A$ is $h$-numerically trivial, and $K_{X'}+\Delta'+\lambda A'$ is $h'$-numerically trivial.
By Corollary \ref{cor:MMPpreservesKLT}, we know that
$(X',\Delta'+\lambda A')$ is klt, as desired.
\smallskip

Therefore, the new datum $(X',\Delta',\lambda A')$ satisfies the same assumptions as the datum $(X,\Delta,A)$, 
except that it is now possible (and desirable) that $K_{X'}+\Delta'$ is $\pi'$-nef, 
in which case we say that the minimal model program of $(X,\Delta)$ over $Z$ with the scaling of $A$ terminates with a minimal model.
Otherwise, we start over with $(X',\Delta',\lambda A')$.
If, after finitely many steps, we arrive at the situation $\dim X>\dim Y$
(resp.\ $K_{X'}+\Delta'$ is $\pi'$-nef), we say the minimal model program of $(X,\Delta)$ over $Z$ with the scaling of $A$ terminates with a Mori fibration (resp. a minimal model).

Otherwise, we will get an infinite sequence
\[
  (X_1, \Delta_1, \lambda_1A_1) \overset{f_1}{\dashrightarrow} \cdots
  \overset{f_{i-1}}{\dashrightarrow} (X_i, \Delta_i, \lambda_iA_i)
  \overset{f_{i}}{\dashrightarrow} (X_{i+1}, \Delta_{i+1}, \lambda_{i+1}A_{i+1})
  \overset{f_{i+1}}{\dashrightarrow} \cdots
\]
where $X_1=X$, $\Delta_i$ and $A_i$ are the birational transforms of $\Delta$
and $A$, respectively, the triple
$(X_i,\Delta_i,\lambda_{i-1}A_i)$
satisfies the assumptions of Lemma \ref{lem:cl1362},
\[
  \lambda_i=\lambda_{i-1}\cdot\lambda\bigl(X_i/Z,\Delta_i,\lambda_{i-1}A_i\bigr)
  \leq \lambda_{i-1},
\]
and $f_i$ is either a birational contraction or a flip corresponding to a ray as in Lemma \ref{lem:cl1362}.

Since the number $\dim\bigl(N^1(X_i/Z)_{\RR}\bigr)$ 
decreases for a (good) contraction $f_i$ (Definition \ref{def:GoodContrOfR}) and 
remains unchanged for a flip $f_i$ (Lemma \ref{lem:FLIPQfacAndN1same}$(\ref{lem:FLIPN1Same})$), we see that all but finitely many $f_i$ are flips.
\end{proof}

We now prove that a sequence of flips always terminates with additional bigness conditions.
By Lemma \ref{lem:AmpleIsGoodScaling},
this completes the proof of Theorem
\ref{thm:introrelativemmp}$(\ref{setup:introalgebraicspaces})$.

\begin{theorem}[cf. {\cite[Theorem 6]{CL13}}]\label{thm:cl1365}
Let $\pi_1\colon X_1 \to Z$ be a projective morphism of Noetherian algebraic
spaces of equal characteristic zero over a scheme $S$,
such that $X_1$ is integral, normal, and $\QQ$-factorial, and such that $Z$ is
quasi-excellent and has a dualizing complex $\omega_Z^\bullet$.
Denote by $K_{X_1}$ a canonical divisor on $X_1$ associated to $\omega_{X_1}^\bullet =
  \pi^!\omega_Z^\bullet$.
  
Let $\Delta_1$ be an effective $\QQ$-divisor on $X_1$ such that  
$(X_1,\Delta_1)$ is klt. 
Let $A_1$
be a good scaling divisor for the pair $(X_1, \Delta_1)$,
and let $\lambda_1 = \lambda(X_1, \Delta_1, A_1)$. 
Assume that $cK_{X_1}+\Delta_1$ is $\pi_1$-big for some rational number $c\in(-\infty, 1]$.
Then, any sequence
\[
  (X_1, \Delta_1, \lambda_1A_1) \overset{f_1}{\dashrightarrow} \cdots
  \xdashrightarrow{f_{i-1}} (X_i, \Delta_i, \lambda_iA_i)
  \overset{f_{i}}{\dashrightarrow} (X_{i+1}, \Delta_{i+1}, \lambda_{i+1}A_{i+1})
  \xdashrightarrow{f_{i+1}} \cdots
\]
of flips of the Minimal Model Program with scaling of $A_1$ terminates.
\end{theorem}
\begin{proof}
If the sequence of flips does not terminate, we can find an \'etale affine $W
\to Z$ such that, denoting by $U_i$ the base change $X_i\times_Z W$, the birational map $f_{i|U_i}:U_i\dashrightarrow U_{i+1}$ is not an isomorphism (thus a flip of a suitable contraction, see Lemma \ref{lem:ContSmallerOpen}) for infinitely many $i$.
By Lemma \ref{lem:K+Delta+lambdaAInmathcalA},
$[K_{X_1}+\Delta_1+\lambda_1 A_1]$ belongs to the set $\mathcal{A}(X_1/Z)$ as in
Definition \ref{def:cl13thm4cond},
and we may require the numbers $c_a=1$. 
Therefore, after possibly shrinking $W$,
we may assume that 
there exists
a $\RR$-divisor $\Delta_1'\geq 0$ and a $\pi_{|U_1}$-ample $\RR$-divisor $H_1'$ on $U_1$ such that 
$(U_1,\Delta_1')$ is klt and
\[
  [K_{U_1}+\Delta_{1\vert U_1}+\lambda_1 A_{1\vert
  U_1}]=[K_{X_1}+\Delta_1'+H_1']_{\rvert U_1}\in
  N^1\bigl(U_1/V\bigr)_{\RR}.
\]
Since $K_{U_1}+\Delta_{1U_1}+\lambda_1 A_{1U_1}$ is a $\QQ$-divisor by Lemma \ref{lem:cl1362}, 
we may assume that $\Delta_1'$ and $H_1'$ are $\QQ$-divisors.

We may find $\QQ$-divisors $D_1^1,\ldots,D_1^m$ on $U_1$ such that the convex hull $P_1$ of $\{[D_1^1],\ldots,[D_1^m]\}$ is a rational polytope containing $[K_{U_1}+\Delta_{1|U_1}+\lambda_1 A_{1|U_1}]$ in its interior and is contained in $\mathcal{A}$, 
and that $D_1^a-K_{U_1}-\Delta_1'$ is ample for all indices $a$.
Let
\[
  R_1 \coloneqq R\bigl(U_1/V;K_{U_1}+\Delta_{1|U_1},D_1^1,\ldots,D_1^m\bigr).
\]
By Theorem \ref{thm:cl1332},
$R_1$ is finitely generated
over $H^0(W,\cO_W)$.

Write
\[
  g_i=f_{i|U_i}\circ\cdots\circ f_{1|U_1}\colon U_1\dashrightarrow U_{i+1}
\]
for all \(i\geq 0\),
$D_{i+1}^a=g_{i*}D_1^a$, and
\[
  R_{i}=R\bigl(U_{i+1}/V;K_{U_i}+\Delta_{i|U_i},D_i^1,\ldots,D_i^m\bigr).
\]
Then 
each $g_i$ induces an isomorphism $R_1\cong R_{i+1}$, so each $R_{i}$ is finitely generated over $H^0(W,\cO_W)$.
Put $V_{i}=\RR(K_{U_i}+\Delta_{i|U_i})+\sum_a\RR D_i^a$, so we have a commutative diagram
\[
  \begin{tikzcd}[column sep=2.8em]
V_1\rar{f_{1*}}\dar{\varphi_1} &\cdots\rar{f_{i-1,*}} 
&V_i\dar{\varphi_i}\rar{f_{i*}}& V_{i+1}\dar{\varphi_{i+1}}\rar{f_{i+1,*}} &\cdots\\
N^1(U_1/W)_{\RR}\rar{f_{1*}} &\cdots\rar{f_{i-1,*}} 
&N^1(U_i/W)_{\RR}\rar{f_{i*}}& N^1(U_{i+1}/W)_{\RR}\rar{f_{i+1,*}} 
&\cdots
\end{tikzcd}
\]
in which $\varphi_i$ are the canonical maps as in Lemma \ref{lem:cl1339}, and $f_j$ by abuse of notation means $f_{j|U_j}$.
Notice that by Lemma \ref{lem:FLIPQfacAndN1same}, the horizontal arrows are all isomorphisms of real vector spaces.

Let $Q_i$ be the convex hull of $\{[K_{U_i}+\Delta_{i|U_i}],[D_i^1],\ldots,[D_i^m]\}$. 
Then by construction, $Q_1$ contains $[K_{U_1}+\Delta_{1|U_1}+\lambda A_{1|U_1}]$ in its interior for all positive $\lambda\leq \lambda_1$.
Thus, $Q_i$ contains
\[
  [K_{U_i}+\Delta_{i|U_i}+\lambda_i A_{i|U_i}]\in\Nef\bigl(U_i/W\bigr)
\]
in its interior.
Therefore $Q_i\cap \Amp(U_i/W)\neq\emptyset$, so $\mathrm{Supp}(R_i)\cap \Amp(U_i/W)\neq\emptyset$.

Let
\begin{align*}
  \Supp(R_1)&=\bigsqcup_p \cC_1^p
\intertext{be the coarsest subdivision into rational polyhedral cones such that $o_v$ is
linear on each $\cC_1^p$ (see Theorem \ref{thm:cl1335}$(\ref{thm:cl13353})$).
Writing $\cC_{i+1}^p=g_{i*}\cC_1^p$, we see from Lemma \ref{lem:cl1352} that}
  \Supp(R_i)&=\bigsqcup_p \cC_i^p
\end{align*}
is the coarsest subdivision into rational polyhedral cones such that $o_v$ is linear on each $\cC_i^p$. 
Now since
\[
  \Supp(R_i)\cap \Amp(U_i/W)\neq\emptyset,
\]
by Lemma \ref{lem:cl1339} we see that for each $i$ there exists a $p_i$ such that
\[
  \varphi_i^{-1}\bigl(\Nef(X_i/Z)\bigr)\cap\Supp(R_i)=\cC_i^{p_i}.
\]
Since there exists only finitely many indices $p$, there exists an $i$ and infinitely many $j>i\geq 1$ such that $p_j=p_i$.
We pick a $j$ such that there exists $k\in\ZZ$, $i\leq k<j$ with $f_{k|U_k}$ not an isomorphism.

Since $p_i=p_j$, there exists a $\pi$-ample $\QQ$-invertible sheaf $H_i$ on $U_i$
and a  $\pi$-nef $\QQ$-invertible sheaf $H_j$ on $U_j$ such that $H_j$ is the birational transform of $H_i$.
By Lemma \ref{lem:cl1353}, the rational map
\begin{align*}
  f_{j-1|U_{j-1}}\circ\cdots\circ f_{i|U_i}&\colon U_i\dashrightarrow U_j
  \intertext{is a morphism.
  By symmetry, we have that}
  f_{i|U_i}^{-1}\circ\cdots\circ f_{j-1|U_{j-1}}^{-1}&\colon U_j\dashrightarrow U_i
\end{align*}
is a morphism as well.
By \cite[Lemma 7]{CL13}, they are isomorphisms inverse to each other.
However, $f_{j-1|U_{j-1}}\circ\cdots\circ f_{i|U_i}$ is a composition of several isomorphisms and at least one flip, so it cannot be an isomorphism by Corollary \ref{cor:MMPnotCycleBack}, a contradiction.
\end{proof}

\begin{corollary}[cf. {\cite[Corollary 4]{CL13}}]\label{cor:cl1366}
Let $\pi\colon X \to Z$ be a projective morphism of Noetherian algebraic
spaces of equal characteristic zero over a scheme $S$, such that $X$ is
integral, normal and $\QQ$-factorial and such that $Z$ is quasi-excellent and has a dualizing complex $\omega_Z^\bullet$.
Denote by $K_{X}$ a canonical divisor on $X$ associated to $\omega_{X}^\bullet =
  \pi^!\omega_Z^\bullet$.

Let $\Delta\geq 0$ be a $\mathbf Q$-Weil divisor on $X$ such that $K_X+\Delta$ is $\QQ$-Cartier and such that $(X,\Delta)$ is klt. 
  Let $A$ be a good scaling divisor for the pair $(X,\Delta)$. 
  Assume that $cK_{X}+\Delta$ is $\pi$-big for some rational number $c\in(-\infty, 1]$.
  Then, the following hold.
  
  \begin{enumerate}[label=$(\roman*)$,ref=\roman*]
      
      \item\label{cor:cl13661} If $K_X+\Delta$ is $\pi$-pseudoeffective
      then any process of the Minimal Model Program of the pair $(X,\Delta)$ over $Z$ with the scaling of $A$ 
      terminates with a minimal model.
      
      \item\label{cor:cl13662} If $K_X+\Delta$ is not $\pi$-pseudoeffective, 
      then any process of the Minimal Model Program of the pair $(X,\Delta)$ over $Z$ with the scaling of $A$ 
      terminates with a Mori fibration.
      
  \end{enumerate}
\end{corollary}
\begin{proof}
At the end of the proof of Theorem \ref{rem:MMP}, we have noticed that if the process does not terminate, we will have an infinite sequence of flips.
Our assumption and Theorem \ref{thm:cl1365} ensures that such an infinite sequence cannot exist, so the process terminates.

Since whether or not $K_X+\Delta$ is $\pi$-pseudoeffective will not change in the process (Corollary \ref{cor:MMPpreservesKeffective}), 
we see that if the process terminates
and $K_X+\Delta$ is $\pi$-pseudoeffective (resp.\ not $\pi$-pseudoeffective) 
then the process terminates with a minimal model (resp.\ Mori fibration), as desired.
\end{proof}

\begin{corollary}[cf. {\cite[Corollary 5]{CL13}}]\label{cor:cl1367}
Let $\pi\colon X \to Z$ be a projective morphism of Noetherian algebraic
spaces of equal characteristic zero over a scheme $S$, such that $X$ is
integral, normal and $\QQ$-factorial and such that $Z$ is quasi-excellent and has a dualizing complex $\omega_Z^\bullet$.
Denote by $K_{X}$ a canonical divisor on $X$ associated to $\omega_{X}^\bullet =
  \pi^!\omega_Z^\bullet$.

Let $\Delta\geq 0$ be a $\mathbf Q$-Weil divisor on $X$ such that $K_X+\Delta$ is $\QQ$-Cartier and such that $(X,\Delta)$ is klt. 
  Let $A$ be a good scaling divisor for the pair $(X,\Delta)$. 
 If $K_X+\Delta$ is not $\pi$-pseudoeffective, 
      then any process of the Minimal Model Program of the pair $(X,\Delta)$ over $Z$ with the scaling of $A$ 
      terminates with a Mori fibration.
\end{corollary}
\begin{proof}
Reasoning as in the proof of Corollary \ref{cor:cl1366}, it suffices to show any sequence of flips 
\[
  (X_1, \Delta_1, \lambda_1A_1) \overset{f_1}{\dashrightarrow} \cdots
  \xdashrightarrow{f_{i-1}} (X_i, \Delta_i, \lambda_iA_i)
  \overset{f_{i}}{\dashrightarrow} (X_{i+1}, \Delta_{i+1}, \lambda_{i+1}A_{i+1})
  \xdashrightarrow{f_{i+1}} \cdots
\]
as in the proof of Theorem \ref{rem:MMP} terminates, 
and as in the proof of Theorem \ref{thm:cl1365}, 
we may replace $Z$ by any \'etale affine whose image in $Z$
intersects $\pi(X)$.
By Condition $(\ref{lem:cl1362condition1})$ in Definition \ref{def:GoodScalingDivisors},
we may thus assume that there exists $A'\in |A|_{\QQ}$ such that 
$(X,\Delta+A')$ is klt.
Let $A'_i\in |A_i|_{\QQ}$ be the birational transform of $A'$ on $X_i.$

Let $\mu\in \QQ_{>0}$ be such that $K_X+\Delta+\mu A$ not $\pi$-pseudoeffective.
We know from Lemma \ref{lem:ContrPreservesPsEff} that
$K_{X_i}+\Delta_i+\mu A_i$ is not pseudoeffective over $Z$ for each $i$, so $\lambda_i>\mu$.
The divisor $K_{X_i}+\Delta_i+\mu A'_i$ is $\QQ$-linearly equivalent to the combination
\[
  (1-r)(K_{X_i}+\Delta_i)+r(K_{X_i}+\Delta_i+\lambda_iA_i)
\]
where $r=r_i\coloneqq\frac{\mu}{\lambda_i}\in (0,1)$.
Thus the sequence of flips of concern is also a sequence of flips for the pair $(X,\Delta+\mu A')$ with the scaling of $(1-\mu)A$, in symbols
\[
\biggl(X_1, \Delta_1+\mu A'_1, \frac{\lambda_1-\mu}{1-\mu}(1-\mu)A_1\biggr)
\overset{f_1}{\dashrightarrow} \cdots
\xdashrightarrow{f_{i-1}}
\biggl(X_i, \Delta_i+\mu A'_i, \frac{\lambda_i-\mu}{1-\mu}(1-\mu)A_i\biggr)
\overset{f_{i}}{\dashrightarrow} \cdots.
\]
Such a sequence terminates by Theorem \ref{thm:cl1365} (with $c=0$), as $\Delta+\mu A'$ is $\pi$-big.
\end{proof}

\section{Existence of \texorpdfstring{$\QQ$}{Q}-factorializations and
terminalizations for schemes}

In this section, we show that $\QQ$-factorializations and terminalizations exist
for klt pairs.
For simplicity, we restrict to the case of schemes and klt pairs.

\begin{theorem}[cf.\ {\cite[Corollary 1.4.3]{BCHM10}}]\label{thm:TerminalizationGeneral}
  Let $X$ be an integral normal excellent Noetherian scheme of equal characteristic zero
  that has a dualizing complex
  $\omega_X^\bullet$ with associated canonical divisor $K_X$.
  Let 
  $\Delta$ an effective $\RR$-Weil divisor on $X$ such that $K_X+\Delta$ is
  $\RR$-Cartier and $(X,\Delta)$ is klt.
  Let $g\colon Y\to X$ be a projective log resolution.
  Let $\mathfrak{E}$ be a set of $g$-exceptional prime divisors such that for
  every \(E \in \mathfrak{E}\), we have 
  \(a(E,X,\Delta) \le 0\).
  Then, there exists a projective birational morphism $h\colon Z\to X$ with $Z$ $\QQ$-factorial
  such that the $h$-exceptional prime divisors are exactly the birational
  transforms of divisors in the set $\mathfrak{E}$.
\end{theorem}
\begin{proof}
  By \cite[Proposition 2.21]{Kol13},
  there exists an effective $\QQ$-Weil divisor $\Delta'$ on $X$ such that the support
  of $\Delta$ and $\Delta'$ are the same,
  $K_X+\Delta'$ is $\QQ$-Cartier, $(X,\Delta')$ is klt, and 
  each divisor in $\mathfrak{E}$ has discrepancy at most \(0\) with respect to
  the pair $(X,\Delta')$.
  We may therefore replace $\Delta$ by $\Delta'$ to assume that $\Delta$ is a
  $\QQ$-Weil divisor and $K_X+\Delta$ is $\QQ$-Cartier.
  \par Write 
  \[
  K_Y+\Delta_Y\sim_\QQ g^*(K_X+\Delta)+\Gamma
  \]
  where 
  \begin{align*}
    \Delta_Y&=\sum_{\substack{E\subseteq Y\\a(E,X,\Delta)\leq 0}}-a(E,X,\Delta)\,E.
    \intertext{and}
    \Gamma&=\sum_{\substack{E\subseteq Y\\a(E,X,\Delta)> 0}}a(E,X,\Delta)\,E.
  \end{align*}
  Since $(X,\Delta)$ is klt, the coefficients of $\Delta_Y$ are less than
  1, so $(Y,\Delta_Y)$ is klt by \cite[Corollary 2.13]{Kol13}.
  
  \par Let $F$ be the sum of the $g$-exceptional prime divisors not in
  $\mathfrak{E}$ that do not appear in \(\Gamma\).
  Since \(Y\) is regular, $F$ is an effective Cartier divisor on $Y$.
Let $\varepsilon\in\QQ_{>0}$ be sufficiently small such that
$(Y,\Delta_Y+\varepsilon F)$ is klt (see Lemma
\ref{lem:kltFacts}$(\ref{lem:kltContinuous})$).
  There exists a $g$-ample Cartier divisor $A$ on $Y$ and $Y$ is $\QQ$-factorial since it is regular.
  Thus, we may run the MMP for the pair $(Y,\Delta_Y+\varepsilon F)$ with 
  scaling of $A$ over \(X\).
  See Theorem \ref{rem:MMP} and Lemma \ref{lem:AmpleIsGoodScaling}.
  Since $g$ is birational, $\Delta_{Y}$ is $g$-big and $K_{Y}+\Delta_{Y}+\varepsilon F$ is $g$-pseudoeffective,
  so the MMP terminates with a minimal model
  \[
    h\colon\bigl(Z,\varphi_*(\Delta_{Y}+\varepsilon F)\bigr)\longrightarrow X
  \]
  where $\varphi\colon Y\dashrightarrow Z$ is a composition of divisorial
  contractions and flips (see Corollary \ref{cor:cl1366}$(\ref{cor:cl13661})$).
  Note that $h$ is projective and $Z$ is $\QQ$-factorial, as noted in the proof
  of Theorem \ref{rem:MMP}.
  
  Since the rational map $\varphi^{-1}\colon Z\dashrightarrow Y$ does not
  contract any divisors,
  we see that $h$-exceptional divisors are birational transforms of $g$-exceptional divisors, and
  \[
  K_Z+\varphi_*(\Delta_Y+\varepsilon F)\sim_\QQ h^*(K_X+\Delta)+\varphi_*(\Gamma+\varepsilon F).
  \]
  As $h$ is a minimal model, we see $\varphi_*(\Gamma+\varepsilon F)$ is $h$-nef.
  By Lemma \ref{lem:Negativity}, we have $\varphi_*(\Gamma+\varepsilon F)=0$.
  By the definition of $F$,
  this means that all $g$-exceptional prime divisors
  not in $\mathfrak{E}$ are contracted by $\varphi$.
  
  It now suffices to show that no divisor in $\mathfrak{E}$ is contracted by $\varphi$.
  Assume not. 
  Then there is a step
  $\psi_j\colon Y_j\rightarrow Y_{j+1}$
  of the MMP that is a divisorial contraction, and the divisor $E_j$ contracted
  is the birational transform of some $E\in \mathfrak{E}$.
  Denote by $\varphi_j\colon Y\dashrightarrow Y_j$ the rational map coming from the previous steps of the MMP.
  Then $\varphi_j^{-1}$ does not contract any divisor, and
  \[
  K_{Y_j}+(\varphi_j)_*(\Delta_j+\varepsilon F)\sim_\QQ h_j^*(K_X+\Delta)+(\varphi_j)_*(\Gamma+\varepsilon F),
  \]
  where $h_j$ is the map from the $X$-scheme $Y_j$ to $X$.
  Since $\psi_j$ is a step of the MMP,
  we know that $-\left(K_{Y_j}+(\varphi_j)_*(\Delta_j+\varepsilon F)\right)$ is $\psi_j$-ample,
  so $-(\varphi_j)_*(\Gamma+\varepsilon F)$ is $\psi_j$-ample.
  Since $Y_j$ is $\QQ$-factorial,
  $-(\varphi_j)_*(\Gamma+\varepsilon F)+\sigma E_j$
  is $\QQ$-Cartier and $\psi_j$-ample
  for sufficiently small $\sigma\in\QQ_{>0}$.
  Lemma \ref{lem:Negativity} applies and we see $(\varphi_j)_*(\Gamma+\varepsilon F)-\sigma E_j$ is effective,
  so $E_j$ is a component of $(\varphi_j)_*(\Gamma+\varepsilon F)$,
  thus $E\in\mathfrak{E}$ is a component of $\Gamma+\varepsilon F$, contraction.
\end{proof}

\begin{definition}[cf.\ {\cite[p.\ 114]{Kaw88}}]\label{def:QFactorialization}
Let $X$ be an integral normal Noetherian scheme.
A $\QQ$\textsl{-factorialization of} $X$ is an integral $\QQ$-factorial Noetherian scheme $Y$ together with a proper birational morphism 
$g\colon Y\to X$ such that no prime divisor on $Y$ is $g$-exceptional.
\end{definition}

\begin{corollary}[cf.\ {\cite[Corollary 4.5]{Kaw88}}]\label{cor:QfactorializationsExist}
  Let $X$ be an integral normal excellent Noetherian scheme of equal
  characteristic zero that has a dualizing complex
  $\omega_X^\bullet$ with associated canonical divisor $K_X$.
  Let 
  $\Delta$ an effective $\RR$-Weil divisor on $X$ such that $K_X+\Delta$ is $\RR$-Cartier and $(X,\Delta)$ is klt.
  Then, there exists a projective $\QQ$-factorialization $h\colon Z\to X$.
\end{corollary}
\begin{proof}
  Let $g\colon Y\to X$ be a log resolution constructed by blowing up regular centers,
  which exists by \cite[Theorem 1.1.6]{Tem18},
  so $g$ is projective.
  Now take $\mathfrak{E}$ to be the empty set in Theorem \ref{thm:TerminalizationGeneral}.
  The resulting $h\colon Z\to X$ has no exceptional divisors, and $Z$ is $\QQ$-factorial, as desired. 
\end{proof}

\begin{definition}[{cf.\ \citeleft\citen{Rei83}\citemid Main Theorem 0.6\citepunct
  \citen{BCHM10}\citemid p.\ 413\citeright}]\label{def:Terminalization}
  Let $X$ be an integral normal Noetherian scheme that has a dualizing complex
  $\omega_X^\bullet$ with associated canonical divisor $K_X$.
Let
$\Delta$ be an effective $\QQ$-Weil divisor on $X$ such that $K_X+\Delta$ is $\QQ$-Cartier and $(X,\Delta)$ is klt.

A \textsl{terminalization of the pair} $(X,\Delta)$ is a terminal pair $(Y,\Delta_Y)$ together with a proper birational morphism 
$g\colon Y\to X$ such that $g_*\Delta_Y=\Delta$ 
and that $g^*(K_X+\Delta)\sim_{\QQ}K_Y+\Delta_Y$.
The condition is equivalent to saying that $a(E,X,\Delta)\leq 0$ for all prime
divisors $E\subseteq Y$, and that the pair 
\[
  \left(Y,\sum_{E\subseteq Y}-a(E,X,\Delta)\,E\right)
\]
is terminal.
\end{definition}
\begin{lemma}\label{lem:ResolveKLTtoTerminal}
  Let $X$ be an integral normal excellent Noetherian scheme of equal characteristic zero
  that has a dualizing complex
  $\omega_X^\bullet$ with associated canonical divisor $K_X$.
  Let 
  $\Delta$ an effective $\RR$-Weil divisor on $X$ such that $K_X+\Delta$ is $\RR$-Cartier and $(X,\Delta)$ is klt.
  For an integral normal scheme $Y$ proper birational over $X$, we set
  \[
    \Delta_Y=\sum_{\substack{E\subseteq Y\\a(E,X,\Delta)\leq 0}}-a(E,X,\Delta)\,E.
  \]
  Then, there exists a log resolution $g\colon Y\to X$ constructed by blowing up
  regular centers such that the components of $\Delta_Y$ are disjoint.
  Moreover, for any such resolution, the following hold:
  \begin{enumerate}[label=$(\roman*)$,ref=\roman*]
    \item \label{Yterminal}
    $(Y,\Delta_Y)$ is terminal.
    \item \label{YhasallNegativeDiscrep}
    For every proper birational map $Y'\to X$ and every prime Weil divisor $E'$ on $Y'$, if $a(E',X,\Delta)<0$,
    then $E'$ is not contracted by the rational map $Y'\dashrightarrow Y$.
  \end{enumerate}
  
\end{lemma}
\begin{proof}
Let $g_0\colon Y_0\to X$ be a log resolution constructed by blowing up regular
centers, which exists by \cite[Theorem 1.1.6]{Tem18}.
All coefficients of $\Delta_{Y_0}$ are less than $1$ since $(X,\Delta)$ is klt, so $\delta_0:=1-\max\{\text{coefficients of }\Delta_{Y_0}\}>0.$

If $t\geq 2$ components $E_1,\ldots,E_t$ with coefficients $a_1,\ldots,a_t$ meet
and no other components of $\Delta_{Y_0}$ meet $Z\coloneqq E_1\cap\cdots\cap E_t$, 
we consider the blow up $Y_1=\operatorname{Bl}_Z Y_0$.
Note that $Z$ is a regular scheme of pure dimension $\dim Y_0-t$ and may have several connected components.
The preimage of each connected component $C$ of $Z$ in $Y_1$ is a prime divisor $E_C$, and
\[
  a(E_C,X,\Delta)\leq a(E_C,Y_0,\Delta_{Y_0})=1-t+\sum_{i=1}^t a_i\leq 1-t\,\delta_0\leq 1-\delta_0
\]
where the equality follows from \cite[Lemma 2.29]{KM98}.
Along $E_C$, at most $t-1$ of the birational transforms of $E_i$ meet, 
and thus at most $t$ components of $\Delta_{Y_1}$ meet.
If that happens,
the sum of their coefficients is at most
\[
  a(E_C,X,\Delta)+\sum_{i=1}^t a_i-\min_i\{a_i\}\leq 1-t+\sum_{i=1}^t
  a_i+(t-1)(1-\delta_0)=\sum_{i=1}^t
  a_i-t\delta_0.
\]
It is now clear that after finitely many such blow ups we get the desired $Y$.

Now, the coefficients of $\Delta_Y$ are less than 1 since $(X,\Delta)$ is klt, and the components of $\Delta_Y$ are disjoint. 
By \cite[Corollary 2.11]{Kol13}, $(Y,\Delta_Y)$ is terminal.
To show $(\ref{YhasallNegativeDiscrep})$, we may assume that $Y'$ is given by a proper birational map $h:Y'\to Y$.
  Write 
  \[
  K_Y+\Delta_Y\sim_\RR g^*(K_X+\Delta)+\Gamma
  \]
  where the components of $\Gamma$ are exactly the exceptional divisors of $g$ that is not a component of $\Delta_Y$.
  Then $\Gamma$ is effective by the definition of $\Delta_Y$.
Since $Y$ is regular, every component of $\Gamma$ is Cartier, and we have
\[
h^*(K_Y+\Delta_Y) \sim_\RR (h\circ g)^*(K_X+\Delta) + h^*\Gamma.
\]
Thus, $a(E',Y,\Delta_Y)\leq a(E',X,\Delta)<0$. 
Since $(Y,\Delta_Y)$ is terminal, $E'$ must not be $h$-exceptional, as desired.
 \end{proof}

\begin{corollary}[cf.\ \citeleft\citen{Rei83}\citemid Main Theorem 0.6\citepunct
  \citen{BCHM10}\citemid p.\ 413\citeright]\label{cor:TerminalizationsExist}
  Let $X$ be an integral normal excellent Noetherian scheme of equal
  characteristic that has a dualizing complex
  $\omega_X^\bullet$ with associated canonical divisor $K_X$.
  Let $\Delta$ be an effective $\RR$-Weil divisor on $X$ such that $K_X+\Delta$ is $\RR$-Cartier and $(X,\Delta)$ is klt.
  Then, there exists a projective terminalization $h\colon Z\to X$ where $Z$ is $\QQ$-factorial.
\end{corollary}
\begin{proof}
  Let $g\colon Y\to X$ and $\Delta_Y$ be as in Lemma \ref{lem:ResolveKLTtoTerminal},
  and take $\mathfrak{E}$ to be the set of components of $\Delta_Y$ in Theorem \ref{thm:TerminalizationGeneral}. 
  The exceptional prime divisors of the resulting map $h:Z\to X$ are exactly the birational transforms of the components of $\Delta_Y$.
  Thus we have $K_Z+\varphi_*\Delta_Y\sim_{\RR} h^*(K_X+\Delta)$,
  and it suffices to show $(Z,\varphi_*\Delta_Y)$ terminal.
  
  For every proper birational map $Y'\to Z$ and prime divisor $E'$ on $Y'$, the $\RR$-linear equivalence above gives
  $a(E',Z,\varphi_*\Delta_Y)=a(E',X,\Delta)$.
  If $a(E',Z,\varphi_*\Delta_Y)<0$,
  then $a(E',X,\Delta)<0$,
  so by Lemma \ref{lem:ResolveKLTtoTerminal}$(\ref{YhasallNegativeDiscrep})$,
  $E'$ is not contracted by the rational map $Y'\dashrightarrow Y$, and its birational transform is thus a component of $\Delta_Y$ since it has negative discrepancy.
  Thus $E'$ is not exceptional over $Z$, and $(Z,\varphi_*\Delta_Y)$ is terminal.
\end{proof}

\begingroup
\makeatletter
\renewcommand{\@secnumfont}{\bfseries}
\part{Extensions to other categories}\label{part:othercats}
\makeatother
\endgroup
In this part, we extend the relative minimal model program to projective
morphisms of algebraic spaces, formal
schemes, complex analytic spaces, Berkovich analytic spaces, rigid
analytic spaces, and adic spaces.
We work both in equal characteristic zero and in positive/mixed characteristic,
where in the latter context we will assume $\dim(X) \le 3$.
We will also assume the existence of dualizing complexes.
To do so, we first collect some preliminaries for each of these different
categories.\bigskip
\section{Quasi-excellence and dualizing complexes}\label{sect:qedualizingothers}
In this section, we review the notions of quasi-excellence and dualizing
complexes that are analogous to those for schemes in
\S\ref{sect:excellentschemes}.
\subsection{Formal schemes}
We use the definition of formal schemes and Noetherian formal schemes
from \cite[D\'efinition 10.4.2]{EGAInew}.
Quasi-excellence is defined as follows.
\begin{citeddef}[{\citeleft\citen{Tem08}\citemid \S3.1\citepunct
  \citen{Tem12}\citemid \S2.4.3\citeright}]
  Let $X$ be a locally Noetherian formal scheme.
  We say that $X$ is \textsl{quasi-excellent} if for every morphism $\Spf(A)
  \to X$ of finite type, the ring $A$ is quasi-excellent.
\end{citeddef}
\begin{remark}
  The definition above is from \cite{Tem12}, and is equivalent to the original
  definition in \cite{Tem08} by a theorem of Gabber \cite[Theorem 5.1]{KS}.
  See \cite[Remark 3.1.1]{Tem08} and \cite[\S2.4.3]{Tem12}.
\end{remark}
We use the notion of $c$-dualizing complexes from \cite{ATJLL99}.
This notion is distinct from the notion of dualizing complexes due to Yekutieli
\cite[Definition 5.2]{Yek98}.
Yekutieli's notion coincides with what are called \textsl{$t$-dualizing complexes}
in \cite[Definition 2.5.1]{ATJLL99} by \cite[Remark (3) on p.\ 25]{ATJLL99}.
\begin{citeddef}[{\cite[Definition 2.5.1]{ATJLL99}}]
  Let $X$ be a Noetherian formal scheme.
  A complex $\omega_X^\bullet$ on $X$ is a \textsl{$c$-dualizing complex} if
  the following conditions are satisfied.
  \begin{enumerate}[label=$(\roman*)$]
    \item $\omega_X^\bullet$ is an object of $\DD^+_{\mathrm{c}}(X)$.
    \item The natural morphism $\cO_X \to
      \RRHHom(\omega_X^\bullet,\omega_X^\bullet)$ is an isomorphism.
    \item There is an integer $b$ such that for every coherent torsion sheaf
      $\sM$ and for every $i > b$, we have
      $\mathbf{h}^i\RRHHom(\sM,\omega_X^\bullet) = 0$.
  \end{enumerate}
\end{citeddef}
There is a notion of relative analytification for formal schemes
and corresponding GAGA results \citeleft\citen{EGAIII1}\citemid \S5\citepunct
\citen{SGA2}\citemid Expos\'e IX\citeright, which we will refer to by
\textsl{formal GAGA}.
Exceptional pullbacks in the sense of Grothendieck duality exist, preserve
dualizing complexes, and are compatible with formal GAGA in the following sense.
\begin{remark}\label{rem:formalgaga}
  Let $f\colon X \to Y$ be a morphism of Noetherian formal schemes.
  Suppose that \(f\) is \textsl{pseudo-proper}, i.e., there are ideals of
  definition \(\sI \subseteq \cO_Y\) and \(\sJ \subseteq \cO_X\) such that
  \(\sI\cO_X \subseteq \sJ\) and the morphism
  \[
    f_0\colon \bigl(X,\cO_X/\sJ\bigr) \longrightarrow \bigl(Y,\cO_Y/\sI\bigr)
  \]
  of ordinary schemes is proper \cite[1.2.2]{ATJLL99}.
  Consider the functor $f^\sharp$ constructed in \cite[Theorem 2$(b)$]{ATJLL99}.
  \begin{enumerate}[label=$(\roman*)$,ref=\roman*]
    \item If $\omega_Y^\bullet$ is a $c$-dualizing
      complex on $Y$, then $\omega_X^\bullet \coloneqq
      f^\sharp\omega_Y^\bullet$ is a $c$-dualizing complex on $X$ by
      \cite[Proposition 2.5.11]{ATJLL99}.
    \item\label{rem:formalgagasharp}
      Now suppose that \(f\) is \textsl{proper}, i.e., it satisfies the
      following conditions \cite[(3.4.1)]{EGAIII1}:
      \begin{enumerate}[label=\((\arabic*)\),ref=\arabic*]
        \item \(f\) is of finite type in the sense of \cite[D\'efinition
          10.13.3]{EGAInew}, and in particular, adic in the sense of
          \cite[(10.12.1)]{EGAInew}.
        \item\label{cond:proper2}
          Let \(\sI \subseteq \cO_Y\) be an ideal of definition and set
          \(\sJ = \sI \cO_X \subseteq \cO_X\).
          Then, the morphism
          \[
            f_0\colon \bigl(X,\cO_X/\sJ\bigr) \longrightarrow
            \bigl(Y,\cO_Y/\sI\bigr)
          \]
          of ordinary schemes is proper.
      \end{enumerate}
      Suppose, moreover, that locally on $Y$ the morphism $f$ is the completion
      of a morphism of schemes.
      Then, $f^\sharp$ is
      compatible with formal GAGA by \cite[Corollaries 3.3.8 and 6.1.7$(a)$]{ATJLL99}.
  \end{enumerate}
  Note that proper morphisms are pseudo-proper: Since a proper morphism is of
  finite type and in particular, adic,
  we know that \(\sI \subseteq \cO_Y\) and \(\sJ = \sI \cO_X
  \subseteq \cO_X\) are ideals of definition satisfying \(\sI\cO_X \subseteq
  \sJ\).
  Thus, condition \((\ref{cond:proper2})\) of the definition of properness
  implies that \(f\) is pseudo-proper.
\end{remark}
\subsection{Semianalytic germs of complex analytic spaces}
We use the definition of complex analytic spaces from \cite[1.1.5]{GR84}.
We start with the definition of a semianalytic subset of a complex analytic
space.
\begin{citeddef}[{\citeleft\citen{Loj64}\citemid \S1, I\citepunct
  \citen{Fri67}\citemid p.\ 120\citeright}]
  Let $\cX$ be a complex analytic space, and let $a \in X$ be a point.
  Denote by $\mathscr{S}_a$ the minimal class of germs at $a$ of subsets of $X$
  such that the following hold:
  \begin{enumerate}[label=$(\roman*)$]
    \item $\mathscr{S}_a$ is stable under finite unions.
    \item $\mathscr{S}_a$ is stable under complements.
    \item $\mathscr{S}_a$ contains all germs of the form
      $\Set{x \in \cX \given f(x) < 0}_a$,
      where $f(x)$ is a real analytic function in a neighborhood of $a$.
  \end{enumerate}
  A subset $X \subseteq \cX$ is \textsl{semianalytic} if, for every $x \in X$,
  the local germ of $X$ at $x$ is an element of $\mathscr{S}_x$.
\end{citeddef}
We can now define semianalytic germs of complex analytic spaces in the sense of \cite{AT19}.
\begin{citeddef}[{\cite[\S\S B.2--B.3]{AT19}}]
  A \textsl{semianalytic germ of a complex analytic space} is a pair $(\cX,X)$ consisting of
  a complex analytic space $\cX$ and a semianalytic subset $X \subseteq \cX$.
  We call $X$ the \textsl{support} of $(\cX,X)$ and $\cX$ a
  \textsl{representative} of $(\cX,X)$.
  We sometimes use the shorter notation $X$ for the germ $(\cX,X)$.
  The \textsl{structure sheaf} on $X$ is
  \[
    \cO_X \coloneqq (\cO_{\cX})_{\vert X} = i^{-1}\cO_{\cX},
  \]
  where $i\colon X \hookrightarrow \cX$ is the embedding.
  \par A \textsl{morphism} $\phi\colon (\cX,X) \to (\cY,Y)$ of semianalytic
  germs of complex
  analytic spaces consists of a neighborhood $\cX'$ of $X$ and an analytic map
  $f\colon \cX' \to \cY$ taking $X$ to $Y$.
  We say that $f$ is a \textsl{representative} of $\phi$.
\end{citeddef}
We define proper morphisms and closed embeddings of semianalytic germs as follows.
\begin{citeddef}[{\citeleft\citen{AT19}\citemid \S B.5\citeright}]
  \label{def:complexgermmorphisms}
  Let $\phi\colon (\cX,X) \to (\cY,Y)$ be a morphism of semianalytic germs of
  complex analytic spaces.
  \begin{enumerate}[label=$(\roman*)$,ref=\roman*]
    \item We say that $\phi$ is \textsl{without boundary} if there exists a representative
      $f\colon \cX' \to \cY$ of $\phi$ that satisfies $X =
      f^{-1}(Y)$.
    \item We say that $\phi$ is an \textsl{open immersion} (resp.\
      a \textsl{closed immersion})
      if $\phi$ is without boundary and there exists a representative
      $f\colon \cX' \to \cY$ of $\phi$ that is an open immersion (resp.\ a
      closed embedding).
    \item\label{def:complexgermmorphismsproj}
      We say that $\phi$ is \textsl{proper} (resp.\ \textsl{projective})
      if
      there exists a representative
      $f\colon \cX' \to \cY$ of $\phi$ that is proper (resp.\ projective)
      and satisfies $X = f^{-1}(Y)$.
      Note that proper (resp.\ projective) morphisms are without boundary by
      definition.
  \end{enumerate}
\end{citeddef}
We can then define affinoid semianalytic germs as follows.
\begin{citeddef}[{\cite[\S B.6 and \S6.2.4]{AT19}}]
  Let $(\cX,X)$ be a semianalytic germ of a complex analytic space.
  We say that $X$ is \textsl{affinoid} if it admits a closed immersion into a
  germ of the form $(\CC^n,D)$, where $D$ is a closed polydisc.
  A covering $X = \bigcup_i X_i$ of $X$ by affinoids is \textsl{admissible} if
  it admits a finite refinement.
\end{citeddef}
We define dualizing complexes on semianalytic germs of complex analytic spaces.
\begin{definition}[{cf.\ \cite[p.\
  89]{RR71}}]\label{def:complexanalyticdualizing}
  Let $(\cX,X)$ be a semianalytic germ of a complex analytic space.
  A \textsl{dualizing complex} on $X$ is an object $\omega_X^\bullet$ in
  $\DD_\textup{c}^+(X)$ such that the following hold:
  \begin{enumerate}[label=$(\roman*)$,ref=\roman*]
    \item For every $x \in X$, there exists $n(x) \in \ZZ$ such that
      $\Ext^i_{\cO_{X,x}}(\CC,\omega_{X,x}^\bullet) = 0$ for all $i > n(x)$.
      Here, \(\CC\) is the field of complex numbers.
    \item\label{def:complexanalyticdualizingreflexive}
      The natural morphism
      \[
        \id \longrightarrow
        \RRHHom_{\cO_X}\bigl(\RRHHom_{\cO_X}(-,\omega_X^\bullet),
        \omega_X^\bullet\bigr)
      \]
      of $\delta$-functors on $\DD_\textup{c}(X)$ is an isomorphism.
  \end{enumerate}
\end{definition}
\begin{remark}\label{rem:omegaexistsrr71}
  By \cite[\S5]{RR71} (see also \cite[Chapter VII, Theorem 2.6]{BS76}), every
  complex analytic space has a dualizing complex.
  This dualizing complex lies in $\DD_\textup{c}^b(X)$ if $X$ is
  finite-dimensional \citeleft\citen{RR71}\citemid p.\ 89\citepunct
  \citen{BS76}\citemid Chapter VII, Theorem 2.6$(iii)$\citeright.
  Since both conditions in Definition \ref{def:complexanalyticdualizing} can be
  checked at the level of stalks, if $(\cX,X)$ is a semianalytic germ of a
  complex analytic space, then setting $\omega_X^\bullet = i^{-1}\omega_\cX^\bullet$
  gives a dualizing complex on $X$, where $i\colon X \hookrightarrow \cX$ is the
  embedding.
\end{remark}
\begin{convention}
  For semianalytic germs of complex analytic spaces, we will always use the
  dualizing complex $\omega_X^\bullet$ constructed using \cite[\S5]{RR71}.
\end{convention}
\subsection{Non-Archimedean analytic spaces}
Let $k$ be a complete non-Archimedean field.
We use the definition of rigid $k$-analytic spaces from \cite[Definition
9.3.1/4]{BGR84} (in which case we assume that $k$ is non-trivially valued)
and the definition of $k$-analytic spaces from
\cite[\S1]{Ber93} (in which case we allow trivial valuations on $k$).
We sometimes refer to the $k$-analytic spaces from \cite{Ber93} as
\textsl{Berkovich spaces}.
We use the definition of adic spaces from \cite[Definition on
p.\ 521]{Hub94}.\medskip
\par Instead of defining dualizing complexes on rigid $k$-analytic and
Berkovich spaces in a similar fashion to complex analytic spaces
(Definition
\ref{def:complexanalyticdualizing}), we adopt a definition that is more easily
comparable to the scheme-theoretic notion of a 
\textsl{weakly pointwise dualizing complex} from \cite[p.\ 120]{Con00}.
Below, $X_G$ denotes the ringed site where the Grothendieck topology is the
\textsl{$G$-topology} in the sense of \citeleft\citen{BGR84}\citemid Definition
9.3.1/4\citepunct \citen{Ber93}\citemid p.\ 25\citeright.
\begin{definition}\label{def:dualizingcomplexrigid}
  Let $X$ be one of the following:
  \begin{enumerate}[label=$(\alph*)$,ref=\alph*]
    \item A rigid $k$-analytic space, where $k$ is a complete non-trivially
      valued non-Archimedean field.
    \item A $k$-analytic space, where $k$ is a complete
      non-Archimedean field.
  \end{enumerate}
  A \textsl{dualizing complex} on $X$ is an object $\omega_X^\bullet$ in
  $\DD^+_\coh(X_G)$
  such that for every $x \in X$, the object
  $\omega_{X,x}^\bullet$ in $\DD_\coh^+(\cO_{X_G,x})$
  is a dualizing complex in
  the sense of Definition \ref{def:dualizingcomplexschemes} (see also \cite[p.\
  118 and Lemma 3.1.4]{Con00}).
  In either setting, the
  stalks $\cO_{X_G,x}$
  are Noetherian \citeleft\citen{BGR84}\citemid Proposition 7.3.2/7\citepunct
  \citen{Ber93}\citemid Theorem 2.1.4\citeright, and hence we can ask
  whether $\omega_{X,x}^\bullet$ is a dualizing complex.
\end{definition}
\begin{convention}
  If $X$ is a good $k$-analytic space in the sense of \cite[Remark
  1.2.16]{Ber93}, then we drop the subscript $G$ in $X_G$, since in this case
  there is a good notion of a structure sheaf $\cO_X$ on $X$ such that the
  categories of coherent $\cO_X$-modules and coherent $\cO_{X_G}$-modules
  coincide \cite[Proposition 1.3.4$(ii)$]{Ber93}.
  Note that affinoid $k$-analytic spaces and all $k$-analytic spaces that are
  proper over affinoid $k$-analytic spaces are good by
  \citeleft\citen{Ber90}\citemid \S3.1\citepunct \citen{Ber93}\citemid
  \S1.5\citeright.
\end{convention}
For adic spaces, we adopt a different definition.
We do not work with all adic spaces $X$ and stalks $\cO_{X,x}$, and instead work
only with \textsl{Jacobson adic spaces} and the points in the
\textsl{Jacobson--Gelfand spectrum} of $X$, a notion first defined in
\cite{Lou}.
Property \((\ref{def:jacobsongelfand1})\) below is the property (T) from
\cite[p.\ 108]{Hub93thesis}.
\begin{citeddef}[{\cite[Definitions 3.1 and 3.2]{Lou}}]\label{def:jacobsongelfand}
  A strongly Noetherian complete Tate ring $A$ is a \textsl{Jacobson--Tate ring}
  if it satisfies the following properties:
  \begin{enumerate}[label=$(\arabic*)$,ref=\arabic*]
    \item For\label{def:jacobsongelfand1}
      every maximal ideal \(\fm \subseteq A\), the quotient topology on
      \(A/\fm\) is the topology defined by a rank 1 valuation on \(A/\fm\).
    \item For\label{def:jacobsongelfand2}
      every $A$-algebra $B$ topologically of finite type over $A$, the
      induced map $\Spec(B) \to \Spec(A)$ maps maximal ideals to maximal ideals.
  \end{enumerate}
  A locally Noetherian analytic adic space
  $X$ is a \textsl{Jacobson adic space} if it is locally of the form
  $\Spa(A,A^+)$ where $A$ is a Jacobson--Tate ring.
  The \textsl{Jacobson--Gelfand spectrum} of $X$ is the subset
  \[
    \JG(X) \subseteq X
  \]
  of all rank 1 points $x \in X$ for which there exists an affinoid open
  neighborhood $U = \Spa(A,A^+)$ of $x \in X$ such that \(A\) is a
  Jacobson--Tate ring and $\supp(x) \subseteq A$
  is a maximal ideal.
\end{citeddef}
For completeness, we reprove some results from \cite{Lou} on Jacobson--Tate
rings.
\begin{citedprop}[{\cite[Proposition 3.3]{Lou}}]\label{prop:lou33}
  Let \(A\) be a Jacobson--Tate ring.
  \begin{enumerate}[label=\((\roman*)\),ref=\roman*]
    \item Let\label{prop:lou331}
      \(B\) be a ring topologically of finite type over \(A\).
      Then, \(B\) is a Jacobson--Tate ring.
    \item Let\label{prop:lou332}
      \(B\) be a rational localization of \(A\).
      Then, every maximal ideal \(\fn \subseteq B\) is the extension of a unique
      maximal ideal \(\fm \subseteq A\) and the natural map
      \begin{align*}
        (A_\fm)^\wedge &\longrightarrow (B_\fn)^\wedge
      \intertext{is an isomorphism.
      In particular, if \(x \in X = \Spa(A,A^+)\) is
      supported at a maximal ideal \(\fm \subseteq A\), then the canonical map}
        (A_\fm)^\wedge &\longrightarrow \hat{\cO}_{X,x}
      \end{align*}
      is an isomorphism.
    \item For\label{prop:lou333}
      every ring of integral elements \(A^+ \subseteq A\), the support
      map
      \[
        \supp\colon \JG\bigl(\Spa(A,A^+)\bigr) \longrightarrow \Spec(A)
      \]
      is injective with image \(\MaxSpec(A)\) and the inclusion
      \[
        \JG\bigl(\Spa(A,A^+)\bigr) \hooklongrightarrow \Spa(A,A^+)
      \]
      has dense image.
  \end{enumerate}
\end{citedprop}
\begin{proof}
  \((\ref{prop:lou331})\).
  Let \(\fn \subseteq B\) be a maximal ideal.
  Then, \(\fm = \fn \cap A \subseteq A\) is a maximal ideal by condition
  \((\ref{def:jacobsongelfand2})\) of Definition \ref{def:jacobsongelfand}.
  The induced map \(A/\fm \hookrightarrow B/\fn\) is topologically of finite
  type.
  Since the topology on \(A/\fm\) is defined by a rank 1 valuation,
  we can apply results from \cite{BGR84}.
  By a consequence of the non-Archimedean Noether normalization lemma
  \cite[Corollary 6.1.2/3]{BGR84}, \(A/\fm \hookrightarrow B/\fn\) is a
  finite field extension.
  Thus, \(B/\fn\) is complete by \cite[Proposition 3.7.3/3]{BGR84}, proving
  condition
  \((\ref{def:jacobsongelfand1})\) of Definition \ref{def:jacobsongelfand}.
  For condition \((\ref{def:jacobsongelfand2})\) of Definition
  \ref{def:jacobsongelfand}, let \(C\) be a \(B\)-algebra topologically of
  finite type over \(B\).
  We want to show that if \(\fn' \subseteq C\) is a maximal ideal, then \(\fn =
  \fn' \cap B\) is a maximal ideal.
  Set \(\fm = \fn \cap A = \fn' \cap A\) and consider the maps
  \[
    A/\fm \hooklongrightarrow B/\fn \hooklongrightarrow C/\fn'.
  \]
  By \cite[Corollary 6.1.2/3]{BGR84} again, the composition \(A/\fm
  \hookrightarrow C/\fn'\) is a finite field extension.
  We therefore see that \(B/\fn\) is finite over \(A/\fn\), and in particular,
  \(\dim(B/\fn) = 0\).
  Since \(B/\fn\) is a subring of \(C/\fn'\), we know that \(B/\fn\) is a
  domain.
  Thus, \(B/\fn\) is a zero-dimensional domain, that is, a field.
  We conclude that \(\fn \subseteq B\) is a maximal ideal.\smallskip
  \par \((\ref{prop:lou332})\).
  We first show that every maximal ideal \(\fn \subseteq B\) is the extension of
  a unique maximal ideal \(\fm \subseteq A\).
  Since \(B\) is topologically of finite type over \(A\), condition
  \((\ref{def:jacobsongelfand2})\) of Definition \ref{def:jacobsongelfand}
  implies that \(\fn \subseteq B\) contracts to a maximal ideal \(\fm \subseteq
  A\).
  Then, \(A/\fm \hookrightarrow B/\fn\) is a rational localization.
  Since \(A/\fm\) is a field, we see that \(A/\fm \hookrightarrow B/\fn\) is an
  isomorphism.
  Thus, \(\fn = \fm B\).
  Finally, the maximal ideal \(\fm \subseteq A\) is unique since no other
  maximal ideal has \(\fn\) in its fiber.
  \par The isomorphism \((A_\fm)^\wedge \overset{\sim}{\to} \hat{\cO}_{X,x}\)
  holds by \cite[Proposition 3.3.16\((iii)\)]{Hub93thesis}.
  This isomorphism fits into the commutative diagram
  \begin{equation}\label{eq:cdratloc}
    \begin{tikzcd}[column sep=0]
      (A_\fm)^\wedge \arrow{rr}{\sim} \arrow{dr} & & \hat{\cO}_{X,x}\\
      & (B_\fn)^\wedge \arrow{ur}[sloped]{\sim}
    \end{tikzcd}
  \end{equation}
  where the right diagonal map is an isomorphism by \cite[Proposition
  3.3.16\((iii)\)]{Hub93thesis} again.\smallskip
  \par \((\ref{prop:lou333})\).
  Set \(X \coloneqq \Spa(A,A^+)\).
  We first show that
  \[
    \supp\bigl(\JG(X)\bigr) \subseteq \MaxSpec(A).
  \]
  Let \(x \in \JG(X)\) and choose an affinoid open neighborhood \(U =
  \Spa(B,B^+)\) of \(x\) such that \(\supp(x) \subseteq B\) is a maximal ideal
  and \(B\) is a Jacobson--Tate ring.
  We claim that we may assume that \(A \to B\) is a rational localization.
  Since rational subsets form a basis for the topology on \(X\), there
  exists an open subset
  \[
    V = X\biggl(\frac{T_1}{s_1},\frac{T_2}{s_2},\ldots,
    \frac{T_n}{s_n}\biggr)
    \subseteq U
  \]
  that contains \(x\).
  By definition of rational subsets, we have
  \[
    V = \Spa(B,B^+)\biggl(\frac{T_1}{s_1},\frac{T_2}{s_2},\ldots,
    \frac{T_n}{s_n}\biggr)
  \]
  by considering the images in \(B\) of the elements \(s_i\) and the elements in
  \(T_i\).
  In the commutative diagram
  \[
    \begin{tikzcd}[column sep=small]
      V \arrow[hook]{rr}\arrow[hook]{dr} & & X\\
      & U \arrow[hook]{ur}
    \end{tikzcd}
  \]
  both the top horizontal map and the left diagonal map are inclusions of
  rational subsets.
  Since \(\supp(x)\) is a maximal ideal in \(B\), \((\ref{prop:lou332})\)
  implies that \(\supp(x)\) is a maximal ideal in \(\cO_{X}(V)\).
  We may therefore replace \(U\) by \(V\) to assume that \(A \to B\) is a
  rational localization.
  \par We now show that \(\supp(x) \in \MaxSpec(A)\) when \(x \in \Spa(B,B^+)\)
  for a rational localization \(A \to B\) and \(\supp(x) \subseteq B\) is a
  maximal ideal.
  Since \(B\) is a rational localization of \(A\), we know that \(B\) is a
  Jacobson--Tate ring by \((\ref{prop:lou331})\).
  Then, \((\ref{prop:lou332})\) shows that in the commutative
  diagram
  \[
    \begin{tikzcd}
      \Spa(B,B^+) \rar[hook]\dar[swap]{\supp} & \Spa(A,A^+) \dar{\supp}\\
      \Spec(B) \rar & \Spec(A)
    \end{tikzcd}
  \]
  the image of \(\supp(x) \in \Spec(B)\) in \(\Spec(A)\) is maximal.
  Thus, we have \(\supp(x) \in \MaxSpec(A)\).
  \par We now show that \(\supp\) maps \(\JG(X)\) injectively to
  \(\MaxSpec(A)\).
  By the previous two paragraphs, we know that \(\supp(x) \in \MaxSpec(A)\) for
  every \(x \in \JG(X)\).
  Following \cite[p.\ 119]{Hub93thesis}, we set
  \[
    \Max_v(A) \coloneqq \Set[\big]{v \in
    \Cont(A)_{\mathrm{min}} \given \supp(x) \in \MaxSpec(A)}
  \]
  where the subscript min denotes the set of minimal points in a topological
  space.
  By \cite[Lemma 3.1.14\((ii)\), Proof of Proposition 3.3.9, and p.\
  119]{Hub93thesis}, we have
  \[
    \Spa(A,A^+)_{\mathrm{min}} = \Cont(A)_{\mathrm{min}}
    = \Set[\big]{v \in \Cont(A) \given v\ \text{has rank 1}}.
  \]
  We therefore see that \(\JG(X) = \Max_v(A)\) as subspaces of
  \(\Spa(A,A^+)\).
  Since \(\supp\) induces a bijection
  \[
    \supp\colon \Max_v(A) \xrightarrow{\text{bij.}} \MaxSpec(A)
  \]
  by 
  \cite[p.\ 119]{Hub93thesis}, we are done.
  \par We now show that \(\JG(X)\) is dense in \(X\).
  Let \(U \subseteq X\) be an open subset.
  We need to show that \(U \cap \JG(X)\) is nonempty.
  Let \(\Spa(B,B^+) \subseteq U\) be an open subset rational in \(X\).
  Then, there exists a maximal ideal \(\fm \subseteq B\).
  By the previous paragraph, there exists a point \(x \in \JG(\Spa(B,B^+))\)
  such that \(\supp(\fm) = x\).
  Since
  \[
    \JG\bigl(\Spa(B,B^+)\bigr) \hooklongrightarrow \JG(X),
  \]
  we see that \(x \in \Spa(B,B^+) \cap \JG(X) \subseteq U \cap \JG(X)\), as
  needed.
\end{proof}

We now define dualizing complexes on Jacobson adic spaces \(X\).
\begin{definition}\label{def:dualizingcomplexadic}
  Let $X$ be a Jacobson adic space.
  A \textsl{dualizing complex} on $X$ is an object $\omega_X^\bullet$ in
  $\DD^+_\coh(X)$ such that
  for every $x \in \JG(X)$, the object $\omega_{X,x}^\bullet$ in
  $\DD_c^+(\cO_{X,x})$ is a dualizing complex in the sense of
  Definition \ref{def:dualizingcomplexschemes} (see also \cite[p.\
  118 and Lemma 3.1.4]{Con00}).
  Note that the stalks $\cO_{X,x}$ are Noetherian by
  \cite[Proposition 3.3.16\((i)\)]{Hub93thesis}.
\end{definition}

\section{Grothendieck duality, dualizing complexes, and GAGA}\label{sect:dualityandgaga}
The goal of this section is to compare dualizing complexes and exceptional
pullbacks (in the sense of Grothendieck duality) under relative GAGA for
semianalytic germs of complex analytic spaces \cite[Appendix C]{AT19},
rigid analytic spaces \cite{Kop74}, Berkovich spaces \cite{Poi10}, and
adic spaces \cite{Hub07}.
In the complex analytic case, we compare the existing definitions of
exceptional pullbacks for spaces with the scheme-theoretic definition from
\cite{Har66}.
In the non-Archimedean cases, there is no definition for exceptional pullbacks
along arbitrary morphisms of analytic spaces in the literature (although for adic
spaces, one can construct an exceptional pullback functor using condensed
mathematics; see \S\ref{csdiscussion}).
Instead, we will check that the analytification of the scheme-theoretic exceptional
pullback sends dualizing complexes to dualizing complexes.
\begin{convention}\label{convention:analytification}
  We denote the analytification functor in each setting by $(-) \mapsto
  (-)^\an$, and similarly for sheaves and complexes.
  For objects in the essential image of this functor, we denote by $(-)^\al$ the
  corresponding algebraic object, and call this process \textsl{algebraization}.
\end{convention}
\subsection{Equivalences of categories of coherent sheaves yield equivalences of
bounded derived categories}
\par As a preliminary step, we need versions of the GAGA theorems in
\cite[Theorem C.1.1]{AT19},
\citeleft\citen{Kop74}\citemid Folgerung 6.6, Folgerung 6.7, and Theorem
6.8\citeright\ (see also \cite[Example 3.2.6]{Con06}), \cite[Th\'eor\`eme
A.1]{Poi10}, and \citeleft\citen{Hub07}\citemid Corollary 6.4\citepunct
\citen{Zav}\citemid Lemma 6.9\citeright\ 
for bounded derived categories.\medskip
\par We prove the following result deducing equivalences of bounded derived
categories from equivalences of (weak) Serre subcategories of categories of
modules.
The statements $(\ref{thm:weakserreequivgivesderivedext})$ and
$(\ref{thm:weakserreequivgivesderived})$ below are versions of the
first half of the proof of \cite[Theorem 3.7]{Lim}, but we write down the proof
for completeness.
The result in \cite{Lim} gives the
stronger conclusion that $\DD^-_{\mathscr{A}_X}(X) \to \DD^-_{\mathscr{A}_Y}(Y)$
is an equivalence of categories under stronger hypotheses.
See also \cite[Lemma 5.12]{PY16} for a version of this result for
$\infty$-categories.
\begin{theorem}\label{thm:weakserreequivgivesderivedmain}
  Let $h\colon (Y,\cO_Y) \to (X,\cO_X)$ be a flat morphism of ringed sites.
  Fix weak Serre subcategories $\mathscr{A}_Y$ in $\Mod(Y)$ and
  $\mathscr{A}_X$ in $\Mod(X)$.
  Suppose the pullback functor $h^*\colon \Mod(X) \to \Mod(Y)$ restricts to
  a functor
  \begin{align}\label{eq:cohequivalencecomplex}
    h^*\colon \mathscr{A}_X &\longrightarrow \mathscr{A}_Y,
    \intertext{and consider the associated derived functors}
  \label{eq:dbequivalencecomplex}
    h^*\colon \DD^*_{\mathscr{A}_X}(X) &\longrightarrow
    \DD^*_{\mathscr{A}_Y}(Y)
  \end{align}
  where $* \in \{b,+\}$.
  \begin{enumerate}[label=$(\roman*)$,ref=\roman*]
    \item\label{thm:weakserreequivgivesderivedext}
      Suppose the natural morphisms
      \begin{align}\label{eq:extsareisos}
        \Ext^n_{\cO_X}(\sF,\sG) &\longrightarrow
        \Ext^n_{\cO_{Y}}(h^*\sF,h^*\sG)
      \intertext{are isomorphisms for all objects $\sF,\sG$ in $\mathscr{A}_X$ and for all
      $n \in \ZZ$.
      Then, the natural morphisms}
        \RRHom_{\cO_X}(\sF,\sG) &\longrightarrow
        \RRHom_{\cO_{Y}}(h^*\sF,h^*\sG)\label{eq:dbfullyfaithful}
      \end{align}
      are isomorphisms for all objects $\sF,\sG$ in $\DD^*_{\mathscr{A}_X}(X)$.
    \item\label{thm:weakserreequivgivesderivedsheafext}
      Suppose the natural morphisms
      \begin{align}\label{eq:sheafextsareisos}
        h^*\EExt^n_{\cO_X}(\sF,\sG) &\longrightarrow
        \EExt^n_{\cO_{Y}}(h^*\sF,h^*\sG)
      \intertext{are isomorphisms for all objects $\sF,\sG$ in $\mathscr{A}_X$ and for all
      $n \in \ZZ$.
      Then, the natural morphisms}
        h^*\RRHHom_{\cO_X}(\sF,\sG) &\longrightarrow
        \RRHHom_{\cO_{Y}}(h^*\sF,h^*\sG)\nonumber
      \end{align}
      are isomorphisms for all objects $\sF,\sG$ in $\DD^*_{\mathscr{A}_X}(X)$.
    \item\label{thm:weakserreequivgivesderived}
      Suppose \eqref{eq:cohequivalencecomplex} is an equivalence of categories,
      and that
      the natural morphisms \eqref{eq:extsareisos}
      are isomorphisms for all objects $\sF,\sG$ in $\mathscr{A}_X$ and for all
      $n \in \ZZ$.
      Then, \eqref{eq:dbequivalencecomplex} is an equivalence of categories.
    \item\label{thm:weakserrecohgiveshypercoh}
      If \eqref{eq:cohequivalencecomplex}
      induces isomorphisms on cohomology modules, then
      \eqref{eq:dbequivalencecomplex} induces
      isomorphisms on $\RR\Gamma$ and on hypercohomology modules.
      In this case, if the natural morphisms \eqref{eq:sheafextsareisos} are
      isomorphisms, then the natural morphisms in
      $(\ref{thm:weakserreequivgivesderivedext})$ and
      $(\ref{thm:weakserreequivgivesderivedsheafext})$ are all isomorphisms.
  \end{enumerate}
\end{theorem}
\begin{proof}
  Since \(h\) is flat, we know that \(h^*\) commute with cohomology.
  Thus, $h^*$ sends bounded objects in $\DD_{\mathscr{A}_X}(X)$ to
  bounded (resp.\ bounded-below)
  objects in $\DD_{\mathscr{A}_Y}(Y)$.\smallskip
  \par For $(\ref{thm:weakserreequivgivesderivedext})$,
  we first assume that $\sG$ is concentrated in one degree.
  If $\sF$ is concentrated in one degree, then the isomorphism follows from
  the isomorphism \eqref{eq:extsareisos}.
  \par We now show \eqref{eq:dbfullyfaithful} is an isomorphism
  for general $\sF$ when $\sG$ is
  concentrated in one degree.
  First suppose $* = b$, and
  let $n$ be the smallest degree where $\mathbf{h}^n(\sF) \ne 0$.
  Consider the exact triangle
  \[
    \mathbf{h}^n(\sF)[-n] \longrightarrow \sF \longrightarrow \tau_{\ge n+1}\sF
    \xrightarrow{+1}.
  \]
  We then have the commutative diagram
  \[
    \begin{tikzcd}[column sep=1.6em]
      \RRHom_{\cO_X}\bigl(\tau_{\ge n+1} \sF,\sG\bigr) \rar\dar[sloped]{\sim}
      & \RRHom_{\cO_X}(\sF,\sG) \rar\dar
      & \RRHom_{\cO_X}\bigl(\mathbf{h}^n(\sF)[-n],\sG\bigr)
      \rar{+1}\dar[sloped]{\sim} & {}\\
      \RRHom_{\cO_{Y}}\bigl(h^*\tau_{\ge n+1} \sF,h^*\sG\bigr) \rar
      & \RRHom_{\cO_{Y}}(h^*\sF,h^*\sG) \rar
      & \RRHom_{\cO_{Y}}\bigl(h^*\mathbf{h}^n(\sF)[-n],h^*\sG\bigr) \rar{+1} & {}
    \end{tikzcd}
  \]
  with exact rows
  where the left and right vertical arrows are quasi-isomorphisms by the
  inductive hypothesis.
  By \cite[Chapitre II, Corollaire 1.2.3]{Ver67}, we see the middle vertical
  arrow is a quasi-isomorphism.
  This shows \eqref{eq:dbfullyfaithful} is an isomorphism
  when $\sG$ is concentrated in one degree and $* = b$.
  When $* = +$, the argument above shows that
  \[
    \RRHom_{\cO_X}\bigl(\tau_{\le n}\sF,\sG\bigr)
    \overset{\sim}{\longrightarrow} 
    \RRHom_{\cO_Y}\bigl(h^*\tau_{\le n}\sF,h^*\sG\bigr)
  \]
  is an isomorphism for all $n$, and hence \eqref{eq:dbfullyfaithful} is an
  isomorphism when $\sG$ is concentrated in one degree and $* = +$.
  \par To show \eqref{eq:dbfullyfaithful} is an isomorphism
  for general $\sF$ and general $\sG$ when $* = b$,
  we repeat the same argument by induction on the length of $\sG$.
  The case when $\sG$ is concentrated in one degree was shown above.
  If $n$ is the smallest degree where $\mathbf{h}^n(\sG) \ne 0$, the exact
  triangle
  \[
    \mathbf{h}^n(\sG)[-n] \longrightarrow \sG \longrightarrow \tau_{\ge n+1}\sG
    \xrightarrow{+1}
  \]
  yields the commutative diagram
  \[
    \begin{tikzcd}[column sep=scriptsize]
      \RRHom_{\cO_X}\bigl(\sF,\mathbf{h}^n(\sG)[-n]\bigr) \rar\dar[sloped]{\sim}
      & \RRHom_{\cO_X}(\sF,\sG) \rar\dar
      & \RRHom_{\cO_X}\bigl(\sF,\tau_{\ge n+1}\sG\bigr)
      \rar{+1}\dar[sloped]{\sim} & {}\\
      \RRHom_{\cO_{Y}}\bigl(h^*\sF,h^*\mathbf{h}^n(\sG)[-n]\bigr) \rar
      & \RRHom_{\cO_{Y}}(h^*\sF,h^*\sG) \rar
      & \RRHom_{\cO_{Y}}\bigl(h^*\sF,h^*\tau_{\ge n+1}\sG\bigr) \rar{+1} & {}
    \end{tikzcd}
  \]
  with exact rows
  where the left and vertical arrows are quasi-isomorphisms by the inductive
  hypothesis.
  By \cite[Chapitre II, Corollaire 1.2.3]{Ver67}, we see the middle vertical arrow is a
  quasi-isomorphism.
  This shows \eqref{eq:dbfullyfaithful} is an isomorphism for all
  $\sF$ when $\sG$ is bounded.
  Now for $* = +$, we know that
  \[
    \RRHom_{\cO_X}(\sF,\tau_{\le n}\sG) \overset{\sim}{\longrightarrow} 
    \RRHom_{\cO_Y}(h^*\sF,h^*\tau_{\le n}\sG)
  \]
  is an isomorphism for all $n$, and hence \eqref{eq:dbfullyfaithful} is an isomorphism
  for all $\sF$ and $\sG$ that are bounded-below.\smallskip
  \par For $(\ref{thm:weakserreequivgivesderivedsheafext})$, we can repeat the
  same argument as in $(\ref{thm:weakserreequivgivesderivedext})$ replacing
  $\RRHom$ with $\RRHHom$.\smallskip
  \par For $(\ref{thm:weakserreequivgivesderived})$, since the functor
  \eqref{eq:dbequivalencecomplex} is fully
  faithful by $(\ref{thm:weakserreequivgivesderivedext})$, it suffices to show
  the functor \eqref{eq:dbequivalencecomplex} is essentially surjective.
  We start with the case $*=b$.
  Fix an object $\sG$ in $\DD^b_{\mathscr{A}_Y}(Y)$.
  We proceed by induction on the length of $\sG$.
  If $\sG$ is concentrated in one degree, this follows from the equivalence
  \eqref{eq:cohequivalencecomplex}.
  For general $\sG$, let $n$ be the smallest degree where $\mathbf{h}^n(\sG) \ne
  0$, and consider the exact triangle
  \[
    (\tau_{\ge n+1}\sG)[-1] \longrightarrow \mathbf{h}^n(\sG)[-n]
    \longrightarrow \sG \xrightarrow{+1}.
  \]
  By \eqref{eq:cohequivalencecomplex} and the inductive hypothesis, there exist
  objects $\sF,\sF'$ in $\mathscr{A}_X$ such that
  \[
    h^*\sF \simeq (\tau_{\ge n+1}\sG)[-1] \qquad \text{and} \qquad
    h^*\sF' \simeq \mathbf{h}^n(\sG)[-n].
  \]
  By $(\ref{thm:weakserreequivgivesderivedext})$, we know that the morphism $(\tau_{\ge
  n+1}\sG)[-1] \to \mathbf{h}^n(\sG)[-n]$ is the pullback of a morphism
  $\varphi\colon \sF \to \sF'$ in $\DD^b_{\mathscr{A}_X}(X)$.
  It follows that $\sG \simeq h^*\Cone(\sF \to \sF')$ since $h$ is flat.
  \par Next, we consider the case when $* = +$.
  Write
  \[
    \sG \simeq \hocolim_{\raisebox{5pt}{\(\scriptstyle n\)}} \tau_{\le n}\sG.
  \]
  Since \eqref{eq:dbequivalencecomplex} is an equivalence for $* = b$, every
  $\tau_{\le n}\sG$ is of the form $h^*\sF_n$ for $\sF_n$ in
  $\DD^b_{\mathscr{A}_X}(X)$.
  Moreover, the transition morphisms in the homotopy colimit come from
  compatible morphisms of the $h^*\sF_n$ using the faithful fullness of
  \eqref{eq:dbequivalencecomplex}.
  We therefore see that
  \[
    \sG \simeq \hocolim_{\raisebox{5pt}{\(\scriptstyle n\)}} \,
    h^*\sF_n \simeq h^*
    \hocolim_{\raisebox{5pt}{\(\scriptstyle n\)}} \sF_n
  \]
  where the second quasi-isomorphism holds
  since $h$ is flat.
  Since
  \[
    \hocolim_{\raisebox{5pt}{\(\scriptstyle n\)}} \sF_n
  \]
  is an object in
  $\DD^+_{\mathscr{A}_X}(X)$, we are done with the proof of
  $(\ref{thm:weakserreequivgivesderived})$.\medskip
  \par It remains to show $(\ref{thm:weakserrecohgiveshypercoh})$.
  Let $\sF$ be an object in $\DD^b_{\mathscr{A}_X}(X)$.
  We proceed by induction on the length of $\sF$.
  If $\sF$ is concentrated in one degree, this follows from the assumption that
  $h$ preserves cohomology modules.
  In general, let $n$ be the smallest degree where $\mathbf{h}^n(\sF) \ne 0$,
  and consider the exact triangle
  \[
    \mathbf{h}^n(\sF)[-n] \longrightarrow \sF \longrightarrow \tau_{\ge n+1}\sF
    \xrightarrow{+1}.
  \]
  We then have the commutative diagram
  \[
    \begin{tikzcd}[column sep=scriptsize]
      \RR\Gamma\bigl(X,\mathbf{h}^n(\sF)[-n]\bigr) \rar\dar[sloped]{\sim}
      & \RR\Gamma(X,\sF) \rar\dar
      & \RR\Gamma\bigl(X,\tau_{\ge n+1}\sF\bigr) \rar{+1}\dar[sloped]{\sim}
      & {}\\
      \RR\Gamma\bigl(Y,h^*\mathbf{h}^n(\sF)[-n]\bigr) \rar &
      \RR\Gamma(Y,h^*\sF) \rar & \RR\Gamma\bigl(Y,h^*\tau_{\ge n+1}\sF\bigr) \rar{+1} &
      {}
    \end{tikzcd}
  \]
  where the left and right vertical arrows are quasi-isomorphism by the
  inductive hypothesis.
  By \cite[Chapitre II, Corollaire 1.2.3]{Ver67}, we see the middle vertical arrow is a
  quasi-isomorphism.
  The ``in particular'' statement in
  $(\ref{thm:weakserrecohgiveshypercoh})$ now follows by applying $\mathbf{H}^0$.
  The case when $\sF$ is an object in $\DD^+_{\mathscr{A}_X}(X)$ also follows
  since $(\ref{thm:weakserrecohgiveshypercoh})$ holds for $\tau_{\le n}\sF$ for
  all $n$.
\end{proof}
\subsection{Dualizing complexes and relative GAGA for semianalytic germs of
complex analytic spaces}
We first deduce the relative GAGA theorem for bounded derived categories of
semianalytic germs of complex analytic spaces from the statement for categories
of coherent sheaves in \cite{AT19}.
\begin{theorem}[cf.\ {\cite[Theorem C.1.1]{AT19}}]\label{thm:dbgagagerms}
  Let $(\cZ,Z)$ be an affinoid semianalytic germ of a complex analytic space
  with ring of global analytic functions $A$.
  Let $X$ be a projective scheme over $\Spec(A)$.
  Then, the pullback functor
  \begin{equation}\label{eq:derivedequiv}
    h^*\colon \DD^*_\coh(X) \longrightarrow
    \DD^*_\coh(X^\an)
  \end{equation}
  is an equivalence of categories that induces isomorphisms on $\RR\Gamma$,
  hypercohomology
  modules, $\RRHom$, and $\RRHHom$ for $* \in \{b,+\}$.
\end{theorem}
\begin{proof}
  We verify the hypotheses in Theorem
  \ref{thm:weakserreequivgivesderivedmain}$(\ref{thm:weakserreequivgivesderived})$
  and \ref{thm:weakserreequivgivesderivedmain}$(\ref{thm:weakserrecohgiveshypercoh})$
  for the relative analytification morphism $h\colon X^\an \to X$ from
  \cite[Appendix C]{AT19}
  when $\mathscr{A}_{X^\an} = \Coh(X^\an)$ and $\mathscr{A}_X = \Coh(X)$.
  By \cite[Theorem C.1.1]{AT19}, we have an equivalence of categories
  \begin{equation*}
    h^*\colon \Coh(X) \overset{\sim}{\longrightarrow} \Coh(X^\an)
  \end{equation*}
  that induces isomorphisms on cohomology modules.
  Since $h\colon X^\an \to X$ is flat \cite[p.\ 421]{AT19}, the natural
  morphisms
  \[
    h^*\EExt^n_{\cO_X}(\sF,\sG) \longrightarrow
    \EExt^n_{\cO_{X^\an}}(\sF^\an,\sG^\an)
  \]
  are isomorphisms for all objects $\sF,\sG$ in $\Coh(X)$ by \cite[Proposition
  12.3.5]{EGAIII1}.
  We therefore see that \eqref{eq:derivedequiv} is an equivalence
  by Theorem
  \ref{thm:weakserreequivgivesderivedmain}$(\ref{thm:weakserreequivgivesderived})$.
  This equivalence
  induces isomorphisms on $\RR\Gamma$, hypercohomology modules, $\RRHom$, and $\RRHHom$
  by Theorem
  \ref{thm:weakserreequivgivesderivedmain}$(\ref{thm:weakserrecohgiveshypercoh})$.
\end{proof}
We can now show that dualizing complexes are compatible with GAGA.
Below, the notation $\omega^\bullet_{-}$ where the subscript is a semianalytic
germ of a complex analytic space denotes the dualizing complex constructed in
\cite{RR71} (see Remark \ref{rem:omegaexistsrr71}).
\begin{theorem}\label{thm:dualizingcomplexcompatcomplex}
  Let $(\cZ,Z)$ be an affinoid semianalytic germ of a complex analytic space
  with ring of global analytic functions $A$.
  \begin{enumerate}[label=$(\roman*)$,ref=\roman*]
    \item\label{thm:dualizingcomplexcompatcomplexaffinoid}
      Let $\omega_A^\bullet$ denote the object in
      $\DD^b_\coh(\Spec(A))$ corresponding to $\omega_{Z}^\bullet$ under the
      equivalence in Theorem \ref{thm:dbgagagerms}.
      Then, $\omega_A^\bullet$ is a dualizing complex on $\Spec(A)$.
    \item\label{thm:rrvisexceptionalpullback}
      Let $f\colon Y \to X$ be a morphism of schemes projective over
      $\Spec(A)$.
      Then, we have the commutative diagram
      \begin{equation}\label{eq:rrvleftadjoint}
        \begin{tikzcd}[column sep=large,baseline=(hstarsim.base)]
          \DD^+_\coh(X^\an) \rar{(f^\an)^!}
          &\DD^+_\coh(Y^\an)\\
          \DD^+_\coh(X) \rar{f^!}\arrow[u, "h^*"'{name=hstarsim}, "\sim" sloped]
          &\DD^+_\coh(Y)\arrow[u, "h^*"', "\sim" sloped]
        \end{tikzcd}
      \end{equation}
      of functors.
      Here,
      \[
        (f^\an)^! \coloneqq \RRHHom_{\cO_{Y^\an}}\Bigl(\LL
        f^{\an*}\RRHHom_{\cO_{X^\an}}\bigl(-,\omega_{X^\an}^\bullet\bigr),
        \omega_{Y^\an}^\bullet\Bigr)
      \]
      is the exceptional pullback functor from \emph{\cite{RRV71}}.
    \item\label{thm:dualizingcomplexcompatcomplexexcpullback}
      Let $f\colon Y \to X$ be a morphism of schemes projective over $\Spec(A)$.
      We have $(f^!\omega_X^\bullet)^\an \cong \omega_{Y^\an}^\bullet$, and the
      analytification of the Grothendieck trace $\RR f_*
      \omega_Y^\bullet \to \omega_X^\bullet$ is the relative trace from
      \emph{\cite{RRV71}}.
  \end{enumerate}
\end{theorem}
\begin{proof}
  For $(\ref{thm:dualizingcomplexcompatcomplexaffinoid})$, we first note that
  $\Spec(A)$ is Noetherian of finite Krull dimension
  \citeleft\citen{Fri67}\citemid Th\'eor\`eme I, 9\citepunct
  \citen{AT19}\citemid Lemma B.6.1$(i)$\citeright.
  Thus, it
  suffices to show that $\omega_A^\bullet$ is locally a dualizing complex at
  every $x \in \Spec(A)$ by
  \cite[Chapter V, Proposition 8.2]{Har66} (see also \cite[p.\ 120]{Con00}).
  Moreover, it suffices to show that
  $\omega_A^\bullet$ is locally a dualizing complex at every closed point $x \in
  \Spec(A)$ by \cite[Chapter V, Corollary 2.3]{Har66}.
  But this follows from the conditions in Definition
  \ref{def:complexanalyticdualizing} together with the fact that $h$ is flat
  \cite[Lemma B.6.1$(iv)$]{AT19} and induces a bijection on closed points
  \cite[Lemma B.6.1$(iii)$]{AT19},
  since finite injective dimension can be tested with modules of the form
  $\Ext^i_{\cO_{X,x}}(\CC,-)$
  \cite[\href{https://stacks.math.columbia.edu/tag/0AVJ}{Tag
  0AVJ}]{stacks-project}, and both the formation of $\Ext$ and $\RRHHom$ commute
  with $h^*$ by Theorem \ref{thm:dbgagagerms}.
  Note here that while the condition in Definition
  \ref{def:complexanalyticdualizing}$(\ref{def:complexanalyticdualizingreflexive})$
  is a statement about functors on $\DD_\coh(X)$, it suffices to check that the
  morphism is an isomorphism when plugging in $\cO_X$ (resp.\ $\cO_{X^\an}$) by
  \cite[Chapter V, Proposition 2.1]{Har66} (resp.\ \cite[Proposition 1]{RR71}).\smallskip
  \par $(\ref{thm:rrvisexceptionalpullback})$
  We apply Grothendieck duality
  for proper morphisms of complex analytic spaces \cite[p.\ 261]{RRV71} to
  \(f^\an\colon Y^\an \to X^\an\).
  Restricting to $Z$ using the proper base change theorem
  from topology \cite[Chapter VII, Corollary 1.5]{Ive86}, we obtain the
  isomorphism
  \begin{align*}
    \MoveEqLeft[3]\RR f^\an_*
    \RRHHom_{\cO_{Y^\an}}\Bigl(\sF^\bullet,\RRHHom_{\cO_{Y^\an}}\bigl(\LL
    f^{\an*}\RRHHom_{\cO_{X^\an}}(\sG^\bullet,\omega_{X^\an}^\bullet),\omega_{Y^\an}^\bullet
    \bigr)\Bigr)\\
    &\overset{\sim}{\longrightarrow} \RRHHom_{\cO_{X^\an}}(\RR
    f_*^\an\sF^\bullet,\sG^\bullet).
  \end{align*}
  This isomorphism is
  natural in objects $\sF^\bullet$ and $\sG^\bullet$ in $\DD_\coh(Y^\an)$ and
  $\DD^+_\coh(X^\an)$, respectively.
  In particular, this isomorphism holds for $\sF^\bullet$ in
  $\DD^+_\coh(Y^\an)$.
  Taking $\mathbf{H}^0$ and
  applying the equivalence of categories $h^*$ from Theorem \ref{thm:dbgagagerms},
  we see that the top functor in \eqref{eq:rrvleftadjoint} is a right adjoint of
  $\RR f_*$, which preserves $\DD^+_\coh$ by \cite[Theorem 5.20]{PY16}.
  Finally, we obtain the diagram \eqref{eq:rrvleftadjoint} since right adjoints
  are unique.\smallskip
  \par We now show $(\ref{thm:dualizingcomplexcompatcomplexexcpullback})$.
  By $(\ref{thm:rrvisexceptionalpullback})$, it
  suffices to note that
  \begin{align*}
    \MoveEqLeft[3]\RRHHom_{\cO_{Y^\an}}\bigl(\LL
      f^{\an*}\RRHHom_{\cO_{X^\an}}(
    \omega_{X^\an}^\bullet,\omega_{X^\an}^\bullet),\omega_{Y^\an}^\bullet\bigr)\\
    &\cong \RRHHom_{\cO_{Y^\an}}\bigl(\LL
    f^{\an*}\cO_{X^\an},\omega_{Y^\an}^\bullet\bigr)\\
    &\cong \RRHHom_{\cO_{Y^\an}}\bigl(\cO_{Y^\an},\omega_{Y^\an}^\bullet\bigr)\\
    &\cong \omega_{Y^\an}^\bullet.
  \end{align*}
  The last statement about trace now follows since in both settings, the trace
  is the counit morphism for the adjunction from
  $(\ref{thm:rrvisexceptionalpullback})$, where in the scheme case we are using
  \citeleft\citen{Har66}\citemid Appendix, Th\'eor\`eme 2\citepunct
  \citen{Ver}\citemid Theorem 1\citepunct \citen{Nee96}\citemid Proposition
  6.3\citepunct \citen{Lip09}\citemid Theorem 4.1.1\citeright.
\end{proof}
\subsection{Dualizing complexes and relative GAGA for non-Archimedean analytic
spaces}
We first deduce the relative GAGA theorem for bounded derived categories of
rigid analytic spaces, Berkovich spaces, and adic from the statements for
categories of coherent sheaves in \cite{Kop74,Poi10,Hub07} (see also
\cite{Con06,Hal}).
Stronger results for Berkovich spaces are shown in \cite[Theorem 7.1 and Corollary
7.5]{PY16}.
\begin{theorem}[{cf.\ \citeleft\citen{Kop74}\citemid Folgerung 6.6, Folgerung 6.7, and
  Theorem 6.8\citepunct \citen{Poi10}\citemid Th\'eor\`eme
  A.1\citepunct \citen{Hub07}\citemid Corollary
  6.4\citepunct \citen{Zav}\citemid Lemma 6.7\citeright}]\label{thm:derivedgagarigid}
  Let $Z$ be one of the following:
  \begin{enumerate}[label=$(\alph*)$,ref=\alph*]
    \item An affinoid rigid $k$-analytic space, where $k$ is a complete
      non-trivially valued non\hyph{}Archimedean field.
    \item An affinoid $k$-analytic space, where $k$ is a complete
      non-Archimedean field.
    \item\label{thm:derivedgagarigidadic}
      An affinoid analytic adic space such that one of the following
      conditions holds:
      \begin{itemize}
        \item $\cO_Z(Z)$ has a Noetherian ring of definition.
        \item $\cO_Z(Z)$ is strongly Noetherian.
      \end{itemize}
  \end{enumerate}
  Let $R$ be the ring of global functions on $Z$, and let $X$ be a proper scheme
  over $\Spec(R)$.
  Then, the pullback functor
  \begin{equation}\label{eq:derivedequivnonarch}
    h^*\colon \DD^*_\coh(X) \longrightarrow
    \DD^*_\coh(X^\an)
  \end{equation}
  is an equivalence of categories that induces isomorphisms on $\RR\Gamma$,
  hypercohomology
  modules, $\RRHom$, and $\RRHHom$ for $* \in \{b,+\}$.
\end{theorem}
\begin{proof}
  We verify the hypotheses in Theorem
  \ref{thm:weakserreequivgivesderivedmain}$(\ref{thm:weakserreequivgivesderived})$
  and \ref{thm:weakserreequivgivesderivedmain}$(\ref{thm:weakserrecohgiveshypercoh})$
  for the relative analytification morphism $h\colon X^\an \to X$ from
  \cite[Definition 1.4]{Kop74} (see also \cite[Example 2.2.11]{Con06}),
  \cite[\S2.6]{Ber93}, and \citeleft\citen{Hub07}\citemid \S6\citepunct
  \citen{Zav}\citemid \S6\citeright\ 
  when $\mathscr{A}_{X^\an} =
  \Coh(X^\an)$ and $\mathscr{A}_X = \Coh(X)$.
  \par By \cite[Folgerung 6.6, Folgerung 6.7, and Theorem 6.8]{Kop74} (see also
  \cite[Example 3.2.6]{Con06}), \cite[Th\'eor\`eme A.1]{Poi10}, and
  \citeleft\citen{Hub07}\citemid Corollary 6.4\citepunct \citen{Zav}\citemid
  Lemma 6.9\citeright,
  respectively, we have an equivalence of
  categories
  \begin{equation}\label{eq:rigidcohequiv}
    h^*\colon \Coh(X) \overset{\sim}{\longrightarrow} \Coh(X^\an)
  \end{equation}
  that induces isomorphisms on cohomology modules (see also \cite[Example 9.4]{Hal}).
  We note that $h^*$ induces isomorphisms on $\EExt$ sheaves by
  \cite[Satz 3.9]{Kop74} in the rigid analytic case and by \cite[Proposition
  12.3.5]{EGAIII1} in the Berkovich and adic cases since $h$ is flat
  \citeleft\citen{Ber93}\citemid Proposition 2.6.2\citepunct
  \citen{Hub07}\citemid Lemma 6.1\citepunct
  \citen{Zav}\citemid Lemma 6.7\citeright.
  We therefore see that \eqref{eq:derivedequivnonarch} is an equivalence
  by Theorem
  \ref{thm:weakserreequivgivesderivedmain}$(\ref{thm:weakserreequivgivesderived})$.
  Finally, \eqref{eq:derivedequivnonarch} induces isomorphisms on $\RR\Gamma$,
  hypercohomology modules, $\RRHom$, and $\RRHHom$
  by Theorem
  \ref{thm:weakserreequivgivesderivedmain}$(\ref{thm:weakserrecohgiveshypercoh})$.
\end{proof}
\par To understand the behavior of dualizing complexes under analytification for
adic spaces, we prove the
following lemma.
This will also be used later on to prove that singularities of
pairs are compatible with GAGA (Lemma \ref{lem:kltgaga}).
The statement below is an adic version of \cite[Th\'eor\`eme
3.3]{Duc09}, where Ducros proves the Berkovich version of this result.
\begin{theorem}\label{lem:stalksonjgxaregrings}
  Let \(X = \Spa(A,A^+)\) be an affinoid analytic adic space such that \(A\) is
  topologically of finite type over a complete non-trivially valued
  non-Archimedean field \(k\).
  Let
  \[
    Y \coloneqq \Spec(B) \overset{f}{\longrightarrow} \Spec(A)
  \]
  be a finite type morphism of affine schemes.
  Consider the Cartesian diagram
  \begin{equation}\label{eq:adfunctorcart}
    \begin{tikzcd}[baseline=(fmap.base)]
      Y^\an \rar{h}\dar[swap]{f^\an} & Y\arrow[d,"f"{name=fmap}]\\
      X \rar & \Spec(A)
    \end{tikzcd}
  \end{equation}
  of locally ringed spaces.
  Let \(V \coloneqq \Spa(C,C^+) \subseteq Y^\an\) be an open affinoid subset.
  Consider a point
  \[
    y \in \JG\bigl(\Spa(C,C^+)\bigr)
  \]
  and set \(\fn = \supp(y)\).
  In the commutative diagram
  \begin{equation}\label{eq:ducros33foradic}
    \begin{tikzcd}[row sep=2.4em,column sep=3.2em]
      \cO_{V,y} & C_{\fn} \lar[swap]{\varphi_1}\\
      \cO_{Y^\an,y} \uar{\varphi_2}[swap,sloped]{\sim}
      & \cO_{Y,h(y)} \lar[swap]{\varphi_3} \uar[swap,yshift=-1pt]{\varphi_4}
      \arrow{ul}[description]{\varphi_5}
    \end{tikzcd}
  \end{equation}
  of Noetherian local rings, every map is faithfully flat with geometrically
  regular fibers.
\end{theorem}
The diagram \eqref{eq:adfunctorcart} is Cartesian by the definition of
the analytification functor in \cite[p.\ 999]{Hub07} (where \((-)^\an\) is
denoted by \((-)^\mathrm{ad}\) and \(h\) is denoted by \(p\)).
\begin{proof}
  Every ring in \eqref{eq:ducros33foradic} is local.
  Thus, to show faithful flatness of the maps in \eqref{eq:ducros33foradic}, it
  suffices to show flatness.
  We know that \(\varphi_1\) is flat by Proposition
  \ref{prop:lou33}\((\ref{prop:lou332})\), \(\varphi_2\) is an isomorphism by
  the description of stalks for the structure sheaf of an adic space in \cite[p.\
  520]{Hub94}, and \(\varphi_3\) is flat by \cite[Lemma 6.1]{Hub07}.
  Thus, \(\varphi_4\) and \(\varphi_5\) are also flat by the commutativity of
  \eqref{eq:ducros33foradic}.
  \par To show that \(\varphi_1\) has geometrically regular fibers, consider the
  commutative diagram
  \[
    \begin{tikzcd}
      \hat{\cO}_{V,y} & (C_\fn)^\wedge \lar[swap]{\sim}\\
      \cO_{V,y} \uar & C_\fn \lar[swap]{\varphi_1} \uar
    \end{tikzcd}
  \]
  where the top horizontal map is an isomorphism by Proposition
  \ref{prop:lou33}\((\ref{prop:lou332})\).
  The right vertical map has geometrically regular fibers by
  \citeleft\citen{Kie69}\citemid Theorem 3.3\citepunct \citen{Con00}\citemid
  \S1.1\citeright.
  Thus, \(\varphi_1\) has geometrically regular fibers by \cite[Proposition
  6.8.3\((ii)\)]{EGAIV2}.
  \par Next, we show that \(\varphi_3\) has geometrically regular fibers.
  We have the Cartesian diagram
  \[
    \begin{tikzcd}
      Y^\an \rar \dar[hook] & Y \dar[hook]\\
      \bA_A^{n,\an} \rar{h} & \bA_A^n
    \end{tikzcd}
  \]
  where the vertical maps are closed immersions.
  We then have the co-Cartesian diagram
  \[
    \begin{tikzcd}
      \cO_{Y^\an,y} & \lar \cO_{Y,y}\\
      \cO_{\bA_A^{n,\an},y} \uar[twoheadrightarrow]
      & \lar \cO_{\bA_A^{n},y}\mathrlap{.} \uar[twoheadrightarrow]
    \end{tikzcd}
  \]
  It therefore suffices to show that the bottom horizontal map has
  geometrically regular fibers.
  After choosing an appropriate polydisc containing \(y\), we have the
  commutative diagram
  \[
    \begin{tikzcd}
      A\langle X_1,X_2,\ldots,X_n \rangle_{\fm_y} \dar
      & \lar A[X_1,X_2,\ldots,X_n]_{\fm_{h(y)}}\dar\\
      \cO_{\bA_A^{n,\an},y} \dar & \cO_{\bA_A^{n},h(y)} \lar \dar\\
      \hat{\cO}_{\bA_A^{n,\an},y} & \hat{\cO}_{\bA_A^{n},h(y)} \lar[swap]{\sim}
    \end{tikzcd}
  \]
  where \(\fm_y \subseteq A\langle X_1,X_2,\ldots,X_n \rangle\) is a maximal
  ideal.
  The composition along the left column is the completion map by
  Proposition \ref{prop:lou33}\((\ref{prop:lou332})\), and the composition along
  the right column is also the completion map.
  The bottom horizontal map is an isomorphism since both rings are isomorphic to
  \[
    \bigl(A_{\fm_{h(y)} \cap A}\bigr)^\wedge\,
    \llbracket X_1,X_2,\ldots,X_n \rrbracket.
  \]
  Since \(A\) is excellent \citeleft\citen{Kie69}\citemid Theorem 3.3\citepunct
  \citen{Con00}\citemid \S1.1\citeright, the map
  \begin{align*}
    \cO_{\bA_A^{n},h(y)} &\longrightarrow \hat{\cO}_{\bA_A^{n},h(y)}
    \intertext{is flat with geometrically regular fibers.
    Since \(\cO_{\bA_A^{n,\an},y}\) is Noetherian \cite[Proposition
    3.3.16\((i)\)]{Hub93thesis}, the map}
    \cO_{\bA_A^{n,\an},y} &\longrightarrow \hat{\cO}_{\bA_A^{n,\an},y}
  \end{align*}
  is flat.
  We conclude that
  \(\cO_{\bA_A^{n},h(y)} \to \cO_{\bA_A^{n,\an},y}\) is geometrically regular by
  \cite[Proposition 6.8.3\((ii)\)]{EGAIV2}.
  \par To complete the proof, the fact that \(\varphi_2\) is an isomorphism and
  \(\varphi_3\) is flat with geometrically regular fibers implies \(\varphi_5\)
  is flat with geometrically regular fibers.
  Finally, since \(\varphi_5\) is flat with geometrically regular fibers and
  \(\varphi_1\) is flat, we conclude that \(\varphi_4\) has geometrically
  regular fibers by \cite[Proposition 6.8.3\((ii)\)]{EGAIV2}.
\end{proof}
We can now show that dualizing complexes are compatible with GAGA.
For the definition of irreducible components, dimension, and equidimensionality,
see \citeleft\citen{CM98}\citemid p.\
14\citepunct \citen{Con99}\citemid p.\ 496\citeright,
\citeleft\citen{Ber90}\citemid p.\ 34\citepunct \citen{Ber93}\citemid p.\
23\citepunct \citen{Duc09}\citemid p.\ 1455\citeright,
and \cite[Definition 1.8.1]{Hub96},
respectively.
For the definition of smoothness, see \cite[Definition 2.1]{BLR95},
\cite[Definition 3.5.1]{Ber93}, and \cite[Definition 1.6.5]{Hub96}, respectively.
\par In the Berkovich setting, since $X^\an$ and $Z$ are good, smoothness is
equivalent to ``quasi-smooth and boundaryless'' as defined in \cite[Definition
5.2.4]{Duc18} (see \cite[Corollary 5.4.8]{Duc18}).
Note that smoothness in the Berkovich setting is not known to be G-local on the
target without goodness assumptions,
whereas ``quasi-smooth and boundaryless'' is always G-local on the target
\cite[Remark 4.1]{CT21}.
\par For adic spaces,
as mentioned in the paragraph before Definition \ref{def:dualizingcomplexadic},
we will eventually restrict to the class of adic spaces locally of
weakly finite type (see \cite[Definition 1.2.1$(i)$]{Hub96}) over a complete
non-trivially valued non-Archimedean field $k$.
We will do the same here for simplicity.
\begin{theorem}\label{thm:dualizingcomplexcompatrigid}
  Let $Z$ be one of the following:
  \begin{enumerate}[label=$(\alph*)$,ref=\alph*]
    \item\label{thm:dualizingcomplexcompatrigidtate}
      An affinoid rigid $k$-analytic space, where $k$ is a complete
      non-trivially valued non\hyph{}Archimedean field.
    \item\label{thm:dualizingcomplexcompatrigidberk}
      An affinoid Berkovich $k$-analytic space, where $k$ is a complete
      non-Archimedean field.
    \item\label{thm:dualizingcomplexcompatrigidadic}
      An affinoid analytic adic space \(\Spa(A,A^+)\) such that \(A\) is
      topologically of finite type over a complete
      non-trivially valued non\hyph{}Archimedean field \(k\).
  \end{enumerate}
  Let $A$ be the ring of global functions on $Z$.
  Let $\pi\colon X \to \Spec(A)$ be a finite type morphism of schemes.
  \begin{enumerate}[label=$(\roman*)$,ref=\roman*]
    \item\label{thm:dualizingcomplexcompatrigidaffinoid}
      Let $\sK$ be an object in $\DD^+_\coh(X)$.
      Then, $\sK$ is a dualizing complex on $X$ if and only if $\sK^\an$ is a
      dualizing complex on $X^\an$.
    \item\label{thm:dualizingcomplexcompatrigidexcpullback}
      Suppose $\pi$ is separated, and
      let $\omega_Z^\bullet$ be a dualizing complex on $Z$.
      Then, $(\pi^!\omega_Z^{\bullet\al})^\an$ is a 
      dualizing complex on $X^\an$.
    \item\label{thm:dualizingcomplexcompatrigidexcpullbacksmooth}
      Suppose $\pi$ is separated.
      If $X^\an$ is smooth of pure dimension $d$ over $k$, then the sheaf
      $\omega_{X^\an/k}[d]$ of top differential forms shifted by $d$ is a 
      dualizing complex on $X^\an$ for which there
      exists a 
      dualizing complex $\omega_Z^\bullet$ on $Z$ such that
      \[
        \omega_{X^\an/k}[d] \cong
        (\pi^!\omega_Z^{\bullet\al})^\an.
      \]
  \end{enumerate}
\end{theorem}
\begin{proof}
  For $(\ref{thm:dualizingcomplexcompatrigidaffinoid})$, we note that by
  \citeleft\citen{Har66}\citemid Chapter V, Corollary 2.3\citepunct
  \citen{Con00}\citemid p.\ 120\citeright,
  $\sK$ is a dualizing complex on $X$ if and only if $\sK_x$ is a dualizing
  complex on $\cO_{X,x}$ for every $x \in X$.
  \par Next, we note that in each context, $X^\an \to X$ satisfies the following
  set-theoretic properties:
  \begin{enumerate}[label=$(\alph*)$,ref=\alph*]
    \item[{$(\ref{thm:dualizingcomplexcompatrigidtate})$}] $X^\an \to X$ is
      a bijection onto the set of points of $X$ with residue fields of
      finite degree over $k$ by \cite[Example 2.2.11]{Con06}.
    \item[{$(\ref{thm:dualizingcomplexcompatrigidberk})$}]
      $X^\an \to X$ is a surjection that induces a bijection
      \[
        (X^\an)_0 \coloneqq \Set[\big]{x \in X^\an \given [\sH(x):k] < \infty}
        \xrightarrow{\mathrm{bij.}}
        \Set[\big]{x \in X \given [k(x):k] < \infty} \eqqcolon X_0
      \]
      by \cite[Proposition 2.6.2 and Lemma 2.6.3]{Ber93}.
    \item[{$(\ref{thm:dualizingcomplexcompatrigidadic})$}]
      Every closed point of $X$ is contained in the image of $X^\an \to X$
      by \cite[p.\ 1000]{Hub07}.
  \end{enumerate}
  Combined with the previous paragraph,
  it therefore suffices to prove that for every point $\tilde{x} \in X^\an$ with
  image $x = h(\tilde{x}) \in X$, we have $\sK^\an_{\tilde{x}}$ is a dualizing
  complex on $\cO_{X^\an,\tilde{x}}$ if and only if $\sK_x$ is a dualizing
  complex on $\cO_{X,x}$.
  This equivalence holds by \cite[Theorem 5.1]{AF92}
  since $\cO_{X,x} \to \cO_{X^\an,\tilde{x}}$ is a
  regular ring map \cite[Th\'eor\`eme 3.3]{Duc09} (which also applies to
  rigid $k$-analytic spaces using \cite[Theorem 1.6.1]{Ber93}) in cases
  $(\ref{thm:dualizingcomplexcompatrigidtate})$ and
  $(\ref{thm:dualizingcomplexcompatrigidberk})$, and by Theorem
  \ref{lem:stalksonjgxaregrings} in case
  $(\ref{thm:dualizingcomplexcompatrigidadic})$.\smallskip
  \par Next, $(\ref{thm:dualizingcomplexcompatrigidexcpullback})$ follows from 
  $(\ref{thm:dualizingcomplexcompatrigidaffinoid})$ since 
  $\pi^!\omega_Z^{\bullet\al}$ is a dualizing complex on $X$ by Lemma
  \ref{lem:dualizingcomplexpullback}.\smallskip
  \par Finally, we show
  $(\ref{thm:dualizingcomplexcompatrigidexcpullbacksmooth})$.
  Since $Z$ is affinoid, there exists a surjection
  \[
    k\{r^{-1}T\} \longtwoheadrightarrow A,
  \]
  where in the rigid analytic case and the adic
  case, we can assume $r = (1,1,\ldots,1)$.
  Let $i\colon \Spec(A) \hookrightarrow \Spec(k\{r^{-1}T\})$ be the associated
  closed immersion
  with associated closed immersion $i^\an\colon Z
  \hookrightarrow D$ of rigid $k$-analytic spaces or $k$-analytic spaces.
  We can replace $\pi$ by $i \circ \pi$ to assume that $Z = D$:
  We have
  $\pi^!i^! \cong (i \circ \pi)^!$ by \cite[Chapter VII, Corollary
  3.4$(a)$]{Har66}. Thus, if $\omega_D^\bullet$ works for $i \circ \pi$, then
  $(i^!\omega_D^{\bullet\al})^\an$ works for $\pi$.
  \par We now prove $(\ref{thm:dualizingcomplexcompatrigidexcpullbacksmooth})$
  assuming $Z$ is a polydisc with ring of analytic functions $A = k\{r^{-1}T\}$.
  By \cite[p.\ 144]{Har66}, we have
  \begin{align*}
    \pi^!\omega_{k\{r^{-1}T\}/k}[n] &= \pi^*\omega_{k\{r^{-1}T\}/k}[n]
    \otimes_{\cO_X} \omega_{X/k\{r^{-1}T\}}[d-n].
    \intertext{Applying $(-)^\an$, we obtain}
    \bigl(\pi^!\omega_{k\{r^{-1}T\}/k}[n]\bigr)^\an
    &= \bigl((\pi^*\omega_{k\{r^{-1}T\}/k})^\an
    \otimes_{\cO_{X^\an}} \omega_{X/k\{r^{-1}T\}}^\an\bigr)[d]
  \end{align*}
  since sheaves of differentials are compatible with analytification
  \cite[Proposition 3.3.11]{Ber93}.
  The right-hand side is isomorphic to $\omega_{X^\an/k}[d]$ by
  taking determinants in \cite[Corollary
  3.5.10]{Ber93}.
  Thus, by $(\ref{thm:dualizingcomplexcompatrigidaffinoid})$ and
  $(\ref{thm:dualizingcomplexcompatrigidexcpullback})$, we can take
  $\omega_Z^\bullet = (\omega_{k\{r^{-1}T\}/k}[n])^\an$, where
  $\omega_{k\{r^{-1}T\}/k}[n]$ is a dualizing complex by
  \cite[Chapter V, Example 2.2 and Theorem 3.1]{Har66}.
\end{proof}
\begin{remark}\label{rem:grothendieckdualityrigid}
  Theorem \ref{thm:dualizingcomplexcompatrigid} implies that
  the formation of dualizing complexes and Grothendieck duality are compatible
  with GAGA and existing results for Grothendieck duality on rigid analytic
  spaces over an affinoid rigid analytic space $Z$.
  With assumptions as in Theorem
  \ref{thm:dualizingcomplexcompatrigid}$(\ref{thm:dualizingcomplexcompatrigidtate})$,
  Van der Put \cite[Main Theorem 5.1]{vdP92} (see also
  \citeleft\citen{Bey97}\citemid Theorem 5.1.1 and 5.1.2\citepunct
  \citen{AL}\citemid Theorem 5.5.1\citeright)
  showed that if $Y$ is a rigid analytic space that
  is smooth and proper over $Z$, then the sheaf $\omega_{Y/Z}$ satisfies the
  statement of Serre duality.
  If $Y$ is the analytification of a scheme $X$ that is proper over $k$
  (which is necessarily smooth by \cite[Proposition 3.5.8]{Ber93}),
  then Van der Put's results are compatible with GAGA by Theorem
  \ref{thm:dualizingcomplexcompatrigid}$(\ref{thm:dualizingcomplexcompatrigidexcpullbacksmooth})$.
  The compatibility of the trace morphism follows by the same argument using
  the uniqueness of adjoint functors as in the proof of Theorem
  \ref{thm:dualizingcomplexcompatcomplex}$(\ref{thm:dualizingcomplexcompatcomplexexcpullback})$.
\end{remark}
\subsection{Dualizing complexes and relative GAGA for adic spaces}
\label{csdiscussion}
In the adic case, we can prove the full analogue of Theorem
\ref{thm:dualizingcomplexcompatcomplex} using forthcoming work of Clausen and
Scholze.\medskip
\par In private communication, Peter Scholze informed us that in forthcoming work,
Clausen and Scholze prove Grothendieck duality for adic spaces
using condensed mathematics and the same proof as in their lecture
notes on complex geometry \cite{CS22} (see also \cite{CS19}).
For adic spaces locally of weakly finite type over a field, the exact statement
we will need by Clausen and Scholze (which is a special case of their results)
is the following:
\begin{theorem}[{Clausen and Scholze, forthcoming; cf.\
  \citeleft\citen{CS19}\citemid Lecture XI, Theorem 11.1, p.\ 74, and
  Remark 11.7\citepunct 
  \citen{CS22}\citemid Lecture XII, Theorem 12.18 and
  Lecture XIII, pp.\ 120--121\citepunct
  \citen{Man22}\citemid Proposition 2.9.31\citeright}]
  \label{thm:clausenscholze}
  Let $f\colon Y \to X$ be a proper morphism between adic spaces that are
  separated and locally of weakly finite type over a complete non-trivially valued
  non-Archimedean field $k$.
  Then, the functor $\RR f_*$ preserves $\DD^+_\coh$, and
  there exists a functor
  \begin{alignat*}{3}
    f^!\colon{}& \DD^+_\coh(X) &{}\longrightarrow{}& \DD^+_\coh(Y)
    \intertext{such that $(f \circ g)^! \cong g^! \circ f^!$, and such that $f^!$ is
    the right adjoint to the functor}
    \RR f_*\colon{}& \DD^+_\coh(Y) &{}\longrightarrow{}& \DD^+_\coh(X).
  \end{alignat*}
\end{theorem}
\par Using the Grothendieck duality theorem of Clausen and Scholze, we can show the
following:
\begin{theorem}[{relies on Theorem \ref{thm:clausenscholze}}]
  \label{thm:dualizingcomplexcompatadic}
  Let $Z$ be an affinoid analytic adic space such that $A \coloneqq \cO_Z(Z)$ is
  topologically of finite type over a complete non-trivially valued
  non-Archimedean field $k$.
  Let $\pi\colon X \to \Spec(A)$ be a finite type morphism of schemes.
  \begin{enumerate}[label=$(\roman*)$,ref=\roman*]
    \item\label{thm:dualizingcomplexcompatadicexcpullback}
      Let $f\colon Y \to X$ be a morphism of schemes proper over $\Spec(A)$.
      We then have the following commutative diagram of functors:
      \begin{equation}\label{eq:csleftadjoint}
        \begin{tikzcd}[column sep=large]
          \DD^+_\coh(X^\an) \rar{f^{\an!}}
          &\DD^+_\coh(Y^\an)\\
          \DD^+_\coh(X) \rar{f^!}\arrow[u, "h^*"', "\sim" sloped]
          &\DD^+_\coh(Y)\mathrlap{.}\arrow[u, "h^*"', "\sim" sloped]
        \end{tikzcd}
      \end{equation}
      In particular, $f^{\an!}$ sends dualizing complexes to dualizing
      complexes.
      Here, $f^{\an!}$ is the exceptional pullback functor which exists by
      Theorem \ref{thm:clausenscholze}.
    \item\label{thm:dualizingcomplexcompatadicexcpullbacktrace}
      Let $f\colon Y \to X$ be a morphism of schemes proper over \(\Spec(A)\).
      The analytification of the Grothendieck trace $\RR f_*
      f^! \to \id$ of functors $\DD^+_\coh(X) \to \DD^+_\coh(X)$
      is the relative trace which exists by Theorem \ref{thm:clausenscholze}.
  \end{enumerate}
\end{theorem}
\begin{proof}
  We first note that our GAGA result in Theorem \ref{thm:derivedgagarigid}
  applies by setting $D = A$ and $f = 1$ in the second bullet point of
  Theorem
  \ref{thm:derivedgagarigid}$(\ref{thm:derivedgagarigidadic})$.
  \par For $(\ref{thm:dualizingcomplexcompatadicexcpullback})$, 
  it suffices to
  note that both $f^{\an!}$ and $h^* \circ f^! \circ h^{*-1}$
  are right adjoints for $\RR f^\an_*$
  using the
  equivalence of categories $h^*$ from Theorem \ref{thm:dbgagagerms} and the
  uniqueness of right adjoints.
  Here, we are using Theorem \ref{thm:clausenscholze}
  on the adic side and are using
  \citeleft\citen{Har66}\citemid Appendix, Th\'eor\`eme 2\citepunct
  \citen{Ver}\citemid Theorem 1\citepunct \citen{Nee96}\citemid Proposition
  6.3\citepunct \citen{Lip09}\citemid Theorem 4.1.1\citeright\ on the scheme
  side.
  \par For $(\ref{thm:dualizingcomplexcompatadicexcpullbacktrace})$, it suffices
  to note that the trace is the counit morphism for the adjunction from
  $(\ref{thm:dualizingcomplexcompatadicexcpullback})$.
\end{proof}
\section{Setup for the relative MMP with scaling}\label{sect:setupforothercats}
We now give our setup for the relative MMP with scaling in categories other than
schemes and algebraic spaces.
We have made an effort to make definitions consistent with those in the
literature.
\subsection{Categories of spaces}
\par We will work in the following categories of spaces.
We have included $(\ref{setup:algebraicspaces})$ to simplify our discussion in
the rest of this section, although the necessary preliminaries are already
covered in Part \ref{part:prelim}.
\begin{setup}[cf.\ {\cite[\S6.2.1]{AT19}}]\label{setup:spaces}
  A \textsl{category of spaces} is one of the following categories.
  \begin{enumerate}[label=$(\textup{\Roman*})$,ref=\textup{\Roman*}]
  \item[$(0)$]
  \makeatletter
  \protected@edef\@currentlabel{0}
  \phantomsection
  \label{setup:algebraicspaces}
  \makeatother
    The category of quasi-excellent Noetherian algebraic spaces over a scheme
    $S$ admitting dualizing complexes.
    \item\label{setup:formalqschemes}
      The category of quasi-excellent Noetherian formal schemes admitting $c$-dualizing complexes.
    \makeatletter
    \item\label{setup:complexanalyticgerms}
      The category of semianalytic germs $X = (\mathcal{X},X)$ of complex
      analytic spaces.
    \item
      \label{setup:berkovichspaces}
      The category of $k$-analytic spaces, where $k$ is a complete non-Archimedean field.
    \makeatletter
    \item[{$(\ref*{setup:berkovichspaces}')$}]
    \protected@edef\@currentlabel{\ref*{setup:berkovichspaces}'}
    \phantomsection
    \label{setup:rigidanalyticspaces}
    \makeatother
      The category of rigid $k$-analytic spaces, where $k$ is a complete
      non-trivially valued non-Archimedean field.
    \item\label{setup:adicspaces}
      The category of locally Noetherian analytic adic spaces that have an open
      affinoid covering by affinoids \(Z\) satisfying one of the following
      conditions:
      \begin{itemize}
        \item $\cO_Z(Z)$ has a Noetherian ring of definition.
        \item $\cO_Z(Z)$ is strongly Noetherian.
      \end{itemize}
  \end{enumerate}
  We denote any such category by $\Sp$.
  A \textsl{space} is an object in $\Sp$.
  \par A \textsl{category of $\QQ$-spaces} is a space as above,
  except in $(\ref{setup:algebraicspaces})$, $(\ref{setup:formalqschemes})$,
  and $(\ref{setup:adicspaces})$,
  we assume that the spaces are over $\Spec(\QQ)$,
  and in
  $(\ref{setup:berkovichspaces})$ and
  $(\ref{setup:rigidanalyticspaces})$
  we assume that the field $k$ is of
  characteristic zero.
  We denote any such category by $\Sp_\QQ$.
  A \textsl{$\QQ$-space} is an object in $\Sp_\QQ$.
\end{setup}
In each category,
there are good notions of affinoid subdomains,
admissible affinoid coverings,
regularity,
and smooth
morphisms \cite[\S6.2]{AT19}.
See also footnote \ref{note:affinoidsubdomain} on p.\ 
\pageref{note:affinoidsubdomain} for the notion of ``affinoid subdomain.''
For $(\ref{setup:adicspaces})$ (which is not covered in \cite{AT19}),
see \cite[Definition on p.\ 521]{Hub94} for the definition of affinoids and see
\cite[Definition 1.6.5$(i)$]{Hub96} for the definition of smooth morphisms.
An adic space locally of weakly finite type over a complete non-trivially valued
non-Archimedean field $k$ is \textsl{regular} if $X$ can be covered by affinoid
adic spaces of the form $\Spa(A,A^+)$ such that $A$ is regular (cf.\
\cite[Definition 2.3]{Man23}).
\par There is a relative GAGA theorem for proper schemes over
$\Spec(\cO_U(U))$ when $U$ is affinoid, which induces
equivalences on categories of coherent sheaves and isomorphisms on cohomology
modules (see \citeleft\citen{AT19}\citemid \S6.3\citepunct
\citen{Hub07}\citemid \S6\citepunct
\citen{Zav}\citemid \S6\citeright).
For spaces $X$ and schemes $X_0$ that these GAGA theorems apply to,
we use the notions of \textsl{analytification} and \textsl{algebraization} as in
Convention \ref{convention:analytification}.

\subsubsection{Ampleness and locally projective morphisms}
We have good notions of relative ampleness for the categories
$(\ref{setup:formalqschemes})$, $(\ref{setup:complexanalyticgerms})$, and
$(\ref{setup:rigidanalyticspaces})$.
We have adopted definitions for Berkovich and adic spaces that allow us to apply the
relative GAGA theorem in this setting.
These ample invertible sheaves correspond to ample invertible sheaves under
the GAGA correspondence (see \cite[Remark 3.1.3]{Con06} for
$(\ref{setup:rigidanalyticspaces})$).
\begin{definition}[cf.\ {\cite[Chapitre VIII, Remarque 2.3]{Hak72}}]
  \label{def:locproj}
  Let $\pi\colon X \to Z$ be a proper morphism in $\Sp$ in the sense of
  \cite[(3.4.1)]{EGAIII1},
  \cite[p.\ 91]{BS76} (with the adjustment to germs as in \cite[\S B.5]{AT19};
  see Definition
  \ref{def:complexgermmorphisms}$(\ref{def:complexgermmorphismsproj})$),
  \cite[Definition 9.6.2/2]{BGR84}, \cite[Example 1.5.3$(iii)$]{Ber93}, and
  \cite[Definition 1.3.2]{Hub96},
  respectively.
  Let $\sL$ be an invertible sheaf on $X$.
  We say that $\sL$ is \textsl{$\pi$-ample} in each setting of Setup
  \ref{setup:spaces} if the following conditions hold:
  \begin{enumerate}
    \item[$(\ref{setup:formalqschemes})$]
      For every affine open $\Spf(A) \subseteq Z$, if $I$ is the ideal of
      definition of $A$, then $\sL$ restricts to a relatively ample invertible
      sheaf on $X \times_Z \Spec(A/I)$ (see \cite[Th\'eor\`eme 5.4.5]{EGAIII1}).
    \item[$(\ref{setup:complexanalyticgerms})$]
      There exists a proper representative $\cX \to \cZ$ of $\pi$ such that
      $\pi^{-1}(Z) = X$, together with an invertible
      sheaf on $\cX$ restricting to $\sL$ on $X$ that is $\pi$-ample in the
      sense of \cite[p.\ 141]{BS76}.
    \item[$(\ref{setup:berkovichspaces})$]
      For every affinoid subdomain \(V \subseteq Z\), there exists an integer
      \(r \ge 0\) and a factorization
      \[
        \begin{tikzcd}
          \pi^{-1}(V) \rar[hook]\arrow{dr}[swap]{\pi\rvert_{\pi^{-1}(V)}}
          & \Bigl(\PP^r_{\cO_V(V)}\Bigr)^\an \dar\\
          & V
        \end{tikzcd}
      \]
      where \(\pi^{-1}(V) \hookrightarrow \PP^r_V\) is a
      closed immersion and \(\sL \cong (\cO(1))^\an\rvert_{\pi^{-1}(V)}\).
    \item[$(\ref{setup:rigidanalyticspaces})$]
      The invertible sheaf $\sL$ is ample relative to $Z$ in the sense of
      \cite[Definition 3.2.2]{Con06}.
    \item[$(\ref{setup:adicspaces})$]
      For every open affinoid subset \(V \subseteq Z\), there exists an integer
      \(r \ge 0\) and a factorization
      \[
        \begin{tikzcd}
          \pi^{-1}(V) \rar[hook]\arrow{dr}[swap]{\pi\rvert_{\pi^{-1}(V)}}
          & \PP_V\Bigl(\cO_V^{\oplus (r+1)}\Bigr) \dar\\
          & V
        \end{tikzcd}
      \]
      where \(\pi^{-1}(V) \hookrightarrow \PP^r_V\) is a
      closed immersion and \(\sL \cong \cO(1)\rvert_{\pi^{-1}(V)}\).
      Here, \(\PP_V(-)\) and \(\cO(1)\) are defined as in \cite[Remark 7.10 and
      Definition 7.11]{Zav}.
  \end{enumerate}
  If for every affinoid subdomian \(U \subseteq X\), there exists a
  \(\pi_{\rvert\pi^{-1}(U)}\)-ample inverrtible sheaf on \(\pi^{-1}(U)\), we
  say that \(\pi\) is
  \textsl{locally projective}.
  \par A $\kk$-invertible sheaf for $\kk \in \{\QQ,\RR\}$
  is \textsl{$\pi$-ample} if it is a nonzero
  $\kk_{>0}$-linear combination of $\pi$-ample invertible sheaves.
\end{definition}
\begin{remark}[Comparison with \cite{Hak72}]
  Denote by \(\mathsf{Z}\) the topos associated to \(Z\), where we use the
  \'etale topology for algebraic spaces, the Zariski topology for formal
  schemes, the Euclidean topology on complex analytic spaces,
  the G-topology on Berkovich spaces \cite[\S1.3]{Ber93} and rigid
  analytic spaces \cite[Definition 9.3.1/4]{BGR84}, and the usual topology for
  adic spaces \cite[p.\ 521]{Hub94}.
  \par In \cite{Hak72}, Hakim develops the theory of relative quasi-schemes over
  topoi and proves GAGA-type theorems for them in the complex analytic setting.
  In the complex analytic case, Hakim proves that morphisms of relative quasi-schemes
  over \(\mathsf{Z}\) are locally projective if and only if their
  analytifications are locally projective morphisms of complex analytic spaces
  \cite[Chapitre VIII, Proposition 2.6\((iii)\)]{Hak72}.
  The analogue of this comparison result holds in the settings
  \((\ref{setup:formalqschemes})\), 
  \((\ref{setup:berkovichspaces})\), \((\ref{setup:rigidanalyticspaces})\), and
  \((\ref{setup:adicspaces})\) of Setup
  \ref{setup:spaces} as follows.
  \begin{enumerate}[label=\((\roman*)\)]
    \item (cf.\ \cite[Chapitre VIII, Proposition 2.1]{Hak72}) The functor
      \[
        \Hom'_{(\mathsf{Z},\cO_{\mathsf{Z}})}
        \bigl(\mathsf{t}(-),(\mathsf{X},\cO_{\mathsf{X}})\bigr)\colon
        \mathfrak{Sp}^\op/Z \longrightarrow \mathsf{Ens}
      \]
      is representable, where \(\mathsf{t}(-)\) is the ``associated ringed
      topos'' functor.
      The proof in \cite{Hak72} applies in by replacing \cite{Hou61} with
      \cite[Proposition 2.6.1]{Ber93} in setting \((\ref{setup:berkovichspaces})\),
      \cite[Theorem 2.2.5(1)]{Con00} in setting
      $(\ref{setup:rigidanalyticspaces})$, and
      \cite[Proposition 3.8]{Hub94} in settings \((\ref{setup:formalqschemes})\)
      and \((\ref{setup:adicspaces})\).
    \item By \cite[Chapitre VII, Proposition 4.4]{Hak72}, we can construct
      relative quasi-schemes \(\PP_{\mathsf{Z}}(E)\) over \(\mathsf{Z}\) for
      \(\cO_{\mathsf{Z}}\)-modules \(E\) of finite presentation.
    \item By \cite[Chapitre VIII, Proposition 2.6\((iii)\)]{Hak72} and its
      proof, we see that a morphism \(\mathsf{X} \to \mathsf{Z}\) of relative
      quasi-schemes over \(\mathsf{Z}\) is locally projective if and only if
      its relative analytification is locally projective using the
      GAGA theorems in \citeleft\citen{Ber93}\citemid Proposition
      2.6.9\citepunct \citen{Poi10}\citemid Th\'eor\`eme A.1\citeright\ for
      \((\ref{setup:berkovichspaces})\), \cite[Hilfssatz 2.10 and Theorem
      6.8]{Kop74} for \((\ref{setup:rigidanalyticspaces})\),
      and \citeleft\citen{Hub07}\citemid Corollary 6.4\citepunct
      \citen{Zav}\citemid Lemma 6.9\citeright\ for \((\ref{setup:formalqschemes})\)
      and \((\ref{setup:adicspaces})\).
  \end{enumerate}
  Thus, we see that the definitions of spaces locally projective over
  \(Z\) in \cite[Chapitre V, D\'efinition 2.2]{Hak72} and Definition
  \ref{def:locproj} are equivalent.
\end{remark}
\subsubsection{Nefness}
We can define nefness using GAGA.
\begin{definition}\label{def:nefothercats}
  Let $\pi\colon X \to Z$ be a locally projective morphism in $\Sp$.
  \begin{enumerate}[label=$(\roman*)$,ref=\roman*]
    \item A closed subspace $Y \subseteq X$ is \textsl{$\pi$-contracted} if
      $\pi(Y)$ is a zero-dimensional (closed) subspace of $Z$.
      A \textsl{$\pi$-contracted curve} is a $\pi$-contracted closed subspace
      that is integral and of dimension one.
    \item\label{def:nefothercatscontr}
      Suppose that every $\pi$-contracted curve $C \subseteq X$ is the analytification
      of a scheme $C^\al$ over $\{z\}^\al$.
      Let $D \in \Pic_\kk(X)$ for $\kk \in \{\ZZ,\QQ,\RR\}$.
      We say that $D$ is \textsl{$\pi$-nef} if, for every $\pi$-contracted curve
      $C \subseteq X$, we have $\deg_{C^\al}(D^\al) \ge 0$.
  \end{enumerate}
\end{definition}
\begin{remark}
  The condition in Definition
  \ref{def:nefothercats}$(\ref{def:nefothercatscontr})$ on
  $\pi$-contracted curves holds for
  $(\ref{setup:formalqschemes})$ and $(\ref{setup:adicspaces})$
  when $\pi$ is projective.
  In the categories $(\ref{setup:complexanalyticgerms})$,
  $(\ref{setup:berkovichspaces})$, and $(\ref{setup:rigidanalyticspaces})$,
  every $\pi$-contracted curve is the analytification of a scheme $C^\al$ over
  $\{z\}^\al$.
  See \citeleft\citen{dJ95}\citemid Proposition 3.2 and Remark 3.3\citepunct
  \citen{Duc}\citemid Th\'eor\`eme 3.7.2\citeright\ for
  $(\ref{setup:berkovichspaces})$ and 
  see \cite[Th\'eor\`eme 2]{FM86}
  for $(\ref{setup:rigidanalyticspaces})$.
  \par For $(\ref{setup:complexanalyticgerms})$, we follow the proof in
  \cite{nfdc23}.
  Let \(C\) be a reduced compact
  complex analytic space of dimension one with irreducible components \(C_i\).
  Consider \(\sL = \bigotimes_i \cO_C(x_i)\) where \(x_i \in C_i\) is a smooth
  point not lying on any other irreducible component \(C_j\) for every \(i\) and
  denote by \(\nu\colon \tilde{C} \to C\) the normalization map, which is finite
  by a theorem of Oka \cite[(8.2.3)]{GR84}.
  Then, \(\tilde{C}\) is a disjoint union of Riemann surfaces.
  For every coherent sheaf \(\sF\) on \(C\), the morphism
  \begin{align*}
    \sF \otimes_{\cO_C} \sL^{\otimes m} &\longrightarrow
    \nu_*\bigl(\nu^*(\sF \otimes_{\cO_C} \sL^{\otimes m})\bigr)
    \intertext{has kernel and cokernel with finite support.
    Thus, we see that}
    H^1\bigl(C,\sF \otimes_{\cO_C} \sL^{\otimes m}\bigr)
    &\overset{\sim}{\longrightarrow}
    H^1\bigl(\tilde{C},\nu^*(\sF \otimes_{\cO_C} \sL^{\otimes
    m})\bigr) = 0
  \end{align*}
  for all \(m \gg 0\) by Riemann--Roch (see \cite[pp.\ 213--215]{GH94}).
  Thus, \(\sL\) is positive \cite[Theorem 4.7]{Pet94V}, and \(C\) is therefore
  projective by Kodaira's embedding theorem \cite[Theorem 4.4]{Pet94V}.
\end{remark}
\begin{remark}
  We compare the definition in Definition \ref{def:nefothercats} with existing
  definitions for some of the settings of Setup \ref{setup:spaces}.
  \begin{enumerate}
    \item[{\((\ref{setup:algebraicspaces})\)}]
      Definition \ref{def:nefothercats} is equivalent to the usual definition
      for schemes and algebraic spaces (see Definition \ref{def:RelNeronSeveri})
      when one of the equivalent hypotheses in Lemma
      \ref{lem:NefAgainstNonClosedContracted} holds.
    \item[{\((\ref{setup:complexanalyticgerms})\)}]
      Definition \ref{def:nefothercats} is equivalent to \cite[Definition
      1.7]{Nak87}.
      This is because intersection numbers are computed using Euler
      characteristics and sheaf cohomology is preserved under the
      GAGA correspondence \cite[Theorem C.1.1]{AT19}.
      When \(\pi\) is proper but not projective, Definition
      \ref{def:nefothercats} would not be the correct definition because
      \(\pi^{-1}(z)\) may not contain any curves.
    \item[{\((\ref{setup:rigidanalyticspaces})\)}]
      Definition \ref{def:nefothercats} is equivalent to the definition of
      nefness defined using the notions of degree and intersection theory
      on rigid analytic spaces of dimension $\le 2$ from \citeleft\citen{Uen87}\citemid
      \S5$(a)$\citepunct \citen{Mit11}\citemid \S\S A.4--A.5\citeright.
      This is because intersection numbers are computed using Euler
      characteristics and sheaf cohomology is preserved under the
      GAGA correspondence \cite[1.\ GAGA-Satz 4.7]{Kop74}.
  \end{enumerate}
\end{remark}
\subsubsection{Bigness and pseudoeffectivity}
For $\kk \in \{\QQ,\RR\}$ and projective morphisms in $\Sp$,
we define relatively big and relatively pseudoeffective $\kk$-invertible sheaves
as follows:
\begin{definition}
  Let $\pi\colon X \to Z$ be a locally projective morphism in $\Sp$.
  Let $D \in \Pic_\kk(X)$ for $\kk\in\{\ZZ,\QQ,\RR\}$.
  We say that $D$ is \textsl{$\pi$-big} (resp.\ \textsl{$\pi$-pseudoeffective})
  if, for every affinoid subdomain $U \subseteq Z$, the restriction
  $D_{\rvert\pi^{-1}(U)}$ is the relative analytification of a $\pi$-big (resp.\
  \textsl{$\pi$-pseudoeffective}) $\kk$-invertible sheaf on $(\pi^{-1}(U))^\al$
  in the sense of Definition \ref{def:fbig}.
\end{definition}
In case $(\ref{setup:complexanalyticgerms})$, this definition is equivalent to the
definition in \cite[Definition 2.46]{Fuj} by Corollary \ref{lem:kodairachar} as
long as the source space is normal.
\subsubsection{Divisors and \texorpdfstring{$\QQ$}{Q}-factoriality}
\label{subsect:divqfac}
Let $X$ be an irreducible normal space, where for $(\ref{setup:adicspaces})$, we
moreover assume that \(X\) is locally of weakly finite type over a complete
non-trivially valued non-Archimedean field.
For the definition of irreducibility and normality, see
\cite[Lemma 1.2.1]{Con99} for $(\ref{setup:formalqschemes})$,
\cite[p.\ 8 and \S9.1.2]{GR84} for $(\ref{setup:complexanalyticgerms})$,
\citeleft\citen{Ber93}\citemid \S2.2\citepunct \citen{Duc09}\citemid p.\
1455\citeright\ for $(\ref{setup:berkovichspaces})$,
\citeleft\citen{BGR84}\citemid p.\ 300\citepunct \citen{CM98}\citemid Definition on
p.\ 12\citepunct \citen{Con99}\citemid Definition 2.2.2\citeright\ for
$(\ref{setup:rigidanalyticspaces})$, and
\cite[Definitions 2.3 and 2.11]{Man23} for $(\ref{setup:adicspaces})$.
\par Weil divisors are defined as formal sums of integral closed subspaces of
codimension $1$ that are locally finite, i.e., they become finite sums after
restriction to every affinoid subdomain.
See
\cite[Definition 4.10]{Cai} for $(\ref{setup:berkovichspaces})$ and
\cite[p.\ 8]{Bos83} for $(\ref{setup:rigidanalyticspaces})$.
\par
For Cartier divisors, we adopt the following definitions.
For $(\ref{setup:complexanalyticgerms})$, Cartier
divisors are defined as a special type of Weil divisor, following \cite[p.\ 555]{Nak87}.
For $(\ref{setup:formalqschemes})$,
$(\ref{setup:berkovichspaces})$, $(\ref{setup:rigidanalyticspaces})$, and
$(\ref{setup:adicspaces})$,
we use the definition of Cartier divisors on G-ringed spaces from
\cite[Definition 2.2]{Gub98} (see also \cite[Definition 3.6]{Cai}).
Note that for $(\ref{setup:adicspaces})$, the necessary results for the
sheaf $\sM_X$ of meromorphic functions on adic spaces proved in
\cite[\S2.14]{MR23} hold for all adic spaces that
are both taut over $\Spa(k,k^\circ)$ \cite[Definition 0.4.7$(ii)$]{Hub96} and
strongly Noetherian \cite[p.\ 524]{Hub94} (note that all (partially) proper
morphisms are taut \cite[Definition 0.4.2 and p.\ 18]{Hub96}, and hence
tautness holds for the morphisms we will consider in the sequel).
In each of these cases,
we have a cycle map
\[
  \cyc\colon \Div(X) \longrightarrow \WDiv(X).
\]
This follows from definition of Weil and Cartier divisors for
$(\ref{setup:complexanalyticgerms})$.
For $(\ref{setup:rigidanalyticspaces})$, see \cite[pp.\ 8--10]{Bos83}.
For $(\ref{setup:formalqschemes})$, $(\ref{setup:berkovichspaces})$,
$(\ref{setup:rigidanalyticspaces})$, and $(\ref{setup:adicspaces})$
see \cite[2.5]{Gub98}, which gives another version of the construction for
$(\ref{setup:rigidanalyticspaces})$.
In the other categories, the construction in \cite{Gub98}
also works since inclusions of
affinoid subdomains induce flat maps on rings of sections by
\cite[Lemma 2.4.6]{Tem12},
\cite[Proposition 2.2.4$(ii)$]{Ber90}, and
\cite[Proposition 3.3.8$(i)$]{Hub93thesis}, respectively.
\par Finally, linear equivalence is defined using the exact sequence
\eqref{eq:divtopic} (which holds for all locally ringed spaces) in cases
$(\ref{setup:formalqschemes})$, $(\ref{setup:complexanalyticgerms})$,
$(\ref{setup:berkovichspaces})$,
and $(\ref{setup:adicspaces})$.
In case $(\ref{setup:rigidanalyticspaces})$, we use \cite[Proposition
3.1]{Bos83} to pass from Cartier divisors to invertible sheaves.
\begin{remark}
  For formal schemes $(\ref{setup:formalqschemes})$,
  an example of Smith \cite[pp.\ 59--60]{Smi17} shows that
  the cycle map may not be
  injective, even if $X$ is a formal scheme that is rig-smooth over a field in
  the sense of \cite[Definition 3.1]{BLR95}.
  This will affect our definition of $\QQ$-factoriality below.
  See also \cite[Corollary on p.\ 17]{Spe73}, which describes the kernel of the
  map $\Div(X) \to \Div(\mathfrak{X})$ when
  $\mathfrak{X}$ is the formal completion of a regular scheme over a field along a
  connected closed subscheme.
\end{remark}
\par We now define $\kk$-Weil and $\kk$-Cartier divisors and the corresponding
notion of $\QQ$-factoriality.
We note that we work relatively over a base $Z$ in order to be compatible with
GAGA.
See also \cite[Definition 4.13]{Nak87} for the complex analytic case.
\begin{definition}\label{def:qfactorialoverz}
  Let $\pi\colon X \to Z$ be a locally projective
  morphism in $\Sp$, where $X$ is irreducible and
  normal.
  For $(\ref{setup:adicspaces})$, we
  moreover assume that \(X\) is locally of weakly finite type over a complete
  non-trivially valued non-Archimedean field.
  Let $\kk \in \{\QQ,\RR\}$, and
  define $\kk$-Weil divisors and $\kk$-Cartier divisors as in Definition
  \ref{def:kcartdiv}.
  We say that $X$ is \textsl{$\kk$-factorial over $Z$} if for every affinoid
  subdomain $U \subseteq Z$,
  the map 
  \[
    \cyc_\kk\colon \Div_\kk\bigl(\pi^{-1}(U)^\al\bigr) \longrightarrow
    \WDiv_\kk\bigl(\pi^{-1}(U)^\al\bigr)
  \]
  is surjective.
  A $\kk$-Weil divisor on $X$ \textsl{is $\kk$-Cartier} if it lies in the image of
  $\cyc_\kk$ over each $U$.
\end{definition}
We note that regular rigid analytic spaces over a field $k$ are $\QQ$-factorial
over $\operatorname{Sp}(k)$ (in fact, $\cyc$
is an isomorphism) by \cite[Theorem A.9]{Mit11}.
\subsubsection{Canonical divisors and singularities of pairs}
Let \(X\) be as in \S\ref{subsect:divqfac}.
We can define canonical sheaves and divisors in the same way as in Definition
\ref{def:canonicalsheaf} using the notion of dualizing
complexes from \S\ref{sect:qedualizingothers}.
We define singularities of $\QQ$-pairs as in Definition \ref{def:singpairs}, where we
note that the requisite trace morphisms $f_*\omega_Y\to \omega_X$
between canonical sheaves exist by analytifying the corresponding Grothendieck
trace morphisms on schemes.
Since we are working with $\QQ$-pairs, however, instead of working with
$\QQ$-linear equivalences as in Definition
\ref{def:canonicalsheaf}, we can work with isomorphisms of coherent sheaves as in
\cite[(2.4.1)]{Kol13}, which is easier to work with under the GAGA
correspondence.
\par In case $(\ref{setup:formalqschemes})$, the trace morphism from \cite{ATJLL99}
is the analytification of the trace morphism in the scheme case by Remark
\ref{rem:formalgaga}$(\ref{rem:formalgagasharp})$.
In cases $(\ref{setup:complexanalyticgerms})$ and $(\ref{setup:adicspaces})$,
the trace morphisms from \cite{RRV71} and \cite{CS22} are the
analytifications of the trace morphism in the scheme case by Theorems
\ref{thm:dualizingcomplexcompatcomplex}$(\ref{thm:dualizingcomplexcompatcomplexexcpullback})$
and
\ref{thm:dualizingcomplexcompatadic}$(\ref{thm:dualizingcomplexcompatadicexcpullbacktrace})$
(which relies on Theorem \ref{thm:clausenscholze}),
respectively.
In cases $(\ref{setup:berkovichspaces})$,  
$(\ref{setup:rigidanalyticspaces})$, and $(\ref{setup:adicspaces})$,
one can define discrepancies
using isomorphisms of the form in \cite[(2.4.1)]{Kol13}.
\par Moreover, in cases $(\ref{setup:complexanalyticgerms})$,
$(\ref{setup:berkovichspaces})$, 
$(\ref{setup:rigidanalyticspaces})$, and $(\ref{setup:adicspaces})$,
the canonical divisors $K_X$ and canonical
sheaves $\omega_X$ have concrete descriptions as sheaves of top differential
forms after restricting to the smooth locus of $X$ by Theorems
\ref{thm:dualizingcomplexcompatcomplex}$(\ref{thm:dualizingcomplexcompatcomplexexcpullback})$
and
\ref{thm:dualizingcomplexcompatrigid}$(\ref{thm:dualizingcomplexcompatrigidexcpullbacksmooth})$.\smallskip
\par To reduce to the scheme setting, we prove the following:
\begin{lemma}\label{lem:kltgaga}
  Let $\pi\colon X \to Z$ be a locally projective morphism in $\Sp$.
  For $(\ref{setup:adicspaces})$, we
  assume that \(X\) is locally of weakly finite type over a complete
  non-trivially valued non-Archimedean field.
  Moreover, if we are not in $\Sp_\QQ$, we additionally assume that $\dim(X) \le 3$.
  Let $\Delta$ be
  an $\RR$-Weil divisor on $X$ such that $K_X+\Delta$ is klt.
  Then, for every affinoid subdomain $U \subseteq Z$, we have that
  $(\pi^{-1}(U)^\al,\Delta_{\vert \pi^{-1}(U)^\al}^\al)$ is klt.
\end{lemma}
\begin{proof}
  Replacing $Z$ by an affinoid subdomain $U$, we may assume that $U = Z$.
  Note that $\cO_U(U)$ is excellent by
  \citeleft\citen{Fri67}\citemid Th\'eor\`eme I, 9\citepunct
  \citen{Mat73}\citemid Theorem 2.7\citepunct
  \citen{AT19}\citemid Lemma B.6.1$(i)$\citeright\ in the complex-analytic case
  and \citeleft\citen{Kie69}\citemid Theorem 3.3\citepunct \citen{Con00}\citemid
  \S1.1\citepunct \citen{Duc09}\citemid Th\'eor\`eme 2.13\citeright\ 
  in the non-Archimedean case.
  Fix a proper log resolution $f\colon Y \to X^\al$ of $(X^\al,\Delta^\al)$, which exists
  by \cite[Theorem 2.3.6 and Lemma 4.2.4]{Tem08} in equal characteristic zero, and by
  \citeleft\citen{Lip78}\citemid Theorem on p.\ 151\citepunct
  \citen{CP19}\citemid Theorem 1.1\citepunct
  \citen{CJS20}\citemid Corollary 1.5\citepunct
  \citen{BMPSTWW}\citemid Proposition 2.14\citeright\ in arbitrary
  characteristic if $\dim(X) \le 3$.
  Then, $X^\an$ is normal and
  $f^\an\colon Y^\an \to X$ is a log resolution of $(X,\Delta)$ by
  \cite[Proposition 6.3.6]{AT19} except in case $(\ref{setup:adicspaces})$,
  where we apply Lemma \ref{lem:stalksonjgxaregrings} together with
  \cite[Theorem 23.7$(ii)$]{Mat89} instead.
  The claim about klt singularities holds since (after reducing to the case of
  $\QQ$-coefficients using \cite[Proposition 2.21]{Kol13}) the expression
  \[
    K_{Y ^\an} + (f^\an)_*^{-1}\Delta \sim_\QQ f^{\an*}(K_X+\Delta) +
    \sum_{\text{$f$-exceptional $E$}} a(E,X,\Delta)E
  \]
  (or more canonically, the sheaf-theoretic version of this $\QQ$-linear
  equivalence in \cite[(2.4.2)]{Kol13})
  also holds after algebraization.
\end{proof}
\begin{remark}
  Lemma \ref{lem:kltgaga} holds for other singularities of pairs, since we
  showed that the discrepancies are well-behaved under algebraization.
\end{remark}

\section{The relative MMP with scaling\texorpdfstring{\except{toc}{\\}}{}
(Proofs of Theorems
    \texorpdfstring{\ref*{thm:introrelativemmp}}{A},
    \texorpdfstring{A\textsuperscript{\textnormal{\emph{p}}}}{A\textasciicircum p},
    and
\texorpdfstring{\ref*{thm:introfinitegen}}{B})}\label{sect:mmpforothercats}
We now prove Theorems \ref{thm:introrelativemmp},
\ref{thm:introrelativemmpcharp}, and \ref{thm:introfinitegen}.
As in \cite{Kol21qfac,VP,MZ,EH}, our convention for the relative MMP with scaling
is to contract extremal faces instead of extremal rays.
\subsection{Ample models}\label{sect:amplemodels}
To make the outputs of the relative MMP with scaling unique, we need a suitable
abstract characterization for the outputs of the relative MMP with scaling.
We do so using ample models, following \cite[\S3]{EH} (see also \cite[\S4.1]{MZ}
for earlier related results).
Compare the characterization in \cite[Lemma 2.1]{VP}.
\begin{definition}[Ample models {\citeleft\citen{BCHM10}\citemid Definition
  3.6.5\citepunct \citen{EH}\citemid Definition 3.10 and p.\ 17\citeright}]
  Let $\Sp$ be the category of quasi-excellent Noetherian algebraic spaces over
  a scheme \(S\).
  Let $\pi\colon X \to Z$ be a projective morphism of spaces in $\Sp$ such that
  \(X\) is normal.
  Let $D$ be an $\RR$-invertible sheaf on $X$.
  Let \(\varphi\colon X \dashrightarrow X'\) be a rational map to another normal
  space in \(\Sp\) and consider a resolution of indeterminacy
  \[
    X \overset{p}{\longleftarrow} \tilde{X} \overset{q}{\longrightarrow} X'
  \]
  for \(\varphi\), that
  is, a normal space \(\tilde{X}\) fitting into the commutative diagram
  \begin{equation}\label{eq:resindet}
    \begin{tikzcd}[column sep=scriptsize]
      & \tilde{X}\arrow{dl}[swap]{p}\arrow{dr}{q}\\
      X \arrow[dashed]{rr}{\varphi}\arrow{dr}[swap]{\pi} & & X'\arrow{dl}\\
      & Z
    \end{tikzcd}
  \end{equation}
  where \(p,q\) are proper and \(q^\#\colon \cO_{X'} \to q_*\cO_{\tilde{X}}\) is
  an isomorphism.
  \begin{enumerate}[label=\((\roman*)\)]
    \item Suppose that \(\varphi\) is birational.
      If \(\varphi_*D\) is \(\RR\)-Cartier, then we say that \(\varphi\)
      \textsl{has the Cartier pushforward \(\varphi_*D\) of \(D\)}.
      Note that if \(D\) is big (resp.\ pseudo-effective, numerically trivial)
      over \(Z\), then so is \(\varphi_*D\).
    \item We say that \(\varphi\) is an \textsl{ample model of \(D\)} if there
      exist an \(\RR\)-invertible sheaf \(D'\) on \(X'\) that is ample over
      \(Z\) and a commutative diagram of the form \eqref{eq:resindet} for which
      there exists an \(\RR\)-linear equivalence over \(Z\)
      \begin{equation}\label{eq:amplemodelrlineq}
        p^*D \sim_{\RR,Z} q^*D' + E
      \end{equation}
      for some effective \(\RR\)-Cartier divisor \(E\) on \(\tilde{X}\) such
      that \(B \ge E\) for all \(B \in \lvert p^*D \mathbin{/} Z \rvert_\RR\).
      Here, \(\RR\)-linear equivalence over \(Z\) 
      and relative linear systems are defined as
      in \cite[Definitions 3.1.1(3) and 3.5.1]{BCHM10}.
  \end{enumerate}
\end{definition}
\begin{remark}\label{rem:amplemodelbasechange}
  By flat base change \cite[\href{https://stacks.math.columbia.edu/tag/073K}{Tag
  073K}]{stacks-project}, ample models are compatible with base change along
  flat morphisms with geometrically normal fibers on the base.
  The condition on fibers is used to ensure the base changes of \(X\) and \(X'\)
  are still normal \cite[Corollaire 6.5.4\((ii)\) and Proposition 6.8.2]{EGAIV2}.
\end{remark}
We show that ample models are essentially unique and therefore can
be constructed locally on the base.
\begin{lemma}[cf.\ {\citeleft\citen{BCHM10}\citemid Lemma 3.6.6\citepunct
  \citen{MZ}\citemid Lemma 4.4\citepunct
  \citen{EH}\citemid Lemma 3.11 and p.\ 17\citeright}]\label{lem:eh311}
  Let $\Sp$ be the category of quasi-excellent locally Noetherian algebraic
  spaces over a scheme \(S\).
  Let $\pi\colon X \to Z$ be a projective morphism of spaces in $\Sp$ such that
  \(X\) is normal.
  Let $D$ be an $\RR$-invertible sheaf on $X$.
  Then, ample models of \(D\) (if they exist) are unique up to compatible isomorphisms.
  More precisely, for every pair of ample models \(\varphi_i\colon X
  \dashrightarrow X_i\) of \(D\) for \(i \in \{1,2\}\),
  we can assign an isomorphism
  \(\sigma_{21}\colon X_1 \overset{\sim}{\to} X_2\) over \(Z\) fitting into the
  commutative diagram
  \[
    \begin{tikzcd}[column sep=small]
      & X \arrow[dashed]{dl}[swap]{\varphi_1}\arrow[dashed]{dr}{\varphi_2}\\
      X_1 \arrow{rr}{\sigma_{21}}[swap]{\sim} & & X_2
    \end{tikzcd}
  \]
  over \(Z\) such that \(\sigma_{21}\) is the identity when \(\varphi_1 =
  \varphi_2\) and such that for every third ample model \(\varphi_3\colon X
  \dashrightarrow X_3\), the diagram
  \[
    \begin{tikzcd}[sep=large]
      & X \arrow[dashed]{dl}[swap]{\varphi_1}
      \arrow[dashed]{d}[description]{\varphi_2}
      \arrow[dashed]{dr}{\varphi_3}\\
      X_1 \arrow{r}{\sigma_{21}}[swap]{\sim}
      \arrow[bend right=35]{rr}{\sigma_{31}}[swap]{\sim} & X_2
      \arrow{r}{\sigma_{32}}[swap]{\sim} & X_3
    \end{tikzcd}
  \]
  over \(Z\) commutes.
\end{lemma}
\begin{proof}
  Let \(\varphi_i\colon X \dashrightarrow X_i\) be a pair of ample models of
  \(D\) for \(i \in \{1,2\}\).
  Consider a common resolution of indeterminacies
  \[
    X \overset{p}{\longleftarrow} \tilde{X} \overset{q_i}{\longrightarrow}
    X_i
  \]
  for \(\varphi_1\) and \(\varphi_2\).
  Write
  \[
    p^*D \sim_{\RR,Z} q_i^*D'_i + E_i
  \]
  for each \(i\).
  We then have \(E_1 = E_2\) by the same proof as in \cite[Lemma
  3.6.6(1)]{BCHM10}.
  \par To construct \(\sigma_{21}\), we proceed as follows.
  Since \(\tilde{X}\) is a resolution of indeterminacies,
  the normalization \(\bar{X}_{12}\) of its image
  in \(X_1 \times_Z X_2\) only depends on \(\varphi_1,\varphi_2\) and
  does not depend on the choice of \(\tilde{X}\).
  Moreover, the morphism \((q_1,q_2)\colon \tilde{X} \to X_1 \times_Z X_2\)
  induced by the universal property of fiber products
  factors uniquely through \(\bar{X}_{12}\) since \(\tilde{X}\) is normal
  \cite[\href{https://stacks.math.columbia.edu/tag/0823}{Tag
  0823}]{stacks-project}.
  We therefore obtain the commutative diagram
  \[
    \begin{tikzcd}[column sep=2.5em,row sep=2.2em]
      & & & & X_1 \arrow{dr}\\
      X
      & \lar[swap]{p} \tilde{X} \rar{q_{12}}
      \arrow[bend left=15,start anchor={[yshift=2pt]}]{urrr}{q_1}
      \arrow[bend right=15,start anchor={[yshift=-2pt]}]{drrr}[swap]{q_2}
      & \bar{X}_{12} \rar
      \arrow[bend left=10]{urr}[swap]{r_1}
      \arrow[bend right=10]{drr}{r_2}
      & X_1 \times_Z X_2
      \arrow[start anchor=north east]{ur}
      \arrow[start anchor=south east]{dr}
      & & Z\\
      & & & & X_2 \arrow{ur}
    \end{tikzcd}
  \]
  where the right square is Cartesian and the compositions \(X \dashrightarrow
  X_i\) are the \(\varphi_i\).
  Let \(A \coloneqq r_1^*D_1' + r_2^*D_2'\).
  Then, we have
  \begin{equation}\label{eq:qstarais2qidi}
    q_{12}^*A \sim_{\RR,Z} q_1^*D_1' + q_2^*D_2' \sim_{\RR,Z} 2q_i^*D'_i.
  \end{equation}
  We will show that \(r_1\) and \(r_2\) are isomorphisms.
  By construction, the morphisms \(\cO_{X_i} \to r_{i*}\cO_{\bar{X}_{12}}\) are
  isomorphisms for \(i \in \{1,2\}\).
  By way of contradiction, suppose that \(r_1\) is not an isomorphism.
  By Zariski's Main Theorem
  \cite[\href{https://stacks.math.columbia.edu/tag/082K}{Tag
  082K}]{stacks-project}, there exists an integral one-dimensional proper
  subspace \(C \subseteq \bar{X}_{12}\) such that \(r_1(C)\) is a closed point.
  We claim that \(r_2(C)\) is a curve.
  If not, then setting \(x_1 = r_1(C)\), \(x_2 = r_2(C)\), and \(z\) the common
  image of \(x_1\) and \(x_2\) in \(Z\), then
  \[
    C \subseteq \Spec\bigl(\kappa(x_1)\otimes_{\kappa(z)}\kappa(x_2)\bigr),
  \]
  a contradiction.
  On the other hand, consider an integral one-dimensional proper subspace
  \(\tilde{C} \subseteq \tilde{X}\) whose image under \(q\) is \(C\).
  Then, we have \(2q_i^*D_i' \cdot \tilde{C} = 0\), which implies
  \(r_2(C)\) is a point.
  This contradicts the fact that \(r_2(C)\) is a curve.
  By the same argument, \(r_2\) is an isomorphism.
  We can therefore define
  \[
    \sigma_{21} \coloneqq r_2 \circ r_1^{-1}.
  \]
  As noted above, \(\bar{X}_{12}\) only depends on \(\varphi_1,\varphi_2\).
  By construction of the commutative diagram above, we have \(\sigma_{21} \circ
  \varphi_1 = \varphi_2\).
  By construction, we also see that \(\sigma_{21}\) is the identity if
  \(\varphi_1 = \varphi_2\).
  \par Finally, consider a third ample model \(\varphi_3\colon X \dashrightarrow
  X_3\).
  Consider a common resolution of indeterminacies
  \[
    X \overset{p}{\longleftarrow} \tilde{X} \overset{q_i}{\longrightarrow}
    X_i
  \]
  for \(\varphi_1\), \(\varphi_2\), and \(\varphi_3\).
  We then have the commutative diagram
  \[
    \begin{tikzcd}[column sep=2.5em,row sep=2.25em]
      & & & \bar{X}_{12} \arrow{dr}[sloped,below]{\sim}
      & X_1 \arrow{dr} \dar{\sigma_{21}}[below,sloped]{\sim}\\
      X & \lar[swap]{p} \tilde{X} \rar{q_{13}} \arrow[bend left=15,start
      anchor={[yshift=2pt]}]{urr}{q_{12}}
      \arrow[bend right=15,start
      anchor={[yshift=-2pt]}]{drr}[swap]{q_{23}}
      & \bar{X}_{13} \rar
      & \bar{X}_{12} \times_Z \bar{X}_{23} \rar \uar \dar
      & X_2 \arrow{r} \dar{\sigma_{32}}[below,sloped]{\sim} &[0.75em] Z\mathrlap{.}\\
      & & & \bar{X}_{23}\arrow{ur}[sloped]{\sim}
      & X_3\arrow{ur}
      \arrow[from=2-3,to=1-5,crossing over,bend left=3,
      "\sim"{sloped,below,pos=0.6},start anchor={[yshift=2pt]},end anchor={[yshift=-4pt]}]
      \arrow[from=2-3,to=3-5,crossing over,bend right=3,
      "\sim"{sloped,pos=0.6},start anchor={[yshift=-2pt]},end anchor={[yshift=4pt]}]
      \arrow[from=1-4,to=1-5,"\sim"]
      \arrow[from=3-4,to=3-5,"\sim"']
    \end{tikzcd}
  \]
  In this commutative diagram, the morphism \(\bar{X}_{13} \to \bar{X}_{12}
  \times_Z \bar{X}_{23}\) is obtained by constructing the morphism
  \(\bar{X}_{13} \to X_1 \times_Z X_3\) as before and then composing with the
  inverse of the isomorphism
  \[
    \bar{X}_{12} \times_Z \bar{X}_{23} \overset{\sim}{\longrightarrow} X_1
    \times_Z X_3.
  \]
  The argument in the previous paragraph shows the compositions
  \(\bar{X}_{13} \to X_i\) are the isomorphisms \(r_i\) for \(i \in \{1,3\}\).
  By the commutativity of the diagram, we see that \(\sigma_{31} =
  \sigma_{32} \circ \sigma_{21}\).
\end{proof}
\subsection{Outputs of the relative MMP with scaling}
We define the outputs of the relative MMP with scaling following \cite[\S3]{EH}.
Compare the definitions in \citeleft\citen{Kol21qfac}\citemid Definition
1\citepunct \citen{VP}\citemid \S2\citeright.
\begin{citeddef}[{\cite[Definition 3.13]{EH}}]\label{def:rthoutputmmp}
  Let $\Sp$ be the category of quasi-excellent
  algebraic spaces over a scheme \(S\).
  Let $\pi\colon X \to Z$ be a locally projective morphism of spaces
  in $\Sp$ such that \(X\) is normal.
  Let $D$ and $H$ be $\RR$-invertible sheaves on $X$ such that $H$ is big over
  $Z$.
  Assume that the \(\pi\)-pseudo-effective threshold of \(H_{\rvert U}\)
  relative to \(D_{\rvert U}\), denoted by \(\mu\), is constant for any \'etale
  morphism \(U \to Z\) from a space \(U\) such that \(X \times_Z U\) is
  nonempty.
  For \(t > \mu\), a birational map
  \[
    \varphi_r\colon X \dashrightarrow X_r
  \]
  is called the \textsl{\(r\)-th output of the \(\pi\)-relative
  \(D\)-MMP with scaling of \(H\)} if it is an ample model of
  \[
    D+(r-\varepsilon)H
  \]
  for sufficiently small \(\varepsilon > 0\)
  \'etale-locally on \(Z\), that is, there exists an \'etale covering \(\{Z_i
  \to Z\}_i\) and positive numbers \(\{a_i\}_i\) such that the base change of
  \(\varphi_t\) to \(Z_i\) is an ample model of
  \[
    D_{\rvert Y_i} + (r-\varepsilon)H_{\rvert Y_i}
  \]
  for all \(0 < \varepsilon < a_i\).
  We say that the \textsl{\(\pi\)-relative
  \(D\)-MMP with scaling of \(H\) exists} if there exists a \(r\)-th output of a
  \(D\)-MMP with scaling of \(H\) over \(Y\) for every \(r > \mu\).
\end{citeddef}
\begin{remark}\label{rem:mmpbasechange}
  By Remark \ref{rem:amplemodelbasechange}, the \(r\)-th output of the
  \(\pi\)-relative \(D\)-MMP with scaling of \(H\) is compatible with base
  change along flat morphisms with geometrically normal fibers on the base.
\end{remark}
By applying Lemma \ref{lem:eh311} on an \'etale cover of the base \(Z\), we
obtain the following uniqueness result for steps of the relative MMP with
scaling:
\begin{lemma}[{cf.\ \citeleft\citen{MZ}\citemid Lemma 4.4\citepunct
  \citen{EH}\citemid Lemma 3.14\citeright}]\label{lem:eh314}
  Let $\Sp$ be the category of quasi-excellent locally Noetherian algebraic
  spaces over a scheme \(S\).
  Let $\pi\colon X \to Z$ be a locally projective morphism of 
  spaces in $\Sp$ such that \(Z\) is normal.
  Let $D$ and $H$ be $\RR$-invertible sheaves on $X$ such that $H$ is big over
  $Z$.
  Assume that the \(\pi\)-pseudo-effective threshold of \(H_{\rvert U}\)
  relative to \(D_{\rvert U}\), denoted by \(\mu\), is constant for any \'etale
  morphism \(U \to Z\) from a space \(U\) such that \(X \times_Z U\) is
  nonempty.
  For every \(r > \mu\), the \(r\)-th outputs of the \(\pi\)-relative
  \(D\)-MMP with scaling of \(H\) (if they exist) are unique up to compatible
  isomorphisms.
  More precisely, for every pair of \(r\)-th outputs of the \(\pi\)-relative
  \(D\)-MMP with scaling of \(H\) denoted \(\varphi_i\colon X
  \dashrightarrow X_i\) of \(D\) for \(i \in \{1,2\}\),
  we can assign an isomorphism
  \(\sigma_{21}\colon X_1 \overset{\sim}{\to} X_2\) over \(Z\) fitting into the
  commutative diagram
  \[
    \begin{tikzcd}[column sep=small]
      & X \arrow[dashed]{dl}[swap]{\varphi_1}\arrow[dashed]{dr}{\varphi_2}\\
      X_1 \arrow{rr}{\sigma_{21}}[swap]{\sim} & & X_2
    \end{tikzcd}
  \]
  over \(Z\) such that \(\sigma_{21}\) is the identity when \(\varphi_1 =
  \varphi_2\) and such that for every third \(r\)-th output of the
  \(\pi\)-relative \(D\)-MMP with scaling of \(H\) denoted \(\varphi_3\colon X
  \dashrightarrow X_3\), the diagram
  \[
    \begin{tikzcd}[sep=large]
      & X \arrow[dashed]{dl}[swap]{\varphi_1}
      \arrow[dashed]{d}[description]{\varphi_2}
      \arrow[dashed]{dr}{\varphi_3}\\
      X_1 \arrow{r}{\sigma_{21}}[swap]{\sim}
      \arrow[bend right=35]{rr}{\sigma_{31}}[swap]{\sim} & X_2
      \arrow{r}{\sigma_{32}}[swap]{\sim} & X_3
    \end{tikzcd}
  \]
  over \(Z\) commutes.
\end{lemma}
\subsection{Gluing}\label{sect:gluing}
We can now glue steps of the relative $D$-MMP together.
See \cite[Corollary 2.3]{VP} for the corresponding gluing statement for steps
of the MMP for algebraic spaces.
\begin{theorem}\label{thm:gluemmp}
  Let $\Sp$ be as in $(\ref{setup:formalqschemes})$,
  $(\ref{setup:complexanalyticgerms})$,
  $(\ref{setup:berkovichspaces})$, $(\ref{setup:rigidanalyticspaces})$, or
  $(\ref{setup:adicspaces})$ of
  Setup \ref{setup:spaces}.
  For $(\ref{setup:adicspaces})$, we moreover
  assume that \(X\) is locally of weakly finite type over a complete
  non-trivially valued non-Archimedean field.
  Suppose the hypotheses in Definition \ref{def:rthoutputmmp} are satisfied.
  Let $Z = \bigcup_a V_a$ be an affinoid covering, and define
  $X_a = X \times_Z V_a$, $\pi_a = \pi_{\vert X_a}$, $D_a = D_{\vert X_a}$, and $H_a
  = H_{\vert X_a}$.
  Suppose that for each $a$ we know the existence of the $r$-th output of the
  $\pi_a$-relative $D_a$-MMP with scaling of $H_a$.
  Then, the $r$-th output of the $\pi$-relative $D$-MMP with scaling of $H$
  exists.
\end{theorem}
\begin{proof}
  It suffices to show that for every affinoid subdomain $W \subseteq V_a \cap
  V_b$, the restrictions of the $r$-th output of the $\pi_a$-relative $D$-MMP
  with scaling coincides with that of the $\pi_b$-relative $D$-MMP with scaling
  up to compatible isomorphisms.
  \par Let $A_a = \cO_{V_a}(V_a)$, $A_b = \cO_{V_b}(V_b)$, and $B = \cO_W(W)$.
  It suffices to show that the corresponding steps of the relative $D$-MMP with
  scaling over the schemes $\Spec(A_a)$ and $\Spec(A_b)$ under the GAGA
  correspondences in \cite[\S6.3]{AT19} and \cite[\S6]{Hub07}
  coincide with that on $\Spec(B)$, since
  all objects involved are projective over $Z$.
  By Remark \ref{rem:mmpbasechange} and Lemma \ref{lem:eh314},
  it suffices to show that the maps $\Spec(B) \to
  \Spec(A_a)$ and $\Spec(B) \to \Spec(A_b)$ are flat with geometrically normal 
  fibers.
  These maps are
  flat as shown in
  \cite[Lemma 6.2.8]{AT19} and \cite[Proposition 3.3.8$(i)$]{Hub93thesis}.
  It therefore suffices to show that if $W \subseteq V$ is an inclusion of
  affinoid subdomains in $Z$, then the map $\Spec(\cO_W(W)) \to
  \Spec(\cO_{V}(V))$ is has geometrically normal fibers.
  In fact, these morphisms have geometrically regular fibers by \cite[Lemma
  6.2.8]{AT19} and Lemma \ref{lem:stalksonjgxaregrings}.
\end{proof}
\subsection{Proof of Theorems \texorpdfstring{\ref{thm:introrelativemmp}}{A} and
\texorpdfstring{\ref{thm:introrelativemmpcharp}}{A\textasciicircum p}}
We can now prove Theorems \ref{thm:introrelativemmp} and
\ref{thm:introrelativemmpcharp}.
\begin{proof}[Proof of Theorems \ref{thm:introrelativemmp} and
  \ref{thm:introrelativemmpcharp}]
  We first replace $\pi\colon X \to Z$ by its Stein factorization to assume that
  $Z$ is normal.
  Note that Stein factorizations exist for algebraic spaces by
  \cite[\href{https://stacks.math.columbia.edu/tag/0A1B}{Tag
  0A1B}]{stacks-project}, for semianalytic germs of complex analytic spaces
  by applying \cite[10.6.1]{GR84} to a representative for $\pi$, for
  Berkovich spaces by \cite[Proposition 3.3.7]{Ber90}, for rigid analytic
  spaces by \cite[Proposition 9.6.3/5]{BGR84}, and for adic spaces locally
  of weakly finite type over a field by \cite[Theorem 3.9]{Man23}.
  For Theorem \ref{thm:introrelativemmp} (resp.\
  \ref{thm:introrelativemmpcharp}),
  case $(\ref{setup:introalgebraicspaces})$ (resp.\ case
  $(\ref{setup:introalgebraicspaces})$ for schemes quasi-projective over an
  excellent domain admitting a dualizing complex) was shown in Theorems
  \ref{rem:MMP} and \ref{thm:cl1365} (resp.\ in \citeleft\citen{Tan18}\citemid
  Theorem 4.5\citepunct \citen{BMPSTWW}\citemid Theorem G\citeright).
  It therefore suffices to show Theorem \ref{thm:introrelativemmp} (resp.\
  \ref{thm:introrelativemmpcharp}) in the other
  cases.
  Let $A$ be a $\QQ$-invertible sheaf as in the statement of Theorem
  \ref{thm:introrelativemmp} (resp.\
  \ref{thm:introrelativemmpcharp}), and let $Z = \bigcup_a V_a$ be an
  affinoid covering.
  Note that each $\cO_{V_a}(V_a)$ is excellent either by assumption or by
  \citeleft\citen{Fri67}\citemid Th\'eor\`eme I, 9\citepunct
  \citen{Mat73}\citemid Theorem 2.7\citepunct
  \citen{AT19}\citemid Lemma B.6.1$(i)$\citeright\ in the complex-analytic case
  and \citeleft\citen{Kie69}\citemid Theorem 3.3\citepunct \citen{Con00}\citemid
  \S1.1\citepunct \citen{Duc09}\citemid Th\'eor\`eme 2.13\citeright\ 
  in the non-Archimedean case.
  \par For Theorem \ref{thm:introrelativemmp}, we can use 
  Lemma \ref{lem:AmpleIsGoodScaling} to show that after possibly shrinking the
  $V_a$, we have divisors $A_{a} \in \lvert A_{\vert
  \pi^{-1}(V_a)} \rvert$ such that $(X_a,\Delta_{\vert X_a}+A_a)$ is klt.
  The $A_a$ are therefore good scaling divisors in the sense of Definition
  \ref{def:GoodScalingDivisors}.
  We want to apply Theorem \ref{thm:gluemmp} for $D = K_X+\Delta$, $H = A$,
  and the affinoid covering $Z = \bigcup_a V_a$.
  By GAGA and Theorems \ref{rem:MMP} and \ref{thm:cl1365}, each step of the
  $\pi_a$-relative $(K_X+\Delta)_a$-MMP with scaling of $A_a$ exists (note that
  because of the difference in conventions, the outputs of the relative MMP in
  Definition \ref{def:rthoutputmmp} are compositions of many steps in Theorem
  \ref{rem:MMP}).
  In order to apply these theorems, we note that
  the positivity conditions on $A$ and $K_X+\Delta+A$ are
  preserved under algebraization, as well as the klt condition on
  $(X_a,\Delta_{\vert X_a}+A_a)$ (see Lemma \ref{lem:kltgaga}).
  By Theorem \ref{thm:gluemmp}, we can glue these relative MMP steps to obtain
  global MMP steps over $Z$.
  By construction, we see that this relative MMP terminates over each $V_a$ in
  the way described (see Corollaries \ref{cor:cl1366}
  and \ref{cor:cl1367}).
  \par For Theorem \ref{thm:introrelativemmpcharp}, we apply GAGA together with
  \citeleft\citen{Tan18}\citemid Theorem 4.5\citepunct \citen{BMPSTWW}\citemid
  Remark 2.41\citeright) when $\dim(X)
  = 2$, and when $\dim(X) = 3$, we apply GAGA together with
  \citeleft\citen{BMPSTWW}\citemid Theorem G\citeright\ in case
  $(\ref{thm:introrelativemmpcharptanbmpstww})$,
  \citeleft\citen{Kaw94}\citepunct \citen{Kaw99}\citemid \S3\citepunct
  \citen{TY}\citemid Theorem 5.10\citeright\ in case
  $(\ref{thm:introrelativemmpcharpty})$, and \cite[Theorem 9]{Kol21qfac} in case
  $(\ref{thm:introrelativemmpcharpkol})$
  over each
  $V_a$ to say that each step of the $\pi_a$-relative
  $(K_X+\Delta)_a$-MMP with scaling of $A_a$ exists (with the difference in
  conventions as in the previous paragraph) and terminates.
  Note that in case $(\ref{thm:introrelativemmpcharptanbmpstww})$, the
  assumptions on the residue characteristics of local rings of $Z$ imply that
  the residue characteristics $\Spec(\cO_Z(Z))$ do not lie in $\{2,3,5\}$ using
  the bijections used in the second paragraph of the proof of Theorem
  \ref{thm:gluemmp}.
  The rest of the argument now proceeds as in the previous paragraph.
\end{proof}
\subsection{Proof of Theorem \texorpdfstring{\ref{thm:introfinitegen}}{B}}
Finally, we prove Theorem \ref{thm:introfinitegen}.
\begin{proof}[Proof of Theorem \ref{thm:introfinitegen}]
  As before, we have already shown case $(\ref{setup:introalgebraicspaces})$ in
  Theorem \ref{thm:finitegenerationalgspaces}.
  It therefore suffices to show Theorem \ref{thm:introfinitegen} in the other
  cases.
  \par The positivity conditions on the $A_i$ and $c_iK_X+\Delta_i$ are
  preserved under algebraization over every affinoid subdomain $U \subseteq Z$,
  as well as the klt condition on $(X,\Delta_i)$ (see Lemma \ref{lem:kltgaga}).
  Since GAGA preserves cohomology groups \citeleft\citen{EGAIII1}\citemid
  Proposition 5.1.2\citepunct \citen{AT19}\citemid Theorem C.1.1\citepunct
  \citen{Poi10}\citemid Th\'eor\`eme A.1$(i)$\citepunct \citen{Kop74}\citemid
  Folgerung 6.6\citepunct \citen{Hub07}\citemid Corollary 6.4\citeright\ 
  (see also \citeleft\citen{Con06}\citemid Example
  3.2.6\citepunct \citen{Hal}\citemid Example 9.4\citeright), we can apply
  Theorem \ref{thm:finitegenerationalgspaces} over $\Spec(\cO_U(U))$ to deduce
  Theorem \ref{thm:introfinitegen}.
  Note that each $\cO_{U}(U)$ is excellent as shown at the end of
  the first paragraph in the
  proof of Theorems \ref{thm:introrelativemmp} and
  \ref{thm:introrelativemmpcharp}.
\end{proof}

\begingroup
\makeatletter
\renewcommand{\@secnumfont}{\bfseries}
\part{Additional results in other categories}\label{part:additionalresults}
\makeatother
\endgroup
In this part, we apply our gluing method
to prove a
version of the relative minimal model program with scaling that does not require
shrinking the base space $Z$.
We use as input the existing results in \cite{Fuj,DHP} instead of our own
results on the relative minimal model program that we showed earlier on in the
paper.
We then reformulate
the Basepoint-free theorem (Theorem \ref{thm:bpf}) and
the Contraction theorem (Theorem \ref{thm:contraction})
so that they apply to contexts where dualizing
complexes may not exist.\bigskip
\section{The relative MMP with scaling\texorpdfstring{\except{toc}{\\}}{} for
complex analytic spaces without shrinking}
In this section, we use our result on gluing
to show
that the relative minimal model program with scaling
established by Fujino \cite[Theorem 1.7]{Fuj} and Das--Hacon--P\u{a}un
\cite[Theorem 1.4]{DHP} can be performed without shrinking the base space at
each step, as long as the scaling divisor $C$ has stronger positivity properties
to enable gluing.
We note, however, that if $Z$ does not admit a finite cover by affinoid
subdomains, then there may not be a sequence of flips and divisorial
contractions that is globally finite that yields the pair $(X_m,\Delta_m)$.
\par The following result uses the results on the relative minimal model program
proved in \cite{Fuj} instead of our results for schemes proved in Parts
\ref{part:prelim}--\ref{part:relativemmpforschemes}.
\begin{theorem}\label{thm:complexmmpwithoutshrinking}
  Let $\pi\colon X \to Z$ be a projective surjective morphism of semianalytic
  germs of complex analytic spaces, where $X$ and $Z$ are integral and $X$ is
  normal.
  Suppose $X$ is $\QQ$-factorial over every affinoid subdomain in $Z$,
  and let $\Delta$ be an $\pi$-big
  $\RR$-divisor such that $(X,\Delta)$ is klt.
  Let $C$ be an effective $\RR$-divisor on $X$ such that $(X,\Delta+C)$ is
  klt and $K_X+\Delta+rC$ is $\pi$-ample for some $r \in \RR_{>0}$.
  Then, the relative minimal model program with scaling of $C$ over $Z$ exists.
  Moreover, we have the following properties.
  \begin{enumerate}[label=$(\arabic*)$]
    \item The relative minimal model program with scaling of $C$ over $Z$ terminates
      after a finite sequence of flips and divisorial contractions over every
      affinoid subdomain $U \subseteq Z$ starting from
      $(\pi^{-1}(U),\Delta_{\vert \pi^{-1}(U)})$.
    \item The relative minimal model program with scaling of $C$ over $Z$ yields
      a commutative diagram
      \[
        \begin{tikzcd}[column sep=tiny]
          (X,\Delta) \arrow[dashed]{rr}\arrow{dr}[swap]{\pi} & & (X_m,\Delta_m)
          \arrow{dl}{\pi_m}\\
          & Z
        \end{tikzcd}
      \]
      where $X \dashrightarrow X_m$ is a meromorphic map in the sense of
      Remmert.
      Over every affinoid subdomain $U \subseteq Z$, the morphism
      $\pi_m^{-1}(U) \to U$ is either a minimal
      model over $U$ (when $(K_X+\Delta)_{\vert \pi^{-1}(U)}$ is
      $\pi_{\vert\pi^{-1}(U)}$-pseudoeffective) or a Mori fibration
      over $U$ (when $(K_X+\Delta)_{\vert \pi^{-1}(U)}$ is not
      $\pi_{\vert\pi^{-1}(U)}$-pseudoeffective).
  \end{enumerate}
\end{theorem}
\begin{proof}
  By applying Stein factorization \cite[10.6.1]{GR84} to a representative for
  $\pi$, we may assume that $Z$ is normal.
  Over each affinoid subdomain $U \subseteq Z$, the necessary steps of the
  relative minimal model program with scaling in $C$ exist 
  and terminate in the situations listed above by applying \cite[Theorem
  1.7]{Fuj} to a representative $\mathcal{U}$ of the germ $U = (\mathcal{U},U)$.
  Note that the shrinking present in \cite{Fuj} amounts to replacing
  $\mathcal{U}$ by a possibly smaller complex analytic space that still contains
  $U$ (Noetherianity of $\Gamma(U,\cO_{\mathcal{U}})$ holds by
  \cite[Th\'eor\`eme I, 9]{Fri67}).
  Finally, applying Theorem \ref{thm:gluemmp} to the algebraizations of these
  steps over an affinoid covering of $Z$, we see that there exists a
  partially defined map $(X,\Delta) \dashrightarrow (X_m,\Delta_m)$ that is
  meromorphic in the sense of Remmert \cite[Def.\ 15]{Rem57}
  (see also \cite[Definition 1.7]{Pet94}).
  Over each affinoid subdomain $U \subseteq Z$,
  this meromorphic map restricts to a finite sequence of
  flips and divisorial contractions, and the morphism $\pi_m^{-1}(U) \to U$
  is a minimal model or a Mori fibration.
\end{proof}

\section{Basepoint-free and Contraction theorems\texorpdfstring{\except{toc}{\\}}{}
without dualizing complexes}\label{sect:noomega}
In this section, we formulate versions of the Basepoint-free theorem (Theorem
\ref{thm:bpf}) and the Contraction theorem (Theorem \ref{thm:contraction}) that
do not assume that $X$ and $Z$ have dualizing complexes.
Instead, we put conditions on singularities of pairs and the positivity of Cartier
divisors after base change to completions at points in $Z$.
\par Below, the assumption that the formal fibers of $Z$ are geometrically normal
imply that $X \otimes_{\cO_{Z,z}} \hat{\cO}_{Z,z}$ is normal for every $z \in Z$
by \cite[Corollaire 6.5.4 and Proposition 6.8.2]{EGAIV2}.
The rings $\hat{\cO}_{Z,z}$ admit dualizing complexes by \cite[(4) on p.\
299]{Har66}.
\begin{theorem}[Basepoint-free theorem; cf.\ Theorem \ref{thm:bpf}]\label{thm:bpfnoomega}
  Let $\pi\colon X \to Z$ be a proper surjective morphism of integral
  Noetherian schemes of equal characteristic
  zero over a scheme $S$.
  Suppose that $X$ is normal and that the formal fibers of $Z$ are geometrically
  normal.
  \par Let $\Delta$ be an effective $\RR$-Weil divisor on $X$.
  For each $z \in Z$, consider the Cartesian diagram
  \[
    \begin{tikzcd}
      \hat{X}_z \rar\dar[swap]{\hat{\pi}_z} & X \dar{\pi}\\
      \hat{Z}_z \rar & Z
    \end{tikzcd}
  \]
  where $\hat{Z}_z \coloneqq \Spec(\hat{\cO}_{Z,z})$,
  denote by $\hat{\Delta}_z$ the pullback of $\Delta$ to $\hat{X}_z$,
  and choose a canonical divisor $K_{\hat{X}_z}$
  that is compatible with a dualizing complex on
  $\hat{Z}_z$.
  Suppose that for every closed point
  $z \in Z$, the $\RR$-Weil divisor $K_{\hat{X}_z}+\hat{\Delta}_z$
  is $\RR$-Cartier. 
  \par Let $H \in \Pic(X)$ be $\pi$-nef.
  Suppose the pair $(\hat{X}_z,\hat{\Delta}_z)$
  is dlt (or more generally, weakly log terminal) (resp.\ klt) for every closed point
  $z\in Z$
  and that there exists some $a_z \in \ZZ_{>0}$ such that $a_z\hat{H}_z -
  (K_{\hat{X}_z} + \hat{\Delta}_z)$
  is $\hat{\pi}_z$-ample (resp.\ $\hat{\pi}_z$-big and $\hat{\pi}_z$-nef) for
  every closed point $z \in Z$, where $\hat{H}_z$ is the pullback of $H$ to $\hat{X}_z$.
  Then, there exists $m_0 \in \ZZ_{>0}$ such that $mH$ is $\pi$-generated
  for all $m \ge m_0$.
\end{theorem}
\begin{proof}
  After replacing $Z$ by the image of $X$, we may assume that $\pi$ is
  surjective.
  Note the assumptions on the formal fibers of $Z$ are not affected by
  \cite[Th\'eor\`eme 7.4.4]{EGAIV2}.
  We make the following claim:
  \begin{claim}\label{claim:pnbpfnoomega}
    For every prime number $p$, the Cartier divisor $p^nH$ is $\pi$-generated
    for $n \gg 0$.
  \end{claim}
  Showing Claim \ref{claim:pnbpfnoomega} would imply the theorem, since then the
  monoid of natural numbers $m \in \NN$ such that $mH$ is $\pi$-generated would
  contain all sufficiently large integers by \cite[Theorem 1.0.1]{RA05}.
  Since for all $n,n' \in \NN$ such that $n' \ge n$, we have the inclusion 
  \begin{align*}
    \MoveEqLeft[5]
    \Supp\Bigl(\cok\bigl(\pi^*\pi_*\cO_X(p^{n'}H)
    \longrightarrow \cO_X(p^{n'}H)\bigr)\Bigr)\\
    &\subseteq \Supp\Bigl(\cok\bigl(\pi^*\pi_*\cO_X(p^{n}H)
    \longrightarrow \cO_X(p^{n}H)\bigr)\Bigr),
  \end{align*}
  the Noetherianity of $Z$ implies there exist some $n_0$ such that these
  inclusions stabilize for all $n' \ge n \ge n_0$.
  \par We claim that
  \[
    \Supp\Bigl(\cok\bigl(\pi^*\pi_*\cO_X(p^{n_0}H)
    \longrightarrow \cO_X(p^{n_0}H)\bigr)\Bigr) = \emptyset.
  \]
  Suppose not, in which case there exists a closed point $z \in Z$ in this
  support by \cite[(2.1.2)]{EGAInew}.
  We can then apply Theorem \ref{thm:bpf} to the base change
  $(\hat{X}_z,\hat{\Delta}_z)$ to see there exists $n \in \NN$ such that
  \[
    z \notin \Supp\Bigl(\cok\bigl(\pi^*\pi_*\cO_X(p^{n}H)
    \longrightarrow \cO_X(p^{n}H)\bigr)\Bigr).
  \]
  This contradicts the assumption that the chain of inclusions of supports
  stabilized for all $n \ge n_0$.
\end{proof}
For the Contraction theorem, we have the following:
\begin{theorem}[Contraction theorem; cf.\ Theorem \ref{thm:contraction}]
  \label{thm:contractionnoomega}
  Let $\pi\colon X \to Z$ be a projective surjective morphism of integral
   Noetherian schemes of equal characteristic zero over
  a scheme $S$.
  Suppose that $X$ is normal and that the formal fibers of $Z$ are geometrically
  normal.
  \par Let $\Delta$ be an effective $\RR$-Weil divisor on $X$.
  For each $z \in Z$, consider the Cartesian diagram
  \[
    \begin{tikzcd}
      \hat{X}_z \rar\dar[swap]{\hat{\pi}_z} & X \dar{\pi}\\
      \hat{Z}_z \rar & Z
    \end{tikzcd}
  \]
  where $\hat{Z}_z \coloneqq \Spec(\hat{\cO}_{Z,z})$,
  denote by $\hat{\Delta}_z$ the pullback of $\Delta$ to $\hat{X}_z$,
  and choose a canonical divisor $K_{\hat{X}_z}$
  that is compatible with a dualizing complex on
  $\hat{Z}_z$.
  \par Suppose that for every closed point
  $z \in Z$, the $\RR$-Weil divisor $K_{\hat{X}_z} + \hat{\Delta}_z$
  is $\RR$-Cartier and that $(\hat{X}_z,\hat{\Delta}_z)$ is
  dlt (or more generally, weakly log terminal).
  Let $H \in \Pic(X)$ be $\pi$-nef such that for every $z \in Z$, we have
  \[
    F \coloneqq \bigl(\hat{H}^\perp_z \cap \NEbar(\hat{X}_z/\hat{Z}_z)\bigr) - \{0\}
    \subseteq \Set[\Big]{
      \beta \in N_1(\hat{X}_z/\hat{Z}_z) \given
      \bigl(K_{\hat{X}_z} + \hat{\Delta}_z
      \bigr) \cdot \beta < 0
    }
  \]
  where $\hat{H}_z$ is the pullback of $H$ to $\hat{X}_z$ and
  $\hat{H}^\perp_z \coloneqq \Set{\beta \in N_1(\hat{X}_z/\hat{Z}_z) \given
  (\hat{H}_z \cdot \beta) = 0}$.
  Then, the morphism $\varphi$ in the Stein factorization
  \[
    X \overset{\varphi}{\longrightarrow} Y
    \longrightarrow \Proj_Z\Biggl( \bigoplus_{m=0}^\infty
    \pi_*\cO_X(mH)\biggr)
  \]
  is a projective and surjective morphism to
  an integral normal quasi-excellent Noetherian scheme \(Y\)
  projective over $Z$.
  The morphism \(\varphi\) satisfies the following properties:
  \begin{enumerate}[label=$(\roman*)$,ref=\roman*]
    \item\label{thm:contractionnoomegai}
      For every integral one-dimensional subscheme $C \subseteq X$ such that
      $\pi(C)$ is a point, the image $\varphi(C)$ is a point if
      and only if $(H \cdot C) = 0$, i.e., if and only if $[C] \in F$.
    \item\label{thm:contractionnoomegaii}
      $\cO_Y \to \varphi_*\cO_X$ is an isomorphism.
  \end{enumerate}
  \par Moreover, consider a projective surjective morphism \(\varphi'\colon X \to Y'\)
  fitting into the commutative diagram
  \[
    \begin{tikzcd}[column sep=tiny]
      X \arrow{rr}{\varphi'}\arrow{dr}[swap]{\pi} & & Y' \arrow{dl}{\sigma'}\\
      & Z
    \end{tikzcd}
  \]
  where \(Y'\) is an integral normal quasi-excellent Noetherian scheme
  projective over \(Z\).
  Suppose that \(\varphi'\) satisfies properties $(\ref{thm:contractioni})$ and
  $(\ref{thm:contractionii})$.
  Then, \(\varphi'\) is isomorphic to \(\varphi\) over \(Z\), and \(\varphi'\) satisfies the
  following additional property:
  \begin{enumerate}[resume,label=$(\roman*)$,ref=\roman*]
    \item\label{thm:contractionnoomegaiii}
      $H = \varphi^{\prime*}A$ for some $\sigma'$-ample $A \in \Pic(Y)$.
  \end{enumerate}
\end{theorem}
\begin{proof}
  By the fact that
  relative ampleness can be detected over closed points \cite[Propositin
  2.7]{Kee03}, we can apply Kleiman's criterion (Proposition
  \ref{lem:AmpleIsPositiveOnNE}) to say for each $z\in Z$, there exists a $a \in
  \NN$ such that $a_z\hat{H}_z - (K_{\hat{X}} + \hat{\Delta}_z )$
  is $\hat{\pi}_z$-ample.
  Here, we use the fact that relative ampleness can be detected over closed
  points to say that the curves contracted by the morphism $\hat{\pi}_z$
  map to curves contracted by the morphism
  $\pi$ (the completion map $\cO_{Z,z} \to \hat{\cO}_{Z,z}$ induces an
  isomorphism of residue fields).
  Thus, by the Basepoint-free theorem (Theorem \ref{thm:bpfnoomega}), we know
  that $mH$ is $\pi$-generated for $m \gg 0$.
  The rest of the proof now proceeds as in the proof of Theorem
  \ref{thm:contraction}.
\end{proof}

\addtocontents{toc}{\protect\medskip}
\bookmarksetup{startatroot}

\end{document}